\documentclass[letterpaper,leqno,english]{amsart}
\usepackage[T1]{fontenc}
\usepackage[latin9]{inputenc}
\usepackage{color}
\usepackage{babel}
\usepackage{array}
\usepackage{float}
\usepackage{booktabs}
\usepackage{mathtools}
\usepackage{enumitem}
\usepackage{multirow}
\usepackage{amstext}
\usepackage{amsthm}
\usepackage{amssymb}
\usepackage{stmaryrd}
\usepackage{setspace}
\usepackage{wasysym}
\usepackage[all]{xy}
\usepackage{microtype}
\usepackage[unicode=true,
 bookmarks=true,bookmarksnumbered=true,bookmarksopen=false,
 breaklinks=false,pdfborder={0 0 0},pdfborderstyle={},backref=false,colorlinks=true]
 {hyperref}
\hypersetup{
 citecolor=blue, linkcolor=red}

\makeatletter

\let\SF@@footnote\footnote
\def\footnote{\ifx\protect\@typeset@protect
    \expandafter\SF@@footnote
  \else
    \expandafter\SF@gobble@opt
  \fi
}
\expandafter\def\csname SF@gobble@opt \endcsname{\@ifnextchar[
  \SF@gobble@twobracket
  \@gobble
}
\edef\SF@gobble@opt{\noexpand\protect
  \expandafter\noexpand\csname SF@gobble@opt \endcsname}
\def\SF@gobble@twobracket[#1]#2{}
\providecommand{\tabularnewline}{\\}

\numberwithin{equation}{section}
\numberwithin{figure}{section}
\numberwithin{table}{section}
\theoremstyle{plain}
\newtheorem{thm}{\protect\theoremname}[section]
\theoremstyle{definition}
\newtheorem{defn}[thm]{\protect\definitionname}
\theoremstyle{definition}
\newtheorem{example}[thm]{\protect\examplename}
\theoremstyle{remark}
\newtheorem{rem}[thm]{\protect\remarkname}
\theoremstyle{plain}
\newtheorem{prop}[thm]{\protect\propositionname}
\theoremstyle{plain}
\newtheorem{lem}[thm]{\protect\lemmaname}
\theoremstyle{plain}
\newtheorem{cor}[thm]{\protect\corollaryname}
\theoremstyle{remark}
\newtheorem{notation}[thm]{\protect\notationname}
\theoremstyle{remark}
\newtheorem{note}[thm]{\protect\notename}
\theoremstyle{plain}
\newtheorem{lyxalgorithm}[thm]{\protect\algorithmname}

\usepackage{tikz,fancyhdr,enumitem}
\usetikzlibrary{scopes,calc,intersections,through,backgrounds,arrows,decorations.pathmorphing,positioning,fit,petri,decorations.markings,matrix,angles}
\usepackage{hyperref,bookmark,dsfont,bbm,nicefrac}
\usepackage[all]{xy}
\hypersetup{linktocpage}
\usepackage{relsize} 
\usepackage{graphicx} 

\newtheorem{mainthm}{Theorem}

\newtheorem{propdef}[thm]{Proposition-Definition}

\newcommand{\osubset}{\mathbin{\mathpalette\make@circleda\subset}}
\newcommand{\make@circleda}[2]{%
    \mathrel{\ooalign{\hss$\m@th#1\smallbigcirc{#1}$\hss\cr
\kern0.3ex\raise0.154ex\hbox{\scalebox{0.75}
{$\m@th#1#2$}}}}
}

\newcommand{\osubeq}{\mathbin{\mathpalette\make@circledc\subseteq}}
\newcommand{\make@circledc}[2]{%
    \mathrel{\ooalign{\hss$\m@th#1\smallbigcirk{#1}$\hss\cr
\kern0.265ex\raise0.07ex\hbox{\scalebox{0.85}
{$\m@th#1#2$}}}}
}

\newcommand{\oleq}{\mathbin{\mathpalette\make@circledd\leqslant}}
\newcommand{\make@circledd}[2]{%
    \mathrel{\ooalign{\hss$\m@th#1\smallbigcirk{#1}$\hss\cr
\kern0ex\raise0.025ex\hbox{\scalebox{0.94}
{$\m@th#1#2$}}}}
}

\newcommand{\oless}{\mathbin{\mathpalette\make@circlede\le}}
\newcommand{\make@circlede}[2]{%
    \mathrel{\ooalign{\hss$\m@th#1\smallbigcirk{#1}$\hss\cr
\kern0.275ex\raise0.07ex\hbox{\scalebox{0.818}
{$\m@th#1#2$}}}}
}

\newcommand{\ocap}{\mathbin{\mathpalette\make@circledb\cap}}
\newcommand{\make@circledb}[2]{%
    \mathrel{\ooalign{\hss$\m@th#1\smallbigcirc{#1}$\hss\cr
\kern0.32ex\raise-0.04ex\hbox{\scalebox{0.75}
{$\m@th#1#2$}}}}
}

\newcommand{\smallbigcirk}[1]{%
  \vcenter{\hbox{\scalebox{0.9}{$\m@th#1\bigcirc$}}}%
}

\newcommand{\smallbigcirc}[1]{%
  \vcenter{\hbox{\scalebox{0.77778}{$\m@th#1\bigcirc$}}}%
}

\renewcommand\ell{l}

\subjclass[2020]{Primary 05B35; Secondary 52B40, 52C35, 52B20, 52C22, 06A12, 14N20, 14M15, 14M25, 14T15}

\makeatother

\providecommand{\algorithmname}{Algorithm}
\providecommand{\corollaryname}{Corollary}
\providecommand{\definitionname}{Definition}
\providecommand{\examplename}{Example}
\providecommand{\lemmaname}{Lemma}
\providecommand{\notationname}{Notation}
\providecommand{\notename}{Note}
\providecommand{\propositionname}{Proposition}
\providecommand{\remarkname}{Remark}
\providecommand{\theoremname}{Theorem}

\begin{document}
\title[Birational Matroid Geometry]{Birational Geometry of Matroids and Abstract Hyperplane Arrangements}
\author{Jaeho Shin}
\address{Department of Mathematical Sciences, Seoul National University, Gwanak-ro
1, Gwanak-gu, Seoul 08826, South Korea}
\email{j.shin@snu.ac.kr}
\begin{abstract}
A matroid is a machine capturing linearity of mathematical objects
and producing combinatorial structures. Matroid structure arises everywhere
since linearity is a ubiquitous concept. One natural way to obtain
matroids is by considering hyperplane arrangements, which give rise
to convex polytopes called matroid polytopes. Much research has been
conducted on these three areas: matroids, matroid polytopes, and hyperplane
arrangements. However, substantial gaps in our knowledge remain, and
the correspondence diagram between those areas needs to be more extensive.
For instance, currently, there is no matroid counterpart of a matroid
subdivision, and only some matroid subdivisions are associated with
stable hyperplane arrangements. Moreover, we need a deeper understanding
of the face structure of a matroid polytope and how to glue or subdivide
base polytopes; the latter requires overcoming Mnëv's universality
theorem. Another interesting question is whether the birational geometry
of hyperplane arrangements can be implemented over matroids. In this
paper, we develop a theory that integrates the three areas into a
trinity relationship and provide solutions to the aforementioned questions
while answering as many as possible.
\end{abstract}

\maketitle
\begin{quote}
\begin{center}
``Rem tene, verba sequentur.'' \textemdash{} Cato the Elder
\par\end{center}

\end{quote}
\tableofcontents{}

\section*{Introduction}

A matroid is a machine capturing linearity of mathematical objects
and producing combinatorial structures. Matroid structure arises everywhere
since linearity is a ubiquitous concept. One of the most natural appearances
of matroids is in the Grassmannians: every generic point of them is
identified with an arrangement of hyperplanes whose linear functionals
form a matroid, while the Chow polytope of the point is the base polytope
of the matroid. In general, matroids are in a one-to-one correspondence
with matroid polytopes by which we mean base, independence, spanning-set,
and any other convex polytopes associated with a matroid. A matroid
subdivision, which is a subdivision of a base polytope into smaller
base polytopes, corresponds to a stable hyperplane arrangement \cite{Kapranov,HKT,Ale08},
which is a certain kind of limit of hyperplane arrangements.

Although much research has been conducted on these three areas: matroids,
matroid polytopes, and hyperplane arrangements, there remain substantial
gaps in our knowledge, and the correspondence diagram between those
areas needs to be more extensive. For instance, currently, there is
no matroid counterpart of a matroid subdivision, and only some matroid
subdivisions are associated with stable hyperplane arrangements. Moreover,
we need a deeper understanding of the face structure of a matroid
polytope and how to glue or subdivide base polytopes. The latter requires
overcoming \emph{Mnëv's universality theorem} in any way, which warns
that it can be arbitrarily complicated. Another interesting question
is whether the birational geometry of hyperplane arrangements can
be implemented over matroids.

In this paper, we develop a theory that integrates the three areas
into a trinity relationship, which we call the \emph{birational matroid
geometry}, and provide solutions to the previously mentioned questions
while answering as many as possible. We adopt a direct approach rather
than relying on traditional methods, such as the use of ``dual''
structures including cones, fans, vertex figures, secondary polytopes,
etc.\footnote{One can read off cohomology-related information from our theory and
translate it back and forth.} Our direct approach places emphasis on the flats of a matroid and
rank functions, while giving less attention to circuits, in contrast
to existing references such as \cite{Kapranov}, \cite[Chapter 7]{GKZ94},
and others. This approach is what sets our theory apart, making it
concise and elegant. The following are two main results presented
at the end of the final section, where Theorem \ref{thm:mainB} is
derived from Theorem \ref{thm:mainA} as an application to the realizable
case, which implies that Mnëv's universality theorem has been successfully
overcome.

\begin{mainthm}[Lemma \ref{lem:wt<adm} and Theorems \ref{thm:complete-ext} and \ref{thm:Counterexample}] \label{thm:mainA}
For $\scalebox{0.95}{\ensuremath{n\le9}}$, every rank-$\scalebox{0.95}{\ensuremath{3}}$
and $\scalebox{0.95}{\ensuremath{(n-1)}}$-dimensional weighted or
admissible matroid tiling in $\scalebox{0.95}{\ensuremath{\mathbb{R}^{n}}}$
extends to a matroid subdivision of hypersimplex $\scalebox{0.95}{\ensuremath{\Delta(3,n)}}$.
The bound $\scalebox{0.95}{\ensuremath{n=9}}$ is tight.\end{mainthm}

\begin{mainthm}[Theorem \ref{thm:Surjectivity}] \label{thm:mainB}
For $\scalebox{0.95}{\ensuremath{n\le9}}$ and any two weight vectors
$\scalebox{0.95}{\ensuremath{\mathbf{w},\mathbf{v}\in\mathbb{R}^{n}}}$
with $\scalebox{0.95}{\ensuremath{\mathbf{w}>\mathbf{v}}}$, the reduction
morphism $\scalebox{0.95}{\ensuremath{\rho_{\mathbf{w},\mathbf{v}}:\overline{\mathrm{M}}_{\mathbf{w}}(3,n)\rightarrow\overline{\mathrm{M}}_{\mathbf{v}}(3,n)}}$
between the moduli  spaces of weighted stable $n$-line arrangements
is surjective. The bound $\scalebox{0.95}{\ensuremath{n=9}}$ is sharp.
\end{mainthm}

We begin our journey by observing that an \emph{abstract hyperplane
arrangement} can be constructed from a matroid, making a usual hyperplane
arrangement a realization of an abstract one. The details are described
in Subsections \ref{subsec:HA-1} and \ref{subsec:HA}. It follows
that a matroid is realizable if and only if its abstract hyperplane
arrangement is. We study the\emph{ degeneration of abstract hyperplanes}
as a continuation of the earlier discussion of \emph{matroid degeneration}
by Gelfand, Goresky, MacPherson, and Serganova \cite[Section 5.2]{GGMS}.
It turns out that the hyperplane degeneration is equivalent to cutting
off corners of a base polytope to obtain a smaller base polytope,
as explained in Subsection \ref{subsec:Degeneration}.\medskip{}

\noindent \textbf{Polyhedral geometry of matroids.} To analyze the
face of a matroid polytope, we first view a matroid as a discrete
metric space equipped with a grading. Using this simple structure,
we provide an alternative proof of a theorem of Borovik, Gelfand,
and White \cite[Theorem 1.12.8]{Coxeter}, which classifies the 2-dimensional
faces of a base polytope through representation theory; see Proposition
\ref{prop:2-dim faces}. The 2-dimensional faces of an independence
polytope are classified in the same manner.

Also, regarding matroids as combinatorial structures on sets, we define
the \emph{pullback} and \emph{pushforward} of a matroid (Subsection
\ref{subsec:PBPF}). By choosing a suitable function between sets,
we demonstrate that cutting any base polytope with a hyperplane of
the form $\scalebox{0.95}{\ensuremath{\left\{ \sum_{i\in A}x_{i}=1\right\} }}$
always produces a matroid subdivision (Lemma \ref{lem:Cut-w-1}).
However, this is not the case for hyperplanes of the form $\scalebox{0.95}{\ensuremath{\left\{ \sum_{i\in A}x_{i}=\rho\right\} }}$
for an integer $\scalebox{0.95}{\ensuremath{\rho>1}}$, except when
cutting a hypersimplex (the base polytope of a uniform matroid), which
does result in a matroid subdivision (Lemma \ref{lem:Cutting-1}). 

For a matroid $\scalebox{0.95}{\ensuremath{M}}$ and a subset $\scalebox{0.95}{\ensuremath{A\subseteq E\left(M\right)}}$,
we define the \emph{attaché operation on }$\scalebox{0.95}{\ensuremath{M}}$\emph{
by $\scalebox{0.95}{\ensuremath{A}}$} as obtaining a matroid $\scalebox{0.95}{\ensuremath{M\left(A\right):=M|A\oplus M/A}}$
(Subsection \ref{subsec:Attache-1}). For the base or independence
polytope $\scalebox{0.95}{\ensuremath{\mathrm{P}}}$, we define the
\emph{attaché operation on }$\scalebox{0.95}{\ensuremath{\mathrm{P}}}$\emph{
by $\scalebox{0.95}{\ensuremath{A}}$} as obtaining the intersection
$\scalebox{0.95}{\ensuremath{\mathrm{P}\cap\left\{ \sum_{i\in A}x_{i}=r_{M}(A)\right\} }}$,
a nonempty face of $\scalebox{0.95}{\ensuremath{\mathrm{P}}}$, where
$\scalebox{0.95}{\ensuremath{r_{M}}}$ denotes the rank function of
$\scalebox{0.95}{\ensuremath{M}}$. It turns out that attaché operations
can be recursively applied and every nonempty face of $\scalebox{0.95}{\ensuremath{\mathrm{P}}}$
is obtained this way. Hence, we call the matroid of a face of $\scalebox{0.95}{\ensuremath{\mathrm{P}}}$
a \emph{face matroid} of $\scalebox{0.95}{\ensuremath{\mathrm{P}}}$.
This correspondence from base polytopes (or independence polytopes
and their faces) to matroids forms a \emph{functor}, with attaché
operations being morphisms in both categories (Subsection \ref{subsec:Attache-2}).\medskip{}

\noindent \textbf{Lattice geometry of matroids.} For multiple faces
of a base or independence polytope $\scalebox{0.95}{\ensuremath{\mathrm{P}}}$
with a nonempty intersection, the intersection is obtained by applying
to $\scalebox{0.95}{\ensuremath{\mathrm{P}}}$ the associated attaché
operations at once, regardless of the order in which they are applied,
although attaché operations do not commute in general (Theorem \ref{thm:face-matroids}).
This makes the collection of face matroids of $\scalebox{0.95}{\ensuremath{\mathrm{P}}}$
a lattice, which can be managed through the use of matroidal expressions.
A matroidal expression is a direct sum of minor expressions, where
a minor expression is a sequence of restrictions and contractions
(as described in Subsections \ref{subsec:Minor-expressions} and \ref{subsec:Matroidal-expressions}).
Theorem \ref{thm:face-matroids} further extends to Theorem \ref{thm:face-intersection}.
In Lemma \ref{lem:Ridges of BP}, we investigate the codimension-$\scalebox{0.95}{\ensuremath{2}}$
face of a base polytope.

A base polytope is a representative matroid polytope. We pay particular
attention to loopless base polytopes and their loopless faces, where
``loopless'' for polytopes means not contained within a coordinate
hyperplane. A base polytope is loopless if and only if its matroid
is loopless. For a rank-$\scalebox{0.95}{\ensuremath{k}}$ loopless
matroid, it is connected, or equivalently, its base polytope is full-dimensional
if its abstract hyperplane arrangement has $\scalebox{0.95}{\ensuremath{k+1}}$
hyperplanes in general position, and the converse holds only when
$\scalebox{0.95}{\ensuremath{k=2,3}}$; see Subsection \ref{subsec:gen-pos}.

For a loopless base or independence polytope $\scalebox{0.95}{\ensuremath{\mathrm{P}}}$,
any loopless face of it can be obtained by applying to $\scalebox{0.95}{\ensuremath{\mathrm{P}}}$
a sequence of attaché operations determined by a \emph{flag} (as explained
in Proposition \ref{prop:flaces-flags}). The matroids of loopless
faces of $\scalebox{0.95}{\ensuremath{\mathrm{P}}}$ form a lattice,
which is coatomistic,  and we  call the collection of its coatoms
a \emph{puzzle-piece} of \emph{dimension} $k-\kappa$ where $k$ is
the rank and $\kappa$ is the number of connected components of the
given matroid. A puzzle-piece can be obtained from an abstract hyperplane
arrangement by lattice operations of \emph{matroidal blowups} and
\emph{collapsings}, a realization of which, if it exists, is obtained
from a realization of the abstract hyperplane arrangement by the corresponding
algebro-geometric operations, see Subsections \ref{subsec:matoidal-blowup},
\ref{subsec:Collapsing-operation}, \ref{subsec:minimal-model}, and
\ref{subsec:Grassmann}.

\smallskip{}

\noindent \textbf{Grassmann Geometry. }Consider the Grassmannian $\scalebox{0.95}{\ensuremath{\mathrm{G}\left(k,n\right)}}$
with an action of the algebraic torus $\scalebox{0.95}{\ensuremath{H=\mathbb{G}_{m}^{n}}}$.
Gelfand and Serganova \cite[Theorems 1.1 and 2.2]{GS87} established
a one-to-one correspondence between the torus orbits in $\scalebox{0.95}{\ensuremath{\overline{H\cdot\left[X\right]}}}$
for any $\scalebox{0.95}{\ensuremath{\left[X\right]\in\mathrm{G}\left(k,n\right)}}$
and the faces of the base polytope of the matroid $\scalebox{0.95}{\ensuremath{M}}$
determined by $\scalebox{0.95}{\ensuremath{X}}$, which preserves
the dimension. For $\scalebox{0.95}{\ensuremath{\left[X\right]\in\mathrm{G}\left(k,n\right)}}$
with $\scalebox{0.95}{\ensuremath{X\subset\mathbb{P}^{n-1}}}$ that
is not contained in the coordinate hyperplanes, the correspondence
implies a one-to-one correspondence between those torus orbits that
are not contained in the coordinate hyperplanes and the face matroids
of the puzzle-piece of $\scalebox{0.95}{\ensuremath{M}}$, preserving
the codimension (Proposition \ref{prop:GS-thm-2}). Moreover, gluing
torus orbits is equivalent to gluing base polytopes. This has a matroidal
shadow: gluing puzzle-pieces.

For birational algebraic geometers, we remark that a realization of
a puzzle-piece is the log canonical model of a realization of its
associated abstract hyperplane arrangement. Thus, our lattice geometry
of matroids tells us that the log canonical model of a hyperplane
arrangement exists and how to obtain it.\smallskip{}

\noindent \textbf{Extension theory for matroids.} A \emph{matroid
tiling} is the collection of inclusionwise maximal members of a polyhedral
complex, consisting of base polytopes whose union is connected in
codimension $\scalebox{0.95}{\ensuremath{1}}$. Our primary focus
is on extending a matroid tiling by gluing base polytopes, despite
the warning of Mnëv's universality theorem. One of the major challenges
is that extending a matroid tiling is a global process and can have
a huge computational complexity. So, we introduce the notion of a
\emph{matroid semitiling}, an intermediate collection of base polytopes
connected in codimension $\scalebox{0.95}{\ensuremath{1}}$ that arises
during tiling extension. According to our theorem, extending a matroid
tiling to a matroid semitiling in a ``locally convex'' manner results
in the semitiling being a tiling (Theorem \ref{thm:locally-convex}),
and the support of the extension is a base polytope, making it a matroid
subdivision (Lemma \ref{lem:locally-convex} and Corollary \ref{cor:locally-convex}).

Moreover, base polytopes glue through faces not contained in the boundary
of hypersimplex, allowing for a switch from extending a tiling to
a puzzle. Here, a semipuzzle is the collection of the puzzle-pieces
corresponding to the base polytopes of a semitiling, and a puzzle
is a semipuzzle whose semitiling is a tiling. Switching provides the
advantage of dimension drop by $\scalebox{0.95}{\ensuremath{n-k}}$
where $\scalebox{0.95}{\ensuremath{n}}$ is the size of the ground
set and $\scalebox{0.95}{\ensuremath{k}}$ is the rank. Consequently,
the extension process converts to a local operation of less dimension.

Using our extension theory, we describe how to extend an arbitrary
rank-$\scalebox{0.95}{\ensuremath{2}}$ tiling in Subsection \ref{subsec:Rank-2-tilings}.
We further discuss the rank-$\scalebox{0.95}{\ensuremath{3}}$ case
in Section \ref{sec:Tiling Extension}, the final section. When $\scalebox{0.95}{\ensuremath{k=3}}$,
puzzles become $\scalebox{0.95}{\ensuremath{2}}$-dimensional objects,
enabling manual computations and turning the task of extending a matroid
tiling into a fun jigsaw puzzle game. However, not all rank-$\scalebox{0.95}{\ensuremath{3}}$
tilings are extendable, and we  limit our discussion to a specific
class of rank-$\scalebox{0.95}{\ensuremath{3}}$ tilings, which we
call \emph{admissible}. We show that every admissible tiling without
any \emph{forked} facets extends to a complete tiling, a matroid subdivision
of a hypersimplex. The details can be found in Algorithm \ref{alg:extend_tilings}
and Theorem \ref{thm:unforked}.

Furthermore, we prove that every rank-$\scalebox{0.95}{\ensuremath{3}}$
weighted or admissible tiling in $\scalebox{0.95}{\ensuremath{\Delta\left(3,n\right)}}$
with $\scalebox{0.95}{\ensuremath{n\le9}}$ extends to a complete
tiling (Lemma \ref{lem:wt<adm} and Theorem \ref{thm:complete-ext}),
while for $\scalebox{0.95}{\ensuremath{n=10}}$, we present a counterexample
(Theorem \ref{thm:Counterexample}). Combining these two produces
Theorem \ref{thm:mainA}, which oversolves an open problem from an
algebro-geometric context \cite[Question 2.10]{Ale08}. We derive
Theorem \ref{thm:realizable-complete} from Theorem \ref{thm:mainA}
as an application to the realizable case by addressing issues related
to gluing automorphisms over the given field. This leads to a thorough
answer to the question of whether reduction morphisms are surjective,
which is Theorem \ref{thm:mainB} (Theorem \ref{thm:Surjectivity}).\medskip{}

Weighted objects are discussed at the end of Sections \ref{sec:Lattice Geometry of Matroids},
\ref{sec:Matroid Tilings}, and \ref{sec:Puzzle}, as well as in Section
\ref{sec:Tiling Extension}, with the latter mostly dedicated to the
topic. All computations are performed manually with pen and paper
using our theory, without resorting to computers. See, for example,
Remarks \ref{rem:diff of isos}, \ref{rem:why-pp}, \ref{rem:general position},
\ref{rem:degenerate-flat}, and \ref{rem:puzzle}, Examples \ref{exa:GA for D(2,4)},
\ref{exa:GA for D(3,5)}, \ref{exa:GA for D(3,6)}, and \ref{exa:GA for D(3,6)-1},
Algorithm \ref{alg:extend_tilings}, and Theorems \ref{thm:Uniform-finiteness},
\ref{thm:complete-ext}, \ref{thm:Counterexample}, \ref{thm:realizable-complete},
and \ref{thm:Surjectivity}. \medskip{}

To present a comprehensive theory, we look into the basics of matroid
theory from a geometric viewpoint and introduce new notions in Section
\ref{sec:Essentials of Matroids}; asterisks {*} mark the subsections
containing new notions so readers familiar with the usual matroid
theory can focus on them. We use \textbf{boldface} to indicate defined
terms and \emph{italics} for emphasis throughout the paper.

\subsection*{Acknowledgements}

The author began this research to provide a complete answer to an
algebro-geometric question, as a follow-up to his Ph.D. dissertation
years ago (Theorem \ref{thm:Surjectivity} now fully answers the question),
which has grown into the theory of ``birational matroid geometry.''
He is truly grateful to Thomas Zaslavsky for cordial advice and invaluable
conversations. He would also like to thank Alexandre Borovik, Sergey
Fomin, June Huh, James Oxley, Bernd Sturmfels, and many others not
listed here for their attention.

\section{\label{sec:Essentials of Matroids}Essentials of Matroid Theory and
New Notions}

The purpose of this section is to provide background knowledge to
readers for understanding the matroid theoretic content of this paper.

\subsection{Basic terms}

Let $\scalebox{0.95}{\ensuremath{S}}$ be any finite set. Its power
set $\scalebox{0.95}{\ensuremath{2^{S}:=\left\{ A:A\subseteq S\right\} }}$
has the natural inclusion relation $\subseteq$. We consider an extra
structure on $\scalebox{0.95}{\ensuremath{2^{S}}}$:
\begin{defn}[Matroid rank axioms]
 A\textbf{ submodular rank function} $r$ on $\scalebox{0.95}{\ensuremath{2^{S}}}$
is a $\scalebox{0.95}{\ensuremath{\mathbb{Z}_{\ge0}}}$-valued function
with the following properties:
\begin{enumerate}[label=(R\arabic*)]
\item \label{enu:(R1)}$\scalebox{0.95}{\ensuremath{0\le r\left(A\right)\le\left|A\right|\quad\text{for }A\in2^{S}}}$.
\item \label{enu:(R2)}$\scalebox{0.95}{\ensuremath{r\left(A\right)\le r\left(B\right)\quad\text{for }A,B\in2^{S}\text{ with }A\subseteq B}}$.
\item \label{enu:(R3)}(Submodularity)\enspace{}$\scalebox{0.95}{\ensuremath{r\left(A\cup B\right)+r\left(A\cap B\right)\le r\left(A\right)+r\left(B\right)\quad\text{for }A,B\in2^{S}}}$.
\end{enumerate}
\end{defn}

\begin{defn}
A (finite) \textbf{matroid} $\scalebox{0.95}{\ensuremath{M}}$ on
$\scalebox{0.95}{\ensuremath{S}}$ is a poset $\scalebox{0.95}{\ensuremath{2^{S}}}$
equipped with a submodular rank function $r$ on $\scalebox{0.95}{\ensuremath{2^{S}}}$,
which we denote by a pair $\scalebox{0.95}{\ensuremath{\left(S,r\right)}}$.
We call $\scalebox{0.95}{\ensuremath{S}}$ the \textbf{ground set}
of $\scalebox{0.95}{\ensuremath{M}}$ and denote it by $\scalebox{0.95}{\ensuremath{E(M)}}$.
The function $\scalebox{0.95}{\ensuremath{r=r_{M}}}$ is called the
\textbf{rank function }of $\scalebox{0.95}{\ensuremath{M}}$. The
number $\scalebox{0.95}{\ensuremath{r(S)}}$ is called the \textbf{rank
}of $\scalebox{0.95}{\ensuremath{M}}$ and also denoted by $\scalebox{0.95}{\ensuremath{r(M)}}$.
\end{defn}

We refer to the empty set $\scalebox{0.95}{\ensuremath{\emptyset}}$
as the \textbf{empty matroid} with empty ground set and rank function
$\scalebox{0.95}{\ensuremath{\emptyset\rightarrow0}}$. Two matroids
$\scalebox{0.95}{\ensuremath{\bigl(\tilde{S},\tilde{r}\bigr)}}$ and
$\scalebox{0.95}{\ensuremath{\bigl(S,r\bigr)}}$ are called \textbf{isomorphic}
if there is a bijection $\scalebox{0.95}{\ensuremath{f:\tilde{S}\rightarrow S}}$
with $\tilde{r}=r\circ f$. A \textbf{$\scalebox{0.95}{\ensuremath{\left(k,S\right)}}$-matroid}
refers to a rank-$k$ matroid on $\scalebox{0.95}{\ensuremath{S}}$.
We often omit ``$\scalebox{0.95}{\ensuremath{\left(k,S\right)}}$-''
if the context is clear. A \textbf{$\scalebox{0.95}{\ensuremath{\left(k,n\right)}}$-matroid}
refers to a $\scalebox{0.95}{\ensuremath{\left(k,\left[n\right]\right)}}$-matroid
for $\scalebox{0.95}{\ensuremath{\left[n\right]:=\left\{ 1,\dots,n\right\} }}$.

A rank-$k$ matroid $\scalebox{0.95}{\ensuremath{M}}$ on $\scalebox{0.95}{\ensuremath{S}}$
is called \textbf{realizable}\footnote{It is also called \emph{representable}, but we reserve the word \emph{represent}
for later use in Subsections \ref{subsec:Minor-expressions} and \ref{subsec:Matroidal-expressions},
 and we  prefer \emph{realizable}.}\textbf{ over a field} or simply \textbf{realizable} if there is a
pair of a $k$-dimensional vector space $\scalebox{0.95}{\ensuremath{V}}$
over a field and a spanning set $\scalebox{0.95}{\ensuremath{\left\{ \mathbf{v}_{i}:i\in S\right\} }}$
such  that for all $\scalebox{0.95}{\ensuremath{A\in2^{S}}}$:
\begin{equation}
r\left(A\right)=\dim\mathrm{span}\left\{ \mathbf{v}_{i}:i\in A\right\} .\label{eq:realizable}
\end{equation}
This spanning set is called a \textbf{realization} of $\scalebox{0.95}{\ensuremath{M}}$
(over the given field). Conversely, for any vector space $\scalebox{0.95}{\ensuremath{V}}$
with a spanning set $\scalebox{0.95}{\ensuremath{\left\{ \mathbf{v}_{i}:i\in S\right\} }}$,
let $r$ be the function on $\scalebox{0.95}{\ensuremath{2^{S}}}$
defined by \eqref{eq:realizable}, then $\scalebox{0.95}{\ensuremath{\left(S,r\right)}}$
is a matroid of rank $\scalebox{0.95}{\ensuremath{\dim V}}$. A \textbf{realizable}
matroid is a matroid that is realizable over some field. A \textbf{regular}
matroid is a matroid that is  realizable  over any field.

A $\scalebox{0.95}{\ensuremath{\left(k,S\right)}}$\textbf{-uniform
}matroid or a rank-$k$ \textbf{uniform }matroid on\textbf{ }$\scalebox{0.95}{\ensuremath{S}}$
is the matroid on $\scalebox{0.95}{\ensuremath{S}}$ defined by a
rank function $\scalebox{0.95}{\ensuremath{A\mapsto\min\left(k,\left|A\right|\right)}}$,
which we denote by $\scalebox{0.95}{\ensuremath{U_{S}^{k}}}$ or $\scalebox{0.95}{\ensuremath{U_{k,S}}}$
(we prefer the former). When $\scalebox{0.95}{\ensuremath{S=\left[n\right]}}$,
we will use the notation $\scalebox{0.95}{\ensuremath{U_{n}^{k}}}$
instead of $\scalebox{0.95}{\ensuremath{U_{\left[n\right]}^{k}}}$.
A uniform matroid is a realizable matroid.

Let $\scalebox{0.95}{\ensuremath{G}}$ be a graph with the edge set
$\scalebox{0.95}{\ensuremath{\left\{ e_{i}:i\in S\right\} }}$. For
any $\scalebox{0.95}{\ensuremath{A\in2^{S}}}$, every maximal subset
$\scalebox{0.95}{\ensuremath{I\subseteq A}}$ with $\left\{ e_{i}:i\in I\right\} $
being a forest has the same size, and let $\scalebox{0.95}{\ensuremath{r\left(A\right)}}$
be this number. Then, $r$ is a submodular rank function on $\scalebox{0.95}{\ensuremath{2^{S}}}$,
and $\scalebox{0.95}{\ensuremath{\left(S,r\right)}}$ is a matroid
called a \textbf{graphic }matroid. A graphic matroid is regular.

\subsection{Flats}

Let $\scalebox{0.95}{\ensuremath{M=\left(S,r\right)}}$ be a matroid.
For $\scalebox{0.95}{\ensuremath{i=0,1,\dots,r\left(M\right)}}$,
the $i$\textbf{-th graded piece} $\scalebox{0.95}{\ensuremath{M^{\left(i\right)}}}$
of $2^{S}$ is $\scalebox{0.95}{\ensuremath{\left\{ A\in2^{S}:r\left(A\right)=i\right\} }}$.
For any $\scalebox{0.95}{\ensuremath{A\in M^{\left(i\right)}}}$:
\[
\scalebox{0.92}{\ensuremath{\overline{A}:=\left\{ s\in S:r\left(A\cup\left\{ s\right\} \right)=r\left(A\right)\right\} \subseteq S}}
\]
 is the maximum member of $\scalebox{0.95}{\ensuremath{M^{\left(i\right)}}}$
containing $\scalebox{0.95}{\ensuremath{A}}$, called a \textbf{flat}
of rank $\scalebox{0.95}{\ensuremath{i}}$. Then, $\scalebox{0.95}{\ensuremath{S}}$
is a flat. Any flat different from $\scalebox{0.95}{\ensuremath{S}}$
is called \textbf{proper}. A rank-$0$ singleton is called a \textbf{loop}.
The collection of loops of $\scalebox{0.95}{\ensuremath{M}}$, denoted
by $\scalebox{0.95}{\ensuremath{\overline{\emptyset}_{M}}}$ or $\scalebox{0.95}{\ensuremath{\overline{\emptyset}}}$,
is the unique rank-$0$ flat. Let $\scalebox{0.95}{\ensuremath{\mathcal{L}^{\left(i\right)}}}$
denote the collection of flats in $\scalebox{0.95}{\ensuremath{M^{\left(i\right)}}}$.
Then, the collection $\scalebox{0.95}{\ensuremath{\bigcup_{i=0}^{r\left(M\right)}\mathcal{L}^{\left(i\right)}}}$
is called the \textbf{flat} \textbf{lattice} of $\scalebox{0.95}{\ensuremath{M}}$
and denoted as $\scalebox{0.95}{\ensuremath{\mathcal{L}\left(M\right)}}$.

\begin{propdef}[Matroid flat axioms]

A collection $\scalebox{0.95}{\ensuremath{\mathcal{A}\subseteq2^{S}}}$
with $\scalebox{0.95}{\ensuremath{S\in\mathcal{A}}}$ is the flat
lattice of a matroid on $\scalebox{0.95}{\ensuremath{S}}$ if it satisfies
the following:
\begin{enumerate}[label=(F\arabic*)]
\item \label{enu:(F1)} For any $\scalebox{0.95}{\ensuremath{F,L\in\mathcal{A}}}$,
one has $\scalebox{0.95}{\ensuremath{F\cap L\in\mathcal{A}}}$.
\item \label{enu:(F2)} For any $\scalebox{0.95}{\ensuremath{F\in\mathcal{A}}}$
and $\scalebox{0.95}{\ensuremath{s\in S-F}}$, by \ref{enu:(F1)},
there exists the smallest member $\scalebox{0.95}{\ensuremath{L}}$
of $\scalebox{0.95}{\ensuremath{\mathcal{A}}}$ containing $\scalebox{0.95}{\ensuremath{F\cup\left\{ s\right\} }}$.
There are no members of $\scalebox{0.95}{\ensuremath{\mathcal{A}}}$
between $\scalebox{0.95}{\ensuremath{F}}$ and $\scalebox{0.95}{\ensuremath{L}}$.
\end{enumerate}
\end{propdef}

For the flat lattice $\scalebox{0.95}{\ensuremath{\mathcal{A}}}$
of a matroid on $\scalebox{0.95}{\ensuremath{S}}$, the rank of a
subset $\scalebox{0.95}{\ensuremath{A\subseteq S}}$ is the \emph{lattice
rank} of the smallest member of $\scalebox{0.95}{\ensuremath{\mathcal{A}}}$
containing $\scalebox{0.95}{\ensuremath{A}}$. One can switch back
and forth between rank axioms and flat axioms.

\subsection{Restriction, deletion, and contraction}

The following are three basic matroid operations. Note that we occasionally
and intentionally confuse a matroid operation with the matroid that
the operation constructs.

Let $\scalebox{0.95}{\ensuremath{M=\left(S,r\right)}}$ be a matroid
with a subset $\scalebox{0.95}{\ensuremath{A\subseteq E\left(M\right)}}$.
The \textbf{restriction} of $\scalebox{0.95}{\ensuremath{M}}$ \textbf{to}
$\scalebox{0.95}{\ensuremath{A}}$ is the matroid on $\scalebox{0.95}{\ensuremath{A}}$,
defined by a submodular rank function on $\scalebox{0.95}{\ensuremath{2^{A}}}$
mapping $\scalebox{0.95}{\ensuremath{I\mapsto r\left(I\right)}}$.
We denote it by $\scalebox{0.95}{\ensuremath{M|_{A}}}$ or $\scalebox{0.95}{\ensuremath{M|A}}$.
We prefer the former notation for a few reasons: it is more compact
than the latter, it stresses in its form that a restriction matroid
is nothing new, and it is consistent with notations used in other
areas of mathematics. We also call $\scalebox{0.95}{\ensuremath{M|_{A}}}$
the \textbf{submatroid}\footnote{Welsh used the term \emph{submatroid} to denote any matroid obtained
from a given matroid by matroid operations, \cite[Chapter 4]{Welsh}.
However, we will only use the term for a restriction matroid as if
it denotes a submatrix of a matrix.} of $\scalebox{0.95}{\ensuremath{M}}$ \textbf{on} $\scalebox{0.95}{\ensuremath{A}}$.
The \textbf{deletion} $\scalebox{0.95}{\ensuremath{M\backslash A}}$
is the restriction $\scalebox{0.95}{\ensuremath{M|_{S-A}}}$.

The \textbf{contraction} \textbf{of} $\scalebox{0.95}{\ensuremath{A}}$
\textbf{in} $\scalebox{0.95}{\ensuremath{M}}$ or \textbf{contraction
of $\scalebox{0.95}{\ensuremath{M}}$ over $\scalebox{0.95}{\ensuremath{A}}$}
is a matroid on $\scalebox{0.95}{\ensuremath{A^{c}}}$, whose rank
function is $\scalebox{0.95}{\ensuremath{I\mapsto r\left(I\cup A\right)-r\left(A\right)}}$,
which is denoted by $\scalebox{0.95}{\ensuremath{M/A}}$.

Deletions and contractions commute.

\subsection{\label{subsec:HA-1}Abstract hyperplanes and subspaces{*}\protect\footnote{Asterisks {*} mark the subsections that introduce new notions.}}

An \textbf{abstract hyperplane} of a matroid $\scalebox{0.95}{\ensuremath{M}}$
is a contraction matroid $\scalebox{0.95}{\ensuremath{M/F}}$ over
a rank-$1$ flat $\scalebox{0.95}{\ensuremath{F}}$ of $\scalebox{0.95}{\ensuremath{M}}$.
We call $\scalebox{0.95}{\ensuremath{\left|F\right|}}$ the \textbf{multiplicity}
of the hyperplane $\scalebox{0.95}{\ensuremath{M/F}}$. This definition
is a natural abstraction of the fact that in a vector space $\scalebox{0.95}{\ensuremath{V}}$
with a spanning set $\scalebox{0.95}{\ensuremath{\left\{ \mathbf{v}_{i}:i\in S\right\} }}$,
the zero set of a linear functional $\scalebox{0.95}{\ensuremath{\left\langle \mathbf{v}_{i},\underline{\hspace{0.7em}}\,\right\rangle }}$
for each nonzero vector $\mathbf{v}_{i}$ is a hyperplane. An \textbf{abstract}
\textbf{subspace} of $\scalebox{0.95}{\ensuremath{M}}$ is a contraction
matroid $\scalebox{0.95}{\ensuremath{M/L}}$ over a flat $\scalebox{0.95}{\ensuremath{L}}$.
Let $\scalebox{0.95}{\ensuremath{\mathcal{S}\left(M\right)}}$ be
the collection of abstract subspaces of $\scalebox{0.95}{\ensuremath{M}}$.
This is a lattice in which the \emph{lattice rank} of a member $\scalebox{0.95}{\ensuremath{M/L}}$
is its matroid rank $\scalebox{0.95}{\ensuremath{r\left(M/L\right)}}$.
Thus, as a lattice, $\scalebox{0.95}{\ensuremath{\mathcal{S}\left(M\right)}}$
is isomorphic to the dual lattice of the flat lattice of $\scalebox{0.95}{\ensuremath{M}}$.
We call it the \textbf{subspace lattice} of $\scalebox{0.95}{\ensuremath{M}}$.
The \textbf{subspace dimension} of $\scalebox{0.95}{\ensuremath{M/F}}$
is defined as
\[
\scalebox{0.95}{\ensuremath{\mathrm{sdim}\,M/F:=r\left(M/F\right)-1}}.
\]
 In particular, $\scalebox{0.95}{\ensuremath{\mathrm{sdim}\,\overline{\emptyset}=-1}}$
and $\scalebox{0.95}{\ensuremath{\mathrm{sdim}\,M/\overline{\emptyset}=r\left(M\right)-1}}$
where $\scalebox{0.95}{\ensuremath{M/\overline{\emptyset}=M\backslash\overline{\emptyset}}}$.
We call an abstract subspace of dimension $\scalebox{0.95}{\ensuremath{0}}$
or $\scalebox{0.95}{\ensuremath{1}}$ a \textbf{point} or a \textbf{line},
respectively.

It is not necessarily true that two points are connected by a line.
A matroid in which every two points are connected by a line is called
\textbf{hypermodular}. However, at most one line passes through two
points:
\begin{example}
\label{exa:2-Lines}Let $\scalebox{0.95}{\ensuremath{M}}$ be a matroid
whose rank is at least $\scalebox{0.95}{\ensuremath{3}}$. Suppose
that two lines $\scalebox{0.95}{\ensuremath{M/F}}$ and $\scalebox{0.95}{\ensuremath{M/L}}$
pass through two distinct points $\scalebox{0.95}{\ensuremath{M/A}}$
and $\scalebox{0.95}{\ensuremath{M/T}}$. Then, $\scalebox{0.95}{\ensuremath{F}}$
and $\scalebox{0.95}{\ensuremath{L}}$ are rank-$\scalebox{0.95}{\ensuremath{\left(k-2\right)}}$
flats, and $\scalebox{0.95}{\ensuremath{A}}$ and $\scalebox{0.95}{\ensuremath{T}}$
are distinct rank-$\scalebox{0.95}{\ensuremath{\left(k-1\right)}}$
flats with $\scalebox{0.95}{\ensuremath{F\cup L\subseteq A\cap T}}$.
It follows $\scalebox{0.95}{\ensuremath{k-2\le r\left(F\cup L\right)\le r\left(A\cap T\right)\le k-2}}$
and $\scalebox{0.95}{\ensuremath{k-2=r\left(F\cup L\right)}}$.  Therefore,
$\scalebox{0.95}{\ensuremath{F=L}}$, and the two lines are the same
line.
\end{example}

\subsection{\label{subsec:HA}Abstract hyperplane arrangement{*}}

For a matroid $\scalebox{0.95}{\ensuremath{M}}$, loopless or not,
the \textbf{abstract} \textbf{hyperplane arrangement} of a matroid
$\scalebox{0.95}{\ensuremath{M}}$ is the collection of its abstract
hyperplanes, whose \textbf{intersection lattice} is defined as the
subspace lattice $\scalebox{0.95}{\ensuremath{\mathcal{S}(M)}}$ of
$\scalebox{0.95}{\ensuremath{M}}$. We say that the abstract hyperplane
arrangement of $\scalebox{0.95}{\ensuremath{M}}$ is \textbf{realizable}
if there is a pair of an $\scalebox{0.95}{\ensuremath{r(M)}}$-dimensional
vector space $\scalebox{0.95}{\ensuremath{V}}$ and a collection of
hyperplanes in $\scalebox{0.95}{\ensuremath{V}}$ containing the origin
whose intersection lattice is isomorphic to $\scalebox{0.95}{\ensuremath{\mathcal{S}(M)}}$.
The linear functionals of these hyperplanes form a matroid isomorphic
to $\scalebox{0.95}{\ensuremath{M\backslash\overline{\emptyset}}}$.
The projectivization of the pair, that is, the pair of $\scalebox{0.95}{\ensuremath{\mathbb{P}\left(V\right)}}$
and the projectivizations of the hyperplanes, is called a \textbf{realization}
of the abstract hyperplane arrangement. For a subspace of $\scalebox{0.95}{\ensuremath{M}}$,
we call its corresponding subspace in $\scalebox{0.95}{\ensuremath{\mathbb{P}\left(V\right)}}$
a \textbf{realization} of it. Both a subspace of $\scalebox{0.95}{\ensuremath{M}}$
and a realization of it have the same dimension. Thus, the abstract
hyperplane arrangement of $\scalebox{0.95}{\ensuremath{M}}$ is realizable
if and only if $\scalebox{0.95}{\ensuremath{M}}$ is realizable.

The characteristic polynomial of a loopless matroid $\scalebox{0.95}{\ensuremath{M}}$
in $x$ is written in terms of restrictions and contractions as follows
($\mu$ is the Möbius invariant): 
\[
\scalebox{0.95}{\ensuremath{{\displaystyle p_{M}\left(x\right)=\sum_{A\in\mathcal{L}\left(M\right)}\mu\left(M|_{A}\right)x^{r\left(M/A\right)}}}}.
\]
 This is an invariant of the abstract hyperplane arrangement of $\scalebox{0.95}{\ensuremath{M}}$.
\begin{defn}
The \textbf{characteristic polynomial} of the abstract hyperplane
arrangement of a matroid $\scalebox{0.95}{\ensuremath{M}}$ is the
characteristic polynomial of $\scalebox{0.95}{\ensuremath{M\backslash\overline{\emptyset}}}$.
\end{defn}

This definition is well-defined because an abstract hyperplane arrangement
and any realization of it have the same characteristic polynomial,
cf. \cite[Proposition 3.11.3]{StanleyEC} and \cite[Proposition 7.2.1]{Zaslavsky}. 

Since every flat is a union of rank-$\scalebox{0.95}{\ensuremath{1}}$
flats, any abstract subspace $\scalebox{0.95}{\ensuremath{M/F}}$
of $\scalebox{0.95}{\ensuremath{M}}$ is an intersection of abstract
hyperplanes. The following is \textbf{Bézout's theorem} for abstract
hyperplane arrangements:

\begin{propdef}

Let $\scalebox{0.95}{\ensuremath{F_{1},\dots,F_{m}}}$ be flats of
a matroid $\scalebox{0.95}{\ensuremath{M}}$. Then, $\scalebox{0.95}{\ensuremath{M/F_{1},\dots,M/F_{m}}}$
meet at a point if and only if $\scalebox{0.95}{\ensuremath{r\left(F_{1}\cup\cdots\cup F_{m}\right)=r\left(M\right)-1}}.$
\end{propdef}
\begin{rem}
In this paper, the term ``line arrangement'' will refer to a hyperplane
arrangement of rank $\scalebox{0.95}{\ensuremath{3}}$. Also, ``hyperplanes''
and ``subspaces'' will be used to refer to abstract hyperplanes
and subspaces, respectively, without explicitly using the term ``abstract,''
unless such terminology causes confusion.
\end{rem}

\subsection{\label{subsec:Labelled}Labelled hyperplane arrangements and weight
structures{*}}

For a matroid $\scalebox{0.95}{\ensuremath{M}}$, a \textbf{labelled
hyperplane} is a pair $\scalebox{0.95}{\ensuremath{\bigl(i,M/\overline{\left\{ i\right\} }\bigr)}}$
of a nonloop $\scalebox{0.95}{\ensuremath{i}}$ and the hyperplane
$\scalebox{0.95}{\ensuremath{M/\overline{\left\{ i\right\} }}}$,
where $i$ is called a \textbf{label}. Then, $\scalebox{0.95}{\ensuremath{M}}$
has $\scalebox{0.95}{\ensuremath{\bigl|E(M)-\overline{\emptyset}\bigr|}}$
labelled hyperplanes. The \textbf{labelled hyperplane arrangement}
$\scalebox{0.95}{\ensuremath{\mathrm{HA}_{M}}}$ of $\scalebox{0.95}{\ensuremath{M}}$
is the collection of its labelled hyperplanes: 
\[
\scalebox{0.95}{\ensuremath{\mathrm{HA}_{M}=\left\{ \bigl(i,M/\overline{\left\{ i\right\} }\bigr):i\in M\backslash\overline{\emptyset}\right\} }}.
\]
 Henceforth, we will always consider a hyperplane arrangement with
its label structure. A \textbf{hyperplane locus} refers to a hyperplane
without a label.

For a labelled hyperplane arrangement $\scalebox{0.95}{\ensuremath{\mathrm{HA}_{M}}}$,
we assign to each label $i$ a real number $\scalebox{0.95}{\ensuremath{0\le w_{i}\le1}}$,
which we call a \textbf{weight},  and we  say that the labels are
\textbf{weighted}. A \textbf{weighted hyperplane arrangement} is a
labelled hyperplane arrangement with weighted labels. An unweighted
hyperplane arrangement is regarded as a weighted hyperplane arrangement
whose labels all have weights $1$.

For a  weighted hyperplane arrangement $\scalebox{0.95}{\ensuremath{\mathrm{HA}_{M}}}$,
by assigning 0 to each loop of $\scalebox{0.95}{\ensuremath{M}}$,
we regard a loop as a label with weight 0 and call the vector $\scalebox{0.95}{\ensuremath{\mathbf{w}=\left(w_{i}\right)_{i\in E(M)}\in\left[0,1\right]^{E(M)}}}$
a \textbf{weight vector}. We will frequently assume that our matroids
under consideration are loopless, without explicitly stating so.

\subsection{Independent sets and bases}

Fix a matroid $\scalebox{0.95}{\ensuremath{M}}$. If $\scalebox{0.95}{\ensuremath{r\left(A\right)=\left|A\right|}}$
in \ref{enu:(R1)}, then $A$ is called an \textbf{independent} \textbf{set}
of $\scalebox{0.95}{\ensuremath{M}}$. All maximal independent sets
of $\scalebox{0.95}{\ensuremath{M}}$ have the same size, $\scalebox{0.95}{\ensuremath{r\left(M\right)}}$,
and are called the \textbf{bases} of $\scalebox{0.95}{\ensuremath{M}}$.

If $\scalebox{0.95}{\ensuremath{r\left(A\right)<\left|A\right|}}$,
then $A$ is called a \textbf{dependent} \textbf{set}, and a minimal
dependent set is called a \textbf{circuit}.

We use $\scalebox{0.95}{\ensuremath{\mathcal{I}\left(M\right)}}$,
$\scalebox{0.95}{\ensuremath{\mathcal{B}\left(M\right)}}$, and $\scalebox{0.95}{\ensuremath{\mathcal{C}\left(M\right)}}$
to denote the collections of independent sets, bases, and circuits
of $\scalebox{0.95}{\ensuremath{M}}$, respectively. Below are the
matroid axiom systems for independent sets and bases, respectively.
The matroid axiom system for circuits is not included here. \begin{propdef}[Matroid independence axioms]

\noindent A nonempty collection $\scalebox{0.95}{\ensuremath{\mathcal{A}\subseteq2^{S}}}$
is the independent-set collection of a matroid on $\scalebox{0.95}{\ensuremath{S}}$
if
\begin{enumerate}[label=(I\arabic*),itemsep=1pt]
\item \label{enu:(I1)}$\scalebox{0.95}{\ensuremath{A\in\mathcal{A}}}$
and $\scalebox{0.95}{\ensuremath{B\subseteq A}}$ implies $\scalebox{0.95}{\ensuremath{B\in\mathcal{A}}}$,
and 
\item \label{enu:(I2)}for $\scalebox{0.95}{\ensuremath{A,B\in\mathcal{A}}}$
with $\scalebox{0.95}{\ensuremath{\left|A\right|<\left|B\right|}}$,
there exists $\scalebox{0.95}{\ensuremath{b\in B-A}}$ with $\scalebox{0.95}{\ensuremath{A\cup\left\{ b\right\} \in\mathcal{A}}}$.
\end{enumerate}
\noindent The condition \ref{enu:(I1)} implies $\scalebox{0.95}{\ensuremath{\emptyset\in\mathcal{A}}}$,
while \ref{enu:(I2)} is called the \emph{exchange property}.\end{propdef}

\begin{propdef}[Matroid base axioms]

A nonempty collection $\scalebox{0.95}{\ensuremath{\mathcal{A}\subseteq2^{S}}}$
is the base collection of a matroid on $\scalebox{0.95}{\ensuremath{S}}$
if it satisfies the following: 
\noindent \begin{center}
For $\scalebox{0.95}{\ensuremath{A,B\in\mathcal{A}}}$, if $\scalebox{0.95}{\ensuremath{a\in A-B}}$,
then $\scalebox{0.95}{\ensuremath{\left(A-\left\{ a\right\} \right)\cup\left\{ b\right\} \in\mathcal{A}}}$
for some $\scalebox{0.95}{\ensuremath{b\in B-A}}$.
\par\end{center}

\noindent This condition is called the \emph{base exchange property}.\end{propdef}

All $r_{M}$, $\scalebox{0.95}{\ensuremath{\mathcal{L}\left(M\right)}}$,
$\scalebox{0.95}{\ensuremath{\mathcal{I}\left(M\right)}}$, $\scalebox{0.95}{\ensuremath{\mathcal{B}\left(M\right)}}$,
and $\scalebox{0.95}{\ensuremath{\mathcal{C}\left(M\right)}}$ are
recovered from one another,  and we  use the \emph{pair} of $S$ and
any of those to denote $\scalebox{0.95}{\ensuremath{M}}$. Here is
a cheap relation between $\scalebox{0.95}{\ensuremath{\mathcal{L}\left(M\right)}}$,
$\scalebox{0.95}{\ensuremath{\mathcal{I}\left(M\right)}}$, and $\scalebox{0.95}{\ensuremath{\mathcal{B}\left(M\right)}}$:
\[
\scalebox{0.95}{\ensuremath{E\left(M\right)=\bigcup\left.\mathcal{L}\left(M\right)\right.\supseteq\bigcup\left.\mathcal{I}\left(M\right)\right.=\bigcup\left.\mathcal{B}\left(M\right)\right.}.}
\]

\subsection{Modular pairs, dual matroid, and direct sum}

An (unordered) pair $\scalebox{0.95}{\ensuremath{\left\{ A,B\right\} }}$
of subsets of $\scalebox{0.95}{\ensuremath{E\left(M\right)}}$ is
called a \textbf{modular pair} of $\scalebox{0.95}{\ensuremath{M}}$
if equality holds in \ref{enu:(R3)}: $\scalebox{0.95}{\ensuremath{r\left(A\cup B\right)+r\left(A\cap B\right)=r\left(A\right)+r\left(B\right)}}$. 

The \textbf{dual matroid} $\scalebox{0.95}{\ensuremath{M^{\ast}}}$
of $\scalebox{0.95}{\ensuremath{M}}$ is a matroid on $\scalebox{0.95}{\ensuremath{E\left(M\right)}}$
defined by a submodular rank function $r^{\ast}$ given by $\scalebox{0.95}{\ensuremath{A\mapsto\left|A\right|-r\left(S\right)+r\left(S-A\right)}}$.

Let $\scalebox{0.95}{\ensuremath{M_{1}=(S_{1},r_{1})}}$ and $\scalebox{0.95}{\ensuremath{M_{2}=(S_{2},r_{2})}}$
be two matroids. For any members $\scalebox{0.95}{\ensuremath{A_{1}\in2^{S_{1}}}}$
and $\scalebox{0.95}{\ensuremath{A_{2}\in2^{S_{2}}}}$, denote by
$\scalebox{0.95}{\ensuremath{A_{1}\oplus A_{2}}}$ the disjoint union
of $\scalebox{0.95}{\ensuremath{A_{1}}}$ and $\scalebox{0.95}{\ensuremath{A_{2}}}$,
and by $r_{1}\oplus r_{2}$ the submodular rank function on $\scalebox{0.95}{\ensuremath{2^{S_{1}\oplus S_{2}}}}$
defined by $\scalebox{0.95}{\ensuremath{A_{1}\oplus A_{2}\mapsto r_{1}\left(A_{1}\right)+r_{2}\left(A_{2}\right)}}$.
Then, $\scalebox{0.95}{\ensuremath{M_{1}\oplus M_{2}:=\left(S_{1}\oplus S_{2},r_{1}\oplus r_{2}\right)}}$
is a matroid called the \textbf{direct sum} of $\scalebox{0.95}{\ensuremath{M_{1}}}$
and $\scalebox{0.95}{\ensuremath{M_{2}}}$.

\subsection{\label{subsec:PBPF}Pullback and pushforward{*}}

Regarding matroids as combinatorial structures on sets, we define
the pullback and pushforward of a matroid.

\begin{propdef}[Pullback]

Let $\scalebox{0.95}{\ensuremath{M=(S,r)}}$ be a matroid with a map
$\scalebox{0.95}{\ensuremath{f:\tilde{S}\rightarrow S}}$ between
sets. We define the set $\scalebox{0.95}{\ensuremath{f^{\ast}\bigl(\mathcal{I}\bigl(M\bigr)\bigr)}}$
as: 
\[
\scalebox{0.95}{\ensuremath{f^{\ast}\bigl(\mathcal{I}\bigl(M\bigr)\bigr):=\left\{ A\subseteq\tilde{S}:f\left(A\right)\in\mathcal{I}\left(M\right),\left|A\right|=\left|f(A)\right|\right\} }}.
\]
 The pair $\scalebox{0.95}{\ensuremath{\bigl(\tilde{S},f^{\ast}\bigl(\mathcal{I}\bigl(M\bigr)\bigr)\bigr)}}$
is a matroid, which we call the \textbf{pullback} of $\scalebox{0.95}{\ensuremath{M}}$
under the map $\scalebox{0.95}{\ensuremath{f}}$ and denote as $\scalebox{0.95}{\ensuremath{f^{\ast}(M)}}$
where $\scalebox{0.95}{\ensuremath{\mathcal{I}\bigl(f^{\ast}(M)\bigr)=f^{\ast}\bigl(\mathcal{I}\bigl(M\bigr)\bigr)}}$.\end{propdef}

The proof is straightforward. The pullback has the following properties:
\[
\scalebox{0.95}{\ensuremath{r_{f^{\ast}(M)}}}=\scalebox{0.95}{\ensuremath{r_{M}\circ f}}\quad\text{and}\quad\scalebox{0.95}{\ensuremath{\mathcal{L}\left(f^{\ast}\left(M\right)\right)}}=\scalebox{0.95}{\ensuremath{f^{-1}\left(\mathcal{L}\left(M\right)\right)}}.
\]
 The latter formula implies that $f$ is a \emph{strong map} from
$\scalebox{0.95}{\ensuremath{f^{\ast}\left(M\right)}}$ to $\scalebox{0.95}{\ensuremath{M}}$.

\begin{propdef}[Pushforward]

Let $\scalebox{0.95}{\ensuremath{\tilde{M}=\bigl(\tilde{S},\tilde{r}\bigr)}}$
be a matroid with a map $\scalebox{0.95}{\ensuremath{f:\tilde{S}\rightarrow S}}$
between sets. We define the set $\scalebox{0.95}{\ensuremath{f_{\ast}\bigl(\mathcal{I}\bigl(\tilde{M}\bigr)\bigr)}}$
as: 
\[
\scalebox{0.95}{\ensuremath{f_{\ast}\bigl(\mathcal{I}\bigl(\tilde{M}\bigr)\bigr):=\left\{ f\left(I\right)\subseteq S:I\in\mathcal{I}\bigl(\tilde{M}\bigr)\right\} }}.
\]
 The pair $\scalebox{0.95}{\ensuremath{\bigl(S,f_{\ast}\bigl(\mathcal{I}\bigl(\tilde{M}\bigr)\bigr)\bigr)}}$
is a matroid, which we call the \textbf{pushforward} of $\scalebox{0.95}{\ensuremath{\tilde{M}}}$
under $f$ and denote as $\scalebox{0.95}{\ensuremath{f_{\ast}\bigl(\tilde{M}\bigr)}}$
where $\scalebox{0.95}{\ensuremath{\mathcal{I}\bigl(f_{\ast}(\tilde{M})\bigr)=f_{\ast}\bigl(\mathcal{I}\bigl(\tilde{M}\bigr)\bigr)}}$.\end{propdef}

The proof can be found in \cite[Theorem 42.1]{Schrijver}. Here, $f$
is a \emph{weak map} from $\scalebox{0.95}{\ensuremath{\tilde{M}}}$
to $\scalebox{0.95}{\ensuremath{f_{\ast}\bigl(\tilde{M}\bigr)}}$. 
\begin{rem}
Let $\scalebox{0.95}{\ensuremath{M}}$ and $\scalebox{0.95}{\ensuremath{\tilde{M}}}$
be the matroids mentioned above with $\scalebox{0.95}{\ensuremath{f:\tilde{S}\rightarrow S}}$.
\begin{enumerate}
\item If $f$ is injective, then $\scalebox{0.95}{\ensuremath{\tilde{M}=f^{\ast}\bigl(f_{\ast}\bigl(\tilde{M}\bigr)\bigr)}}$.
\item If $f$ is surjective, then $\scalebox{0.95}{\ensuremath{M=f_{\ast}\left(f^{\ast}\left(M\right)\right)}}$.
\end{enumerate}
\end{rem}

Later in Section \ref{sec:Hyperplanes}, by choosing a suitable function
for $f$, we show that cutting any base polytope with a hyperplane
of the form $\scalebox{0.95}{\ensuremath{\left\{ \sum_{i\in A}x_{i}=1\right\} }}$
always produces a matroid subdivision (Lemma \ref{lem:Cut-w-1}).\vspace{2pt}

One can take the following examples as definitions of restrictions,
matroid unions, transversal matroids, and simple matroids, respectively.
\begin{example}
Let $\scalebox{0.95}{\ensuremath{M}}$ be a matroid. For an inclusion
map $\scalebox{0.95}{\ensuremath{\iota:A\hookrightarrow E\left(M\right)}}$,
we have $\scalebox{0.95}{\ensuremath{M|_{A}=\iota^{\ast}\left(M\right)}}$;
hence, any restriction is a pullback.
\end{example}

\begin{example}
A \emph{matroid union} is a pushforward. For any matroids $\scalebox{0.95}{\ensuremath{M_{1}=(S_{1},\mathcal{I}_{1})}}$
and $\scalebox{0.95}{\ensuremath{M_{2}=(S_{2},\mathcal{I}_{2})}}$,
let $\scalebox{0.95}{\ensuremath{f:S_{1}\oplus S_{2}\rightarrow S_{1}\cup S_{2}}}$
be the map with $\scalebox{0.95}{\ensuremath{f|_{S_{i}}=\mathrm{id}_{S_{i}}}}$
for $\scalebox{0.95}{\ensuremath{i=1,2}}$, which we call a \emph{natural
surjection}. The matroid union $\scalebox{0.95}{\ensuremath{M_{1}\vee M_{2}}}$
is the pushforward $\scalebox{0.95}{\ensuremath{f_{\ast}\left(M_{1}\oplus M_{2}\right)}}$.
This is a \emph{common upper bound} of $\scalebox{0.95}{\ensuremath{M_{1}}}$
and $\scalebox{0.95}{\ensuremath{M_{2}}}$ in the sense that $\scalebox{0.95}{\ensuremath{\mathcal{I}_{1}\cup\mathcal{I}_{2}\subseteq\mathcal{I}\left(M_{1}\vee M_{2}\right)}}.$
However, it is not the smallest nor a minimal such.
\end{example}

\begin{example}
A \emph{transversal matroid} is a pushforward. For a collection $\scalebox{0.95}{\ensuremath{\left\{ S_{1},\dots,S_{m}\right\} }}$
of finite sets, its induced transversal matroid is the pushforward
$\scalebox{0.95}{\ensuremath{f_{\ast}\bigl(\bigoplus_{i\in\left[m\right]}U_{S_{i}}^{1}\bigr)}}$
where $f$ is the natural surjection.
\end{example}

\begin{example}
Let $f$ be a map on $\scalebox{0.95}{\ensuremath{S-\bar{\emptyset}}}$
defined by $i\mapsto\overline{\left\{ i\right\} }$. Then 
\[
\scalebox{0.95}{\ensuremath{M\backslash\bar{\emptyset}=f^{\ast}\left(f_{\ast}\left(M\backslash\bar{\emptyset}\right)\right)}}.
\]
 We call $\scalebox{0.95}{\ensuremath{f_{\ast}\left(M\backslash\bar{\emptyset}\right)}}$
the \textbf{simplification} of $\scalebox{0.95}{\ensuremath{M}}$,
and $f$ the \textbf{simplification map} for $\scalebox{0.95}{\ensuremath{M}}$.
Then, $\scalebox{0.95}{\ensuremath{M}}$ is \emph{simple} if and only
if it is isomorphic to its simplification.
\end{example}

\subsection{\label{subsec:Matroid-union}Matroid intersection, base union, and
base intersection{*}}

For any two matroids $\scalebox{0.95}{\ensuremath{M_{1}}}$ and $\scalebox{0.95}{\ensuremath{M_{2}}}$,
their \textbf{matroid intersection} $\scalebox{0.95}{\ensuremath{M_{1}\wedge M_{2}}}$
is defined as a pair as follows, which is \emph{not} a matroid in
general: 
\[
\scalebox{0.95}{\ensuremath{M_{1}\wedge M_{2}:=\left(S_{1}\cap S_{2},\mathcal{I}\left(M_{1}\right)\cap\mathcal{I}\left(M_{2}\right)\right)}}.
\]
 Similarly, we define the \textbf{base union} $\scalebox{0.95}{\ensuremath{M_{1}\cup M_{2}}}$,
and the \textbf{base intersection} $\scalebox{0.95}{\ensuremath{M_{1}\cap M_{2}}}$,
of $\scalebox{0.95}{\ensuremath{M_{1}}}$ and $\scalebox{0.95}{\ensuremath{M_{2}}}$,
as follows, which are \emph{not} matroids in general either: 
\noindent \begin{center}
\begin{tabular}{c}
\noalign{\vskip\doublerulesep}
$\scalebox{0.95}{\ensuremath{M_{1}\cup M_{2}:=\left(S_{1}\cup S_{2},\mathcal{B}\left(M_{1}\right)\cup\mathcal{B}\left(M_{2}\right)\right)}}$,\tabularnewline[\doublerulesep]
\noalign{\vskip\doublerulesep}
$\scalebox{0.95}{\ensuremath{M_{1}\cap M_{2}:=\left(S_{1}\cap S_{2},\mathcal{B}\left(M_{1}\right)\cap\mathcal{B}\left(M_{2}\right)\right)}}$.\tabularnewline[\doublerulesep]
\end{tabular}
\par\end{center}

\noindent By abuse of notation, we mean by $\scalebox{0.95}{\ensuremath{M_{1}\wedge M_{2}}}$,
$\scalebox{0.95}{\ensuremath{M_{1}\cup M_{2}}}$, and $\scalebox{0.95}{\ensuremath{M_{1}\cap M_{2}}}$
the intersection of independent-set collections, the union of base
collections, and the intersection of base collections, of $\scalebox{0.95}{\ensuremath{M_{1}}}$
and $\scalebox{0.95}{\ensuremath{M_{2}}}$, respectively.

The operation $\wedge$ is not compatible with $\oplus$. Indeed,
for a matroid $\scalebox{0.95}{\ensuremath{N}}$: 
\[
\scalebox{0.95}{\ensuremath{\left(M_{1}\oplus M_{2}\right)\wedge N\subset\left(M_{1}\wedge N\right)\oplus\left(M_{2}\wedge N\right)}}
\]
 but not the other way around. Note that, as collections of subsets:
\[
\scalebox{0.95}{\ensuremath{M_{1}\cap M_{2}\subseteq M_{1}\wedge M_{2}}}.
\]
 If $\scalebox{0.95}{\ensuremath{M_{1}\cap M_{2}}}$ is nonempty,
then $\scalebox{0.95}{\ensuremath{M_{1}}}$ and $\scalebox{0.95}{\ensuremath{M_{2}}}$
have the same rank. If $\scalebox{0.95}{\ensuremath{M_{1}\cap M_{2}}}$
is a nonempty matroid, so is $\scalebox{0.95}{\ensuremath{M_{1}\wedge M_{2}}}$,
and they are the same matroid.

\subsection{Connectivity}

For a matroid $\scalebox{0.95}{\ensuremath{M=(S,r)}}$, the \textbf{connectivity
}function $c_{M}$ of $\scalebox{0.95}{\ensuremath{M}}$ is the following
$\scalebox{0.95}{\ensuremath{\mathbb{Z}_{\ge0}}}$-valued function
on $\scalebox{0.95}{\ensuremath{2^{S}}}$: $\scalebox{0.95}{\ensuremath{A\mapsto r\left(A\right)+r^{\ast}\left(A\right)-\left|A\right|}}$.

A member $\scalebox{0.95}{\ensuremath{A\in2^{S}}}$ is called a \textbf{separator}
of $\scalebox{0.95}{\ensuremath{M}}$ if $\scalebox{0.95}{\ensuremath{c_{M}\left(A\right)=0}}$.
Then, $\emptyset$ and $\scalebox{0.95}{\ensuremath{S}}$ are separators
called \textbf{trivial separators}. A size-$1$ separator of rank
$1$ is called a \textbf{coloop}, while a loop is a size-$1$ separator
of rank $0$. A coloop of $\scalebox{0.95}{\ensuremath{M}}$ is a
loop of $\scalebox{0.95}{\ensuremath{M^{\ast}}}$,  and we  denote
by $\scalebox{0.95}{\ensuremath{\bar{\emptyset}^{\ast}}}$ the collection
of coloops of $\scalebox{0.95}{\ensuremath{M}}$. Then, $\scalebox{0.93}{\ensuremath{\bar{\emptyset}^{\ast}}}=\scalebox{0.93}{\ensuremath{\bigcap\left.\mathcal{B}\right.}}$
and $\scalebox{0.93}{\ensuremath{\bar{\emptyset}}}=\scalebox{0.93}{\ensuremath{\ensuremath{S-\bigcup\left.\mathcal{B}\right.}}}$,
and $\scalebox{0.95}{\ensuremath{M|_{\bar{\emptyset}^{\ast}}}}$ and
$\scalebox{0.95}{\ensuremath{M|_{\bar{\emptyset}}}}$ are uniform
matroids of rank $\scalebox{0.9}{\ensuremath{\left|\bar{\emptyset}^{\ast}\right|}}$
and $0$, respectively.

We call $\scalebox{0.95}{\ensuremath{M}}$ \textbf{loopless} if $\bar{\emptyset}=\emptyset$
and \textbf{relevant} if $\bar{\emptyset}=\emptyset=\bar{\emptyset}^{\ast}$.
We also call $\scalebox{0.95}{\ensuremath{M}}$ \textbf{connected}
if it has no nontrivial separators and \textbf{disconnected} otherwise.
We call a subset $\scalebox{0.95}{\ensuremath{A\subseteq S}}$ \textbf{connected}
if $\scalebox{0.95}{\ensuremath{M|_{A}}}$ is connected and \textbf{disconnected}
otherwise. Let $\scalebox{0.95}{\ensuremath{A_{1},\dots,A_{\kappa\left(M\right)}}}$
be the nonempty minimal separators of $\scalebox{0.95}{\ensuremath{M}}$,
then $\scalebox{0.95}{\ensuremath{M}}$ is written as $\scalebox{0.95}{\ensuremath{M=M|_{A_{1}}\oplus\cdots\oplus M|_{A_{\kappa\left(M\right)}}}}$
with  $\scalebox{0.95}{\ensuremath{\kappa\left(M\right)\le r\left(M\right)+\left|\bar{\emptyset}\right|}}$.
Each summand $\scalebox{0.95}{\ensuremath{M|_{A_{i}}}}$ is a connected
matroid called a \textbf{connected component} of $\scalebox{0.95}{\ensuremath{M}}$.
We denote the number of connected components of $\scalebox{0.95}{\ensuremath{M}}$
as $\scalebox{0.95}{\ensuremath{\kappa\left(M\right)}}$.
\begin{rem}[]
\begin{enumerate}
\item A matroid with a loop is connected if and only if it is a rank-$0$
matroid whose ground set is a singleton.
\item The collection of separators is closed under set complement, union,
and intersection. Further, $\scalebox{0.95}{\ensuremath{M}}$ and
$\scalebox{0.95}{\ensuremath{M^{\ast}}}$ have the same separators.
\end{enumerate}
\end{rem}

\subsection{Non-degenerate subsets{*}}
\begin{defn}[{\cite[2.5]{GS87}}]
 For a connected matroid $\scalebox{0.95}{\ensuremath{M}}$, a nonempty
proper subset $\scalebox{0.95}{\ensuremath{A\subset E\left(M\right)}}$
is called \textbf{non-degenerate}\footnote{The name comes from the torus action on Grassmannians $\mathrm{G}\left(k,n\right)$,
see \cite{GS87}.}\textbf{ }or a \textbf{non-degenerate subset of $\scalebox{0.95}{\ensuremath{M}}$}
if both $\scalebox{0.95}{\ensuremath{M|_{A}}}$ and $\scalebox{0.95}{\ensuremath{M/A}}$
are connected, and \textbf{degenerate} or a \textbf{degenerate subset
of $\scalebox{0.95}{\ensuremath{M}}$} otherwise.
\end{defn}

We classify below the non-degenerate flats of a connected matroid
$\scalebox{0.95}{\ensuremath{M}}$ of rank $\scalebox{0.95}{\ensuremath{\le3}}$.
Let $\lambda$ denote the number of rank-$1$ flats.
\begin{enumerate}
\item If $\scalebox{0.95}{\ensuremath{r\left(M\right)\le1}}$, no non-degenerate
flats exist.
\item If $\scalebox{0.95}{\ensuremath{r\left(M\right)=2}}$, the non-degenerate
flats are exactly the rank-$1$ flats, and $\scalebox{0.95}{\ensuremath{E\left(M\right)}}$
is written as the disjoint union of the non-degenerate flats.
\item If $\scalebox{0.95}{\ensuremath{r\left(M\right)=3}}$, a nonempty
proper flat $A$ of $\scalebox{0.95}{\ensuremath{M}}$ is non-degenerate
if and only if $\scalebox{0.95}{\ensuremath{\lambda\left(M/A\right)\ge3}}$
or $\scalebox{0.95}{\ensuremath{\lambda\left(M|_{A}\right)\ge3}}$.
\end{enumerate}
We generalize the above definition:
\begin{defn}
\label{def:nondeg flat}Let $\scalebox{0.95}{\ensuremath{M}}$ be
a matroid, connected or not. Then, a nonempty proper subset $\scalebox{0.95}{\ensuremath{A\subset E\left(M\right)}}$
is called \textbf{non-degenerate} or a \textbf{non-degenerate subset
of} $\scalebox{0.95}{\ensuremath{M}}$ if 
\[
\scalebox{0.95}{\ensuremath{\kappa\left(M|_{A}\oplus M/A\right)=\kappa\left(M\right)+1}}
\]
 and \textbf{degenerate} or a \textbf{degenerate subset of} $\scalebox{0.95}{\ensuremath{M}}$
otherwise.
\end{defn}

Suppose that $\scalebox{0.95}{\ensuremath{A_{1},\dots,A_{m}\subset E\left(M\right)}}$
are different non-degenerate subsets with 
\[
\scalebox{0.95}{\ensuremath{M|_{A_{1}}\oplus M/A_{1}=\cdots=M|_{A_{m}}\oplus M/A_{m}}}.
\]
 Then, there exists the smallest subset among them. When $\scalebox{0.95}{\ensuremath{M}}$
is connected, $\scalebox{0.95}{\ensuremath{m=1}}$.\vspace{2pt}

Nondegeneracy is preserved under dualizing, i.e., if $\scalebox{0.95}{\ensuremath{A}}$
is a non-degenerate subset of a matroid $\scalebox{0.95}{\ensuremath{M}}$,
its complement $\scalebox{0.95}{\ensuremath{E\left(M\right)-A}}$
is a non-degenerate subset of $\scalebox{0.95}{\ensuremath{M^{\ast}}}$.\vspace{2pt}

Later in Section \ref{sec:Polyhedral Geometry of Matroids}, we show
how to recover the flat lattice of a matroid from its collection of
non-degenerate flats  (Proposition~\ref{prop:nondegen-flats}).

\subsection{\label{subsec:Attache-1}Attaché operation over matroids{*}}

For a subset $\scalebox{0.95}{\ensuremath{A}}$ of $\scalebox{0.95}{\ensuremath{E\left(M\right)}}$,
we define 
\[
\scalebox{0.95}{\ensuremath{M\left(A\right):=M|_{A}\oplus M/A=M/A\oplus M|_{A}}}.
\]
The operation of obtaining $\scalebox{0.95}{\ensuremath{M\left(A\right)}}$
from $\scalebox{0.95}{\ensuremath{M}}$ is said to be the \textbf{attaché
operation by }$\scalebox{0.95}{\ensuremath{A}}$. If $\scalebox{0.95}{\ensuremath{A_{1},\dots,A_{m}}}$
are subsets of $\scalebox{0.95}{\ensuremath{E\left(M\right)}}$, then
$\scalebox{0.95}{\ensuremath{\left(\cdots\left(\left(M\left(A_{1}\right)\right)\left(A_{2}\right)\right)\cdots\right)\left(A_{m}\right)}}$
is well-defined,  and we  write $\scalebox{0.95}{\ensuremath{M\left(A_{1}\right)\left(A_{2}\right)\cdots\left(A_{m}\right)}}$
to denote it.\smallskip{}

Attaché operations do not commute, and they do not affect the rank,
ground set, or decrease the $\kappa$ value. So, $\scalebox{0.95}{\ensuremath{\kappa\left(M(A)\right)=\kappa\left(M\right)}}$
if and only if $\scalebox{0.95}{\ensuremath{A}}$ is a separator of
$\scalebox{0.95}{\ensuremath{M}}$.\smallskip{}

See Subsection \ref{subsec:Attache-2} for the attaché operation over
matroid polytopes.

\subsection{\label{subsec:Minor-expressions}Minor expression and minor{*}}

For a matroid $\scalebox{0.95}{\ensuremath{M}}$, a finite sequence
of restrictions and contractions is called a \textbf{minor expression}
of $\scalebox{0.95}{\ensuremath{M}}$. It is denoted by concatenating
$\scalebox{0.95}{\ensuremath{M}}$ and those operations in order from
left to right. We say that a minor expression \textbf{represents}
the matroid it constructs. A \textbf{minor} of $\scalebox{0.95}{\ensuremath{M}}$
is the matroid a minor expression of $\scalebox{0.95}{\ensuremath{M}}$
represents. A minor expression is said to be \textbf{empty} and denoted
$\emptyset$ if it represents the empty matroid, i.e., if its ground
set is empty. Two minor expressions are said to be \textbf{disjoint}
if the minors they represent have disjoint ground sets.

Since deletions and contractions commute, every minor expression can
be reduced to $\scalebox{0.95}{\ensuremath{M|_{A}/B}}$ and $\scalebox{0.95}{\ensuremath{M/C|_{D}}}$
for some $\scalebox{0.95}{\ensuremath{A,B,C,D\subseteq E\left(M\right)}}$
with $\scalebox{0.95}{\ensuremath{B\subseteq A}}$ and $\scalebox{0.95}{\ensuremath{C\cap D=\emptyset}}$,
which we call \textbf{simple} minor expressions. The former expression
is said to be of the \emph{restriction-contraction form} and the latter
the \emph{contraction-restriction form}.

We call two minor expressions \textbf{equivalent} if they can be reduced
to the same simple minor expression.

\subsection{\label{subsec:Matroidal-expressions}Matroidal expression{*}}

A \textbf{matroidal expression} or simply an \textbf{expression} is
a finite direct sum of minor expressions. The matroid that a matroidal
expression \textbf{represents} is  the direct sum of the matroids
its summands represent. An \textbf{empty expression}, denoted $\emptyset$,
is an\textbf{ }expression that represents the empty matroid.

For two expressions $\scalebox{0.95}{\ensuremath{X}}$ and $\scalebox{0.95}{\ensuremath{Y}}$,
we say that $\scalebox{0.95}{\ensuremath{X}}$ is a \textbf{sub-expression}
of $\scalebox{0.95}{\ensuremath{Y}}$ if there is an injection from
the nonempty summands of $\scalebox{0.95}{\ensuremath{X}}$ to those
of $\scalebox{0.95}{\ensuremath{Y}}$. We say that $\scalebox{0.95}{\ensuremath{X}}$
and $\scalebox{0.95}{\ensuremath{Y}}$ are \textbf{equivalent} and
denote $\scalebox{0.95}{\ensuremath{X\equiv Y}}$ if they are sub-expressions
of each other. Here, ``$\equiv$'' is an equivalence relation on
the collection of matroidal expressions. Note that two non-equivalent
expressions can represent the same matroid.

Let $\phi$ be the map sending every matroidal expression to the matroid
it represents and every matroid to itself. We call this map $\phi$
the \textbf{forgetful map}. For two expressions $\scalebox{0.95}{\ensuremath{X}}$
and $\scalebox{0.95}{\ensuremath{Y}}$ with $\scalebox{0.95}{\ensuremath{\phi\left(X\right)=\phi\left(Y\right)}}$,
we write $\scalebox{0.95}{\ensuremath{X=Y}}$ for simplicity. Note
that $\phi$ composed with itself is equal to $\phi$, that is, $\phi\circ\phi=\phi$.
We omit the symbol $\phi$ unless confusion could arise.

\subsection{Flats, independent sets, bases, and matroid intersections of minors}

Flats of a matroid behave like closed sets in a topological space.
Restriction to and contraction over a flat behave like restriction
and quotient morphisms, respectively. We describe below flats of minors
as well as their independent sets, bases, and matroid intersections.
Let $\scalebox{0.95}{\ensuremath{M}}$ be a matroid with a subset
$\scalebox{0.95}{\ensuremath{A\subseteq E\left(M\right)}}$.
\begin{enumerate}
\item If $\scalebox{0.95}{\ensuremath{F}}$ is a flat of $\scalebox{0.95}{\ensuremath{M}}$,
then $\scalebox{0.95}{\ensuremath{F\cap A}}$ is a flat of $\scalebox{0.95}{\ensuremath{M|_{A}}}$
for any $\scalebox{0.95}{\ensuremath{A\in2^{E\left(M\right)}}}$.
Conversely, if $\scalebox{0.95}{\ensuremath{F\in2^{E\left(M\right)}}}$
is a flat of $\scalebox{0.95}{\ensuremath{M|_{A}}}$, then $\scalebox{0.95}{\ensuremath{F=\overline{F}\cap A}}$.
\item A member $\scalebox{0.95}{\ensuremath{A\in2^{E\left(M\right)}}}$
is a flat of $\scalebox{0.95}{\ensuremath{M}}$ if and only if $\scalebox{0.95}{\ensuremath{M/A}}$
is loopless. A member $\scalebox{0.95}{\ensuremath{F\in2^{E(M)}}}$
is a flat of $\scalebox{0.95}{\ensuremath{M/A}}$ if and only if $\scalebox{0.95}{\ensuremath{F\cup A}}$
is a flat of $\scalebox{0.95}{\ensuremath{M}}$.
\item The independent sets of $\scalebox{0.95}{\ensuremath{M|_{A}}}$ are
the independent sets of $\scalebox{0.95}{\ensuremath{M}}$ that are
contained in $\scalebox{0.95}{\ensuremath{A}}$.
\item The independent sets of $\scalebox{0.95}{\ensuremath{M/A}}$ are the
independent sets $\scalebox{0.95}{\ensuremath{I}}$ of $\scalebox{0.95}{\ensuremath{M\backslash A}}$
with $\scalebox{0.95}{\ensuremath{M|_{I\cup A}=M|_{I}\oplus M|_{A}}}$,
i.e., $\scalebox{0.95}{\ensuremath{\left\{ I,A\right\} }}$ being
a modular pair of $\scalebox{0.95}{\ensuremath{M}}$.
\item For a base $\scalebox{0.95}{\ensuremath{B}}$ of $\scalebox{0.95}{\ensuremath{M}}$,
if $\scalebox{0.95}{\ensuremath{B\cap A}}$ is a base of $\scalebox{0.95}{\ensuremath{M|_{A}}}$,
then $\scalebox{0.95}{\ensuremath{B-A}}$ is a base of $\scalebox{0.95}{\ensuremath{M/A}}$
and vice versa. Every base of $\scalebox{0.95}{\ensuremath{M|_{A}\oplus M/A}}$
is a base of $\scalebox{0.95}{\ensuremath{M}}$.
\item For $\scalebox{0.95}{\ensuremath{C,D\in2^{E\left(M\right)}}}$, one
has $\scalebox{0.95}{\ensuremath{M|_{C}\wedge M|_{D}}}\hspace{2.7bp}=\scalebox{0.95}{\ensuremath{M|_{C\cap D}}}$,
$\scalebox{0.95}{\ensuremath{M|_{C}\wedge M/D}}=\scalebox{0.95}{\ensuremath{M|_{C\cup D}/D}}$,
and $\scalebox{0.95}{\ensuremath{M/C\wedge M/D}}\supseteq\scalebox{0.95}{\ensuremath{\mathcal{I}\left(M/\left(C\cup D\right)\right)}}$
where $\scalebox{0.95}{\ensuremath{\supseteq}}$ in the last formula
can be strict.
\end{enumerate}

\subsection{\label{subsec:Free-product}Free product of matroids}

Let $\scalebox{0.95}{\ensuremath{M}}$ and $\scalebox{0.95}{\ensuremath{N}}$
be matroids. If $\scalebox{0.95}{\ensuremath{X}}$ is a matroid on
$\scalebox{0.95}{\ensuremath{E\left(M\right)\oplus E\left(N\right)}}$
with $\scalebox{0.95}{\ensuremath{X|_{E\left(M\right)}=M}}$ and $\scalebox{0.95}{\ensuremath{X/E(M)=N}}$,
then a base of $\scalebox{0.95}{\ensuremath{X}}$ is a disjoint union
of an independent set of $\scalebox{0.95}{\ensuremath{M}}$ and a
spanning set of $\scalebox{0.95}{\ensuremath{N}}$. Conversely, let
$\scalebox{0.95}{\ensuremath{\mathcal{B}_{M,N}}}$ denote the collection
of $\scalebox{0.95}{\ensuremath{\left(r(M)+r(N)\right)}}$-element
subsets $\scalebox{0.95}{\ensuremath{I\oplus J}}$ of $\scalebox{0.95}{\ensuremath{E\left(M\right)\oplus E\left(N\right)}}$
with $\scalebox{0.95}{\ensuremath{I}}$ being an independent set of
$\scalebox{0.95}{\ensuremath{M}}$ and $\scalebox{0.95}{\ensuremath{J}}$
being a spanning set of $\scalebox{0.95}{\ensuremath{N}}$. Then,
$\scalebox{0.95}{\ensuremath{M\boxempty N:=\left(E\left(M\right)\oplus E\left(N\right),\mathcal{B}_{M,N}\right)}}$
is a matroid with $\scalebox{0.95}{\ensuremath{\mathcal{B}_{M,N}=\mathcal{B}\left(M\boxempty N\right)}}$,
called the \textbf{free product} of $\scalebox{0.95}{\ensuremath{M}}$
and $\scalebox{0.95}{\ensuremath{N}}$. In particular: 
\[
\scalebox{0.95}{\ensuremath{\left(M\boxempty N\right)\left(E(M)\right)=M\oplus N}}.
\]
 In other words, $\scalebox{0.95}{\ensuremath{M\oplus N}}$ is a face
matroid of $\scalebox{0.95}{\ensuremath{M\boxempty N}}$, obtained
by applying the attaché operation by $\scalebox{0.95}{\ensuremath{E(M)}}$
to $\scalebox{0.95}{\ensuremath{M\boxempty N}}$. So, by construction,
$\scalebox{0.95}{\ensuremath{M\boxempty N}}$ is the (inclusionwise)
\emph{largest} matroid on $\scalebox{0.95}{\ensuremath{E\left(M\right)\oplus E\left(N\right)}}$
with that property. Thus, $\scalebox{0.95}{\ensuremath{M\boxempty N}}$
is a universal object. This can be better understood with the notion
of attaché operation over matroid polytopes, cf. Subsection \ref{subsec:Attache-2}.
Also, it can be seen easily that $\scalebox{0.95}{\ensuremath{E\left(M\right)\oplus\overline{\emptyset}_{N}}}$
is a flat of $\scalebox{0.95}{\ensuremath{M\boxempty N}}$. The binary
operation of free product is \emph{not} commutative. But, it is associative,
that is, for matroids $\scalebox{0.95}{\ensuremath{M_{1}}}$, $\scalebox{0.95}{\ensuremath{M_{2}}}$,
and $\scalebox{0.95}{\ensuremath{M_{3}}}$: 
\[
\scalebox{0.95}{\ensuremath{\left(M_{1}\boxempty M_{2}\right)\boxempty M_{3}=M_{1}\boxempty\left(M_{2}\boxempty M_{3}\right)}}.
\]
 One can also deduce the associativity from the universal property
with Lemma \ref{lem:commuting}. Later in Section \ref{sec:Hyperplanes},
we will use the free product of matroids to show that one obtains
a matroid subdivision of a hypersimplex by cutting it with a hyperplane
of the form $\scalebox{0.95}{\ensuremath{\left\{ \sum_{i\in A}x_{i}=\rho\right\} }}$
for a fixed integer $\scalebox{0.95}{\ensuremath{\rho>1}}$ (Lemma
\ref{lem:Cutting-1}).

\section{\label{sec:Polyhedral Geometry of Matroids}Polyhedral Geometry of
Matroids}

A matroid polytope is a convex hull of indicator vectors of sets.
Here, we mean by matroid polytopes base, independence, spanning-set,
and any other polytopes associated with a matroid. A spanning-set
polytope is expressed as an involution of an independence polytope,
 and we  pay particular attention to independence and base polytopes.
The objective of this section is to investigate the face structure
of a matroid polytope and its face lattice.

\subsection{Basic terminology}

For any nonempty finite set $\scalebox{0.95}{\ensuremath{S}}$, we
denote by $\scalebox{0.95}{\ensuremath{\mathbb{R}^{S}}}$ the product
of $\scalebox{0.95}{\ensuremath{\left|S\right|}}$ copies of $\scalebox{0.95}{\ensuremath{\mathbb{R}}}$,
labelled by the elements of $\scalebox{0.95}{\ensuremath{S}}$, one
for each, where $\scalebox{0.95}{\ensuremath{\mathbb{R}^{S}}}$ can
also be understood as the space of all functions $\scalebox{0.95}{\ensuremath{S\rightarrow\mathbb{R}}}$.
A convex polytope in $\scalebox{0.95}{\ensuremath{\mathbb{R}^{S}}}$
is a convex hull of a finite number of points in $\scalebox{0.95}{\ensuremath{\mathbb{R}^{S}}}$,
denoted by $\scalebox{0.95}{\ensuremath{\mathrm{conv}\,\mathcal{A}}}$,
where $\scalebox{0.95}{\ensuremath{\mathcal{A}}}$ is the collection
of those points.
\begin{defn}
\label{def:indicating vector}Fix a finite set $\scalebox{0.95}{\ensuremath{S}}$.
For $\scalebox{0.95}{\ensuremath{A\in2^{S}}}$, its \textbf{indicator
vector} is defined as a vector $\scalebox{0.95}{\ensuremath{1^{A}\in\mathbb{R}^{S}}}$
whose $i$-th entry is $1$ if $\scalebox{0.95}{\ensuremath{i\in A}}$
and $0$ otherwise. In particular, we denote $\scalebox{0.95}{\ensuremath{\mathbbm{1}:=1^{S}}}$.
The \textbf{0/1-polytope} $\scalebox{0.95}{\ensuremath{\mathrm{P}_{\mathcal{A}}}}$
of a nonempty subcollection $\scalebox{0.95}{\ensuremath{\mathcal{A}\subseteq2^{S}}}$
is defined as $\scalebox{0.95}{\ensuremath{\mathrm{P}_{\mathcal{A}}:=\mathrm{conv}\left(1^{A}:A\in\mathcal{A}\right)\subset\mathbb{R}^{S}}}$.
Denote $\scalebox{0.95}{\ensuremath{\mathbbm{1}-\mathrm{P}_{\mathcal{A}}:=\left\{ \mathbbm{1}-x:x\in\mathrm{P}_{\mathcal{A}}\right\} }}$,
then 
\[
\scalebox{0.95}{\ensuremath{\mathbbm{1}-\mathrm{P}_{\mathcal{A}}=\mathrm{conv}\left(\mathbbm{1}-1^{A}:A\in\mathcal{A}\right)}}.
\]
 We call this involution the \textbf{dual polytope} of $\scalebox{0.95}{\ensuremath{\mathrm{P}_{\mathcal{A}}}}$.
\end{defn}

\begin{rem}
\label{rem:identify}No vertex of a 0/1-polytope $\scalebox{0.95}{\ensuremath{\mathrm{P}_{\mathcal{A}}}}$
is a convex combination of the other vertices; hence the number of
vertices of $\scalebox{0.95}{\ensuremath{\mathrm{P}_{\mathcal{A}}}}$
is $\scalebox{0.95}{\ensuremath{\left|\mathcal{A}\right|}}$. Further,
$\scalebox{0.95}{\ensuremath{\mathrm{P}_{\mathcal{A}}=\mathrm{P}_{\mathcal{A}'}}}$
if and only if $\scalebox{0.95}{\ensuremath{\mathcal{A}=\mathcal{A}'}}$,
 and we  identify a subcollection of $\scalebox{0.95}{\ensuremath{2^{S}}}$
with its 0/1-polytope.
\end{rem}

\begin{defn}
For $\scalebox{0.95}{\ensuremath{A,B\in2^{S}}}$, the \textbf{distance}
between them is defined as
\begin{equation}
\scalebox{0.95}{\ensuremath{d\left(A,B\right):=\tfrac{1}{2}\left|A\,\triangle\,B\right|}}\label{eq:metric on U(n,n)}
\end{equation}
 where $\scalebox{0.95}{\ensuremath{A\,\triangle\,B}}$ denotes the
symmetric difference $\scalebox{0.95}{\ensuremath{A\cup B-A\cap B}}$.
We write $\scalebox{0.95}{\ensuremath{d\left(1^{A},1^{B}\right)}}$
for $\scalebox{0.95}{\ensuremath{d(A,B)}}$ under the identification
of Remark \ref{rem:identify}. If $\scalebox{0.95}{\ensuremath{A\neq B}}$,
then $\scalebox{0.95}{\ensuremath{d\left(A,B\right)}}$ is called
the \textbf{length} of the line segment $\scalebox{0.95}{\ensuremath{\overline{1^{A}1^{B}}:=\mathrm{conv}\left(1^{A},1^{B}\right)}}$.
\end{defn}

For the uniform matroid $\scalebox{0.95}{\ensuremath{U_{S}^{k}}}$,
the 0/1-polytope of its base collection is called the $\scalebox{0.95}{\ensuremath{\left(k,S\right)}}$-\textbf{hypersimplex}\footnote{Gelfand and Serganova named the 0/1-polytope of the base collection
of a $\scalebox{0.95}{\ensuremath{\left(k,S\right)}}$-matroid an
\emph{$\scalebox{0.95}{\ensuremath{\left(S,k\right)}}$-hypersimplex}
\cite[Section 2.2]{GS87}, which we call a base polytope; see Definition
\ref{def:matroid-poly}.} denoted by $\scalebox{0.95}{\ensuremath{\Delta_{S}^{k}}}$ or $\scalebox{0.95}{\ensuremath{\Delta\left(k,S\right)}}$,
where we prefer the former notation. We call a 0/1-polytope in $\scalebox{0.95}{\ensuremath{\Delta_{S}^{k}}}$
a $\scalebox{0.95}{\ensuremath{\left(k,S\right)}}$\textbf{-polytope}.

Let $\scalebox{0.95}{\ensuremath{\mathrm{P}_{\mathcal{A}}}}$ with
$\scalebox{0.95}{\ensuremath{\mathcal{A}\subseteq2^{S}}}$ be a 0/1-polytope.
We call $\scalebox{0.95}{\ensuremath{\mathrm{P}_{\mathcal{A}}}}$
\textbf{loopless} if it is not contained in a coordinate hyperplane
of $\scalebox{0.95}{\ensuremath{\mathbb{R}^{S}}}$ and call it \textbf{relevant}
if it is loopless and its dual polytope is also loopless. If $\scalebox{0.95}{\ensuremath{\mathcal{A}=\mathcal{B}\left(M\right)}}$
for some matroid $\scalebox{0.95}{\ensuremath{M}}$, one has $\scalebox{0.95}{\ensuremath{\mathbbm{1}-\mathrm{P}_{\mathcal{A}}=\mathrm{P}_{\mathcal{A}^{\ast}}}}$
with $\scalebox{0.95}{\ensuremath{\mathcal{A}^{\ast}=\mathcal{B}\left(M^{\ast}\right)}}$.
Hence, $\scalebox{0.95}{\ensuremath{\mathrm{P}_{\mathcal{A}}}}$ is
loopless or relevant if and only if $\scalebox{0.95}{\ensuremath{M}}$
is loopless or relevant, respectively.

If every edge of a 0/1-polytope has length $\scalebox{0.95}{\ensuremath{\le1}}$,
we say that it has the \textbf{edge length property}. The edge length
property with \ref{enu:(I1)} is equivalent to the exchange property
\ref{enu:(I2)}. For a $\scalebox{0.95}{\ensuremath{\left(k,S\right)}}$-polytope,
the edge length property alone is equivalent to the base exchange
property. These facts lead us to the following definitions:
\begin{defn}
\label{def:matroid-poly}Suppose that $\scalebox{0.95}{\ensuremath{\mathrm{P}_{\mathcal{A}}}}$
with $\scalebox{0.95}{\ensuremath{\mathcal{A}\subseteq2^{S}}}$ has
the edge length property.
\begin{enumerate}
\item $\scalebox{0.95}{\ensuremath{\mathrm{P}_{\mathcal{A}}}}$ is called
an \textbf{independence polytope} if it satisfies \ref{enu:(I1)}.
The pair $\scalebox{0.95}{\ensuremath{M=\left(S,\mathcal{A}\right)}}$
is a matroid with $\scalebox{0.95}{\ensuremath{\scalebox{0.95}{\ensuremath{\mathcal{I}\left(M\right)=\mathcal{A}}}}}$,
 and we   denote $\scalebox{0.95}{\ensuremath{\mathrm{IP}_{M}=\mathrm{P}_{\mathcal{A}}}}$.
\item $\scalebox{0.95}{\ensuremath{\mathrm{P}_{\mathcal{A}}}}$ is called
a \textbf{base polytope} if it is a $\scalebox{0.95}{\ensuremath{\left(k,S\right)}}$-polytope.
The pair $\scalebox{0.95}{\ensuremath{M=\left(S,\mathcal{A}\right)}}$
is a matroid with $\scalebox{0.95}{\ensuremath{\scalebox{0.95}{\ensuremath{\mathcal{B}\left(M\right)=\mathcal{A}}}}}$,
 and we   denote $\scalebox{0.95}{\ensuremath{\mathrm{BP}_{M}=\mathrm{P}_{\mathcal{A}}}}$.
\end{enumerate}
We call $\scalebox{0.95}{\ensuremath{\mathrm{IP}_{M}}}$ or $\scalebox{0.95}{\ensuremath{\mathrm{BP}_{M}}}$
\emph{the} independence or base polytope of $\scalebox{0.95}{\ensuremath{M}}$,
respectively, and for $\scalebox{0.95}{\ensuremath{\mathrm{P}=\mathrm{IP}_{M}}}$
or $\scalebox{0.95}{\ensuremath{\mathrm{P}=\mathrm{BP}_{M}}}$, we
denote by $\scalebox{0.95}{\ensuremath{\mathrm{MA}_{\mathrm{P}}}}$
the matroid $\scalebox{0.95}{\ensuremath{M}}$.
\end{defn}

The following proposition is immediate.
\begin{prop}
\label{prop:base-poly-cpx}Every face of a base polytope is again
a base polytope.
\end{prop}

\subsection{Descriptions of independence and base polytopes }

A convex polytope can be expressed as a bounded intersection of finitely
many half spaces. Let $\scalebox{0.95}{\ensuremath{M}}$ be a matroid
on $\scalebox{0.95}{\ensuremath{S}}$. Its independence polytope $\scalebox{0.95}{\ensuremath{\mathrm{IP}_{M}}}$
is the intersection of $\scalebox{0.95}{\ensuremath{\left\{ x\in\mathbb{R}^{S}:x_{i}\ge0\right\} }}$
for all $\scalebox{0.95}{\ensuremath{i\in S}}$ and the following
half spaces:
\begin{equation}
\scalebox{0.95}{\ensuremath{\left\{ x\in\mathbb{R}^{S}:x\left(A\right)\le r\left(A\right)\right\} }}\quad\text{for all }\scalebox{0.95}{\ensuremath{A\in2^{S}}}\label{eq:defining-ineq}
\end{equation}
 where $\scalebox{0.95}{\ensuremath{x\left(A\right)=\sum_{i\in A}x_{i}}}$.
The formula \eqref{eq:defining-ineq} can be replaced by: 
\begin{equation}
\scalebox{0.95}{\ensuremath{\left\{ x\in\mathbb{R}^{S}:x\left(F\right)\le r\left(F\right)\right\} }}\quad\text{for all flats }\scalebox{0.95}{\ensuremath{F\in2^{S}}}.\label{eq:def-ineq-flat}
\end{equation}
 Later in Lemma~\ref{lem:Ineq improved}, \eqref{eq:def-ineq-flat}
further reduces to a minimal system of half spaces. The base polytope
$\scalebox{0.95}{\ensuremath{\mathrm{BP}_{M}}}$ is the intersection
of $\scalebox{0.95}{\ensuremath{\mathrm{IP}_{M}}}$ and $\scalebox{0.95}{\ensuremath{\left\{ x\in\mathbb{R}^{n}:x\left(S\right)=r\left(M\right)\right\} }}$,
which is a face of $\scalebox{0.95}{\ensuremath{\mathrm{IP}_{M}}}$.
Meanwhile, the \emph{dimensions} of $\scalebox{0.95}{\ensuremath{\mathrm{IP}_{M}}}$
and $\scalebox{0.95}{\ensuremath{\mathrm{BP}_{M}}}$ are $\scalebox{0.95}{\ensuremath{\left|S\right|-\left|\bar{\emptyset}\right|}}$
and $\scalebox{0.95}{\ensuremath{\left|S\right|-\kappa\left(M\right)}}$,
respectively. So, a loopless matroid $\scalebox{0.95}{\ensuremath{M}}$
is connected if and only if $\scalebox{0.95}{\ensuremath{\mathrm{BP}_{M}}}$
is full-dimensional, if and only if $\scalebox{0.95}{\ensuremath{\mathrm{BP}_{M}}}$
is a facet of $\scalebox{0.95}{\ensuremath{\mathrm{IP}_{M}}}$.

For a vector $\scalebox{0.95}{\ensuremath{\mathbf{v}=\left(v_{i}\right)_{i\in S}\in\mathbb{R}^{S}}}$,
we denote $\scalebox{0.95}{\ensuremath{\mathbf{v}\left(A\right):={\textstyle \sum_{i\in A}v_{i}}}}$
for $\scalebox{0.95}{\ensuremath{A\in2^{S}}}$.
\begin{defn}
Let $\scalebox{0.95}{\ensuremath{Q}}$ be a face of $\scalebox{0.95}{\ensuremath{\mathrm{P}_{\mathcal{A}}\subset\mathbb{R}^{S}}}$
with $\scalebox{0.95}{\ensuremath{\mathcal{A}\subseteq2^{S}}}$. We
say that $\scalebox{0.95}{\ensuremath{\mathbf{v}\in\mathbb{R}^{S}}}$
is a \textbf{characterizer vector} \textbf{of }$\scalebox{0.95}{\ensuremath{Q}}$\textbf{
in }$\scalebox{0.95}{\ensuremath{\mathrm{P}_{\mathcal{A}}}}$ if,
for a number $\scalebox{0.95}{\ensuremath{c\in\mathbb{R}}}$: 
\[
\scalebox{0.95}{\ensuremath{\mathbf{v}(A)=\left\langle \mathbf{v},1^{A}\right\rangle \le c\quad\mbox{for all }A\in\mathcal{A}}}
\]
 and equality holds precisely when $\scalebox{0.95}{\ensuremath{1^{A}}}$
is a vertex of $\scalebox{0.95}{\ensuremath{Q}}$. Here, $\left\langle \,,\,\right\rangle $
denotes the usual inner product of $\scalebox{0.95}{\ensuremath{\mathbb{R}^{S}}}$.
\end{defn}

\begin{example}
For $\scalebox{0.95}{\ensuremath{\mathrm{P}=\mathrm{IP}_{M}}}$ or
$\scalebox{0.95}{\ensuremath{\mathrm{P}=\mathrm{BP}_{M}}}$, every
indicator vector $\scalebox{0.95}{\ensuremath{1^{A}}}$ of $\scalebox{0.95}{\ensuremath{A\subseteq E\left(M\right)}}$
is a characterizer vector, cf.~\eqref{eq:defining-ineq}. Moreover,
any characterizer vector is \emph{outward-pointing normal} to $\scalebox{0.95}{\ensuremath{\mathrm{P}}}$.
\end{example}

The intersection of two independence polytopes or two base polytopes
is not an independence polytope or a base polytope in general. However,
the following proposition shows that there is a meaningful relationship.
\begin{prop}[{\cite[Corollaries 41.12b,d]{Schrijver}}]
\label{prop:common_polytopes} For any matroids $\scalebox{0.95}{\ensuremath{M}}$
and $\scalebox{0.95}{\ensuremath{N}}$ with $\scalebox{0.95}{\ensuremath{E\left(M\right)=E\left(N\right)}}$,
one has $\scalebox{0.95}{\ensuremath{\mathrm{IP}_{M}\cap\mathrm{IP}_{N}=\mathrm{P}_{M\wedge N}}}$
and $\scalebox{0.95}{\ensuremath{\mathrm{BP}_{M}\cap\mathrm{BP}_{N}=\mathrm{P}_{M\cap N}}}$.
\end{prop}

\subsection{Two-dimensional faces}

The edge length property describes the $\scalebox{0.95}{\ensuremath{1}}$-dimensional
face of $\scalebox{0.95}{\ensuremath{\mathrm{IP}_{M}}}$ and $\scalebox{0.95}{\ensuremath{\mathrm{BP}_{M}}}$
for a matroid $\scalebox{0.95}{\ensuremath{M}}$. In their book \cite{Coxeter},
Borovik, Gelfand, and White classified the two-dimensional faces of
$\scalebox{0.95}{\ensuremath{\mathrm{BP}_{M}}}$ using representation
theory, which played a significant role in their original proof. However,
we present another proof without using representation theory (Proposition
\ref{prop:2-dim faces}). Similarly, the two-dimensional faces of
$\scalebox{0.95}{\ensuremath{\mathrm{IP}_{M}}}$ are classified (Proposition
\ref{prop:2-dim faces-1}) using the same method.
\begin{defn}
Fix $\scalebox{0.95}{\ensuremath{\mathcal{A}\subseteq2^{S}}}$ with
$\scalebox{0.95}{\ensuremath{A,B\in\mathcal{A}}}$. If $\scalebox{0.95}{\ensuremath{A\,\triangle\,B=\left(A\,\triangle\,C\right)\cup\left(B\,\triangle\,C\right)}}$
for some $\scalebox{0.95}{\ensuremath{C\in\mathcal{A}}}$, we say
that $\scalebox{0.95}{\ensuremath{\left(A\,\triangle\,C\right)\cup\left(B\,\triangle\,C\right)}}$
is a \textbf{$\delta$-decomposition }of\textbf{ }$\scalebox{0.95}{\ensuremath{A\,\triangle\,B}}$
in $\scalebox{0.95}{\ensuremath{\mathcal{A}}}$.
\end{defn}

\begin{rem}
\label{rem:Delta-decomp}We have two useful equivalences:
\begin{enumerate}
\item $\scalebox{0.95}{\ensuremath{d\left(A,B\right)\lneq d\left(A,C\right)+d\left(B,C\right)}}$
if and only if $\scalebox{0.95}{\ensuremath{A\,\triangle\,B\subsetneq\left(A\,\triangle\,C\right)\cup\left(B\,\triangle\,C\right)}}$.
\item $\scalebox{0.95}{\ensuremath{A\,\triangle\,B=\left(A\,\triangle\,C\right)\cup\left(B\,\triangle\,C\right)}}$
if and only if $\scalebox{0.95}{\ensuremath{A\cap B\subseteq C\subseteq A\cup B}}$.
\end{enumerate}
\end{rem}

\begin{prop}
\label{prop:Delta-decomp}For $\scalebox{0.95}{\ensuremath{\mathcal{A}\subseteq2^{S}}}$
with $\scalebox{0.95}{\ensuremath{A,B\in\mathcal{A}}}$, if $\scalebox{0.95}{\ensuremath{A\,\triangle\,B}}$
has no nontrivial $\delta$-decomposition in $\scalebox{0.95}{\ensuremath{\mathcal{A}}}$,
then the line segment $\scalebox{0.95}{\ensuremath{\overline{1^{A}1^{B}}}}$
is an edge of $\scalebox{0.95}{\ensuremath{\mathrm{P}_{\mathcal{A}}}}$.
\end{prop}

\begin{proof}
Let $\scalebox{0.95}{\ensuremath{\mathbf{v}:=1^{A\cap B}-1^{\left(A\cup B\right)^{c}}}}$.
Then, $\scalebox{0.95}{\ensuremath{\left\langle \mathbf{v},1^{U}\right\rangle \le\left|A\cap B\right|}}$
for any vertex $\scalebox{0.95}{\ensuremath{1^{U}}}$ of $\scalebox{0.95}{\ensuremath{\mathrm{P}_{\mathcal{A}}}}$
in which equality holds exactly when $\scalebox{0.95}{\ensuremath{A\cap B\subseteq U\subseteq A\cup B}}$.
Let 
\[
\scalebox{0.95}{\ensuremath{Q:=\left\{ \mathbf{u}\in\mathrm{P}_{\mathcal{A}}:\left\langle \mathbf{v},\mathbf{u}\right\rangle =\left|A\cap B\right|\right\} }}.
\]
Then, $\scalebox{0.95}{\ensuremath{Q}}$ is a face of $\scalebox{0.95}{\ensuremath{\mathrm{P}_{\mathcal{A}}}}$
with $\scalebox{0.95}{\ensuremath{1^{A},1^{B}\in Q}}$, and $\scalebox{0.95}{\ensuremath{\mathbf{v}}}$
is a characterizer vector of $\scalebox{0.95}{\ensuremath{Q}}$. So,
if $\scalebox{0.95}{\ensuremath{1^{U}\in Q}}$, then $\scalebox{0.95}{\ensuremath{A\cap B\subseteq U\subseteq A\cup B}}$
and $\scalebox{0.95}{\ensuremath{\left(A\,\triangle\,U\right)\cup\left(B\,\triangle\,U\right)}}$
is a $\delta$-decomposition of $\scalebox{0.95}{\ensuremath{A\,\triangle\,B}}$;
hence, $\scalebox{0.95}{\ensuremath{U=A}}$ or $\scalebox{0.95}{\ensuremath{U=B}}$
by assumption. Thus, $\scalebox{0.95}{\ensuremath{Q=\mathrm{conv}(1^{A},1^{B})=\overline{1^{A}1^{B}}}}$.
\end{proof}
\begin{prop}
\label{prop:Delta-decomp-1}Let $\scalebox{0.95}{\ensuremath{1^{A}}}$
and $\scalebox{0.95}{\ensuremath{1^{B}}}$ be two distinct vertices
of $\scalebox{0.95}{\ensuremath{\mathrm{P}_{\mathcal{A}}}}$ with
$\scalebox{0.95}{\ensuremath{\mathcal{A}\subseteq2^{S}}}$. If $\scalebox{0.95}{\ensuremath{\overline{1^{A}1^{B}}}}$
has the smallest length among all edges, then it is an edge of $\scalebox{0.95}{\ensuremath{\mathrm{P}_{\mathcal{A}}}}$.
\end{prop}

\begin{proof}
Because $\scalebox{0.95}{\ensuremath{\overline{1^{A}1^{B}}}}$ has
the smallest length, $\scalebox{0.95}{\ensuremath{A\,\triangle\,B}}$
has no nontrivial $\delta$-decomposition in $\scalebox{0.95}{\ensuremath{\mathcal{A}}}$,
and $\scalebox{0.95}{\ensuremath{\overline{1^{A}1^{B}}}}$ is an edge
of $\scalebox{0.95}{\ensuremath{\mathrm{P}_{\mathcal{A}}}}$ by Proposition
\ref{prop:Delta-decomp}.
\end{proof}
Now, we prove the theorem of Borovik, Gelfand, and White.
\begin{prop}[{\cite[Theorem 1.12.8]{Coxeter}}]
\label{prop:2-dim faces} Every $2$-dimensional face of a base polytope
is either a regular triangle or a square.
\end{prop}

\begin{proof}
Let $\scalebox{0.95}{\ensuremath{Q}}$ be a $2$-dimensional face
of $\scalebox{0.95}{\ensuremath{\mathrm{BP}_{M}\subseteq\Delta_{S}^{k}}}$.
Let $\scalebox{0.95}{\ensuremath{\overline{1^{A}1^{B}}}}$ and $\scalebox{0.95}{\ensuremath{\overline{1^{B}1^{C}}}}$
be two distinct edges of $\scalebox{0.95}{\ensuremath{Q}}$. Then,
$\scalebox{0.95}{\ensuremath{d\left(A,C\right)=1,2}}$ is implied
by the following:
\[
\scalebox{0.95}{\ensuremath{d\left(A,C\right)\le d\left(A,B\right)+d\left(B,C\right)=1+1=2}}.
\]
If $\scalebox{0.95}{\ensuremath{d\left(A,C\right)=1}}$, then $\scalebox{0.95}{\ensuremath{\overline{1^{A}1^{C}}}}$
is an edge of $\scalebox{0.95}{\ensuremath{Q}}$ by Proposition \ref{prop:Delta-decomp-1}.
So, $\scalebox{0.95}{\ensuremath{\overline{1^{A}1^{B}}}}$, $\scalebox{0.95}{\ensuremath{\overline{1^{B}1^{C}}}}$,
and $\scalebox{0.95}{\ensuremath{\overline{1^{A}1^{C}}}}$ are three
edges of $\scalebox{0.95}{\ensuremath{Q}}$ with $\scalebox{0.95}{\ensuremath{\dim Q=2}}$;
hence, $\scalebox{0.95}{\ensuremath{Q=\mathrm{conv}\left(1^{A},1^{B},1^{C}\right)}}$.
This is a regular triangle. So, assume $\scalebox{0.95}{\ensuremath{Q}}$
is a non-triangle with $\scalebox{0.95}{\ensuremath{d\left(A,C\right)=2}}$.
Then, observe that $\scalebox{0.95}{\ensuremath{1^{A}+1^{C}}}$ is
a characterizer vector of $\scalebox{0.95}{\ensuremath{Q}}$, and
any vertex of $\scalebox{0.95}{\ensuremath{Q}}$ gives a $\delta$-decomposition
of $\scalebox{0.95}{\ensuremath{A\,\triangle\,C}}$. If $\scalebox{0.95}{\ensuremath{1^{D}}}$
is a new vertex of $\scalebox{0.95}{\ensuremath{Q}}$, then $\scalebox{0.95}{\ensuremath{d\left(A,C\right)=d\left(A,D\right)+d\left(C,D\right)}}$
and $\scalebox{0.95}{\ensuremath{d\left(A,D\right)=1}}$.  Therefore,
$\scalebox{0.95}{\ensuremath{\overline{1^{A}1^{D}}}}$ is an edge
of $\scalebox{0.95}{\ensuremath{Q}}$ and $\scalebox{0.95}{\ensuremath{Q=\mathrm{conv}\left(1^{A},1^{B},1^{C},1^{D}\right)}}$,
which is a square.
\end{proof}
\begin{defn}
For a nonnegative integer $c$, the \textbf{\emph{c}}\textbf{-truncation
matroid} $\scalebox{0.95}{\ensuremath{M^{\left(\le c\right)}}}$ of
$\scalebox{0.95}{\ensuremath{M}}$ or simply a \textbf{truncation}
of $M$ is the matroid defined by a submodular rank function $\scalebox{0.95}{\ensuremath{A\mapsto\min\left(c,r\left(A\right)\right)}}$.
Here, $\scalebox{0.95}{\ensuremath{M^{\left(\le k\right)}=M}}$ for
$\scalebox{0.95}{\ensuremath{k=r\left(M\right)}}$.
\end{defn}

\begin{prop}
\label{prop:2-dim faces-1}For any matroid $\scalebox{0.95}{\ensuremath{M}}$
of rank $k$, the $\scalebox{0.95}{\ensuremath{2}}$-dimensional faces
of its independence polytope $\scalebox{0.9}{\ensuremath{\mathrm{IP}_{M}}}$
are classified into the following four cases, as illustrated in Figure~\ref{fig:2-dim faces-1};
in this figure, the angles are Euclidean, and the numbers in parentheses
indicate the ranks of the vertices.
\begin{enumerate}
\item \label{enu:2dim-faces-1}A regular triangle with a side length of
$1$, contained in $\scalebox{0.9}{\ensuremath{\mathrm{BP}_{M}}}$.
\item \label{enu:2dim-faces-2}A square with a side length of $1$, contained
in $\scalebox{0.9}{\ensuremath{\mathrm{BP}_{M}}}$.
\item \label{enu:2dim-faces-3}An isosceles right triangle with side lengths
of $\frac{1}{2},\frac{1}{2},1$, connecting two base polytopes $\scalebox{0.9}{\ensuremath{\mathrm{BP}_{M^{(\le j)}}}}$
and $\scalebox{0.9}{\ensuremath{\mathrm{BP}_{M^{(\le j+1)}}}}$ for
some $\scalebox{0.9}{\ensuremath{j+1\in\left[k\right]}}$.
\item \label{enu:2dim-faces-4}A square with a side length of $\frac{1}{2}$,
connecting three base polytopes $\scalebox{0.9}{\ensuremath{\mathrm{BP}_{M^{(\le j-1)}}}}$,
$\scalebox{0.9}{\ensuremath{\mathrm{BP}_{M^{(\le j)}}}}$, and $\scalebox{0.9}{\ensuremath{\mathrm{BP}_{M^{(\le j+1)}}}}$,
with two vertices contained in the middle base polytope $\scalebox{0.9}{\ensuremath{\mathrm{BP}_{M^{(\le j)}}}}$.\vspace{-0.5cm}
\end{enumerate}
\begin{spacing}{0}
\noindent 
\begin{figure}[H]
\noindent \centering{}\noindent \begin{center}
\def\size{0.7}


\par\end{center}\vspace{-0.2cm}
\caption{\label{fig:2-dim faces-1}The 2-dimensional face of independence polytope.}
\end{figure}
\end{spacing}
\end{prop}

\begin{proof}
\eqref{enu:2dim-faces-1} and \eqref{enu:2dim-faces-2} are Proposition
\ref{prop:2-dim faces}, and \eqref{enu:2dim-faces-3} and \eqref{enu:2dim-faces-4}
are similarly proved.
\end{proof}

\subsection{\label{subsec:Attache-2}Attaché operation over matroid polytopes}

Let $\scalebox{0.95}{\ensuremath{\mathrm{P}=\mathrm{IP}_{M}}}$ or
$\scalebox{0.95}{\ensuremath{\mathrm{P}=\mathrm{BP}_{M}}}$. For a
subset $\scalebox{0.95}{\ensuremath{A\subseteq E\left(M\right)}}$,
define 
\[
\scalebox{0.95}{\ensuremath{{\rm P}\left(A\right)=\left\{ x\in{\rm P}:x\left(A\right)=r\left(A\right)\right\} }}.
\]
The formula \eqref{eq:defining-ineq} tells that $\scalebox{0.95}{\ensuremath{{\rm P}\left(A\right)}}$
is a face of $\scalebox{0.95}{\ensuremath{\mathrm{P}}}$, which is
\emph{nonempty} since $\scalebox{0.95}{\ensuremath{M|_{A}}}$ has
a base $\scalebox{0.95}{\ensuremath{B}}$ and $\scalebox{0.95}{\ensuremath{\mathrm{P}}}$
contains $\scalebox{0.95}{\ensuremath{1^{B}}}$. Obtaining the nonempty
face $\scalebox{0.95}{\ensuremath{{\rm P}\left(A\right)}}$ from $\scalebox{0.95}{\ensuremath{{\rm P}}}$
is said to be the \textbf{attaché operation by} $\scalebox{0.95}{\ensuremath{A}}$.
\begin{rem}
For technical reasons, one can have a factor of the form $\scalebox{0.95}{\ensuremath{\mathrm{P}_{\emptyset}}}$
in a Cartesian product of polytopes, which should be ignored and erased
in the product. For instance, $\scalebox{0.95}{\ensuremath{\mathrm{P}_{\emptyset}\times\mathrm{P}_{M}}}$
or $\scalebox{0.95}{\ensuremath{\mathrm{P}_{M}\times\mathrm{P}_{\emptyset}}}$
is understood as $\scalebox{0.95}{\ensuremath{\mathrm{P}_{M}}}$,
cf. $\scalebox{0.95}{\ensuremath{\emptyset\oplus M=M\oplus\emptyset=M}}$.
\end{rem}

\begin{lem}
\label{lem:Faces of BP}Let $\scalebox{0.95}{\ensuremath{M}}$ be
a loopless matroid. For any $\scalebox{0.95}{\ensuremath{A\subseteq E\left(M\right)}}$,
one has 
\[
\scalebox{0.95}{\ensuremath{{\rm BP}_{M}(A)={\rm BP}_{M/A}\times{\rm BP}_{M|_{A}}}}\quad\text{and}\quad\scalebox{0.95}{\ensuremath{{\rm IP}_{M}(A)={\rm IP}_{M/A}\times{\rm BP}_{M|_{A}}}}.
\]
 Here, $\scalebox{0.95}{\ensuremath{\mathrm{P}\left(A\right)}}$ for
$\scalebox{0.95}{\ensuremath{\mathrm{P}=\mathrm{IP}_{M}}}$ or $\scalebox{0.95}{\ensuremath{\mathrm{P}=\mathrm{BP}_{M}}}$
is loopless if and only if $\scalebox{0.95}{\ensuremath{A}}$ is a
flat.
\end{lem}

\begin{proof}
The following equivalent statements show $\scalebox{0.95}{\ensuremath{{\rm IP}_{M}(A)={\rm IP}_{M/A}\times{\rm BP}_{M|_{A}}}}$:
\begin{itemize}
\item The indicator vector $1^{J}$ of $\scalebox{0.95}{\ensuremath{J\subseteq E\left(M\right)}}$
is a point of $\scalebox{0.95}{\ensuremath{{\rm IP}_{M}(A)}}$.
\item $J$ is an independent set of $M$ and $\scalebox{0.95}{\ensuremath{\left|J\cap A\right|=1^{J}\left(A\right)=r\left(A\right)}}$.
\item $\scalebox{0.95}{\ensuremath{J\cap A}}$ is a base of $\scalebox{0.95}{\ensuremath{M|_{A}}}$,
and $\scalebox{0.95}{\ensuremath{J-A}}$ is an independent set of
$\scalebox{0.95}{\ensuremath{M/A}}$.
\item $1^{J}$ is a vertex of $\scalebox{0.95}{\ensuremath{{\rm BP}_{M|_{A}}\times\mathrm{IP}_{M/A}}}$.
\end{itemize}
Similarly, $\scalebox{0.95}{\ensuremath{{\rm BP}_{M}(A)={\rm BP}_{M/A}\times{\rm BP}_{M|_{A}}}}$.
Now that $\scalebox{0.95}{\ensuremath{M}}$ is loopless, so is $\scalebox{0.95}{\ensuremath{M|_{A}}}$,
and the following are equivalent.
\begin{itemize}
\item $\scalebox{0.95}{\ensuremath{\mathrm{BP}_{M}(A)}}$ and $\scalebox{0.95}{\ensuremath{\mathrm{IP}_{M}(A)}}$
are contained in a coordinate hyperplane.
\item $\scalebox{0.95}{\ensuremath{\mathrm{BP}_{M/A}}}$ and $\scalebox{0.95}{\ensuremath{\mathrm{IP}_{M/A}}}$
are contained in a coordinate hyperplane.
\item $\scalebox{0.95}{\ensuremath{M/A}}$ has a loop.
\item $\scalebox{0.95}{\ensuremath{A}}$ is a non-flat of $\scalebox{0.95}{\ensuremath{M}}$.
\end{itemize}
The proof is done.
\end{proof}
For $\scalebox{0.95}{\ensuremath{A_{1},\dots,A_{m}\subseteq E\left(M\right)}}$,
$\scalebox{0.95}{\ensuremath{\left(\cdots\left(\left(\mathrm{P}\left(A_{1}\right)\right)\left(A_{2}\right)\right)\cdots\right)\left(A_{m}\right)}}$
is well-defined due to Lemma \ref{lem:Faces of BP}, and we  write
$\scalebox{0.95}{\ensuremath{\mathrm{P}\left(A_{1}\right)\left(A_{2}\right)\cdots\left(A_{m}\right)}}$
for it. Here, $\scalebox{0.95}{\ensuremath{\mathrm{P}\left(A_{1}\right)\left(A_{2}\right)\cdots\left(A_{i}\right)}}$
for $\scalebox{0.95}{\ensuremath{i\in\left[m\right]}}$ is a nonempty
face of $\scalebox{0.95}{\ensuremath{\mathrm{P}\left(A_{1}\right)\left(A_{2}\right)\cdots\left(A_{i-1}\right)}}$
with $\scalebox{0.95}{\ensuremath{\mathrm{P}\left(A_{0}\right):=\mathrm{P}}}$.
In particular 
\[
\scalebox{0.95}{\ensuremath{\mathrm{BP}_{M}\left(A_{1}\right)\left(A_{2}\right)\cdots\left(A_{m}\right)=\mathrm{BP}_{M\left(A_{1}\right)\left(A_{2}\right)\cdots\left(A_{m}\right)}}}.
\]
 Moreover, write $\scalebox{0.95}{\ensuremath{{\displaystyle M\left(A_{1}\right)\left(A_{2}\right)\cdots\left(A_{m}\right)=M/\left(A_{1}\cup\cdots\cup A_{m}\right)\oplus M'}}}$,
then 
\begin{align}
\scalebox{0.95}{\ensuremath{\mathrm{IP}_{M}\left(A_{1}\right)\left(A_{2}\right)\cdots\left(A_{m}\right)}} & =\scalebox{0.95}{\ensuremath{\mathrm{IP}_{M/\left(A_{1}\cup\cdots\cup A_{m}\right)}\times\mathrm{BP}_{M'}}}\nonumber \\
 & \subseteq\scalebox{0.95}{\ensuremath{\mathrm{IP}_{M/\left(A_{1}\cup\cdots\cup A_{m}\right)}\times\mathrm{IP}_{M'}}}\label{eq:equality}\\
 & =\scalebox{0.95}{\ensuremath{\mathrm{IP}_{M\left(A_{1}\right)\left(A_{2}\right)\cdots\left(A_{m}\right)}}}.\nonumber 
\end{align}
Here, the inclusion of \eqref{eq:equality} is strict if and only
if $\scalebox{0.95}{\ensuremath{M'}}$ has nonzero rank. \smallskip{}

For a loopless matroid $\scalebox{0.95}{\ensuremath{M}}$, every loopless
face of $\scalebox{0.95}{\ensuremath{\mathrm{P}=\mathrm{IP}_{M}}}$
or $\scalebox{0.95}{\ensuremath{\mathrm{P}=\mathrm{BP}_{M}}}$ is
obtained by recursive attaché operations. If $\scalebox{0.95}{\ensuremath{\mathrm{P}'}}$
is a face of $\scalebox{0.95}{\ensuremath{\mathrm{P}}}$ and $\scalebox{0.95}{\ensuremath{i\in E\left(M\right)}}$
is \emph{not} a coloop, then $\scalebox{0.95}{\ensuremath{\mathrm{P}'\cap\left\{ x_{i}=0\right\} =\mathrm{P}'\left(E\left(M\right)-\left\{ i\right\} \right)}}$,
which is a nonempty face of $\scalebox{0.95}{\ensuremath{\mathrm{P}}}$.
Thus, for a general matroid $\scalebox{0.95}{\ensuremath{M}}$, every
face of $\scalebox{0.95}{\ensuremath{\mathrm{P}=\mathrm{IP}_{M}}}$
or $\scalebox{0.95}{\ensuremath{\mathrm{P}=\mathrm{BP}_{M}}}$ is
obtained from $\scalebox{0.95}{\ensuremath{\mathrm{P}}}$ by recursive
attaché operations, and the following notion is well-defined:
\begin{defn}
For $\scalebox{0.95}{\ensuremath{A_{1},\dots,A_{m}\subseteq E\left(M\right)}}$,
the matroid $\scalebox{0.95}{\ensuremath{{\displaystyle M\left(A_{1}\right)\cdots\left(A_{m}\right)}}}$
is called a \textbf{face matroid}.
\end{defn}

The following corollary and proposition are immediate.
\begin{cor}
\label{cor:face-IP}Any face of an independence polytope is a Cartesian
product of polytopes, each of which is a base polytope or an independence
polytope.
\end{cor}

\begin{prop}
For $\scalebox{0.95}{\ensuremath{A_{1},\dots,A_{m}\subseteq E\left(M\right)}}$,
the following are equivalent.
\begin{enumerate}
\item $\scalebox{0.95}{\ensuremath{{\displaystyle A_{1}\cup\cdots\cup A_{m}=E\left(M\right)}}}$.
\item $\scalebox{0.95}{\ensuremath{\mathrm{IP}_{M}\left(A_{1}\right)\left(A_{2}\right)\cdots\left(A_{m}\right)}}=\scalebox{0.95}{\ensuremath{\mathrm{BP}_{M\left(A_{1}\right)\left(A_{2}\right)\cdots\left(A_{m}\right)}}}.$
\end{enumerate}
\end{prop}

Thus, the correspondence from base polytopes (or independence polytopes
and their faces) to matroids is a \emph{functor}, with attaché operations
being morphisms in both categories, where attaché operations by $\scalebox{0.95}{\ensuremath{\emptyset}}$
are identity morphisms.

Attaché operations do \emph{not} commute in general, but they do when
``nonemptiness'' is ensured, see Lemma \ref{lem:commuting} and
Theorem \ref{thm:face-intersection}.

\subsection{\label{subsec:Intersection-1}Intersection of faces, I}
\begin{lem}
\label{lem:commuting}Let $\scalebox{0.95}{\ensuremath{F}}$ and $\scalebox{0.95}{\ensuremath{L}}$
be subsets of $\scalebox{0.95}{\ensuremath{E\left(M\right)}}$. The
following are equivalent.
\begin{enumerate}
\item \label{enu:mod-pair}$\scalebox{0.95}{\ensuremath{\left\{ F,L\right\} }}$
is a modular pair of $\scalebox{0.95}{\ensuremath{M}}$.
\item \label{enu:nonempty}$\scalebox{0.95}{\ensuremath{M\left(F\right)\cap M\left(L\right)\neq\emptyset}}$.
\item \label{enu:all-equal}$\scalebox{0.95}{\ensuremath{M\left(F\right)\cap M\left(L\right)=M\left(F\cup L\right)\left(F\cap L\right)=M\left(F\right)\left(L\right)=M\left(L\right)\left(F\right)}}$. 
\item \label{enu:two-equal}Two of $\scalebox{0.95}{\ensuremath{M(F)\cap M(L)}}$,
$\scalebox{0.95}{\ensuremath{M\left(F\cup L\right)\left(F\cap L\right)}}$,
$\scalebox{0.95}{\ensuremath{M\left(F\right)\left(L\right)}}$, and
$\scalebox{0.95}{\ensuremath{M\left(L\right)\left(F\right)}}$ are
the same matroid.
\end{enumerate}
\end{lem}

\begin{proof}
Obviously, \eqref{enu:all-equal}$\Rightarrow$\eqref{enu:two-equal}
is true, and \eqref{enu:all-equal}$\Rightarrow$\eqref{enu:nonempty}
is also clear because if \eqref{enu:all-equal} holds, then $\scalebox{0.95}{\ensuremath{r\left(M\left(F\right)\cap M\left(L\right)\right)=r\left(M\right)}}$,
and $\scalebox{0.95}{\ensuremath{M\left(F\right)\cap M\left(L\right)}}$
is nonempty. To prove \eqref{enu:nonempty}$\Rightarrow$\eqref{enu:mod-pair},
suppose $\scalebox{0.95}{\ensuremath{M\left(F\right)\cap M\left(L\right)\neq\emptyset}}$.
Then, for any $\scalebox{0.95}{\ensuremath{B\in M\left(F\right)\cap M\left(L\right)}}$:
\begin{align*}
\scalebox{0.95}{\ensuremath{r\left(F\right)+r\left(L\right)}} & =\scalebox{0.95}{\ensuremath{\left|B\cap F\right|+\left|B\cap L\right|}}\\
 & =\scalebox{0.95}{\ensuremath{\left|B\cap F\cap L\right|+\left|B\cap\left(F\cup L\right)\right|}}\\
 & \le\scalebox{0.95}{\ensuremath{r\left(F\cap L\right)+r\left(F\cup L\right)}}
\end{align*}
 where equality holds by submodularity of $r$, and $\scalebox{0.95}{\ensuremath{\left\{ F,L\right\} }}$
is a modular pair.  Thus, \eqref{enu:all-equal}$\Rightarrow$\eqref{enu:nonempty}$\Rightarrow$\eqref{enu:mod-pair}
holds. To prove \eqref{enu:mod-pair}$\Rightarrow$\eqref{enu:all-equal},
suppose that  $\scalebox{0.95}{\ensuremath{\left\{ F,L\right\} }}$
is a modular pair. Then, $\scalebox{0.95}{\ensuremath{r_{M\left(F\right)}\left(L\right)=r\left(L\right)}}$
and by Lemma~\ref{lem:Faces of BP}:
\begin{align*}
\scalebox{0.95}{\ensuremath{\mathrm{BP}_{M\left(F\right)}\cap\mathrm{BP}_{M\left(L\right)}}} & =\scalebox{0.95}{\ensuremath{\left\{ x\in\mathrm{BP}_{M}:x\left(F\right)=r\left(F\right),x\left(L\right)=r\left(L\right)\right\} }}\\
 & =\scalebox{0.95}{\ensuremath{\left\{ x\in\mathrm{BP}_{M\left(F\right)}:x\left(L\right)=r_{M\left(F\right)}\left(L\right)\right\} }}\\
 & =\scalebox{0.95}{\ensuremath{\mathrm{BP}_{M\left(F\right)\left(L\right)}}}.
\end{align*}
 So, $\scalebox{0.95}{\ensuremath{M\left(F\right)\cap M\left(L\right)=M\left(F\right)\left(L\right)}}$
and $\scalebox{0.95}{\ensuremath{M\left(F\right)\left(L\right)=M\left(L\right)\left(F\right)}}$
by symmetry where 
\[
\scalebox{0.95}{\ensuremath{M\left(F\right)\left(L\right)=\left(M|_{F\cap L}\right)\oplus\left(M|_{F}/\left(F\cap L\right)\right)\oplus\left(M|_{F\cup L}/F\right)\oplus\left(M/\left(F\cup L\right)\right)}}.
\]
 Then, $\scalebox{0.95}{\ensuremath{M|_{F}/\left(F\cap L\right)=M|_{F\cup L}/L}}$.
Further, by submodularity, $\scalebox{0.95}{\ensuremath{r\left(F\cup L\right)-r\left(F\cap L\right)}}$
equals $\scalebox{0.95}{\ensuremath{\left(r\left(L\right)-r\left(F\cap L\right)\right)+\left(r\left(F\right)-r\left(F\cap L\right)\right)}}$,
which tells that $\scalebox{0.95}{\ensuremath{L-F}}$ and $\scalebox{0.95}{\ensuremath{F-L}}$
are separators of $\scalebox{0.95}{\ensuremath{M|_{F\cup L}/\left(F\cap L\right)}}$,
and 
\begin{align*}
\scalebox{0.95}{\ensuremath{M|_{F\cup L}/\left(F\cap L\right)}} & =\scalebox{0.95}{\ensuremath{\left(M|_{F}/\left(F\cap L\right)\right)\oplus\left(M|_{L}/\left(F\cap L\right)\right)}}\\
 & =\scalebox{0.95}{\ensuremath{\left(M|_{F\cup L}/F\right)\oplus\left(M|_{F\cup L}/L\right)}}.
\end{align*}
 So, $\scalebox{0.95}{\ensuremath{M\left(F\cup L\right)\left(F\cap L\right)=M\left(F\right)\left(L\right)}}$
and \eqref{enu:all-equal} holds.  Thus, \eqref{enu:mod-pair}$\Leftrightarrow$\eqref{enu:nonempty}$\Leftrightarrow$\eqref{enu:all-equal}$\Rightarrow$\eqref{enu:two-equal}
holds. Finally, to prove \eqref{enu:two-equal}$\Rightarrow$\eqref{enu:mod-pair},
suppose \eqref{enu:two-equal}. If $\scalebox{0.95}{\ensuremath{M(F)\cap M(L)}}$
equals one of the other three, then it is a nonempty matroid of rank
$r(M)$, and \eqref{enu:mod-pair} holds by \eqref{enu:nonempty}$\Rightarrow$\eqref{enu:mod-pair}.
If two of $\scalebox{0.95}{\ensuremath{M\left(F\cup L\right)\left(F\cap L\right)}}$,
$\scalebox{0.95}{\ensuremath{M\left(F\right)\left(L\right)}}$, and
$\scalebox{0.95}{\ensuremath{M\left(L\right)\left(F\right)}}$ are
the same, then equating their ranks shows that $\scalebox{0.95}{\ensuremath{\left\{ F,L\right\} }}$
is a modular pair of $\scalebox{0.95}{\ensuremath{M}}$, and \eqref{enu:mod-pair}
holds.
\end{proof}
If $\scalebox{0.95}{\ensuremath{F}}$ and $\scalebox{0.95}{\ensuremath{L}}$
are subsets of $\scalebox{0.95}{\ensuremath{E(M)}}$ with $\scalebox{0.95}{\ensuremath{F\subseteq L}}$
or $\scalebox{0.95}{\ensuremath{F\supseteq L}}$, then $\left\{ F,L\right\} $
is a modular pair, and $\scalebox{0.95}{\ensuremath{M(F)\cap M(L)\neq\emptyset}}$
by the above lemma. Moreover, if $\scalebox{0.95}{\ensuremath{\left(F_{1},\dots,F_{m}\right)}}$
is an increasing sequence of subsets of $\scalebox{0.95}{\ensuremath{E(M)}}$,
then for any permutation $\scalebox{0.95}{\ensuremath{\pi\in\mathfrak{S}_{m}}}$:
\[
\scalebox{0.95}{\ensuremath{M\left(F_{1}\right)\cdots\left(F_{m}\right)}}=\scalebox{0.95}{\ensuremath{M\left(F_{\pi\left(1\right)}\right)\cdots\left(F_{\pi\left(m\right)}\right)}}
\]
 where $\mathfrak{S}_{m}$ denotes the group of permutations on $\scalebox{0.95}{\ensuremath{\left\{ 1,\dots,m\right\} }}$.
We can say further:
\begin{thm}
\label{thm:face-matroids}For $\scalebox{0.95}{\ensuremath{F_{1},\dots,F_{m}\subseteq E\left(M\right)}}$,
the following are equivalent.
\begin{enumerate}
\item \label{enu:intersect-1}$\scalebox{0.95}{\ensuremath{\bigcap_{i\in\left[m\right]}M\left(F_{i}\right)\neq\emptyset}}$.
\item \label{enu:commute-1}$\scalebox{0.95}{\ensuremath{M\left(F_{\pi\left(1\right)}\right)\cdots\left(F_{\pi\left(m\right)}\right)}}$
for all $\scalebox{0.95}{\ensuremath{\pi\in\mathfrak{S}_{m}}}$ are
the same matroid.
\item \label{enu:face-all}$\scalebox{0.95}{\ensuremath{M\left(F_{\pi\left(1\right)}\right)\cdots\left(F_{\pi\left(m\right)}\right)=\bigcap_{i\in\left[m\right]}M\left(F_{i}\right)}}$
for all $\scalebox{0.95}{\ensuremath{\pi\in\mathfrak{S}_{m}}}$.
\item \label{enu:face-some}$\scalebox{0.95}{\ensuremath{M\left(F_{\pi\left(1\right)}\right)\cdots\left(F_{\pi\left(m\right)}\right)=\bigcap_{i\in\left[m\right]}M\left(F_{i}\right)}}$
for some $\scalebox{0.95}{\ensuremath{\pi\in\mathfrak{S}_{m}}}$.
\end{enumerate}
\end{thm}

\begin{proof}
Obviously, \eqref{enu:face-all}$\Rightarrow$\eqref{enu:face-some}
and \eqref{enu:face-all}$\Rightarrow$\eqref{enu:commute-1} hold.
Suppose \eqref{enu:face-some}, then $\scalebox{0.95}{\ensuremath{\bigcap_{i\in\left[m\right]}M\left(F_{i}\right)}}$
is a nonempty matroid of rank $\scalebox{0.95}{\ensuremath{r(M)}}$,
and \eqref{enu:face-some}$\Rightarrow$\eqref{enu:intersect-1} holds.
 Thus, \eqref{enu:face-all}$\Rightarrow$\eqref{enu:face-some}$\Rightarrow$\eqref{enu:intersect-1}
holds. To show \eqref{enu:face-all}$\Rightarrow$\eqref{enu:commute-1}$\Rightarrow$\eqref{enu:intersect-1},
we prove \eqref{enu:commute-1}$\Rightarrow$\eqref{enu:intersect-1}
by induction on $m$. The base case $m=2$ is true by Lemma \ref{lem:commuting}.
Assume \eqref{enu:commute-1}$\Rightarrow$\eqref{enu:intersect-1}
for all $m\le l$ with $l\ge2$. Suppose \eqref{enu:commute-1} for
$m=l+1$. Take any $\pi\in\mathfrak{S}_{l+1}$, then 
\begin{align*}
\scalebox{0.95}{\ensuremath{{\displaystyle \bigcap_{i\in\left[l+1\right]-\left\{ 1\right\} }\left(M\left(F_{\pi\left(1\right)}\right)\right)\left(F_{\pi\left(i\right)}\right)}}} & =\scalebox{0.95}{\ensuremath{{\displaystyle \bigcap_{i\in\left[l+1\right]-\left\{ 1\right\} }M\left(F_{\pi\left(1\right)}\right)\cap M\left(F_{\pi\left(i\right)}\right)}}}\\
 & =\scalebox{0.95}{\ensuremath{{\displaystyle \bigcap_{i\in\left[l+1\right]}M\left(F_{i}\right)}}}
\end{align*}
 which is nonempty by the induction hypothesis, and hence \eqref{enu:intersect-1}.
 Thus, \eqref{enu:face-all}$\Rightarrow$\eqref{enu:commute-1}$\Rightarrow$\eqref{enu:intersect-1}
holds. Now, it suffices to show \eqref{enu:intersect-1}$\Rightarrow$\eqref{enu:face-all}.
We prove by induction on $m$. The base case $m=2$ holds by Lemma
\ref{lem:commuting}. Assume \eqref{enu:intersect-1}$\Rightarrow$\eqref{enu:face-all}
for all $m\le l$ with $l\ge2$, and suppose $\scalebox{0.95}{\ensuremath{\bigcap_{i=1}^{l+1}M(F_{i})\neq\emptyset}}$.
Then, for any $\pi\in\mathfrak{S}_{l+1}$:
\[
\scalebox{0.95}{\ensuremath{{\displaystyle \bigcap_{i\in\left[l+1\right]-\left\{ 1\right\} }\left(M\left(F_{\pi\left(1\right)}\right)\right)\left(F_{\pi\left(i\right)}\right)=\bigcap_{i\in\left[l+1\right]}M\left(F_{i}\right)\neq\emptyset}}}
\]
 and by the induction hypothesis:
\begin{align*}
\scalebox{0.95}{\ensuremath{{\displaystyle \bigcap_{i\in\left[l+1\right]-\left\{ 1\right\} }\left(M\left(F_{\pi\left(1\right)}\right)\right)\left(F_{\pi\left(i\right)}\right)}}} & =\scalebox{0.95}{\ensuremath{{\displaystyle \left(M\left(F_{\pi(1)}\right)\right)\left(F_{\pi\left(2\right)}\right)\cdots\left(F_{\pi\left(l+1\right)}\right)}}}\\
 & =\scalebox{0.95}{\ensuremath{{\displaystyle M\left(F_{\pi\left(1\right)}\right)\left(F_{\pi\left(2\right)}\right)\cdots\left(F_{\pi\left(l+1\right)}\right)}}}.
\end{align*}
  Thus, \eqref{enu:intersect-1}$\Rightarrow$\eqref{enu:face-all}
proves true. The proof is complete.
\end{proof}
\begin{prop}
\label{prop:Boolean}For $\scalebox{0.95}{\ensuremath{F_{1},\dots,F_{m}\subseteq E\left(M\right)}}$,
let $\scalebox{0.95}{\ensuremath{\mathcal{A}}}$ be the \emph{Boolean
algebra they} generate with unions and intersections where we assume
$\scalebox{0.95}{\ensuremath{E\left(M\right)\in\mathcal{A}}}$. If
$\scalebox{0.95}{\ensuremath{\bigcap_{i\in\left[m\right]}M\left(F_{i}\right)}}$
is nonempty, then $\scalebox{0.95}{\ensuremath{\bigcap_{i\in\left[m\right]}M\left(F_{i}\right)}}=\scalebox{0.95}{\ensuremath{\bigcap_{A\in\mathcal{A}}M\left(A\right)}}.$
If $\scalebox{0.95}{\ensuremath{\bigcap_{i\in\left[m\right]}M\left(F_{i}\right)}}$
is a nonempty loopless matroid, then every member of $\scalebox{0.95}{\ensuremath{\mathcal{A}}}$
is a flat of $\scalebox{0.95}{\ensuremath{M}}$.
\end{prop}

\begin{proof}
If $\scalebox{0.95}{\ensuremath{\bigcap_{i\in\left[m\right]}M\left(F_{i}\right)\neq\emptyset}}$,
the attaché operations by $\scalebox{0.95}{\ensuremath{F_{1},\dots,F_{m}}}$
commute and are idempotent by Theorem \ref{thm:face-matroids}. By
Lemma~\ref{lem:commuting}, the attaché operations by $\scalebox{0.95}{\ensuremath{A}}$
for all $\scalebox{0.95}{\ensuremath{A\in\mathcal{A}}}$ on $\scalebox{0.95}{\ensuremath{M\left(F_{1}\right)\cdots\left(F_{m}\right)}}$
do \emph{not} change the face matroid. Suppose that $\scalebox{0.95}{\ensuremath{\bigcap_{i\in\left[m\right]}M\left(F_{i}\right)}}$
is a nonempty loopless matroid. Then, $\scalebox{0.95}{\ensuremath{M}}$
is loopless in the first place because attaché operations never vanish
loops, if any. Also, $\scalebox{0.95}{\ensuremath{M\left(A\right)}}$
for any $\scalebox{0.95}{\ensuremath{A\in\mathcal{A}}}$ is loopless,
and $\scalebox{0.95}{\ensuremath{A}}$ is a flat of $\scalebox{0.95}{\ensuremath{M}}$
by Lemma \ref{lem:Faces of BP}.
\end{proof}
The following is a corollary to Lemma~\ref{lem:Faces of BP}, which
describes the facets of a full-dimensional base polytope. The lemma
extends to Corollary~\ref{cor:GS-thm-ext}.
\begin{lem}[{\cite[Theorem 2.2]{GS87}}]
\label{lem:GS-thm-1}For a connected matroid $\scalebox{0.95}{\ensuremath{M}}$,
any facet of $\scalebox{0.95}{\ensuremath{\mathrm{BP}_{M}}}$ is written
as $\scalebox{0.95}{\ensuremath{\mathrm{BP}_{M\left(F\right)}}}$
for a unique non-degenerate subset $\scalebox{0.95}{\ensuremath{F}}$
of $\scalebox{0.95}{\ensuremath{M}}$.
\end{lem}

\begin{cor}
\label{cor:GS-thm-ext}Let $\scalebox{0.95}{\ensuremath{M}}$ be a
general matroid. Then, every facet matroid of $\scalebox{0.95}{\ensuremath{M}}$
is written as $\scalebox{0.95}{\ensuremath{M\left(F\right)}}$ for
a non-degenerate subset $\scalebox{0.95}{\ensuremath{F}}$ of $\scalebox{0.95}{\ensuremath{M}}$.
\end{cor}

\begin{proof}
Let $\scalebox{0.95}{\ensuremath{M_{1}\oplus\cdots\oplus M_{\ell}}}$
be the decomposition of $\scalebox{0.95}{\ensuremath{M}}$ into its
connected components, then $\scalebox{0.95}{\ensuremath{\mathrm{BP}_{M}=\mathrm{BP}_{M_{1}}\times\cdots\times\mathrm{BP}_{M_{l}}}}$.
For a facet $\scalebox{0.95}{\ensuremath{Q}}$ of $\scalebox{0.95}{\ensuremath{\mathrm{BP}_{M}}}$,
there is $\scalebox{0.95}{\ensuremath{i_{0}\in\left[l\right]}}$ such
that 
\begin{align*}
\scalebox{0.95}{\ensuremath{Q=Q_{1}\times\cdots\times Q_{\ell}}} & \quad\text{with }\scalebox{0.95}{\ensuremath{Q_{i}}}=\begin{cases}
\scalebox{0.95}{\ensuremath{\mathrm{BP}_{M_{i}}}} & \text{if }i\neq i_{0},\\
\scalebox{0.95}{\ensuremath{\text{a facet of }\mathrm{BP}_{M_{i_{0}}}}} & \text{if }i=i_{0}.
\end{cases}
\end{align*}
By Lemma \ref{lem:GS-thm-1}, $\scalebox{0.95}{\ensuremath{Q_{i_{0}}=\mathrm{BP}_{M_{i_{0}}}(F)}}$
for a unique non-degenerate subset $\scalebox{0.95}{\ensuremath{F}}$
of $\scalebox{0.95}{\ensuremath{M_{i_{0}}}}$, which is a non-degenerate
subset of $\scalebox{0.95}{\ensuremath{M}}$ with $\scalebox{0.95}{\ensuremath{Q=\mathrm{BP}_{M\left(F\right)}}}$.
\end{proof}
This produces the following lemma, which shows that there is a minimal
system of describing inequalities of an independence polytope:
\begin{lem}
\label{lem:Ineq improved}Let $\scalebox{0.95}{\ensuremath{M=\bigoplus_{j\in\left[\kappa\right]}M_{j}}}$
be a rank-$k$ matroid with $\scalebox{0.95}{\ensuremath{\kappa}}$
connected components. Then, its independence polytope is determined
by $\left|\bar{\emptyset}\right|$ equations $\scalebox{0.95}{\ensuremath{x(l)=0}}$
with $\scalebox{0.95}{\ensuremath{l\in\bar{\emptyset}}}$ and a minimal
system of inequalities: 
\begin{equation}
\scalebox{0.95}{\ensuremath{\begin{cases}
x_{i}\ge0 & \mbox{ for all }i\in E\left(M\right)-\bar{\emptyset},\\
x\left(F\right)\le r\left(F\right) & \mbox{ for all minimal non-degenerate flats }F\mbox{ of }M.
\end{cases}}}\label{eq:Ineq improved}
\end{equation}
 Its base polytope is determined by $\kappa$ equations $\scalebox{0.95}{\ensuremath{x\left(E(M_{j})\right)=r\left(M_{j}\right)}}$
and \eqref{eq:Ineq improved}.
\end{lem}

When $\scalebox{0.95}{\ensuremath{M}}$ is connected, \eqref{eq:Ineq improved}
is the unique minimal system of describing inequalities of an independence
polytope.\vspace{2pt}

Given a flat collection of a matroid, one can find the collection
of non-degenerate flats. The converse also holds:
\begin{prop}
\label{prop:nondegen-flats} If the collection $\scalebox{0.95}{\ensuremath{\mathcal{N}}}$
of non-degenerate flats of a matroid $\scalebox{0.95}{\ensuremath{M}}$
is nonempty, the flat collection of $\scalebox{0.95}{\ensuremath{M}}$
is recovered from $\scalebox{0.95}{\ensuremath{\mathcal{N}}}$.
\end{prop}

\begin{proof}
We may assume $\scalebox{0.95}{\ensuremath{M}}$ is a connected matroid
of rank $k\ge2$. Let $\scalebox{0.95}{\ensuremath{\mathrm{BP}_{M_{1}},\dots,\mathrm{BP}_{M_{\alpha}}}}$
be all inclusionwise minimal loopless faces of $\scalebox{0.95}{\ensuremath{\mathrm{BP}_{M}}}$.
Each $\scalebox{0.95}{\ensuremath{\mathrm{BP}_{M_{i}}}}$ is an intersection
of loopless facets of $\scalebox{0.95}{\ensuremath{\mathrm{BP}_{M}}}$,
and $\scalebox{0.95}{\ensuremath{M_{i}}}$ is an intersection of $\scalebox{0.95}{\ensuremath{M(F)}}$
for some non-degenerate flats $\scalebox{0.95}{\ensuremath{F}}$ of
$\scalebox{0.95}{\ensuremath{M}}$. Let $\scalebox{0.95}{\ensuremath{\mathcal{A}_{i}}}$
be the Boolean algebra generated with unions and intersections by
all those non-degenerate flats. By Proposition \ref{prop:Boolean},
$\scalebox{0.95}{\ensuremath{\mathcal{A}_{i}}}$ is contained in $\scalebox{0.95}{\ensuremath{\mathcal{L}\left(M\right)}}$.
Conversely, for any nonempty proper flat $\scalebox{0.95}{\ensuremath{L}}$
of $\scalebox{0.95}{\ensuremath{M}}$, the base polytope $\scalebox{0.95}{\ensuremath{\mathrm{BP}_{M\left(L\right)}}}$
is a loopless face of $\scalebox{0.95}{\ensuremath{\mathrm{BP}_{M}}}$
by Lemma \ref{lem:Faces of BP} and contains some $\scalebox{0.95}{\ensuremath{\mathrm{BP}_{M_{i}}}}$,
so $\scalebox{0.95}{\ensuremath{L\in\mathcal{A}_{i}}}$. 
\end{proof}

\subsection{\label{subsec:Latt-operation}Binary operation \textquotedblleft$\odot$\textquotedblright{}
on matroidal expressions}

Fix a matroid $\scalebox{0.95}{\ensuremath{M}}$. We first reduce
all minor expressions to simple minor expressions of contraction-restriction
form. We then define a binary operation ``$\odot$'' on the collection
of all minor expressions of $\scalebox{0.95}{\ensuremath{M}}$ as
follows:
\[
\scalebox{0.95}{\ensuremath{\left(M/A|_{C}\right)\odot\left(M/B|_{D}\right):=M/\left(A\cup B\right)|_{C\cap D}}}.
\]
 This operation is both commutative and associative. Next, we extend
the operation to the collection of all finite direct sums of minor
expressions as follows: for minor expressions $\scalebox{0.95}{\ensuremath{X,Y_{1},\dots,Y_{m}}}$,
define:
\[
\scalebox{0.95}{\ensuremath{{\displaystyle X\odot\Biggl(\bigoplus_{i=1}^{m}Y_{i}\Biggr):=\bigoplus_{i=1}^{m}\left(X\odot Y_{i}\right)}}}.
\]
 Then, the operation $\odot$ is commutative and distributive over
direct sum $\oplus$, thus making it associative.
\begin{rem}
The operations $\wedge$ and $\cap$ do not necessarily produce matroids,
but $\odot$ does. However, the definitions of $\wedge$ and $\cap$
are expression-free while that of $\odot$ is not, and $\odot$ can
produce different matroids for different expressions.
\end{rem}

\subsection{\label{subsec:V=00003DV(M)}The lattice $\mathcal{V}(M)$ of direct
sums of disjoint minor expressions}

Define a partial order ``$\osubeq$'' on the collection of all minor
expressions: 
\[
\scalebox{0.95}{\ensuremath{\left(M/A|_{C}\right)\osubeq\left(M/B|_{D}\right)\quad\text{if }A\supseteq B\text{ and }C\subseteq D}}.
\]
 Consider a finite collection $\scalebox{0.95}{\ensuremath{\mathcal{V}=\mathcal{V}(M)}}$:
\[
\scalebox{0.95}{\ensuremath{\mathcal{V}:=\left\{ \text{all direct sums of disjoint minor expressions of }M\right\} }}.
\]
\emph{Extend} $\osubeq$ to $\scalebox{0.95}{\ensuremath{\mathcal{V}}}$
such that for $\scalebox{0.95}{\ensuremath{X=\bigoplus_{i=1}^{l}X_{i}}}$
and $\scalebox{0.95}{\ensuremath{Y=\bigoplus_{j=1}^{m}Y_{j}}}$ in
$\scalebox{0.95}{\ensuremath{\mathcal{V}}}$: 
\[
\scalebox{0.95}{\ensuremath{X\osubeq Y\quad\text{if for each }i\text{ there is }j\text{ with }X_{i}\osubeq Y_{j}}}.
\]
Then, $\scalebox{0.95}{\ensuremath{\left(\mathcal{V},\osubeq\right)}}$
is a \emph{poset}. The operation $\odot$ is idempotent over $\scalebox{0.95}{\ensuremath{\mathcal{V}}}$.
Thus, $\scalebox{0.95}{\ensuremath{\left(\mathcal{V},\osubeq\right)}}$
is a \emph{meet-semilattice} with meet $\odot$. Further, $\scalebox{0.95}{\ensuremath{\mathcal{V}}}$
is a finite semilattice, while $M=M/\emptyset|_{E\left(M\right)}$
is its largest member.  Therefore, $\scalebox{0.95}{\ensuremath{\mathcal{V}}}$
is a \emph{lattice}.
\begin{rem}[]
\begin{enumerate}
\item If $\scalebox{0.95}{\ensuremath{M/A|_{C}\osubeq M/B|_{D}}}$, then
we have $\scalebox{0.95}{\ensuremath{E\left(M/A|_{C}\right)\subseteq E\left(M/B|_{D}\right)}}$
and $\scalebox{0.95}{\ensuremath{\mathcal{I}\left(M/A|_{C}\right)\subseteq\mathcal{I}\left(M/B|_{D}\right)}}$. 
\item For any $\scalebox{0.95}{\ensuremath{X,Y\in\mathcal{V}}}$, one has
$\scalebox{0.95}{\ensuremath{X\odot Y\osubeq X}}$ and $\scalebox{0.95}{\ensuremath{X\odot Y\osubeq Y}}$.
Moreover: 
\begin{align}
\scalebox{0.95}{\ensuremath{E\left(X\odot Y\right)}} & =\scalebox{0.95}{\ensuremath{E\left(X\right)\cap E\left(Y\right)}},\nonumber \\
\scalebox{0.95}{\ensuremath{\mathcal{I}\left(X\odot Y\right)}} & \subseteq\scalebox{0.95}{\ensuremath{X\wedge Y}},\label{eq:odot-wedge}\\
\scalebox{0.95}{\ensuremath{r\left(X\odot Y\right)}} & \le\scalebox{0.95}{\ensuremath{\min\left\{ r\left(X\right),r\left(Y\right)\right\} }}.\nonumber 
\end{align}
 If $\scalebox{0.95}{\ensuremath{X\osubeq Y}}$, then $\scalebox{0.95}{\ensuremath{E\left(X\right)\subseteq E\left(Y\right)}}$,
$\scalebox{0.95}{\ensuremath{\mathcal{I}\left(X\right)\subseteq\mathcal{I}\left(Y\right)}}$,
and $\scalebox{0.95}{\ensuremath{\phi\left(X\right)\subseteq\phi\left(Y\right)}}$.
\end{enumerate}
\end{rem}

\subsection{\label{subsec:Intersection-2}Intersection of faces, II}

Let $\scalebox{0.95}{\ensuremath{M}}$ be a matroid and $\scalebox{0.95}{\ensuremath{F_{1},\dots,F_{m}}}$
subsets of $\scalebox{0.95}{\ensuremath{E(M)}}$. The collection $\scalebox{0.95}{\ensuremath{\bigwedge_{i\in\left[m\right]}M\left(F_{i}\right)}}$
has a lower bound:
\begin{prop}
\label{prop:Lower bound}For subsets $\scalebox{0.95}{\ensuremath{F_{1},\dots,F_{m}}}$
of $\scalebox{0.95}{\ensuremath{E\left(M\right)}}$: 
\[
\scalebox{0.95}{\ensuremath{{\displaystyle \mathcal{I}\Biggl(\bigodot_{i\in\left[m\right]}M\left(F_{i}\right)\Biggr)\subseteq\bigwedge_{i\in\left[m\right]}M\left(F_{i}\right)}}}.
\]
\end{prop}

\begin{proof}
Because $\scalebox{0.95}{\ensuremath{M(F)}}$ for all $\scalebox{0.95}{\ensuremath{F\subseteq E\left(M\right)}}$
are members of $\scalebox{0.95}{\ensuremath{\mathcal{V}}}$, we repeatedly
apply the formula \eqref{eq:odot-wedge} to $\scalebox{0.95}{\ensuremath{\bigodot_{i\in\left[m\right]}M\left(F_{i}\right)}}$
and obtain the desired inclusion.
\end{proof}
Next, we show that the collection $\scalebox{0.95}{\ensuremath{\bigcap_{i\in\left[m\right]}M(F_{i})}}$
has an upper bound (Proposition \ref{prop:Upper bound}) and extend
Theorem \ref{thm:face-matroids} to Theorem \ref{thm:face-intersection}.
We also express $\scalebox{0.95}{\ensuremath{\bigodot_{i\in\left[m\right]}M(F_{i})}}$
as a direct sum \eqref{eq:odot-oplus} and as a matroid intersection
\eqref{eq:odot-int}. This requires an a priori setup. For any subset
$\sigma$ of $\left[m\right]$, denote
\[
\scalebox{0.95}{\ensuremath{F_{\sigma}^{+}:=\bigcup_{i\notin\sigma}F_{i}}}\quad\text{and}\quad\scalebox{0.95}{\ensuremath{F_{\sigma}^{-}:=\begin{cases}
E\left(M\right) & \text{if }\sigma=\emptyset,\\
\bigcap_{j\in\sigma}F_{j} & \text{otherwise}.
\end{cases}}}
\]

\noindent Then, $\scalebox{0.95}{\ensuremath{F_{\sigma}^{-}-F_{\sigma}^{+}}}$
for all $\scalebox{0.95}{\ensuremath{\sigma\in2^{\left[m\right]}}}$
are pairwise disjoint, whose union is $\scalebox{0.95}{\ensuremath{E\left(M\right)}}$.
Denote 
\[
\scalebox{0.95}{\ensuremath{M_{\sigma}:=M/F_{\sigma}^{+}|_{\left(F_{\sigma}^{-}-F_{\sigma}^{+}\right)}}}.
\]
Then, a computation shows:
\begin{equation}
\scalebox{0.95}{\ensuremath{{\displaystyle \bigodot_{i\in\left[m\right]}M\left(F_{i}\right)=\bigoplus_{\sigma\in2^{\left[m\right]}}M_{\sigma}}}}.\label{eq:odot-oplus}
\end{equation}
This is a matroid on $\scalebox{0.95}{\ensuremath{E\left(M\right)}}$.
For each number $\scalebox{0.95}{\ensuremath{j\in\left[2^{m}\right]-1:=\left\{ l-1:l\in\left[2^{m}\right]\right\} }}$,
let $\scalebox{0.95}{\ensuremath{d_{1j}d_{2j}\cdots d_{mj}}}$ be
the binary number converted from $j$. Assign to each $\scalebox{0.95}{\ensuremath{\sigma\in2^{\left[m\right]}}}$
the decimal number $\scalebox{0.95}{\ensuremath{j=j\left(\sigma\right)}}$
with 
\[
\scalebox{0.95}{\ensuremath{d_{ij}=\begin{cases}
0 & \text{if }i\in\sigma,\\
1 & \text{otherwise}.
\end{cases}}}
\]
 This assignment is a bijection between $\scalebox{0.95}{\ensuremath{2^{\left[m\right]}}}$
and $\scalebox{0.95}{\ensuremath{\left[2^{m}\right]-1}}$; hence $\scalebox{0.95}{\ensuremath{\sigma=\sigma\left(j\right)}}$.
Let 
\[
\diamondsuit_{ij}=\begin{cases}
\cup & \text{if }\scalebox{0.95}{\ensuremath{d_{ij}}}=0,\\
\cap & \text{if }\scalebox{0.95}{\ensuremath{d_{ij}}}=1.
\end{cases}
\]
 For a permutation $\pi\in\mathfrak{S}_{m}$, denote 
\[
F_{\sigma}^{\pi+}=F_{j}^{\pi+}:=F_{\pi(1)}\,\diamondsuit_{1j}\,\left(\cdots\left(F_{\pi(m-1)}\,\diamondsuit_{(m-1)j}\left(F_{\pi(m)}\,\diamondsuit_{mj}\,\emptyset\right)\right)\cdots\right)\subseteq F_{\sigma}^{+}.
\]
 Then, $\scalebox{0.95}{\ensuremath{\left(F_{0}^{\pi+},\dots,F_{2^{m}-1}^{\pi+}\right)}}$
is a decreasing sequence of subsets of $\scalebox{0.95}{\ensuremath{E(M)}}$
with $\scalebox{0.95}{\ensuremath{F_{0}^{\pi+}=F_{\left[m\right]}^{\pi+}=}}$
$\scalebox{0.95}{\ensuremath{\bigcup_{i\in\left[m\right]}F_{i}}}$
and $\scalebox{0.95}{\ensuremath{F_{2^{m}-1}^{\pi+}=F_{\emptyset}^{\pi+}=\emptyset}}$.
Also, let 
\[
\diamondsuit_{ij}^{-}=\begin{cases}
- & \text{if }\scalebox{0.95}{\ensuremath{d_{ij}}}=0,\\
\cap & \text{if }\scalebox{0.95}{\ensuremath{d_{ij}}}=1.
\end{cases}
\]
 Then, $\scalebox{0.95}{\ensuremath{\diamondsuit_{ij}=\diamondsuit_{ij}^{-}=\cap}}$
only when $\scalebox{0.95}{\ensuremath{d_{ij}=1}}$. Denote: 
\[
F_{\sigma}^{\pi-}=F_{j}^{\pi-}:=\left(\cdots\left(\left(E(M)\:\diamondsuit_{1j}^{-}\:F_{\pi(1)}\right)\,\diamondsuit_{2j}^{-}\:F_{\pi(2)}\right)\cdots\right)\,\diamondsuit_{mj}^{-}\:F_{\pi(m)}.
\]

\noindent Then $\scalebox{0.95}{\ensuremath{F_{\sigma}^{\pi-}=F_{\sigma}^{-}-F_{\sigma}^{+}}}$.
Define a matroid $\scalebox{0.95}{\ensuremath{M_{\sigma}^{\pi}}}$
on $\scalebox{0.95}{\ensuremath{F_{\sigma}^{\pi-}}}$: 
\[
M_{\sigma}^{\pi}:=M/F_{\sigma}^{\pi}|_{F_{\sigma}^{\pi-}}.
\]
 Then
\[
\scalebox{0.95}{\ensuremath{{\displaystyle \bigoplus_{\sigma\in2^{\left[m\right]}}M_{\sigma}^{\pi}=M\left(F_{\pi\left(1\right)}\right)\cdots\left(F_{\pi\left(m\right)}\right)}}}.
\]
 Since $\scalebox{0.95}{\ensuremath{M_{\sigma}\osubset M_{\sigma}^{\pi}}}$
for all $\pi\in\mathfrak{S}_{m}$, we have $\ensuremath{\scalebox{0.95}{\ensuremath{\left(\bigoplus_{\sigma\in2^{\left[m\right]}}M_{\sigma}\right)}}\osubeq\scalebox{0.95}{\ensuremath{\left(\bigoplus_{\sigma\in2^{\left[m\right]}}M_{\sigma}^{\pi}\right)}}}$
and: 
\[
\scalebox{0.95}{\ensuremath{{\displaystyle \bigodot_{i\in\left[m\right]}M(F_{i})}}}\,\osubeq\,\scalebox{0.95}{\ensuremath{{\displaystyle M\left(F_{\pi\left(1\right)}\right)\cdots\left(F_{\pi\left(m\right)}\right)}}}.
\]
 For $\scalebox{0.95}{\ensuremath{\sigma=\left\{ i_{1},\cdots,i_{l}\right\} }}$
with $\scalebox{0.95}{\ensuremath{i_{1}<\cdots<i_{l}}}$, there is
a permutation $\pi\in\mathfrak{S}_{m}$ with $\scalebox{0.95}{\ensuremath{\pi\left(1\right)=i_{1}}}$,
$\scalebox{0.95}{\ensuremath{\dots}}$, $\scalebox{0.95}{\ensuremath{\pi\left(l\right)=i_{l}}}$
so that $\scalebox{0.95}{\ensuremath{M_{\sigma}=M_{\sigma}^{\pi}}}$.
Therefore:
\begin{equation}
\scalebox{0.95}{\ensuremath{{\displaystyle \bigodot_{i\in\left[m\right]}M\left(F_{i}\right)=\bigwedge_{\pi\in\mathfrak{S}_{m}}M\left(F_{\pi\left(1\right)}\right)\cdots\left(F_{\pi\left(m\right)}\right)}}}.\label{eq:odot-int}
\end{equation}
 This implies that the intersection of independence polytopes $\scalebox{0.95}{\ensuremath{{\displaystyle \mathrm{IP}_{M\left(F_{\pi\left(1\right)}\right)\cdots\left(F_{\pi\left(m\right)}\right)}}}}$
for all $\scalebox{0.95}{\ensuremath{{\displaystyle \pi\in\mathfrak{S}_{m}}}}$
is an independence polytope: 
\[
\scalebox{0.95}{\ensuremath{{\displaystyle \mathrm{IP}_{\bigodot_{i\in\left[m\right]}M\left(F_{i}\right)}=\bigcap_{\pi\in\mathfrak{S}_{m}}\mathrm{IP}_{M\left(F_{\pi\left(1\right)}\right)\cdots\left(F_{\pi\left(m\right)}\right)}}}}.
\]

Now, we prove that $\scalebox{0.95}{\ensuremath{\bigcap_{i\in\left[m\right]}M\left(F_{i}\right)}}$
has an upper bound:
\begin{prop}
\label{prop:Upper bound}For subsets $\scalebox{0.95}{\ensuremath{F_{1},\dots,F_{m}}}$
of $\scalebox{0.95}{\ensuremath{E\left(M\right)}}$: 
\[
\scalebox{0.95}{\ensuremath{{\displaystyle \bigcap_{i\in\left[m\right]}M\left(F_{i}\right)\subseteq\mathcal{B}\Biggl(\bigodot_{i\in\left[m\right]}M\left(F_{i}\right)\Biggr)}}}.
\]
\end{prop}

\begin{proof}
Let $\scalebox{0.95}{\ensuremath{\mathcal{A}}}$ be the Boolean algebra
that $\scalebox{0.95}{\ensuremath{F_{1},\dots,F_{m}}}$ generate with
intersections and unions. Let $\scalebox{0.95}{\ensuremath{B}}$ be
a common base of $\scalebox{0.95}{\ensuremath{M\left(F_{1}\right),\dots,M\left(F_{m}\right)}}$,
then $\scalebox{0.95}{\ensuremath{B}}$ is a base of $\scalebox{0.95}{\ensuremath{M\left(A\right)}}$
for any $\scalebox{0.95}{\ensuremath{A\in\mathcal{A}}}$ by Proposition
\ref{prop:Boolean}, and $\scalebox{0.95}{\ensuremath{r\left(B\cap A\right)=r\left(A\right)}}$.
For any $\sigma\in2^{\left[m\right]}$, let $\scalebox{0.95}{\ensuremath{B_{\sigma}:=B\cap\left(F_{\sigma}^{-}-F_{\sigma}^{+}\right)}}.$
Then, $\scalebox{0.95}{\ensuremath{r\left(\left(B\cap F_{\sigma}^{-}\right)\cup F_{\sigma}^{+}\right)=r\left(B_{\sigma}\cup F_{\sigma}^{+}\right)}}$,
and because $\scalebox{0.95}{\ensuremath{F_{\sigma}^{-}\cup F_{\sigma}^{+}\in\mathcal{A}}}$,
we have $\scalebox{0.95}{\ensuremath{r\left(B\cap\left(F_{\sigma}^{-}\cup F_{\sigma}^{+}\right)\right)=r\left(F_{\sigma}^{-}\cup F_{\sigma}^{+}\right)}}$.
Therefore: 
\begin{align*}
\scalebox{0.95}{\ensuremath{r_{M_{\sigma}}\left(B_{\sigma}\right)}} & \le\scalebox{0.95}{\ensuremath{r\left(M_{\sigma}\right)}}=\scalebox{0.95}{\ensuremath{r\left(F_{\sigma}^{-}\cup F_{\sigma}^{+}\right)-r\left(F_{\sigma}^{+}\right)}}\\
 & =\scalebox{0.95}{\ensuremath{r\left(B\cap\left(F_{\sigma}^{-}\cup F_{\sigma}^{+}\right)\right)-r\left(F_{\sigma}^{+}\right)}}\\
 & \le\scalebox{0.95}{\ensuremath{r\left(\left(B\cap F_{\sigma}^{-}\right)\cup F_{\sigma}^{+}\right)-r\left(F_{\sigma}^{+}\right)}}\\
 & =\scalebox{0.95}{\ensuremath{r\left(B_{\sigma}\cup F_{\sigma}^{+}\right)-r\left(F_{\sigma}^{+}\right)}}\\
 & =\scalebox{0.95}{\ensuremath{r_{M_{\sigma}}\left(B_{\sigma}\right)}}.
\end{align*}
 Thus, $\scalebox{0.95}{\ensuremath{B_{\sigma}}}$ is a base of $\scalebox{0.95}{\ensuremath{M_{\sigma}}}$.
So, $\scalebox{0.95}{\ensuremath{B=\bigcup_{\sigma\in2^{\left[m\right]}}B_{\sigma}}}$
is a base of $\scalebox{0.95}{\ensuremath{\bigoplus_{\sigma\in2^{\left[m\right]}}M_{\sigma}}}$
and $\scalebox{0.95}{\ensuremath{\bigodot_{i\in\left[m\right]}M\left(F_{i}\right)}}$
by \eqref{eq:odot-oplus}.  The desired inclusion is proved.
\end{proof}
We extend Theorem \ref{thm:face-matroids} to another equivalence
theorem:
\begin{thm}
\label{thm:face-intersection}For $\scalebox{0.95}{\ensuremath{F_{1},\dots,F_{m}\subseteq E\left(M\right)}}$,
the following  are equivalent.
\begin{enumerate}
\item \label{enu:nonempty-int}$\scalebox{0.95}{\ensuremath{\bigcap_{i\in\left[m\right]}M\left(F_{i}\right)\neq\emptyset}}$.
\item \label{enu:commute-1-1}$\scalebox{0.95}{\ensuremath{M\left(F_{\pi\left(1\right)}\right)\cdots\left(F_{\pi\left(m\right)}\right)}}$
for all $\pi\in\mathfrak{S}_{m}$ are the same.
\item \label{enu:face-all-1}$\scalebox{0.95}{\ensuremath{M\left(F_{\pi\left(1\right)}\right)\cdots\left(F_{\pi\left(m\right)}\right)=\bigcap_{i\in\left[m\right]}M\left(F_{i}\right)}}$
for all $\pi\in\mathfrak{S}_{m}$.
\item \label{enu:face-some-1}$\scalebox{0.95}{\ensuremath{M\left(F_{\pi\left(1\right)}\right)\cdots\left(F_{\pi\left(m\right)}\right)=\bigcap_{i\in\left[m\right]}M\left(F_{i}\right)}}$
for some $\pi\in\mathfrak{S}_{m}$.
\item \label{enu:odot-rank}$\scalebox{0.95}{\ensuremath{\bigodot_{i\in\left[m\right]}M\left(F_{i}\right)}}$
has rank $\scalebox{0.95}{\ensuremath{r\left(M\right)}}$.
\item \label{enu:odot-intersect}$\scalebox{0.95}{\ensuremath{\bigodot_{i\in\left[m\right]}M\left(F_{i}\right)=\bigcap_{i\in\left[m\right]}M\left(F_{i}\right)}}$.
\item \label{enu:odot-faces-all}$\scalebox{0.95}{\ensuremath{\bigodot_{i\in\left[m\right]}M\left(F_{i}\right)=M\left(F_{\pi\left(1\right)}\right)\cdots\left(F_{\pi\left(m\right)}\right)}}$
for all $\pi\in\mathfrak{S}_{m}$.
\item \label{enu:odot-faces-some}$\scalebox{0.95}{\ensuremath{\bigodot_{i\in\left[m\right]}M\left(F_{i}\right)=M\left(F_{\pi\left(1\right)}\right)\cdots\left(F_{\pi\left(m\right)}\right)}}$
for some $\pi\in\mathfrak{S}_{m}$.
\end{enumerate}
\end{thm}

\begin{proof}
If $\scalebox{0.95}{\ensuremath{\bigcap_{i\in\left[m\right]}M\left(F_{i}\right)\neq\emptyset}}$,
then $\scalebox{0.95}{\ensuremath{r\bigl(\bigodot_{i\in\left[m\right]}M\left(F_{i}\right)\bigr)=r\left(M\right)}}$
by Proposition \ref{prop:Upper bound}, and \eqref{enu:nonempty-int}$\Rightarrow$\eqref{enu:odot-rank}
holds. Conversely, \eqref{enu:odot-rank}$\Rightarrow$\eqref{enu:nonempty-int}
holds by Proposition \ref{prop:Lower bound}. Since \eqref{enu:nonempty-int}$\Leftrightarrow$\eqref{enu:commute-1-1}$\Leftrightarrow$\eqref{enu:face-all-1}$\Leftrightarrow$\eqref{enu:face-some-1}
is Theorem \ref{thm:face-matroids}, \eqref{enu:nonempty-int}$\Leftrightarrow$\eqref{enu:commute-1-1}$\Leftrightarrow$\eqref{enu:face-all-1}$\Leftrightarrow$\eqref{enu:face-some-1}$\Leftrightarrow$\eqref{enu:odot-rank}
holds.

Suppose \eqref{enu:nonempty-int}, then $\scalebox{0.95}{\ensuremath{\bigcap_{i\in\left[m\right]}M\left(F_{i}\right)}}$
is a matroid by Theorem \ref{thm:face-matroids}, and is the same
matroid as $\scalebox{0.95}{\ensuremath{\bigwedge_{i\in\left[m\right]}M\left(F_{i}\right)}}$,
which equals $\scalebox{0.95}{\ensuremath{\bigodot_{i\in\left[m\right]}M\left(F_{i}\right)}}$
by Propositions \ref{prop:Lower bound} and \ref{prop:Upper bound};
hence \eqref{enu:odot-intersect}. Suppose \eqref{enu:odot-intersect},
then \eqref{enu:nonempty-int} holds since $\scalebox{0.95}{\ensuremath{\bigodot_{i\in\left[m\right]}M\left(F_{i}\right)}}$
is a nonempty matroid, and \eqref{enu:odot-faces-all} holds by Theorem
\ref{thm:face-matroids}. Now, \eqref{enu:odot-faces-all}$\Rightarrow$\eqref{enu:odot-faces-some}
is obvious, and \eqref{enu:odot-faces-some}$\Rightarrow$\eqref{enu:odot-rank}
holds since $\scalebox{0.95}{\ensuremath{M\left(F_{\pi\left(1\right)}\right)\cdots\left(F_{\pi\left(m\right)}\right)}}$
has rank $\scalebox{0.95}{\ensuremath{r\left(M\right)}}$.  Thus,
\eqref{enu:nonempty-int}$\Leftrightarrow$\eqref{enu:odot-intersect}$\Rightarrow$\eqref{enu:odot-faces-all}$\Rightarrow$\eqref{enu:odot-faces-some}$\Rightarrow$\eqref{enu:odot-rank}
holds, and the equivalence is proved.
\end{proof}

\subsection{\label{subsec:Face-lattice}The lattice of loopless faces of a base
polytope}

Let $\scalebox{0.95}{\ensuremath{M}}$ be a loopless matroid, then
$\scalebox{0.95}{\ensuremath{\mathrm{BP}_{M}}}$ is a loopless base
polytope of codimension $\scalebox{0.95}{\ensuremath{\kappa\left(M\right)}}$.
The loopless faces of $\scalebox{0.95}{\ensuremath{\mathrm{BP}_{M}}}$
form a lattice whose meet operation is described as follows: the meet
of loopless faces $\scalebox{0.95}{\ensuremath{\mathrm{BP}_{\phi\left(X\right)}}}$
and $\scalebox{0.95}{\ensuremath{\mathrm{BP}_{\phi\left(Y\right)}}}$
for matroidal expressions $\scalebox{0.95}{\ensuremath{X}}$ and $\scalebox{0.95}{\ensuremath{Y}}$
is $\scalebox{0.95}{\ensuremath{\mathrm{BP}_{\phi\left(X\odot Y\right)}}}$
if $\scalebox{0.95}{\ensuremath{\phi\left(X\odot Y\right)}}$ is a
full-rank loopless matroid and $\scalebox{0.95}{\ensuremath{\emptyset}}$
otherwise, cf. Theorem \ref{thm:face-intersection} and Proposition
\ref{prop:Boolean}. The lattice is coatomistic, whose coatoms are
precisely the loopless facets of $\scalebox{0.95}{\ensuremath{\mathrm{BP}_{M}}}$.
Here, \emph{coatomistic} means that all members other than the largest
member are generated by coatoms with meets. Its lattice rank is $\scalebox{0.95}{\ensuremath{r\left(M\right)-\kappa\left(M\right)+1}}$.

\subsection{Codimension-two faces}

We investigate the codimension-$2$ loopless face of the base polytope
of a connected matroid.
\begin{lem}
\label{lem:Ridges of BP}Let $\scalebox{0.95}{\ensuremath{M}}$ be
a connected matroid of rank at least $\scalebox{0.95}{\ensuremath{3}}$,
and $\scalebox{0.95}{\ensuremath{F}}$ and $\scalebox{0.95}{\ensuremath{L}}$
be two non-degenerate subsets of $\scalebox{0.95}{\ensuremath{M}}$
with $\scalebox{0.95}{\ensuremath{X=M\left(F\right)\left(L\right)=M\left(L\right)\left(F\right)}}$
and $\scalebox{0.95}{\ensuremath{\kappa\left(X\right)=3}}$. Let $\scalebox{0.95}{\ensuremath{A}}$
and $\scalebox{0.95}{\ensuremath{T}}$ be the minimal non-degenerate
subsets of $\scalebox{0.95}{\ensuremath{M\left(F\right)}}$ and $\scalebox{0.95}{\ensuremath{M\left(L\right)}}$,
respectively, with $\scalebox{0.95}{\ensuremath{X=M\left(F\right)\left(A\right)=M\left(L\right)\left(T\right)}}.$
Then, precisely one of the following four cases happens for the quadruple
$\scalebox{0.95}{\ensuremath{\left(F,L,A,T\right)}}$.
\begin{table}[H]
\noindent \centering{}%
\begin{tabular}{c||c|c|c|c|c}
\hline 
\multirow{1}{*}{} & \scalebox{0.9}{$A$} & \scalebox{0.9}{$T$} & \scalebox{0.9}{$M/\left(F\cap L\right)$} & \scalebox{0.9}{$M|_{F\cup L}$} & \scalebox{0.9}{$M\left(F\right)\cap M\left(L\right)$}\tabularnewline
\hline 
\hline 
\scalebox{0.9}{$F\cup L=E(M)$} & \scalebox{0.9}{$F\cap L$} & \scalebox{0.9}{$F\cap L$} & \scalebox{0.9}{$M/F\oplus M/L$} & \scalebox{0.9}{$M$} & \scalebox{0.9}{$M\left(F\cap L\right)$}\tabularnewline
\hline 
\scalebox{0.9}{$F\cap L=\emptyset$} & \scalebox{0.9}{$L$} & \scalebox{0.9}{$F$} & \scalebox{0.9}{$M$} & \scalebox{0.9}{$M|_{F}\oplus M|_{L}$} & \scalebox{0.9}{$M\left(F\cup L\right)$}\tabularnewline
\hline 
\scalebox{0.9}{$F\subsetneq L$} & \scalebox{0.9}{$L-F$} & \scalebox{0.9}{$F$} & \scalebox{0.9}{$M/F$} & \scalebox{0.9}{$M|_L$} & \scalebox{0.8}{$M/L\oplus M|_L/F\oplus M|_F$}\tabularnewline
\hline 
\scalebox{0.9}{$F\supsetneq L$} & \scalebox{0.9}{$L$} & \scalebox{0.9}{$F-L$} & \scalebox{0.9}{$M/L$} & \scalebox{0.9}{$M|_F$} & \scalebox{0.8}{$M/F\oplus M|_F/L\oplus M|_L$}\tabularnewline
\hline 
\end{tabular}\caption{\label{tab:Faces of BP2}The loopless codimension-2 face matroid.}
\end{table}
\end{lem}

\begin{proof}
We know $\scalebox{0.95}{\ensuremath{X=M\left(F\right)\cap M\left(L\right)=M\left(F\right)\odot M\left(L\right)}}$
by Theorem \ref{thm:face-intersection} where 
\[
\scalebox{0.95}{\ensuremath{M\left(F\right)\odot M\left(L\right)}}=\scalebox{0.95}{\ensuremath{M/\left(F\cup L\right)\oplus M|_{F\cup L}/F\oplus M|_{F\cup L}/L\oplus M|_{F\cap L}}}.
\]
This has at least one vanishing summand since $\scalebox{0.95}{\ensuremath{\kappa\left(X\right)=3}}$,
i.e., at least one of the following four equalities holds: $\scalebox{0.95}{\ensuremath{F\cup L=E\left(M\right)}}$,
$\scalebox{0.95}{\ensuremath{F\cap L=\emptyset}}$, $\scalebox{0.95}{\ensuremath{F-L=\emptyset}}$
and $\scalebox{0.95}{\ensuremath{L-F=\emptyset}}$. We show that exactly
one of these four holds. Suppose $\scalebox{0.95}{\ensuremath{F-L=\emptyset}}$,
then $\scalebox{0.95}{\ensuremath{F\subset L}}$ and none of the other
three holds; hence, no pair of the four involving $\scalebox{0.95}{\ensuremath{F-L=\emptyset}}$
holds. Also, by symmetry, no pair involving $\scalebox{0.95}{\ensuremath{L-F=\emptyset}}$
holds. So, suppose that the first two hold, then $\scalebox{0.95}{\ensuremath{M\left(F\right)\odot M\left(L\right)}}=\scalebox{0.95}{\ensuremath{M/F\oplus M/L}}$
and $\scalebox{0.95}{\ensuremath{r\left(M\right)=r\left(F\right)+r\left(L\right)}}$,
which implies that $\scalebox{0.95}{\ensuremath{F}}$ and $\scalebox{0.95}{\ensuremath{L}}$
are separators, a contradiction. Thus, exactly one of the four summands
of $\scalebox{0.95}{\ensuremath{M\left(F\right)\odot M\left(L\right)}}$
vanishes, and the other three are connected.
\begin{enumerate}
\item If $\scalebox{0.95}{\ensuremath{F\cup L=E\left(M\right)}}$, then
$\scalebox{0.95}{\ensuremath{X=M\left(F\cap L\right)}}$ by Lemma
\ref{lem:commuting}\eqref{enu:all-equal}, and $\scalebox{0.95}{\ensuremath{M/\left(F\cap L\right)=}}$
$\scalebox{0.95}{\ensuremath{M/F\oplus M/L}}$. Then, $\scalebox{0.95}{\ensuremath{X=M\left(F\right)\left(F\cap L\right)}}$
by Proposition \ref{prop:Boolean}, and $\scalebox{0.95}{\ensuremath{A=F\cap L}}$
because $\scalebox{0.95}{\ensuremath{A}}$ is a minimal non-degenerate
subset. By symmetry, $\scalebox{0.95}{\ensuremath{T=F\cap L}}$.
\item If $\scalebox{0.95}{\ensuremath{F\cap L=\emptyset}}$, then $\scalebox{0.95}{\ensuremath{X=M\left(F\cup L\right)}}$
and $\scalebox{0.95}{\ensuremath{M|_{F\cup L}=M/L|_{F}\oplus M/F|_{L}}}$.
Thus, $\scalebox{0.95}{\ensuremath{M|_{F\cup L}=M|_{F}\oplus M|_{L}}}$
and $\scalebox{0.95}{\ensuremath{A=L}}$. By symmetry, $\scalebox{0.95}{\ensuremath{T=F}}$.
\item[(3)(4)] If $\scalebox{0.95}{\ensuremath{F-L=\emptyset}}$, then $\scalebox{0.95}{\ensuremath{X=M(L)(F)=M/L\oplus M|_{L}/F\oplus M|_{F}}}$
and $\scalebox{0.95}{\ensuremath{T=F}}$. Also, $\scalebox{0.95}{\ensuremath{M(F)(L)=M(F\cap L)(F-L)(L-F)=M(F)(L-F)}}$
and $\scalebox{0.95}{\ensuremath{A=L-F\neq\emptyset}}$. If $\scalebox{0.95}{\ensuremath{L-F=\emptyset}}$,
then $\scalebox{0.95}{\ensuremath{X=M/F\oplus M|_{F}/L\oplus M|_{L}}}$
with $\scalebox{0.95}{\ensuremath{A=L}}$ and $\scalebox{0.95}{\ensuremath{T=F-L\neq\emptyset}}$.
\end{enumerate}
The proof is complete.
\end{proof}

\section{\label{sec:Lattice Geometry of Matroids}Lattice Geometry of Matroids}

We develop the lattice geometry of matroids and show that the subspace
lattice of a matroid $\scalebox{0.95}{\ensuremath{M}}$ is transformed
into the lattice of loopless face matroids of $\scalebox{0.95}{\ensuremath{\mathrm{BP}_{M\backslash\overline{\emptyset}}}}$
through lattice operations of \emph{matroidal blowups} and \emph{collapsings}.

\subsection{Flag and flace}
\begin{defn}
A \textbf{flace}\footnote{The etymology of ``flace'' is flat $+$ face, mimicking ``flacet''
of \cite{Nested}. Our flace is not a flat but a sequence of flats.} of a matroid $\scalebox{0.95}{\ensuremath{M}}$ is a sequence $\scalebox{0.95}{\ensuremath{\left(F_{1},\dots,F_{m}\right)}}$
of subsets of $\scalebox{0.95}{\ensuremath{E\left(M\right)}}$ such
that each $\scalebox{0.95}{\ensuremath{F_{i}}}$ is a flat of $\scalebox{0.95}{\ensuremath{M_{i-1}=M\left(F_{0}\right)\cdots\left(F_{i-1}\right)}}$
with $\scalebox{0.95}{\ensuremath{\kappa\left(M_{i-1}\right)<\kappa\left(M_{i}\right)}}$
where $\scalebox{0.95}{\ensuremath{F_{0}=\bar{\emptyset}}}$ and $\scalebox{0.95}{\ensuremath{M=M\left(F_{0}\right)}}$.
We call $m$ the \textbf{length} of the flace. A flace is said to
be \textbf{full} if $\scalebox{0.95}{\ensuremath{m=\kappa\left(M_{m}\right)-\kappa\left(M_{1}\right)+1}}$,
i.e., if each $\scalebox{0.95}{\ensuremath{F_{i}}}$ is a non-degenerate
flat of $\scalebox{0.95}{\ensuremath{M_{i-1}}}$. A flace with length
$\scalebox{0.95}{\ensuremath{m=r\left(M\right)-1}}$ is said to be
\textbf{complete}.
\end{defn}

\begin{defn}
A \textbf{flag} of a matroid $\scalebox{0.95}{\ensuremath{M}}$ is
a strictly increasing sequence $\scalebox{0.95}{\ensuremath{\left(L_{1},\dots,L_{l}\right)}}$
of flats with $\scalebox{0.95}{\ensuremath{\bar{\emptyset}\subsetneq L_{1}}}$
and $\scalebox{0.95}{\ensuremath{L_{l}\subsetneq E\left(M\right)}}$.
We call $l$ its \textbf{length}. A \textbf{full} flag is a flag of
length $\scalebox{0.95}{\ensuremath{r\left(L_{l}\right)-r\left(L_{1}\right)+1}}$,
while a \textbf{complete} flag is a flag of length $\scalebox{0.95}{\ensuremath{r\left(M\right)-1}}$.
\end{defn}

Both a complete flace and a complete flag are full. A flace is not
a flag and vice versa, but a relationship exists between them:
\begin{prop}
\label{prop:flaces-flags}For a flace $\scalebox{0.95}{\ensuremath{\left(F_{1},\dots,F_{m}\right)}}$
of $\scalebox{0.95}{\ensuremath{M}}$, there is a same-length flag
$\scalebox{0.95}{\ensuremath{\left(L_{1},\dots,L_{m}\right)}}$ with
$\scalebox{0.95}{\ensuremath{M\left(F_{1}\right)\cdots\left(F_{m}\right)=M\left(L_{1}\right)\cdots\left(L_{m}\right)}}$.
\end{prop}

\begin{proof}
We may assume $\scalebox{0.95}{\ensuremath{\left(F_{1},\dots,F_{m}\right)}}$
is a full flace. With $\scalebox{0.95}{\ensuremath{M_{0}:=M}}$, let
$\scalebox{0.95}{\ensuremath{M_{i}:=M_{i-1}\left(F_{i}\right)}}$
for $\scalebox{0.95}{\ensuremath{i=1,\dots,m}}$, and take the smallest
subset $\scalebox{0.95}{\ensuremath{F_{i}'}}$ of $\scalebox{0.95}{\ensuremath{E(M)}}$
with $\scalebox{0.95}{\ensuremath{M_{i}=M_{i-1}\left(F_{i}'\right)}}$.
Then, either $\scalebox{0.95}{\ensuremath{F_{i}'\subset F_{i-1}'}}$
or $\scalebox{0.95}{\ensuremath{F_{i}'\cap F_{i-1}'=\emptyset}}$
with $\scalebox{0.95}{\ensuremath{F_{i}'\cup F_{i-1}'\neq E(M)}}$.
Let $\scalebox{0.95}{\ensuremath{L_{i}:=\bigcup_{j\in\left[i\right]}F_{m+1-j}'}}$,
then $\scalebox{0.95}{\ensuremath{L_{1}=F_{m}'\supsetneq\overline{\emptyset}}}$
and $\scalebox{0.95}{\ensuremath{L_{m}=F_{1}'\cup\cdots\cup F_{m}'\neq E(M)}}$.
The flag $\scalebox{0.95}{\ensuremath{\left(L_{1},\dots,L_{m}\right)}}$
satisfies $\scalebox{0.95}{\ensuremath{M\left(F_{1}\right)\cdots\left(F_{m}\right)=M\left(L_{1}\right)\cdots\left(L_{m}\right)}}$
due to Lemma \ref{lem:commuting} and Proposition \ref{prop:Boolean}.
\end{proof}

\subsection{\label{subsec:Poset-U(M)}A poset $\mathcal{U}\left(M\right)\subset\mathcal{V}\left(M\right)$}

Let $\scalebox{0.95}{\ensuremath{M}}$ be a matroid. We assume that
$\scalebox{0.95}{\ensuremath{M}}$ is loopless by deleting all loops.
Given a flag $\scalebox{0.95}{\ensuremath{\left(L_{l},\dots,L_{1}\right)}}$\footnote{The subscripts of flats are written in decreasing order for technical
reasons.} of $\scalebox{0.95}{\ensuremath{M}}$, we write: 
\[
\scalebox{0.95}{\ensuremath{M(L_{l},\dots,L_{1})=M/L_{1}\oplus M|_{L_{1}}/L_{2}\oplus\cdots\oplus M|_{L_{l-1}}/L_{l}\oplus M|_{L_{l}}}}.
\]
 We define a subcollection $\scalebox{0.95}{\ensuremath{\mathcal{U}}}$
of $\scalebox{0.95}{\ensuremath{\mathcal{V}}}$ as:
\[
\scalebox{0.95}{\ensuremath{\mathcal{U}=\left\{ \text{all sub-expressions of }M(L_{l},\dots,L_{1})\text{ for all flags }\left(L_{l},\dots,L_{1}\right)\right\} \cup\left\{ \emptyset,M\right\} }}.
\]
A member of a poset other than the largest member (if it exists) is
said to be \textbf{proper}. Take any nonempty proper member $\scalebox{0.95}{\ensuremath{X\in\mathcal{U}}}$,
and let $\scalebox{0.95}{\ensuremath{\left(L_{l},\dots,L_{1}\right)}}$
be an increasing sequence of the flats appearing in the \emph{restriction-contraction}
forms of summands of $X$. Denote 
\[
\scalebox{0.95}{\ensuremath{\mathrm{Fl}\left(X\right)=\left(L_{l},\dots,L_{1}\right)\quad\text{and}\quad\left|\mathrm{Fl}\left(X\right)\right|=L_{1}}}.
\]
Denote by $\scalebox{0.95}{\ensuremath{M\bigl(\mathrm{Fl}\left(X\right)\bigr)}}$
or $\scalebox{0.95}{\ensuremath{\epsilon\left(X\right)}}$ the expression
$\scalebox{0.95}{\ensuremath{M\left(L_{l},\dots,L_{1}\right)}}$.
Then, $\scalebox{0.95}{\ensuremath{X\equiv\epsilon\left(X\right)}}$
if and only if $\scalebox{0.95}{\ensuremath{r\left(X\right)=r\left(M\right)}}$.
Denote by $\scalebox{0.95}{\ensuremath{X^{\perp}}}$ the direct sum
of the summands of $\scalebox{0.95}{\ensuremath{\epsilon\left(X\right)}}$
that are not also summands of $\scalebox{0.95}{\ensuremath{X}}$,
then $\scalebox{0.95}{\ensuremath{\epsilon\left(X\right)\equiv X\oplus X^{\perp}}}$
with 
\[
\scalebox{0.95}{\ensuremath{X\odot X^{\perp}\equiv\emptyset}}.
\]

\noindent In particular, if $\scalebox{0.95}{\ensuremath{X\equiv\epsilon\left(X\right)}}$
or $\scalebox{0.95}{\ensuremath{X\equiv\emptyset}}$, then $\scalebox{0.95}{\ensuremath{X^{\perp}\equiv\emptyset}}$
or $\scalebox{0.95}{\ensuremath{X^{\perp}\equiv\epsilon\left(X\right)}}$,
respectively.\smallskip{}

We define a partial order ``$\oleq$'' on the collection of all
minor expressions by: 
\[
\scalebox{0.95}{\ensuremath{\left(M|_{A}/C\right)\oleq\left(M|_{B}/D\right)\quad\text{if }A\subseteq B\text{ and }C\supseteq D}}.
\]
We extend $\oleq$ to $\scalebox{0.95}{\ensuremath{\mathcal{U}}}$
as done for $\osubeq$ in Subsection \ref{subsec:V=00003DV(M)}, then
$\scalebox{0.95}{\ensuremath{\left(\mathcal{U},\oleq\right)}}$ is
a \emph{poset}. Here, for $\scalebox{0.95}{\ensuremath{X,Y\in\mathcal{U}}}$,
$\scalebox{0.95}{\ensuremath{X\oleq Y}}$ implies $\scalebox{0.95}{\ensuremath{X\osubeq Y}}$.
Moreover 
\[
\scalebox{0.95}{\ensuremath{X\oleq Y}}\Longrightarrow\scalebox{0.95}{\ensuremath{Y^{\perp}\oleq X^{\perp}}}.
\]

\subsection{\label{subsec:meet-U(M)}$\mathcal{U}\left(M\right)$ is a lattice}

In this subsection, we show that the finite poset $\scalebox{0.95}{\ensuremath{\left(\mathcal{U},\oleq\right)}}$
is a meet-semilattice and a lattice since it contains the largest
member, $\scalebox{0.95}{\ensuremath{M}}$. Let $\scalebox{0.95}{\ensuremath{X}}$
and $\scalebox{0.95}{\ensuremath{Y}}$ be any nonempty proper members
of $\scalebox{0.95}{\ensuremath{\mathcal{U}}}$. These can be uniquely
written as: 
\begin{align*}
\scalebox{0.95}{\ensuremath{X}} & =\scalebox{0.95}{\ensuremath{M|_{F_{0}}/F_{1}\oplus M|_{F_{2}}/F_{3}\oplus\cdots\oplus M|_{F_{2m}}/F_{2m+1}}}\\
\scalebox{0.95}{\ensuremath{Y}} & =\scalebox{0.95}{\ensuremath{M|_{L_{0}}/L_{1}\oplus M|_{L_{2}}/L_{3}\oplus\cdots\oplus M|_{L_{2l}}/L_{2l+1}}}
\end{align*}
for \emph{decreasing} sequences of flats $\scalebox{0.95}{\ensuremath{\left(F_{0},F_{1},\dots,F_{2m+1}\right)}}$
and $\scalebox{0.95}{\ensuremath{\left(L_{0},L_{1},\dots,L_{2l+1}\right)}}$
such that no terms are empty. Note that $\scalebox{0.95}{\ensuremath{F_{0}}}$
and $\scalebox{0.95}{\ensuremath{L_{0}}}$ do not necessarily equal
$\scalebox{0.95}{\ensuremath{E\left(M\right)}}$. We denote the \emph{greatest
lower bound} of $\scalebox{0.95}{\ensuremath{X}}$ and $\scalebox{0.95}{\ensuremath{Y}}$
by $\scalebox{0.95}{\ensuremath{X\circledast Y}}$ if it exists.\vspace{2pt}

For $\scalebox{0.95}{\ensuremath{V=M|_{A}/A'\in\mathcal{U}}}$, suppose
$\scalebox{0.95}{\ensuremath{V\oleq X}}$ and $\scalebox{0.95}{\ensuremath{V\oleq Y}}$,
then there is $\scalebox{0.95}{\ensuremath{i\in\left[m\right]\cup\left\{ 0\right\} }}$
and $\scalebox{0.95}{\ensuremath{j\in\left[l\right]\cup\left\{ 0\right\} }}$
with $\scalebox{0.95}{\ensuremath{A\subseteq F_{2i}\cap L_{2j}}}$
and $\scalebox{0.95}{\ensuremath{A'\supseteq\overline{F_{2i+1}\cup L_{2j+1}}}}$.
In this case:
\begin{equation}
\scalebox{0.95}{\ensuremath{F_{2i}\supseteq L_{2j+1}}}\text{ and \scalebox{0.95}{\ensuremath{L_{2j}\supseteq F_{2i+1}}}}.\label{eq:link-cond}
\end{equation}
 Let $\scalebox{0.95}{\ensuremath{U_{i,j}:=M|_{F_{2i}\cap L_{2j}}/\overline{F_{2i+1}\cup L_{2j+1}}}}$,
then $\scalebox{0.95}{\ensuremath{U_{i,j}\oleq X}}$ and $\scalebox{0.95}{\ensuremath{U_{i,j}\oleq Y}}$
with $\scalebox{0.95}{\ensuremath{V\oleq U_{i,j}}}$. Thus, if the
direct sum $\scalebox{0.95}{\ensuremath{U}}$ of all those $\scalebox{0.95}{\ensuremath{U_{i,j}}}$
for $\scalebox{0.95}{\ensuremath{i}}$ and $\scalebox{0.95}{\ensuremath{j}}$
with \eqref{eq:link-cond} is a member of $\scalebox{0.95}{\ensuremath{\mathcal{U}}}$,
then $\scalebox{0.95}{\ensuremath{U=X\circledast Y}}$. This is our
idea to show the existence of $\scalebox{0.95}{\ensuremath{X\circledast Y}}$.

Suppose that all four pairs of $\scalebox{0.95}{\ensuremath{\left\{ i,i'\right\} \times\left\{ j,j'\right\} \subseteq\left(\left[m\right]\cup\left\{ 0\right\} \right)\times\left(\left[l\right]\cup\left\{ 0\right\} \right)}}$
with $\scalebox{0.95}{\ensuremath{i<i'}}$ and $\scalebox{0.95}{\ensuremath{j<j'}}$
satisfy \eqref{eq:link-cond}, then 
\[
\scalebox{0.95}{\ensuremath{F_{2i+1}\supseteq F_{2i'}\supseteq L_{2j+1}\supseteq L_{2j'}\supseteq F_{2i+1}}}
\]
 and $\scalebox{0.95}{\ensuremath{F_{2i+1}=F_{2i+2}=\cdots=F_{2i'}=L_{2j+1}=L_{2j+2}=\cdots=L_{2j'}}}$.
Because $\scalebox{0.95}{\ensuremath{X}}$ and $\scalebox{0.95}{\ensuremath{Y}}$
have no empty summands, $\scalebox{0.95}{\ensuremath{2i+2=2i'}}$
and $\scalebox{0.95}{\ensuremath{2j+2=2j'}}$, that is, $\scalebox{0.95}{\ensuremath{i'=i+1}}$
and $\scalebox{0.95}{\ensuremath{j'=j+1}}$, with $\scalebox{0.95}{\ensuremath{F_{2i+1}=F_{2i'}=L_{2j+1}=L_{2j'}}}$. 

So, let $\scalebox{0.95}{\ensuremath{\mathfrak{i}\subseteq\left[m\right]\cup\left\{ 0\right\} }}$
and $\scalebox{0.95}{\ensuremath{\mathfrak{j}\subseteq\left[l\right]\cup\left\{ 0\right\} }}$
be maximal subsets such that any $\scalebox{0.95}{\ensuremath{i\in\mathfrak{i}}}$
and any $\scalebox{0.95}{\ensuremath{j\in\mathfrak{j}}}$ satisfy
\eqref{eq:link-cond}. Then, both $\scalebox{0.95}{\ensuremath{\mathfrak{i}}}$
and $\scalebox{0.95}{\ensuremath{\mathfrak{j}}}$ are collections
of consecutive numbers, and only one of the following four cases happens
for the pair $\scalebox{0.95}{\ensuremath{\left(\mathfrak{i},\mathfrak{j}\right)}}$:
\begin{enumerate}
\item $\scalebox{0.95}{\ensuremath{\left|\mathfrak{i}\right|=2}}$ and $\scalebox{0.95}{\ensuremath{\left|\mathfrak{j}\right|=2}}$.
\item $\scalebox{0.95}{\ensuremath{\left|\mathfrak{i}\right|=1}}$ and $\scalebox{0.95}{\ensuremath{\left|\mathfrak{j}\right|\ge2}}$.
\item $\scalebox{0.95}{\ensuremath{\left|\mathfrak{i}\right|\ge2}}$ and
$\scalebox{0.95}{\ensuremath{\left|\mathfrak{j}\right|=1}}$.
\item $\scalebox{0.95}{\ensuremath{\left|\mathfrak{i}\right|=1}}$ and $\scalebox{0.95}{\ensuremath{\left|\mathfrak{j}\right|=1}}$.
\end{enumerate}
Let $\scalebox{0.95}{\ensuremath{\left(\mathfrak{i}',\mathfrak{j}'\right)}}$
be another pair of maximal subsets of $\scalebox{0.95}{\ensuremath{\left[m\right]\cup\left\{ 0\right\} }}$
and $\scalebox{0.95}{\ensuremath{\left[l\right]\cup\left\{ 0\right\} }}$,
respectively, such that any $\scalebox{0.95}{\ensuremath{i\in\mathfrak{i}'}}$
and any $\scalebox{0.95}{\ensuremath{j\in\mathfrak{j}'}}$ satisfy
\eqref{eq:link-cond}. Then, there is no inclusion between $\scalebox{0.95}{\ensuremath{\mathfrak{i}}}$
and $\scalebox{0.95}{\ensuremath{\mathfrak{i}'}}$, nor between $\scalebox{0.95}{\ensuremath{\mathfrak{j}}}$
and $\scalebox{0.95}{\ensuremath{\mathfrak{j}'}}$. Moreover, $\scalebox{0.95}{\ensuremath{\left|\mathfrak{i}\cap\mathfrak{i}'\right|=\left|\mathfrak{j}\cap\mathfrak{j}'\right|\le1}}$,
where equality holds only if $\scalebox{0.95}{\ensuremath{\left|\mathfrak{i}\right|=\left|\mathfrak{i}'\right|=\left|\mathfrak{j}\right|=\left|\mathfrak{j}'\right|=2}}$,
and either:
\begin{itemize}
\item $\scalebox{0.95}{\ensuremath{\min\left(\mathfrak{i}\right)<\min\left(\mathfrak{i}'\right)}}$
and $\scalebox{0.95}{\ensuremath{\min\left(\mathfrak{j}\right)<\min\left(\mathfrak{j}'\right)}}$,
or
\item $\scalebox{0.95}{\ensuremath{\min\left(\mathfrak{i}\right)>\min\left(\mathfrak{i}'\right)}}$
and $\scalebox{0.95}{\ensuremath{\min\left(\mathfrak{j}\right)>\min\left(\mathfrak{j}'\right)}}$.
\end{itemize}
This enables us to define a partial order ``$<$'' on the collection
of all pairs of those maximal subsets, for temporary use: 
\begin{equation}
\scalebox{0.95}{\ensuremath{\left(\mathfrak{i},\mathfrak{j}\right)<\left(\mathfrak{i}',\mathfrak{j}'\right)}}\quad\text{if }\scalebox{0.95}{\ensuremath{\min\left(\mathfrak{i}\right)<\min\left(\mathfrak{i}'\right)}}\text{, or equivalently, }\scalebox{0.95}{\ensuremath{\min\left(\mathfrak{j}\right)<\min\left(\mathfrak{j}'\right)}}.\label{eq:partial-order}
\end{equation}
 Then, ``$\le$'' is a linear order. We assume $\scalebox{0.95}{\ensuremath{\left(\mathfrak{i},\mathfrak{j}\right)<\left(\mathfrak{i}',\mathfrak{j}'\right)}}$
in the following. We consider case (1) for $\scalebox{0.95}{\ensuremath{\left(\mathfrak{i},\mathfrak{j}\right)}}$
first, so $\scalebox{0.95}{\ensuremath{\mathfrak{i}=\left\{ i,i+1\right\} }}$
and $\scalebox{0.95}{\ensuremath{\mathfrak{j}=\left\{ j,j+1\right\} }}$
for some $\scalebox{0.95}{\ensuremath{i\in\left[m\right]\cup\left\{ 0\right\} }}$
and $\scalebox{0.95}{\ensuremath{j\in\left[l\right]\cup\left\{ 0\right\} }}$.
Then, $\scalebox{0.95}{\ensuremath{F_{2i+1}=F_{2i+2}=L_{2j+1}=L_{2j+2}}}$
and 
\[
\left\{ \begin{array}{rl}
\scalebox{0.95}{\ensuremath{U_{i,j}}} & =\scalebox{0.95}{\ensuremath{M|_{F_{2i}\cap L_{2j}}/\overline{F_{2i+1}\cup L_{2j+1}}}}\\
\scalebox{0.95}{\ensuremath{U_{i,j+1}}} & =\scalebox{0.95}{\ensuremath{M|_{L_{2j+2}}/F_{2i+1}\equiv\emptyset}}\\
\scalebox{0.95}{\ensuremath{U_{i+1,j}}} & =\scalebox{0.95}{\ensuremath{M|_{F_{2i+2}}/L_{2j+1}\equiv\emptyset}}\\
\scalebox{0.95}{\ensuremath{U_{i+1,j+1}}} & =\scalebox{0.95}{\ensuremath{M|_{F_{2i+2}\cap L_{2j+2}}/\overline{F_{2i+3}\cup L_{2j+3}}}}
\end{array}\right\} .
\]
 Therefore, let $\scalebox{0.95}{\ensuremath{U_{\mathfrak{i},\mathfrak{j}}}}$
denote the direct sum of $\scalebox{0.95}{\ensuremath{U_{i,j}}}$
for all $\scalebox{0.95}{\ensuremath{i\in\mathfrak{i}}}$ and $\scalebox{0.95}{\ensuremath{j\in\mathfrak{j}}}$,
then 
\[
\scalebox{0.95}{\ensuremath{U_{\mathfrak{i},\mathfrak{j}}}}=\scalebox{0.95}{\ensuremath{U_{i,j}\oplus U_{i+1,j+1}}}\in\scalebox{0.95}{\ensuremath{\mathcal{U}}}.
\]
We now check each case for $\scalebox{0.95}{\ensuremath{\left(\mathfrak{i}',\mathfrak{j}'\right)}}$.
Suppose case (1) for $\scalebox{0.95}{\ensuremath{\left(\mathfrak{i}',\mathfrak{j}'\right)}}$:
$\scalebox{0.95}{\ensuremath{\left|\mathfrak{i}'\right|=2}}$ and
$\scalebox{0.95}{\ensuremath{\left|\mathfrak{j}'\right|=2}}$. Then,
$\scalebox{0.95}{\ensuremath{\max\left(\mathfrak{i}\right)\le\min\left(\mathfrak{i}'\right)}}$
and $\scalebox{0.95}{\ensuremath{\max\left(\mathfrak{j}\right)\le\min\left(\mathfrak{j}'\right)}}$;
hence, $\scalebox{0.95}{\ensuremath{U_{\mathfrak{i},\mathfrak{j}}\oplus U_{\mathfrak{i}',\mathfrak{j}'}\in\mathcal{U}}}$.
Suppose case (2) for $\scalebox{0.95}{\ensuremath{\left(\mathfrak{i}',\mathfrak{j}'\right)}}$,
then $\scalebox{0.95}{\ensuremath{\mathfrak{i}=\left\{ i'\right\} }}$
and $\scalebox{0.95}{\ensuremath{U_{\mathfrak{i}',\mathfrak{j}'}}}=\scalebox{0.95}{\ensuremath{\bigoplus_{l\in\mathfrak{j}'}U_{i',l}}}\in\scalebox{0.95}{\ensuremath{\mathcal{U}}}.$
Also, $\scalebox{0.95}{\ensuremath{\mathfrak{i}\cap\mathfrak{i}'=\mathfrak{j}\cap\mathfrak{j}'=\emptyset}}$
with $\scalebox{0.95}{\ensuremath{\max\left(\mathfrak{i}\right)<\min\left(\mathfrak{i}'\right)}}$
and $\scalebox{0.95}{\ensuremath{\max\left(\mathfrak{j}\right)<\min\left(\mathfrak{j}'\right)}}$;
hence, $U_{\mathfrak{i},\mathfrak{j}}\oplus U_{\mathfrak{i}',\mathfrak{j}'}\in\mathcal{U}$.
Case (3) for $\scalebox{0.95}{\ensuremath{\left(\mathfrak{i}',\mathfrak{j}'\right)}}$
is symmetric, and case (4) for $\scalebox{0.95}{\ensuremath{\left(\mathfrak{i}',\mathfrak{j}'\right)}}$
is similar. The other three cases for $\scalebox{0.95}{\ensuremath{\left(\mathfrak{i},\mathfrak{j}\right)}}$
are similar to case (1) above,  and we  conclude that $\scalebox{0.95}{\ensuremath{X\circledast Y}}$
exists.\smallskip{}

We will now show that $\scalebox{0.95}{\ensuremath{\left(\mathcal{U},\oleq,\circledast\right)}}$
is a meet-semilattice. It is obvious that $\circledast$ is commutative
and idempotent. The key is to establish the associativity of $\circledast$.
To do so, we introduce another expression without any empty summands:
\[
\scalebox{0.95}{\ensuremath{Z=M|_{T_{0}}/T_{1}\oplus M|_{T_{2}}/T_{3}\oplus\cdots\oplus M|_{T_{2n}}/T_{2n+1}\in\mathcal{U}}}
\]
where $\scalebox{0.95}{\ensuremath{\left(T_{0},T_{1},\dots,T_{2n+1}\right)}}$
is a decreasing sequence of flats. From the previous argument, we
know that $\scalebox{0.95}{\ensuremath{\left(X\circledast Y\right)\circledast Z}}$
is the direct sum of all the terms 
\[
\scalebox{0.95}{\ensuremath{M|_{\left(F_{2i}\cap L_{2j}\right)\cap T_{2k}}/\overline{\overline{F_{2i+1}\cup L_{2j+1}}\cup T_{2k+1}}=M|_{F_{2i}\cap L_{2j}\cap T_{2k}}/\overline{F_{2i+1}\cup L_{2j+1}\cup T_{2k+1}}}}
\]
 with $\scalebox{0.95}{\ensuremath{i}}$, $\scalebox{0.95}{\ensuremath{j}}$,
and $\scalebox{0.95}{\ensuremath{k}}$ satisfying the conditions $\scalebox{0.95}{\ensuremath{F_{2i}\cap L_{2j}\supseteq T_{2k+1}}}$,
$\scalebox{0.95}{\ensuremath{T_{2k}\supseteq\overline{F_{2i+1}\cup L_{2j+1}}}}$,
and \eqref{eq:link-cond}, which is equivalent to satisfying: 
\[
\left\{ \begin{array}{c}
\scalebox{0.95}{\ensuremath{F_{2i}\supseteq L_{2j+1}}}\text{ and }\scalebox{0.95}{\ensuremath{L_{2j}\supseteq F_{2i+1}}}\\
\scalebox{0.95}{\ensuremath{F_{2i}\supseteq T_{2k+1}}}\text{ and }\scalebox{0.95}{\ensuremath{T_{2k}\supseteq F_{2i+1}}}\\
\scalebox{0.95}{\ensuremath{L_{2j}\supseteq T_{2k+1}}}\text{ and }\scalebox{0.95}{\ensuremath{T_{2k}\supseteq L_{2j+1}}}
\end{array}\right\} .
\]
The symmetry of these formulas implies: 
\[
\scalebox{0.95}{\ensuremath{\left(X\circledast Y\right)\circledast Z=X\circledast\left(Y\circledast Z\right)=Y\circledast\left(X\circledast Z\right)}}.
\]
 Thus, the associativity of $\circledast$ is proved. It follows that
$\scalebox{0.95}{\ensuremath{\mathcal{U}}}$ is a finite meet-semilattice
with the greatest member $\scalebox{0.95}{\ensuremath{M}}$, and hence
it is a lattice.
\begin{rem}
\label{rem:relation-bw-glbs}
\begin{enumerate}
\item $\circledast$ is \emph{not} distributive over direct sum, unlike
$\odot$.
\item Take any $\scalebox{0.95}{\ensuremath{X,Y\in\mathcal{U}}}$, then
$\scalebox{0.95}{\ensuremath{X\circledast Y\osubeq X\odot Y}}$. Also,
$\scalebox{0.95}{\ensuremath{X\circledast X^{\perp}\equiv\emptyset}}$.
\end{enumerate}
\end{rem}

\begin{prop}
\label{prop:meet-int}For $\scalebox{0.95}{\ensuremath{X,Y\in\mathcal{U}}}$
with $\scalebox{0.95}{\ensuremath{r(X)=r(Y)=r(M)}}$, the following
are equivalent.
\begin{enumerate}
\item $\scalebox{0.95}{\ensuremath{r\left(X\circledast Y\right)=r\left(M\right)}}$
\item $\scalebox{0.95}{\ensuremath{E\left(X\circledast Y\right)=E\left(M\right)}}$
\item $\scalebox{0.95}{\ensuremath{\mathrm{Fl}\left(X\circledast Y\right)=\mathrm{Fl}\left(X\right)\cup\mathrm{Fl}\left(Y\right)}}$
\end{enumerate}
\end{prop}

\begin{proof}
Just observe that all these three are equivalent to the statement
that for any flats $\scalebox{0.95}{\ensuremath{F\in\mathrm{Fl}\left(X\right)}}$
and $\scalebox{0.95}{\ensuremath{L\in\mathrm{Fl}\left(Y\right)}}$,
either $F\subseteq L$ or $F\supseteq L$.
\end{proof}
\begin{cor}
\label{cor:TwoMeets}Let $\scalebox{0.95}{\ensuremath{X,Y\in\mathcal{U}\left(M\right)}}$
be of full rank. Then, $\scalebox{0.95}{\ensuremath{r\left(X\circledast Y\right)=r\left(M\right)}}$
if and only if $\scalebox{0.95}{\ensuremath{X\circledast Y\equiv X\odot Y}}$.
\end{cor}

\begin{example}
\label{exa:meet}There is one and only one connected matroid on $\scalebox{0.95}{\ensuremath{\left[5\right]}}$
of rank $\scalebox{0.95}{\ensuremath{3}}$ with rank-$\scalebox{0.95}{\ensuremath{2}}$
flats $\scalebox{0.95}{\ensuremath{\left\{ 1,2,3\right\} }}$ and
$\scalebox{0.95}{\ensuremath{\left\{ 3,4,5\right\} }}$, say $\scalebox{0.95}{\ensuremath{M}}$,
cf. Lemma \ref{lem:degen-rank1-flat}. To illustrate, we compute some
examples:\smallskip{}

\noindent %
\begin{tabular}{l}
\hline 
\noalign{\vskip\doublerulesep}
\hspace{6.05em}$\scalebox{0.9}{\ensuremath{M|_{\left\{ 1,2,3\right\} }\circledast M|_{\left\{ 3,4,5\right\} }=M|_{\left\{ 3\right\} }=M|_{\left\{ 1,2,3\right\} }\odot M|_{\left\{ 3,4,5\right\} }}}$,\tabularnewline[\doublerulesep]
\noalign{\vskip\doublerulesep}
\hspace{4.95em}$\scalebox{0.9}{\ensuremath{M/\left\{ 1,2,3\right\} \circledast M/\left\{ 3,4,5\right\} =\emptyset=M/\left\{ 1,2,3\right\} \odot M/\left\{ 3,4,5\right\} }}$.\tabularnewline[\doublerulesep]
\hline 
\noalign{\vskip\doublerulesep}
\hspace{3em}$\scalebox{0.9}{\ensuremath{M|_{\left\{ 1,2,3\right\} }\circledast M/\left\{ 3,4,5\right\} =\emptyset\neq M/\left\{ 3,4,5\right\} |_{\left\{ 1,2\right\} }=M|_{\left\{ 1,2,3\right\} }\odot M/\left\{ 3,4,5\right\} }}$,\tabularnewline[\doublerulesep]
\noalign{\vskip\doublerulesep}
\hspace{3em}$\scalebox{0.9}{\ensuremath{M/\left\{ 1,2,3\right\} \circledast M|_{\left\{ 3,4,5\right\} }=\emptyset\neq M/\left\{ 1,2,3\right\} |_{\left\{ 4,5\right\} }=M/\left\{ 1,2,3\right\} \odot M|_{\left\{ 3,4,5\right\} }}}$.\tabularnewline[\doublerulesep]
\hline 
\noalign{\vskip\doublerulesep}
$\scalebox{0.88}{\ensuremath{\left(M/\left\{ 1,2,3\right\} \oplus M|_{\left\{ 1,2,3\right\} }\right)\circledast\left(M/\left\{ 3,4,5\right\} \oplus M|_{\left\{ 3,4,5\right\} }\right)=M|_{\left\{ 3\right\} }}}$,\tabularnewline[\doublerulesep]
\noalign{\vskip\doublerulesep}
$\scalebox{0.88}{\ensuremath{\left(M/\left\{ 1,2,3\right\} \oplus M|_{\left\{ 1,2,3\right\} }\right)\odot\left(M/\left\{ 3,4,5\right\} \oplus M|_{\left\{ 3,4,5\right\} }\right)=M/\left\{ 1,2,3\right\} \oplus M/\left\{ 3,4,5\right\} \oplus M|_{\left\{ 3\right\} }}}$.\tabularnewline[\doublerulesep]
\hline 
\noalign{\vskip\doublerulesep}
\end{tabular}\medskip{}
\end{example}

\subsection{\label{subsec:join-U(M)}The join of $\mathcal{U}\left(M\right)$}

We now discuss the join operation ``$\ovee$'' of the lattice $\scalebox{0.95}{\ensuremath{\mathcal{U}}}$.
Consider nonempty proper expressions $\scalebox{0.95}{\ensuremath{X,Y,Z\in\mathcal{U}}}$
and write them as before. For $\scalebox{0.95}{\ensuremath{i\in\left[m\right]\cup\left\{ 0\right\} }}$,
$\scalebox{0.95}{\ensuremath{j\in\left[l\right]\cup\left\{ 0\right\} }}$,
and $\scalebox{0.95}{\ensuremath{k\in\left[n\right]\cup\left\{ 0\right\} }}$,
let 
\[
\scalebox{0.95}{\ensuremath{X_{i}=M|_{F_{2i}}/F_{2i+1}}},\scalebox{0.95}{\ensuremath{Y_{j}=M|_{L_{2j}}/L_{2j+1}}},\text{ and }\scalebox{0.95}{\ensuremath{Z_{k}=M|_{T_{2k}}/T_{2k+1}}}.
\]
 Suppose $\scalebox{0.95}{\ensuremath{X_{i}\oleq Z_{k}}}$ and $\scalebox{0.95}{\ensuremath{Y_{j}\oleq Z_{k}}}$,
then $\scalebox{0.95}{\ensuremath{T_{2k}\supseteq\overline{F_{2i}\cup L_{2j}}}}$,
$\scalebox{0.95}{\ensuremath{F_{2i+1}\cap L_{2j+1}\supseteq T_{2k+1}}}$,
and $\scalebox{0.95}{\ensuremath{T_{2k+2}\supseteq\overline{F_{2i+2}\cup L_{2j+2}}}}$.
In this case: 
\begin{equation}
\scalebox{0.95}{\ensuremath{F_{2i+1}\supseteq L_{2j+2}}}\text{ and \scalebox{0.95}{\ensuremath{L_{2j+1}\supseteq F_{2i+2}}}}.\label{eq:link-cond-1}
\end{equation}
 Let $\scalebox{0.95}{\ensuremath{U_{i,j}:=M|_{\overline{F_{2i}\cup L_{2j}}}/(F_{2i+1}\cap L_{2j+1})}}$,
then $\scalebox{0.95}{\ensuremath{X_{i}\oleq U_{i,j}}}$, $\scalebox{0.95}{\ensuremath{Y_{j}\oleq U_{i,j}}}$,
and $\scalebox{0.95}{\ensuremath{U_{i,j}}}\oleq\scalebox{0.95}{\ensuremath{Z_{k}}}$.\smallskip{}

Suppose that all four pairs of $\scalebox{0.95}{\ensuremath{\left\{ i,i'\right\} \times\left\{ j,j'\right\} \subseteq\left(\left[m\right]\cup\left\{ 0\right\} \right)\times\left(\left[l\right]\cup\left\{ 0\right\} \right)}}$
with $\scalebox{0.95}{\ensuremath{i<i'}}$ and $\scalebox{0.95}{\ensuremath{j<j'}}$
satisfy \eqref{eq:link-cond-1}, then 
\[
\scalebox{0.95}{\ensuremath{F_{2i+2}\supseteq F_{2i'+1}\supseteq L_{2j+2}\supseteq L_{2j'+1}\supseteq F_{2i+2}}}
\]
 and $\scalebox{0.95}{\ensuremath{F_{2i+2}=F_{2i+3}=\cdots=F_{2i'+1}=L_{2j+2}=L_{2j+3}=\cdots=L_{2j'+1}}}$.
This contradicts the assumption that $\scalebox{0.95}{\ensuremath{X}}$
and $\scalebox{0.95}{\ensuremath{Y}}$ have no empty summands. So,
let $\scalebox{0.95}{\ensuremath{\mathfrak{i}\subseteq\left[m\right]\cup\left\{ 0\right\} }}$
and $\scalebox{0.95}{\ensuremath{\mathfrak{j}\subseteq\left[l\right]\cup\left\{ 0\right\} }}$
be maximal subsets such that any $\scalebox{0.95}{\ensuremath{i\in\mathfrak{i}}}$
and any $\scalebox{0.95}{\ensuremath{j\in\mathfrak{j}}}$ satisfy
\eqref{eq:link-cond-1}. Then, $\scalebox{0.95}{\ensuremath{\mathfrak{i}}}$
and $\scalebox{0.95}{\ensuremath{\mathfrak{j}}}$ are collections
of consecutive numbers, and there are precisely three cases for the
pair $\scalebox{0.95}{\ensuremath{\left(\mathfrak{i},\mathfrak{j}\right)}}$:
\begin{enumerate}
\item $\scalebox{0.95}{\ensuremath{\left|\mathfrak{i}\right|=1}}$ and $\scalebox{0.95}{\ensuremath{\left|\mathfrak{j}\right|=2}}$.
\item $\scalebox{0.95}{\ensuremath{\left|\mathfrak{i}\right|=2}}$ and $\scalebox{0.95}{\ensuremath{\left|\mathfrak{j}\right|=1}}$.
\item $\scalebox{0.95}{\ensuremath{\left|\mathfrak{i}\right|=1}}$ and $\scalebox{0.95}{\ensuremath{\left|\mathfrak{j}\right|=1}}$.
\end{enumerate}
Let $\scalebox{0.95}{\ensuremath{\left(\mathfrak{i}',\mathfrak{j}'\right)}}$
be another pair of maximal subsets of $\scalebox{0.95}{\ensuremath{\left[m\right]\cup\left\{ 0\right\} }}$
and $\scalebox{0.95}{\ensuremath{\left[l\right]\cup\left\{ 0\right\} }}$,
respectively, such that any $\scalebox{0.95}{\ensuremath{i\in\mathfrak{i}'}}$
and any $\scalebox{0.95}{\ensuremath{j\in\mathfrak{j}'}}$ satisfy
\eqref{eq:link-cond-1}. Then, we have $\scalebox{0.95}{\ensuremath{\mathfrak{i}\cap\mathfrak{i}'=\mathfrak{j}\cap\mathfrak{j}'=\emptyset}}$,
and $\scalebox{0.95}{\ensuremath{\min\left(\mathfrak{i}\right)<\min\left(\mathfrak{i}'\right)}}$
if and only if $\scalebox{0.95}{\ensuremath{\min\left(\mathfrak{j}\right)<\min\left(\mathfrak{j}'\right)}}$.
Thus, the collection of pairs of maximal subsets of our interest is
ordered by the same linear order as \eqref{eq:partial-order}. Therefore,
the direct sum of all $\scalebox{0.95}{\ensuremath{U_{i,j}}}$ with
$\scalebox{0.95}{\ensuremath{i}}$ and $\scalebox{0.95}{\ensuremath{j}}$
satisfying \eqref{eq:link-cond-1} is $\scalebox{0.95}{\ensuremath{X\ovee Y}}$.
\begin{example}
\label{exa:join}Let $\scalebox{0.95}{\ensuremath{M}}$ be the connected
matroid of Example \ref{exa:meet}. Then:\smallskip{}
\noindent \begin{center}
\begin{tabular}{c}
\hline 
\noalign{\vskip\doublerulesep}
$\scalebox{0.95}{\ensuremath{M|_{\left\{ 1,2,3\right\} }\ovee M|_{\left\{ 3,4,5\right\} }=M|_{\left\{ 1,2,3\right\} }\ovee M/\left\{ 3,4,5\right\} =M/\left\{ 1,2,3\right\} \ovee M|_{\left\{ 3,4,5\right\} }=M}}$,\tabularnewline[\doublerulesep]
\noalign{\vskip\doublerulesep}
$\scalebox{0.95}{\ensuremath{M/\left\{ 1,2,3\right\} \ovee M/\left\{ 3,4,5\right\} =M/\left\{ 3\right\} }}$.\tabularnewline[\doublerulesep]
\hline 
\noalign{\vskip\doublerulesep}
$\scalebox{0.95}{\ensuremath{M|_{\left\{ 1\right\} }\ovee M/\left\{ 2\right\} =M/\left\{ 1\right\} \ovee M|_{\left\{ 2\right\} }=M/\left\{ 1\right\} \ovee M/\left\{ 2\right\} =M}}$,\tabularnewline[\doublerulesep]
\noalign{\vskip\doublerulesep}
$\scalebox{0.95}{\ensuremath{M|_{\left\{ 1\right\} }\ovee M|_{\left\{ 2\right\} }=M|_{\left\{ 1,2,3\right\} }}}$.\tabularnewline[\doublerulesep]
\hline 
\noalign{\vskip\doublerulesep}
\end{tabular}
\par\end{center}
\end{example}

\subsection{\label{subsec:Lattice-W(M)}A lattice $\mathcal{W}\left(M\right)\subset\mathcal{U}\left(M\right)$}

Consider a subcollection $\scalebox{0.95}{\ensuremath{\mathcal{W}}}$
of $\scalebox{0.95}{\ensuremath{\mathcal{U}}}$:
\[
\scalebox{0.95}{\ensuremath{\mathcal{W}=\left\{ X\in\mathcal{U}:M/\left|\mathrm{Fl}\left(X\right)\right|\text{ is a (nonempty) summand of }X\right\} \cup\left\{ \emptyset,M\right\} }}.
\]
This is a finite meet-semilattice and is a lattice for which the join
of $\scalebox{0.95}{\ensuremath{X}}$ and $\scalebox{0.95}{\ensuremath{Y}}$
is $\scalebox{0.95}{\ensuremath{X\ovee Y}}$ and the meet of them
is 
\[
\scalebox{0.95}{\ensuremath{X\owedge Y:=\begin{cases}
X\circledast Y & \text{if }\overline{\left|\mathrm{Fl}\left(X\right)\right|\cup\left|\mathrm{Fl}\left(Y\right)\right|}\neq E\left(M\right),\\
\emptyset & \text{otherwise}.
\end{cases}}}
\]
  Thus, $\scalebox{0.95}{\ensuremath{\left(\mathcal{W},\oleq,\owedge,\ovee\right)}}$
is a lattice.
\begin{example}
Let $\scalebox{0.95}{\ensuremath{M}}$ be the connected matroid of
Example \ref{exa:meet}. Then:\smallskip{}
\noindent \begin{center}
\begin{tabular}{l}
\hline 
\noalign{\vskip\doublerulesep}
\hspace{1.55em}$\scalebox{0.9}{\ensuremath{M|_{\left\{ 1,2,3\right\} }\owedge M|_{\left\{ 3,4,5\right\} }=M|_{\left\{ 3\right\} }=M|_{\left\{ 1,2,3\right\} }\circledast M|_{\left\{ 3,4,5\right\} }}}$,\tabularnewline[\doublerulesep]
\noalign{\vskip\doublerulesep}
\hspace{0.45em}$\scalebox{0.9}{\ensuremath{M/\left\{ 1,2,3\right\} \owedge M/\left\{ 3,4,5\right\} =\emptyset=M/\left\{ 1,2,3\right\} \circledast M/\left\{ 3,4,5\right\} }}$,\tabularnewline[\doublerulesep]
\noalign{\vskip\doublerulesep}
\hspace{1.5em}$\scalebox{0.9}{\ensuremath{M|_{\left\{ 1,2,3\right\} }\owedge M/\left\{ 3,4,5\right\} =\emptyset=M|_{\left\{ 1,2,3\right\} }\circledast M/\left\{ 3,4,5\right\} }}$,\tabularnewline[\doublerulesep]
\noalign{\vskip\doublerulesep}
\hspace{1.5em}$\scalebox{0.9}{\ensuremath{M/\left\{ 1,2,3\right\} \owedge M|_{\left\{ 3,4,5\right\} }=\emptyset=M/\left\{ 1,2,3\right\} \circledast M|_{\left\{ 3,4,5\right\} }}}$.\tabularnewline[\doublerulesep]
\hline 
\noalign{\vskip\doublerulesep}
$\scalebox{0.88}{\ensuremath{\left(M/\left\{ 1,2,3\right\} \oplus M|_{\left\{ 1,2,3\right\} }\right)\owedge\left(M/\left\{ 3,4,5\right\} \oplus M|_{\left\{ 3,4,5\right\} }\right)=\emptyset}}$.\tabularnewline[\doublerulesep]
\hline 
\end{tabular}\medskip{}
\par\end{center}
\end{example}

\subsection{\label{subsec:Mat-lat}Matroidal lattice}

Fix a \emph{loopless} matroid $\scalebox{0.95}{\ensuremath{M}}$ of
rank $k$. Denote
\[
\scalebox{0.95}{\ensuremath{\phi\left(\mathcal{W}\right)=\left\{ \phi\left(X\right):X\in\mathcal{W}\right\} }}.
\]
Its rank-$k$ members are loopless face matroids of $\scalebox{0.95}{\ensuremath{M}}$.
Conversely, all loopless face matroids of $\scalebox{0.95}{\ensuremath{M}}$
are members of $\scalebox{0.95}{\ensuremath{\phi\left(\mathcal{W}\right)}}$
by Proposition \ref{prop:flaces-flags}. So, let $\scalebox{0.95}{\ensuremath{\mathcal{F}=\mathcal{F}\left(M\right)}}$
be the collection of all loopless face matroids of $\scalebox{0.95}{\ensuremath{M}}$,
then 
\[
\scalebox{0.95}{\ensuremath{\mathcal{F}=\left\{ \phi\bigl(M(\mathrm{Fl}(X))\bigr):X\ensuremath{\in\mathcal{W}}\right\} \cup\left\{ \emptyset,M\right\} \subset\phi\left(\mathcal{W}\right)}}.
\]
 This is a lattice of rank $\scalebox{0.95}{\ensuremath{k-\kappa\left(M\right)+1}}$,
cf. Subsection \ref{subsec:Face-lattice}. Moreover, it is a coatomistic
lattice whose coatoms are $\scalebox{0.95}{\ensuremath{M\left(F\right)}}$
for all non-degenerate flats $\scalebox{0.95}{\ensuremath{F}}$. For
$\scalebox{0.95}{\ensuremath{\phi\left(X\right),\phi\left(Y\right)\in\mathcal{F}}}$
with $\scalebox{0.95}{\ensuremath{X,Y\in\mathcal{W}}}$, the meet
of $\scalebox{0.95}{\ensuremath{\phi\left(X\right)}}$ and $\scalebox{0.95}{\ensuremath{\phi\left(Y\right)}}$
is:
\[
\scalebox{0.95}{\ensuremath{\phi\left(X\right)\logof\phi\left(Y\right):=\begin{cases}
\phi\left(X\odot Y\right)=\phi\left(X\right)\cap\phi\left(Y\right) & \text{if }\phi\left(X\odot Y\right)\in\mathcal{F},\\
\emptyset & \text{otherwise}.
\end{cases}}}
\]
 This is well-defined by Theorem \ref{thm:face-intersection} and
Proposition \ref{prop:flaces-flags}. In Subsection \ref{subsec:MMP},
we show that $\scalebox{0.95}{\ensuremath{\mathcal{F}}}$ can be obtained
from the subspace lattice $\scalebox{0.95}{\ensuremath{\mathcal{S}=\mathcal{S}\left(M\right)}}$
by a sequence of lattice operations. As a first step in preparation,
we define \textquotedblleft intermediate\textquotedblright{} lattices
that can arise in between:
\begin{defn}
A \textbf{matroidal lattice} on $\scalebox{0.95}{\ensuremath{M}}$
is the pair of $\scalebox{0.95}{\ensuremath{M}}$ and a meet-semilattice
$\scalebox{0.95}{\ensuremath{\left(\mathcal{A},\preccurlyeq,\curlywedge\right)}}$
with $\scalebox{0.95}{\ensuremath{\left\{ \emptyset,M\right\} \subseteq\mathcal{A}}}$
(and hence a lattice) whose members are expressions of $\scalebox{0.95}{\ensuremath{\mathcal{W}}}$
or rank-$k$ matroids of $\scalebox{0.95}{\ensuremath{\phi\left(\mathcal{W}\right)}}$,
satisfying that for $\scalebox{0.95}{\ensuremath{X,Y\in\mathcal{W}}}$:
\begin{enumerate}[label=(ML\arabic*),topsep=5pt]
\item \label{enu:ML1} $\scalebox{0.95}{\ensuremath{X,\phi\left(X\right)\in\mathcal{A}}}$
implies $\scalebox{0.95}{\ensuremath{\phi\left(X\right)\prec X}}$
or $\scalebox{0.95}{\ensuremath{X\curlywedge\phi\left(X\right)=\emptyset}}$.
\item \label{enu:ML2}$\scalebox{0.95}{\ensuremath{X,Y\in\mathcal{A}}}$
implies $\scalebox{0.95}{\ensuremath{X\curlywedge Y=\begin{cases}
X\owedge Y,\text{ or}\\
\phi\left(X\owedge Y\right)\text{ with }r\left(X\owedge Y\right)=k,\text{ or}\\
\emptyset.
\end{cases}}}$ 
\item \label{enu:ML3} $\scalebox{0.95}{\ensuremath{X,\phi\left(Y\right)\in\mathcal{A}}}$
implies $\scalebox{0.95}{\ensuremath{X\curlywedge\phi\left(Y\right)=\begin{cases}
\phi\left(X\owedge Y\right)\text{ with }r\left(X\owedge Y\right)=k,\text{ or}\\
\emptyset.
\end{cases}}}$
\item \label{enu:ML4} $\scalebox{0.95}{\ensuremath{\phi\left(X\right),\phi\left(Y\right)\in\mathcal{A}}}$
implies $\scalebox{0.95}{\ensuremath{\phi\left(X\right)\curlywedge\phi\left(Y\right)=\phi\left(X\right)\logof\phi\left(Y\right)}}$.
\end{enumerate}
We simply write $\scalebox{0.95}{\ensuremath{\mathcal{A}}}$ for $\scalebox{0.95}{\ensuremath{\left(M,\left(\mathcal{A},\preccurlyeq,\curlywedge\right)\right)}}$
and call $\scalebox{0.95}{\ensuremath{M}}$ the \textbf{underlying
matroid} of $\scalebox{0.95}{\ensuremath{\mathcal{A}}}$.
\end{defn}

\begin{notation}
Henceforth, given a matroidal lattice $\scalebox{0.95}{\ensuremath{\mathcal{A}}}$,
``$\preccurlyeq$'' and ``$\curlywedge$'' will denote its partial
order and meet, respectively, unless otherwise noted.
\end{notation}

Note that \ref{enu:ML2} and \ref{enu:ML3} are well-defined by Corollary
\ref{cor:TwoMeets} and Theorem \ref{thm:face-intersection}. Moreover:
\begin{itemize}
\item In \ref{enu:ML2}, if $\scalebox{0.95}{\ensuremath{X\curlywedge Y\neq\emptyset}}$,
then $\scalebox{0.95}{\ensuremath{Y\preccurlyeq X\Longleftrightarrow Y\oleq X}}$.
\item In \ref{enu:ML3}, suppose $\scalebox{0.95}{\ensuremath{X\curlywedge\phi\left(Y\right)\neq\emptyset}}$.
Then, $\scalebox{0.95}{\ensuremath{\phi\left(Y\right)\prec X}}$ implies
$\scalebox{0.95}{\ensuremath{\phi\left(Y\right)\subseteq\phi\left(X\right)}}$.
However, $\scalebox{0.95}{\ensuremath{X\not\preccurlyeq\phi\left(Y\right)}}$
because otherwise $\scalebox{0.95}{\ensuremath{X=X\curlywedge\phi\left(Y\right)\in\phi\left(\mathcal{W}\right)}}$,
and hence, $\scalebox{0.95}{\ensuremath{X\in\mathcal{W}\cap\phi\left(\mathcal{W}\right)=\left\{ \emptyset\right\} }}$
and $\scalebox{0.95}{\ensuremath{X=\emptyset}}$, a contradiction.
Also, $\scalebox{0.95}{\ensuremath{r\left(X\owedge Y\right)=k}}$
and $\scalebox{0.95}{\ensuremath{X\curlywedge\phi\left(Y\right)=\phi\left(X\owedge Y\right)=\phi\left(X\odot Y\right)=\phi\left(X\right)\cap\phi\left(Y\right)}}.$
\item In \ref{enu:ML4}, $\scalebox{0.95}{\ensuremath{\phi\left(Y\right)\preccurlyeq\phi\left(X\right)\Longleftrightarrow\phi\left(Y\right)\subseteq\phi\left(X\right)}}$.
\end{itemize}
Thus, for any $\scalebox{0.95}{\ensuremath{X,Y\in\mathcal{A}}}$,
$\scalebox{0.95}{\ensuremath{Y\prec X}}$ implies $\scalebox{0.95}{\ensuremath{\phi\left(Y\right)\subseteq\phi\left(X\right)}}$.

\subsection{\label{subsec:Matr-arrt}Matroidal arrangement}

Let $\scalebox{0.95}{\ensuremath{\mathcal{A}}}$ be a coatomistic
matroidal lattice on $\scalebox{0.95}{\ensuremath{M}}$. Denote by
$\scalebox{0.95}{\ensuremath{\mathcal{A}^{\lessdot M}}}$ the set
of coatoms of $\scalebox{0.95}{\ensuremath{\mathcal{A}}}$.
\begin{defn}
The \textbf{matroidal arrangement} of a coatomistic matroidal lattice
$\scalebox{0.95}{\ensuremath{\mathcal{A}}}$ on $\scalebox{0.95}{\ensuremath{M}}$
is defined as the triple of $\scalebox{0.95}{\ensuremath{M}}$, $\scalebox{0.95}{\ensuremath{\mathcal{A}}}$,
and $\scalebox{0.95}{\ensuremath{\mathcal{A}^{\lessdot M}}}$. We
frequently denote it simply by the pair of $\scalebox{0.95}{\ensuremath{M}}$
and $\scalebox{0.95}{\ensuremath{\mathcal{A}^{\lessdot M}}}$, or
just by $\scalebox{0.95}{\ensuremath{\mathcal{A}^{\lessdot M}}}$.
We also call it a $\scalebox{0.95}{\ensuremath{\left(k,E\left(M\right)\right)}}$-arrangement.
We call $\scalebox{0.95}{\ensuremath{\mathcal{A}}}$ the \textbf{intersection
lattice} of $\scalebox{0.95}{\ensuremath{\mathcal{A}^{\lessdot M}}}$
and say that $\scalebox{0.95}{\ensuremath{\mathcal{A}}}$ is \textbf{generated}
by $\scalebox{0.95}{\ensuremath{\mathcal{A}^{\lessdot M}}}$.
\end{defn}

\begin{defn}
A \textbf{label map} for a matroidal arrangement $\scalebox{0.95}{\ensuremath{\mathcal{A}^{\lessdot M}}}$
is a surjective map $\scalebox{0.95}{\ensuremath{\lambda:\Lambda\rightarrow\mathcal{A}^{\lessdot M}}}$
on a finite set $\scalebox{0.95}{\ensuremath{\Lambda}}$ where elements
of $\scalebox{0.95}{\ensuremath{\Lambda}}$ are called \textbf{labels}.
We call the function $\scalebox{0.95}{\ensuremath{\mathcal{A}^{\lessdot M}\rightarrow\mathbb{Z}_{\ge0}}}$
defined by $\scalebox{0.95}{\ensuremath{\mathbf{c}\mapsto\left|\lambda^{-1}(\mathbf{c})\right|}}$
the \textbf{multiplicity function}, and $\scalebox{0.95}{\ensuremath{\left|\lambda^{-1}(\mathbf{c})\right|}}$
the \textbf{multiplicity} of the coatom $\mathbf{c}$.
\end{defn}

\begin{defn}
A \textbf{labelled matroidal arrangement} on $\scalebox{0.95}{\ensuremath{M}}$
is a matroidal arrangement $\scalebox{0.95}{\ensuremath{\mathcal{A}^{\lessdot M}}}$
with a label map. Two labelled matroidal arrangements are \textbf{isomorphic}
if a multiplicity-preserving lattice isomorphism exists between their
intersection lattices.
\end{defn}

\begin{defn}
The \textbf{product} of two labelled arrangements $\scalebox{0.95}{\ensuremath{\mathcal{A}_{1}^{\lessdot M_{1}}}}$
and $\scalebox{0.95}{\ensuremath{\mathcal{A}_{2}^{\lessdot M_{2}}}}$
with label maps $\scalebox{0.95}{\ensuremath{\lambda_{1}}}$ on $\scalebox{0.95}{\ensuremath{\Lambda_{1}}}$
and $\scalebox{0.95}{\ensuremath{\lambda_{2}}}$ on $\scalebox{0.95}{\ensuremath{\Lambda_{2}}}$,
respectively, is 
\[
\scalebox{0.95}{\ensuremath{\left(\left(M_{1},M_{2}\right),\bigl(\mathcal{A}_{1}^{\lessdot M_{1}}\times\left\{ M_{2}\right\} \bigr)\cup\bigl(\left\{ M_{1}\right\} \times\mathcal{A}_{2}^{\lessdot M_{2}}\bigr)\right).}}
\]
 Its label map $\scalebox{0.95}{\ensuremath{\left(\lambda_{1},\lambda_{2}\right)}}$
on $\scalebox{0.95}{\ensuremath{\Lambda_{1}\oplus\Lambda_{2}}}$ is
defined by $\left(\lambda_{1},\lambda_{2}\right)|_{\Lambda_{1}}=\lambda_{1}$
and $\left(\lambda_{1},\lambda_{2}\right)|_{\Lambda_{2}}=\lambda_{2}$.
Here, $\scalebox{0.95}{\ensuremath{X_{1}\times\emptyset=\emptyset\times X_{2}=\emptyset}}$,
and $\scalebox{0.95}{\ensuremath{\left(M_{1},M_{2}\right)}}$ equals
$\scalebox{0.95}{\ensuremath{M_{1}\oplus M_{2}}}$ as a matroid.
\end{defn}

\begin{rem}
By removing the label conditions from the above definitions, one can
obtain the corresponding definitions for unlabelled arrangements.
\end{rem}

\begin{example}
\label{exa:lattices}Let $\scalebox{0.95}{\ensuremath{M}}$ be a loopless
matroid of rank $k\ge2$.
\begin{enumerate}
\item Let $\scalebox{0.95}{\ensuremath{\mathcal{T}=\mathcal{T}\left(M\right)}}$
be the collection of the expressions $\scalebox{0.95}{\ensuremath{M/F\in\mathcal{W}}}$
for all flats $\scalebox{0.95}{\ensuremath{F}}$ of $\scalebox{0.95}{\ensuremath{M}}$.
The forgetful map $\phi$ restricted to $\scalebox{0.95}{\ensuremath{\mathcal{T}}}$
is an isomorphism between $\scalebox{0.95}{\ensuremath{\mathcal{T}}}$
and $\scalebox{0.95}{\ensuremath{\mathcal{S}=\mathcal{S}\left(M\right)}}$.
So, $\scalebox{0.95}{\ensuremath{\mathcal{T}}}$ is a coatomistic
matroidal lattice on $\scalebox{0.95}{\ensuremath{M}}$ with $\scalebox{0.95}{\ensuremath{\mathcal{T}^{\lessdot M}\simeq\mathcal{S}^{\lessdot M}=\mathrm{HA}_{M}}}$,
and $\scalebox{0.95}{\ensuremath{\mathcal{T}^{\lessdot M}}}$ is a
labelled matroidal arrangement.
\item Let $\scalebox{0.95}{\ensuremath{\mathcal{F}=\mathcal{F}\left(M\right)}}$
be the coatomistic lattice of all loopless face matroids of $\scalebox{0.95}{\ensuremath{M}}$,
then its coatoms $\scalebox{0.95}{\ensuremath{M\left(F\right)}}$
are labelled $\scalebox{0.95}{\ensuremath{s_{F}}}$ for all non-degenerate
flats $\scalebox{0.95}{\ensuremath{F}}$.  Thus, $\scalebox{0.95}{\ensuremath{\mathcal{F}^{\lessdot M}}}$
is a labelled matroidal arrangement.
\end{enumerate}
\end{example}

\begin{defn}
\label{def:Puzzle-piece}If $\scalebox{0.95}{\ensuremath{\mathcal{F}=\mathcal{F}(M)}}$
is the nonempty coatomistic matroidal lattice of all loopless face
matroids of a loopless matroid $\scalebox{0.95}{\ensuremath{M}}$,
the matroidal arrangement $\scalebox{0.95}{\ensuremath{\mathcal{F}^{\lessdot M}}}$
is said to be the \textbf{puzzle-piece} of $\scalebox{0.95}{\ensuremath{M}}$
and denoted by $\scalebox{0.95}{\ensuremath{\mathrm{PZ}_{M}}}$.
\end{defn}

\begin{rem}
\label{rem:diff of isos}There are two isomorphic puzzle-pieces whose
base polytopes are not affinely isomorphic: Let $\scalebox{0.95}{\ensuremath{M}}$
be the matroid on $\left[5\right]$ of Example \ref{exa:meet}, and
$\scalebox{0.95}{\ensuremath{N:=U_{3}^{2}\oplus U_{3}^{2}}}$. Then,
both $\scalebox{0.95}{\ensuremath{\mathrm{PZ}_{M}}}$ and $\scalebox{0.95}{\ensuremath{\mathrm{PZ}_{N}}}$
are isomorphic to $\scalebox{0.95}{\ensuremath{\mathrm{HA}_{U_{3}^{2}}\times\mathrm{HA}_{U_{3}^{2}}}}$,
but $\scalebox{0.95}{\ensuremath{\mathrm{BP}_{M}}}$ and $\scalebox{0.95}{\ensuremath{\mathrm{BP}_{N}}}$
are not affinely isomorphic because the number of bases of $\scalebox{0.95}{\ensuremath{M}}$
is $8$ while that of $\scalebox{0.95}{\ensuremath{N}}$ is $9$.
See also Remark \ref{rem:why-pp}.
\end{rem}

\subsection{The dimension of a matroid}

Two functions $r$ and $\kappa$ are additive on direct sums of loopless
matroids, and $r-\kappa$ is a nonnegative dimension-like function.
\begin{defn}
The \textbf{dimension} of a nonempty matroid $\scalebox{0.95}{\ensuremath{M}}$
is defined as: 
\[
\scalebox{0.95}{\ensuremath{\dim M:=\left(r-\kappa\right)\left(M\backslash\bar{\emptyset}\right)}}.
\]
 Then, $\scalebox{0.95}{\ensuremath{0\le\dim M\le r\left(M\right)-1}}$.
For the empty matroid, we define $\scalebox{0.95}{\ensuremath{\dim\emptyset:=-1}}$.
\end{defn}

\begin{defn}
\label{def:dimPZ}The \textbf{dimension} of a puzzle-piece $\scalebox{0.95}{\ensuremath{\mathrm{PZ}_{M}}}$
is defined as: 
\[
\scalebox{0.95}{\ensuremath{\dim\mathrm{PZ}_{M}:=\dim M}}.
\]
 A puzzle-piece of dimension $0$ or $1$ is called a \textbf{point-piece}
or a \textbf{line-piece}, respectively.
\end{defn}

\begin{notation}
For $d\ge1$, denote by $\scalebox{0.95}{\ensuremath{\left\lfloor M\right\rfloor _{d}}}$
the direct sum of all connected components of $\scalebox{0.95}{\ensuremath{M}}$
with dimension $\scalebox{0.95}{\ensuremath{\ge d}}$. When the subscript
$d$ is $1$, we omit it and write $\scalebox{0.95}{\ensuremath{\left\lfloor M\right\rfloor }}$
for $\scalebox{0.95}{\ensuremath{\left\lfloor M\right\rfloor _{1}}}$.
We also write $\scalebox{0.95}{\ensuremath{\left\lfloor M\right\rfloor _{0}=M\backslash\bar{\emptyset}}}$
and $\scalebox{0.95}{\ensuremath{\left\lfloor M\right\rfloor _{-1}=M}}$.
\end{notation}

\begin{rem}
\label{rem:why-pp} $\scalebox{0.95}{\ensuremath{\mathrm{PZ}_{M}}}$
is less dimensional than $\scalebox{0.95}{\ensuremath{\mathrm{BP}_{M}}}$,
indicating that when only loopless parts matter, passing to puzzle-pieces
is a good option. For example, let $\scalebox{0.95}{\ensuremath{M}}$
be the matroid of Example \ref{exa:meet}. See Figure \ref{fig:why-pp}
for $\scalebox{0.95}{\ensuremath{\mathrm{HA}_{M}}}$ and $\scalebox{0.95}{\ensuremath{\mathrm{PZ}_{M}}}$.
We see that $\scalebox{0.95}{\ensuremath{\mathrm{PZ}_{M}}}$ has exactly
$6$ line-pieces and $9$ point-pieces, which tells that $\scalebox{0.95}{\ensuremath{\mathrm{BP}_{M}}}$
has $6$ loopless facets and $9$ loopless codimension-$2$ faces,
and no more loopless nonempty proper faces. Also, only two of those
$15$ loopless faces are relevant: $\scalebox{0.95}{\ensuremath{\mathrm{BP}_{M\left(\left\{ 1,2,3\right\} \right)}}}$
and $\scalebox{0.95}{\ensuremath{\mathrm{BP}_{M\left(\left\{ 3,4,5\right\} \right)}}}$.
\begin{figure}[H]
\noindent \begin{center}
\def\sizea{0.22}
\def\size{0.25}
\def\ratioa{0.901}
\def\ratior{0.851}
\def\ratiow{0.761}
\def\ratioe{0.881}


\par\end{center}\vspace{-6pt}

\caption{\label{fig:why-pp}A line arrangement and its puzzle-piece.}
\end{figure}
\end{rem}

\subsection{\label{subsec:expression-dim}The dimension of a matroidal expression}

Let $\scalebox{0.95}{\ensuremath{X=\bigoplus_{i=0}^{m}X_{i}\in\mathcal{W}}}$
be a matroidal expression with no empty summands, then $\scalebox{0.95}{\ensuremath{r\left(X_{i}\right)\ge1}}$.
We define $\scalebox{0.95}{\ensuremath{\mathrm{edim}\left.X\right.}}$
as: 
\[
\scalebox{0.95}{\ensuremath{\mathrm{edim}\left.X\right.:=\sum_{i=0}^{m}\left(r\left(X_{i}\right)-1\right)}}.
\]
Since $\scalebox{0.95}{\ensuremath{r\left(X_{i}\right)-1\ge\dim X_{i}}}$,
we have: 
\[
\scalebox{0.95}{\ensuremath{\mathrm{edim}\left.X\right.\ge\dim X}}.
\]
 For $\scalebox{0.95}{\ensuremath{X\in\phi\left(\mathcal{W}\right)}}$,
we define 
\[
\scalebox{0.95}{\ensuremath{\mathrm{edim}\left.X\right.:=\dim X}}.
\]
 We call $\scalebox{0.95}{\ensuremath{\mathrm{edim}\left.X\right.}}$
the \textbf{expression dimension} of $\scalebox{0.95}{\ensuremath{X}}$.
For a subspace $\scalebox{0.95}{\ensuremath{M/F}}$ of a loopless
matroid $\scalebox{0.95}{\ensuremath{M}}$, $\scalebox{0.95}{\ensuremath{\mathrm{edim}\left.M/F\right.}}$
equals its subspace dimension $\scalebox{0.95}{\ensuremath{\mathrm{sdim}\left.M/F\right.}}$,
cf. Subsection \ref{subsec:HA-1}.
\begin{defn}
\label{def:stability}Let $\scalebox{0.95}{\ensuremath{X}}$ be a
nonempty member of $\scalebox{0.95}{\ensuremath{\mathcal{W}\cup\phi\left(\mathcal{W}\right)}}$.
We say that  $\scalebox{0.95}{\ensuremath{X}}$ is \textbf{stable}
if $\scalebox{0.95}{\ensuremath{\mathrm{edim}\left.X\right.=\dim X}}$,
and \textbf{unstable} if $\scalebox{0.95}{\ensuremath{\mathrm{edim}\left.X\right.>\dim X}}$.
\end{defn}

Every nonempty $\scalebox{0.95}{\ensuremath{X\in\phi\left(\mathcal{W}\right)}}$
is stable by definition.

\subsection{\label{subsec:matoidal-blowup}Blowup operation}

Let $\scalebox{0.95}{\ensuremath{\mathcal{A}}}$ be a matroidal lattice
on a loopless matroid $\scalebox{0.95}{\ensuremath{M}}$ of rank $k$.
Fix a nonempty proper member $\scalebox{0.95}{\ensuremath{X\in\mathcal{A}\cap\mathcal{W}}}$
with $\scalebox{0.95}{\ensuremath{\epsilon\left(X\right)\notin\mathcal{A}}}$.
Denote 
\[
\scalebox{0.95}{\ensuremath{\mathcal{A}_{X}:=\bigcup_{Z\in I_{X}}\mathcal{A}|_{Z}}}
\]
 where $\scalebox{0.95}{\ensuremath{I_{X}:=\left\{ Z\in\mathcal{A}:X\preccurlyeq Z,\epsilon\left(X\right)=\epsilon\left(Z\right)\right\} }}$
and $\scalebox{0.95}{\ensuremath{\mathcal{A}|_{Z}:=\left\{ A\in\mathcal{A}:A\preccurlyeq Z\right\} }}$.
We define a binary operation ``$\scalebox{0.95}{\ensuremath{\ocap}}$'':
Let $\scalebox{0.95}{\ensuremath{Y,Z\in\mathcal{A}-\mathcal{A}_{X}}}$.
\begin{enumerate}[label=(BU\arabic*),topsep=5pt]
\item \label{enu:BU1}$\scalebox{0.95}{\ensuremath{Y\ocap Z:=\begin{cases}
Y\curlywedge Z & \text{if }Y\curlywedge Z\in\mathcal{A}-\mathcal{A}_{X},\\
\emptyset & \text{otherwise.}
\end{cases}}}$
\item \label{enu:BU2}$\scalebox{0.95}{\ensuremath{Y\ocap\epsilon\left(X\right):=\begin{cases}
Y\owedge\epsilon\left(X\right) & \text{if }Y\in\mathcal{W}\text{ with }X\curlywedge Y\neq\emptyset,\\
\emptyset & \text{otherwise.}
\end{cases}}}$
\item \label{enu:BU3}$\scalebox{0.95}{\ensuremath{Y\ocap\left(Z\ocap\epsilon\left(X\right)\right):=\begin{cases}
\left(Y\curlywedge Z\right)\ocap\epsilon\left(X\right) & \text{if }Y\curlywedge Z\in\mathcal{A}-\mathcal{A}_{X},\\
\emptyset & \text{otherwise.}
\end{cases}}}$
\item \label{enu:BU4}$\scalebox{0.95}{\ensuremath{\left(Y\ocap\epsilon\left(X\right)\right)\ocap\left(Z\ocap\epsilon\left(X\right)\right):=\begin{cases}
\left(Y\curlywedge Z\right)\ocap\epsilon\left(X\right) & \text{if }Y\curlywedge Z\in\mathcal{A}-\mathcal{A}_{X},\\
\emptyset & \text{otherwise.}
\end{cases}}}$
\end{enumerate}
Denote $\scalebox{0.95}{\ensuremath{\mathcal{A}^{\epsilon\left(X\right)}:=\epsilon\left(X\right)\ocap\left(\mathcal{A}-\mathcal{A}_{X}\right)}}$.
Then, $\scalebox{0.95}{\ensuremath{\mathcal{A}^{\epsilon\left(X\right)}}}$
contains $\scalebox{0.95}{\ensuremath{\epsilon\left(X\right)=M\ocap\epsilon\left(X\right)}}$
because $\scalebox{0.95}{\ensuremath{M\in\mathcal{A}-\mathcal{A}_{X}}}$.
We assume that $\scalebox{0.95}{\ensuremath{\mathcal{A}^{\epsilon\left(X\right)}}}$
contains $\scalebox{0.95}{\ensuremath{\emptyset}}$ by adding it if
not. Then, $\scalebox{0.95}{\ensuremath{\mathcal{A}^{\epsilon\left(X\right)}}}$
is a finite meet-semilattice with meet $\scalebox{0.95}{\ensuremath{\ocap}}$
and is a lattice. Further, denote 
\[
\scalebox{0.95}{\ensuremath{\mathrm{Bl}_{X}\mathcal{A}:=\left(\mathcal{A}-\mathcal{A}_{X}\right)\cup\mathcal{A}^{\epsilon\left(X\right)}}}.
\]
 Then, $\scalebox{0.95}{\ensuremath{\mathrm{Bl}_{X}\mathcal{A}}}$
with $\scalebox{0.95}{\ensuremath{\emptyset,M\in\mathrm{Bl}_{X}\mathcal{A}}}$
is a finite meet-semilattice with meet $\scalebox{0.95}{\ensuremath{\ocap}}$
and is a lattice. Here, it is a routine to check the associativity
of $\scalebox{0.95}{\ensuremath{\ocap}}$. 
\begin{defn}
Let $\scalebox{0.95}{\ensuremath{\mathcal{A}}}$ be a matroidal lattice
on a loopless matroid $\scalebox{0.95}{\ensuremath{M}}$. For a nonempty
proper member $\scalebox{0.95}{\ensuremath{X\in\mathcal{A}\cap\mathcal{W}}}$
with $\scalebox{0.95}{\ensuremath{\epsilon\left(X\right)\notin\mathcal{A}}}$,
the lattice $\scalebox{0.95}{\ensuremath{\mathrm{Bl}_{X}\mathcal{A}}}$
is said to be the \textbf{blowup of $\scalebox{0.95}{\ensuremath{\mathcal{A}}}$
along} $\scalebox{0.95}{\ensuremath{X}}$. \textbf{Blowing up} is
taking a blowup, and \textbf{blowing down} is the inverse operation.
If $\scalebox{0.95}{\ensuremath{\mathrm{Bl}_{X}\mathcal{A}}}$ is
a coatomistic lattice, we say that the matroidal arrangement $\scalebox{0.95}{\ensuremath{\left(\mathrm{Bl}_{X}\mathcal{A}\right)^{\lessdot M}}}$
is the \textbf{blowup of $\scalebox{0.95}{\ensuremath{\mathcal{A}^{\lessdot M}}}$
along $\scalebox{0.95}{\ensuremath{X}}$}.
\end{defn}

The property \ref{enu:BU1} indicates that blowing up is a local operation,
and \ref{enu:BU2} does not change the meet structure ``too much'':
if $\scalebox{0.95}{\ensuremath{\epsilon\left(X\right)\notin\mathcal{A}}}$
and $\scalebox{0.95}{\ensuremath{Y\in\left(\mathcal{A}-\mathcal{A}_{X}\right)\cap\phi\left(\mathcal{W}\right)}}$,
then $\scalebox{0.95}{\ensuremath{r\left(X\right)<k}}$ and $\scalebox{0.95}{\ensuremath{X\curlywedge Y=\emptyset}}$
by \ref{enu:ML3}. Also, if $\scalebox{0.95}{\ensuremath{\mathcal{A}}}$
is coatomistic, $\scalebox{0.95}{\ensuremath{\mathrm{Bl}_{X}\mathcal{A}}}$
is not necessarily coatomistic. Note that the coatoms of $\scalebox{0.95}{\ensuremath{\mathrm{Bl}_{X}\mathcal{A}}}$
are the coatoms of $\scalebox{0.95}{\ensuremath{\mathcal{A}}}$ or
the coatoms of $\scalebox{0.95}{\ensuremath{\mathcal{A}}}$ union
$\scalebox{0.95}{\ensuremath{\left\{ \epsilon\left(X\right)\right\} }}$.
\smallskip{}

Suppose that both $\scalebox{0.95}{\ensuremath{\mathcal{A}}}$ and
$\scalebox{0.95}{\ensuremath{\mathrm{Bl}_{X}\mathcal{A}}}$ are coatomistic,
and let $\scalebox{0.95}{\ensuremath{\lambda:\Lambda\rightarrow\mathcal{A}^{\lessdot M}}}$
be a label map of $\scalebox{0.95}{\ensuremath{\mathcal{A}^{\lessdot M}}}$.
Due to \ref{enu:ML2}, one has $\scalebox{0.95}{\ensuremath{\left(\mathrm{Bl}_{X}\mathcal{A}\right)^{\lessdot M}=\mathcal{A}^{\lessdot M}\cup\left\{ \epsilon\left(X\right)\right\} }}$,
which is a labelled arrangement with a label map $\tilde{\lambda}$,
as described below:
\begin{itemize}
\item Introduce a new letter $s_{X}$ for $\scalebox{0.95}{\ensuremath{\epsilon\left(X\right)}}$
and define $\scalebox{0.95}{\ensuremath{\tilde{\lambda}\left(s_{X}\right):=\epsilon\left(X\right)}}$.
\item For any $\scalebox{0.95}{\ensuremath{\mathbf{c}\in\mathcal{A}^{\lessdot M}-\left\{ X\right\} }}$,
define $\scalebox{0.95}{\ensuremath{\tilde{\lambda}|_{\lambda^{-1}\left(\mathbf{c}\right)}:=\lambda|_{\lambda^{-1}\left(\mathbf{c}\right)}}}$.
\item $\scalebox{0.95}{\ensuremath{\tilde{\lambda}:\left\{ s_{X}\right\} \cup\bigl(\bigcup_{\mathbf{c}\in\mathcal{A}^{\lessdot M}-\left\{ X\right\} }\lambda^{-1}\left(\mathbf{c}\right)\bigr)\rightarrow\left(\mathrm{Bl}_{X}\mathcal{A}\right)^{\lessdot M}}}$
is a label map.
\end{itemize}

\subsection{\label{subsec:Collapsing-operation}Collapsing operation}

Let $\scalebox{0.95}{\ensuremath{\mathcal{A}}}$ be a matroidal lattice
on a loopless matroid $\scalebox{0.95}{\ensuremath{M}}$ of rank $k$.
A proper member $\scalebox{0.95}{\ensuremath{X\in\mathcal{A}\cap\mathcal{W}}}$
is said to be \textbf{collapsible} if $\scalebox{0.95}{\ensuremath{r\left(Y\right)=k}}$
for all $\scalebox{0.95}{\ensuremath{Y\in\mathcal{A}\cap\mathcal{W}}}$
with $\scalebox{0.95}{\ensuremath{X\curlywedge Y\neq\emptyset}}$,
including $\scalebox{0.95}{\ensuremath{X}}$ itself. Here, $\scalebox{0.95}{\ensuremath{Y=\epsilon\left(Y\right)}}$.
If $\scalebox{0.95}{\ensuremath{X\in\mathcal{A}\cap\mathcal{W}}}$
is not collapsible, we can make it collapsible by recursively blowing
up all $\scalebox{0.95}{\ensuremath{Y\in\mathcal{A}\cap\mathcal{W}}}$
of rank $\scalebox{0.95}{\ensuremath{r\left(Y\right)<k}}$ with $\scalebox{0.95}{\ensuremath{X\curlywedge Y\neq\emptyset}}$.
We call this process the \textbf{collapsibilization} of $\scalebox{0.95}{\ensuremath{X}}$.\vspace{2pt}

Let $\scalebox{0.95}{\ensuremath{X\in\mathcal{A}\cap\mathcal{W}}}$
be collapsible. For $\scalebox{0.95}{\ensuremath{Y,Z\in\mathcal{A}-\mathcal{A}|_{X}}}$,
we define a binary operation ``$\scalebox{0.95}{\ensuremath{\ocap}}$''
as follows:
\begin{enumerate}[label=(CL\arabic*),topsep=5pt]
\item \label{enu:CL1}$\scalebox{0.95}{\ensuremath{Y\ocap Z:=\begin{cases}
Y\curlywedge Z & \text{if }Y\curlywedge Z\in\mathcal{A}-\mathcal{A}|_{X},\\
\emptyset & \text{otherwise.}
\end{cases}}}$
\item \label{enu:CL2}$\scalebox{0.95}{\ensuremath{Y\ocap\phi\left(X\right):=\begin{cases}
\phi\left(Y\right)\logof\phi\left(X\right) & \text{if }X\curlywedge Y\neq\emptyset,\\
\emptyset & \text{otherwise.}
\end{cases}}}$
\item \label{enu:CL3}$\scalebox{0.95}{\ensuremath{Y\ocap\left(Z\ocap\phi\left(X\right)\right):=\begin{cases}
\phi\left(Y\right)\logof\phi\left(Z\right)\logof\phi\left(X\right) & \text{if }X\curlywedge Y\neq\emptyset\neq X\curlywedge Z,\\
\emptyset & \text{otherwise.}
\end{cases}}}$
\item \label{enu:CL4}$\scalebox{0.95}{\ensuremath{\left(Y\ocap\phi\left(X\right)\right)\ocap\left(Z\ocap\phi\left(X\right)\right):=\begin{cases}
\phi\left(Y\right)\logof\phi\left(Z\right)\logof\phi\left(X\right) & \text{if }X\curlywedge Y\neq\emptyset\neq X\curlywedge Z,\\
\emptyset & \text{otherwise.}
\end{cases}}}$
\end{enumerate}
Denote $\scalebox{0.95}{\ensuremath{\phi\left(\mathcal{A}\right)^{X}:=\phi\left(X\right)\ocap\left(\mathcal{A}-\mathcal{A}|_{X}\right)}}$.
This contains $\scalebox{0.95}{\ensuremath{\phi\left(X\right)}}$,
the largest member. We assume that $\scalebox{0.95}{\ensuremath{\phi\left(\mathcal{A}\right)^{X}}}$
contains $\scalebox{0.95}{\ensuremath{\emptyset}}$ by adding it if
not. Then, $\scalebox{0.95}{\ensuremath{\phi\left(\mathcal{A}\right)^{X}}}$
is a finite meet-semilattice with meet $\scalebox{0.95}{\ensuremath{\ocap}}$
and is a lattice. Further, denote 
\[
\scalebox{0.95}{\ensuremath{\mathrm{Kl}_{X}\mathcal{A}:=\left(\mathcal{A}-\mathcal{A}|_{X}\right)\cup\phi\left(\mathcal{A}\right)^{X}}}.
\]
 Then, $\scalebox{0.95}{\ensuremath{\mathrm{Kl}_{X}\mathcal{A}}}$
with $\scalebox{0.95}{\ensuremath{\emptyset,M\in\mathrm{Kl}_{X}\mathcal{A}}}$
is a finite meet-semilattice with meet $\scalebox{0.95}{\ensuremath{\ocap}}$
and is a lattice. Here, it is straightforward to check the associativity
of $\scalebox{0.95}{\ensuremath{\ocap}}$.
\begin{defn}
Let $\scalebox{0.95}{\ensuremath{M}}$ be a loopless matroid and $\scalebox{0.95}{\ensuremath{\mathcal{A}}}$
a matroidal lattice on it. For a collapsible member $\scalebox{0.95}{\ensuremath{X}}$
of $\scalebox{0.95}{\ensuremath{\mathcal{A}}}$, we refer to the lattice
$\scalebox{0.95}{\ensuremath{\mathrm{Kl}_{X}\mathcal{A}}}$ or taking
it as \textbf{collapsing} $\scalebox{0.95}{\ensuremath{X}}$ \textbf{in}
$\scalebox{0.95}{\ensuremath{\mathcal{A}}}$ or \textbf{collapsing}
$\scalebox{0.95}{\ensuremath{\mathcal{A}}}$ \textbf{over} $\scalebox{0.95}{\ensuremath{X}}$.
If $\scalebox{0.95}{\ensuremath{\mathrm{Kl}_{X}\mathcal{A}}}$ is
coatomistic, we refer to $\scalebox{0.95}{\ensuremath{\left(\mathrm{Kl}_{X}\mathcal{A}\right)^{\lessdot M}}}$
or obtaining it as \textbf{collapsing} $\scalebox{0.95}{\ensuremath{X}}$
\textbf{in} $\scalebox{0.95}{\ensuremath{\mathcal{A}^{\lessdot M}}}$
or \textbf{collapsing} $\scalebox{0.95}{\ensuremath{\mathcal{A}^{\lessdot M}}}$
\textbf{over} $\scalebox{0.95}{\ensuremath{X}}$.
\end{defn}

The collapsing operation is a local operation by \ref{enu:CL1}, in
the same way as the blowup operation. However, unlike the blowup operation,
the collapsing operation produces a matroidal lattice with fewer members.
Also, if $\scalebox{0.95}{\ensuremath{\mathcal{A}}}$ is coatomistic,
$\scalebox{0.95}{\ensuremath{\mathrm{Kl}_{X}\mathcal{A}}}$ is always
coatomistic and $\scalebox{0.95}{\ensuremath{\left(\mathrm{Kl}_{X}\mathcal{A}\right)^{\lessdot M}}}$
is a matroidal arrangement. Moreover, for a collapsible $\scalebox{0.95}{\ensuremath{X\in\mathcal{A}\cap\mathcal{W}}}$,
if $\scalebox{0.95}{\ensuremath{\mathrm{codim}\left.\phi\left(X\right)\right.=1}}$,
then $\scalebox{0.95}{\ensuremath{X}}$ is a coatom of $\scalebox{0.95}{\ensuremath{\mathcal{A}}}$
due to \ref{enu:ML2}. So, $\scalebox{0.95}{\ensuremath{\phi\left(X\right)\in\left(\mathrm{Kl}_{X}\mathcal{A}\right)^{\lessdot M}}}$
implies $\scalebox{0.95}{\ensuremath{X\in\mathcal{A}^{\lessdot M}}}$.\vspace{2pt}

Let $\scalebox{0.95}{\ensuremath{\mathcal{A}^{\lessdot M}}}$ be a
labelled arrangement with a label map $\scalebox{0.95}{\ensuremath{\lambda:\Lambda\rightarrow\mathcal{A}^{\lessdot M}}}$.
We show that $\scalebox{0.95}{\ensuremath{\left(\mathrm{Kl}_{X}\mathcal{A}\right)^{\lessdot M}}}$
is also a labelled arrangement. Let $\scalebox{0.95}{\ensuremath{\lambda^{-1}\left(X\right)=\left\{ s_{x}\right\} }}$.
\begin{itemize}
\item If $\scalebox{0.95}{\ensuremath{\phi\left(X\right)\in\left(\mathrm{Kl}_{X}\mathcal{A}\right)^{\lessdot M}}}$,
then $\scalebox{0.95}{\ensuremath{\left(\mathrm{Kl}_{X}\mathcal{A}\right)^{\lessdot M}=\left(\mathcal{A}^{\lessdot M}-\left\{ X\right\} \right)\cup\left\{ \phi\left(X\right)\right\} }}$.
Define $\scalebox{0.95}{\ensuremath{\tilde{\lambda}|_{\Lambda-\left\{ s_{x}\right\} }:=\lambda|_{\Lambda-\left\{ s_{x}\right\} }}}$
and $\scalebox{0.95}{\ensuremath{\tilde{\lambda}\left(s_{x}\right):=\phi\left(X\right)}}$,
then $\scalebox{0.95}{\ensuremath{\tilde{\lambda}}}$ is a label map.
\item If $\scalebox{0.95}{\ensuremath{\phi\left(X\right)\notin\left(\mathrm{Kl}_{X}\mathcal{A}\right)^{\lessdot M}}}$,
then $\scalebox{0.95}{\ensuremath{\mathrm{codim}\left.\phi\left(X\right)\right.>1}}$,
and $\left(\mathrm{Kl}_{X}\mathcal{A}\right)^{\lessdot M}=\mathcal{A}^{\lessdot M}-\left\{ X\right\} $.
Let $\scalebox{0.95}{\ensuremath{\tilde{\lambda}:=\lambda|_{\Lambda-\left\{ s_{X}\right\} }}}$.
Then, $\scalebox{0.95}{\ensuremath{\tilde{\lambda}}}$ is a label
map.
\end{itemize}
In the meanwhile, one can make an unstable member stable by collapsibilizing
it and collapsing the collapsibilized.
\begin{example}
Let $\scalebox{0.95}{\ensuremath{M}}$ be the matroid of Example \ref{exa:meet}.
\begin{enumerate}
\item Blow up the line arrangement $\scalebox{0.95}{\ensuremath{\mathrm{HA}_{M}}}$
at $\scalebox{0.95}{\ensuremath{M/\left\{ 1,2,3\right\} }}$, $\scalebox{0.95}{\ensuremath{M/\left\{ 3,4,5\right\} }}$,
and $\scalebox{0.95}{\ensuremath{M/\left\{ i\right\} }}$ for all
$\scalebox{0.95}{\ensuremath{i\in\left[5\right]}}$, see Figure \ref{fig:blowup},\footnote{The direction of the arrow is ``reversed'' in the sense that there
is a surjection sending atoms of $\scalebox{0.93}{\ensuremath{\mathrm{Bl}_{X}\mathcal{A}}}$
to atoms of $\scalebox{0.93}{\ensuremath{\mathcal{A}}}$. This coincides
with the geometric convention.} where we use the same labels for $\scalebox{0.95}{\ensuremath{M\left(\left\{ i\right\} \right)}}$.
\begin{figure}[H]
\noindent \begin{center}
\def\sizea{0.22}
\def\size{0.25}
\def\ratioa{0.901}
\def\ratior{0.851}
\def\ratiow{0.761}
\def\ratioe{0.881}

\begin{tikzpicture}[font=\scriptsize]

\matrix[column sep=0.6cm, row sep=0.2cm]{

\begin{scope}[line cap=round,rotate=0,scale=\sizea,xshift=0cm,yshift=0cm,rounded corners]
 \path (90:4)--(210:4)--(330:4)--cycle;
 \path (0,0) coordinate (O);
 \foreach \x [count=\xi] in {B0,C0,A0}{
      \path (90+120*\xi:4) coordinate (\x);}
 \draw [blue,thick] (A0)++(60:1)node[right=-2pt,black]{${5}$}--++(240:8.928);
 \draw [red,thick] (A0)++(120:1)node[left=-2pt,black]{${1}$}--++(-60:8.928);
 \draw [blue,thick] (B0)++(-150:1)--++(30:8)node[right=-2pt,black]{${4}$};
 \draw [red,thick] (C0)++(-30:1)--++(150:8)node[left=-2pt,black]{${2}$};
 \draw [green,thick] (B0)++(180:1)--node[below=-2pt,black]{${3}$}++(0:8.928);
\fill[gray] (B0) circle (8pt) (C0) circle (8pt);

\end{scope}

&

 \draw [->]     (1,0) -- (-1,0);
 \path (0,0.4) node[]{Blowups};

&

\begin{scope}[line cap=round,rotate=0,scale=\sizea,xshift=0cm,yshift=0cm,rounded corners]
 \path (90:4)--(210:4)--(330:4)--cycle;
 \path (0,0) coordinate (O);
 \foreach \x [count=\xi] in {B0,C0,A0}{
      \path (90+120*\xi:4) coordinate (\x);}
 \draw [blue,thick] (A0)++(60:1)node[right=-2pt,black]{${5}$} .. controls (150:2) .. (-4.464,-1.5)
  (A0)++(300:2.309)++(30:1.2)node[right=-2pt,black]{${4}$} .. controls (0,0) .. (210:4.616);
 \draw [red,thick] (A0)++(120:1)node[left=-2pt,black]{${1}$} .. controls (30:2) .. (4.464,-1.5)
  (A0)++(240:2.309)++(150:1.2)node[left=-2pt,black]{${2}$} .. controls (0,0) .. (330:4.616);

\draw [green,thick] (-3.531,-3.116) ..node[midway,below=-2pt,black]
    {$3$} controls (-1.732,-2) and (1.732,-2) .. (3.531,-3.116);
 
\draw [gray,very thick,densely dashed]
(-3.564,0.598)node[left=-3pt,black]{${s_{345}}$} .. controls (210:2.5) .. (-1.464,-3.732)
(3.564,0.598)node[right=-2pt,black]{${s_{123}}$} .. controls (330:2.5) .. (1.464,-3.732);

\end{scope}

\\
};
\end{tikzpicture}
\par\end{center}\vspace{-0.15cm}
\caption{\label{fig:blowup}An example of blowup.}
\end{figure}
\item Collapse $\scalebox{0.95}{\ensuremath{M\left(\left\{ 1,2,3\right\} \right)}}$
and $\scalebox{0.95}{\ensuremath{M\left(\left\{ 3,4,5\right\} \right)}}$,
and then $\scalebox{0.95}{\ensuremath{M\left(\left\{ 3\right\} \right)}}$,
see Figure \ref{fig:collapse}.\footnote{Technically, $\scalebox{0.95}{\ensuremath{M\left(F\right)}}$ and
$\scalebox{0.95}{\ensuremath{\phi\left(M\left(F\right)\right)}}$
are indistinguishable in the figure.} The order of application is taken so for illustrative purpose.
\begin{figure}[H]
\noindent \begin{center}
\def\sizea{0.22}
\def\size{0.25}
\def\ratioa{0.901}
\def\ratior{0.851}
\def\ratiow{0.761}
\def\ratioe{0.881}


\par\end{center}\vspace{-0.15cm}
\caption{\label{fig:collapse}An example of collapsing.}
\end{figure}
\end{enumerate}
\end{example}

\subsection{\label{subsec:MMP}Minimal model}

A matroidal lattice $\scalebox{0.95}{\ensuremath{\mathcal{A}}}$ on
a loopless matroid $\scalebox{0.95}{\ensuremath{M}}$ is said to be
\textbf{collapsible} if \emph{all} of its nonempty members have rank
$k$. If $\scalebox{0.95}{\ensuremath{\mathcal{A}}}$ is collapsible,
we can collapse everything and obtain $\scalebox{0.95}{\ensuremath{\phi\left(\mathcal{A}\right)=\left\{ \phi\left(Y\right):Y\in\mathcal{A}\right\} }}$,
a matroidal lattice with meet $\logof$. We call this the \textbf{collapsing}
of $\scalebox{0.95}{\ensuremath{\mathcal{A}}}$.

A matroidal lattice $\scalebox{0.95}{\ensuremath{\mathcal{A}}}$ can
be made collapsible by recursively collapsibilizing the members. Thus,
$\scalebox{0.95}{\ensuremath{\mathcal{A}}}$ can be converted into
a matroidal sublattice of $\scalebox{0.95}{\ensuremath{\mathcal{F}\left(M\right)}}$
by recursive blowups and collapsings, where $\scalebox{0.95}{\ensuremath{\mathcal{F}\left(M\right)}}$
is the intersection lattice of the puzzle-piece $\scalebox{0.95}{\ensuremath{\mathrm{PZ}_{M}}}$.
The sublattice is \emph{uniquely} obtained by construction. We call
it \emph{the} \textbf{minimal model} of $\scalebox{0.95}{\ensuremath{\mathcal{A}}}$.
If $\scalebox{0.95}{\ensuremath{\mathcal{A}}}$ is coatomistic, its
minimal model is also coatomistic, and the collection of coatoms of
the minimal model is said to be the \textbf{minimal model} of $\scalebox{0.95}{\ensuremath{\mathcal{A}^{\lessdot M}}}$.
We refer to the process of obtaining the minimal model as the \textbf{minimal
model program} or \textbf{MMP} for short.

Moreover, by construction, the minimal model of the subspace lattice
$\scalebox{0.95}{\ensuremath{\mathcal{S}\left(M\right)}}$ of $\scalebox{0.95}{\ensuremath{M}}$
is $\scalebox{0.95}{\ensuremath{\mathcal{F}\left(M\right)}}$, and
the minimal model of the abstract hyperplane arrangement $\scalebox{0.95}{\ensuremath{\mathrm{HA}_{M}}}$
is the puzzle-piece $\scalebox{0.95}{\ensuremath{\mathrm{PZ}_{M}}}$.
See also Subsections \ref{subsec:minimal-model} and \ref{subsec:Grassmann}.

\subsection{\label{subsec:w-arrangement}Weighted arrangement and stability condition}

For a coatomistic lattice $\scalebox{0.95}{\ensuremath{\mathcal{A}}}$,
a \textbf{weighted arrangement} is a triple $\scalebox{0.95}{\ensuremath{\left(\mathcal{A}^{\lessdot M},\lambda,\mathbf{w}\right)}}$
with a label map $\scalebox{0.95}{\ensuremath{\lambda:\Lambda\rightarrow\mathcal{A}^{\lessdot M}}}$
whose labels are all weighted. If $w_{s}=0$ for some $\scalebox{0.95}{\ensuremath{s\in\Lambda}}$,
we regard it as the following: 
\[
\scalebox{0.95}{\ensuremath{\Bigl(\lambda\left(\Lambda-\left\{ s\right\} \right),\lambda|_{\Lambda-\left\{ s\right\} },\left(w_{l}\right)_{l\in\Lambda-\left\{ s\right\} }\Bigr)}}.
\]
I.e., the labelled coatom $\scalebox{0.95}{\ensuremath{\left(s,\lambda\left(s\right)\right)}}$
disappears. We usually require $w_{l}>0$ for all $\scalebox{0.95}{\ensuremath{l\in\Lambda}}$.
A usual labelled matroidal arrangement is regarded as a weighted arrangement
whose labels all have weights $1$.

We consider a particular class of weighted arrangements as follows.
Let $\scalebox{0.95}{\ensuremath{M}}$ be a loopless matroid on $\scalebox{0.95}{\ensuremath{S}}$
and $\scalebox{0.95}{\ensuremath{\mathrm{HA}_{M}}}$ a hyperplane
arrangement weighted by a weight vector $\scalebox{0.95}{\ensuremath{\mathbf{w}=\left(w_{s}\right)_{s\in S}\in\mathbb{R}^{S}}}$.
Let $\scalebox{0.95}{\ensuremath{\mathcal{A}}}$ be a coatomistic
lattice on $\scalebox{0.95}{\ensuremath{M}}$ with $\scalebox{0.95}{\ensuremath{\mathcal{A}^{\lessdot M}}}$
being obtained from $\scalebox{0.95}{\ensuremath{\mathrm{HA}_{M}}}$
by a sequence of blowups and collapsings, for which every label not
in $\scalebox{0.95}{\ensuremath{S}}$ is assigned $1$ for its weight.
We call $\scalebox{0.95}{\ensuremath{\mathcal{A}^{\lessdot M}}}$
a \textbf{$\mathbf{w}$-arrangement}. 

For a weight vector $\scalebox{0.95}{\ensuremath{\mathbf{w}=\left(w_{s}\right)_{s\in S}\in\mathbb{R}^{S}}}$,
the $\mathbf{w}$\textbf{-hypersimplex} is defined as: 
\[
\scalebox{0.95}{\ensuremath{\Delta_{\mathbf{w}}=\Delta_{\mathbf{w}}\bigl(k,S\bigr):=\Delta_{S}^{k}\cap\left(\bigcap_{s\in S}\left\{ x_{s}\le w_{s}\right\} \right)}}.
\]
 Henceforth, we assume $\scalebox{0.95}{\ensuremath{\dim\Delta_{\mathbf{w}}=\dim\Delta}}$,
which happens if and only if $\scalebox{0.95}{\ensuremath{\mathbf{w}\left(S\right)>k}}$
and $\scalebox{0.95}{\ensuremath{w_{s}>0}}$ for all $\scalebox{0.95}{\ensuremath{s\in S}}$.
\begin{defn}
\label{def:w-stability}Let $\scalebox{0.95}{\ensuremath{M}}$ be
a rank-$k$ loopless matroid on $\scalebox{0.95}{\ensuremath{S}}$.
For a vector $\mathbf{w}=\left(w_{s}\right)_{s\in S}\in\left[0,1\right]^{S}$,
we say that a nonempty member $\scalebox{0.95}{\ensuremath{X\in\mathcal{W}\cup\phi\left(\mathcal{W}\right)}}$
is \textbf{$\mathbf{w}$-stable} if 
\[
\scalebox{0.95}{\ensuremath{\mathrm{edim}\left.X\right.+\mathrm{codim}_{\Delta_{\mathbf{w}|_{E\left(X\right)}}}\bigl(\mathrm{BP}_{X}\cap\Delta_{\mathbf{w}|_{E\left(X\right)}}\bigr)\le k-1}}
\]
and \textbf{$\mathbf{w}$-unstable} otherwise. Here, $\scalebox{0.95}{\ensuremath{E\left(X\right)\subseteq S}}$,
and $\scalebox{0.95}{\ensuremath{\mathbf{w}|_{E\left(X\right)}\in\left[0,1\right]^{E\left(X\right)}}}$
denotes the vector with $\scalebox{0.95}{\ensuremath{\mathbf{w}|_{E\left(X\right)}\left(s\right)=w_{s}}}$
for $\scalebox{0.95}{\ensuremath{s\in E\left(X\right)}}$.
\end{defn}

The number $\scalebox{0.95}{\ensuremath{k-1-\mathrm{edim}\left(X\right)}}$
is the \emph{expression codimension} of $\scalebox{0.95}{\ensuremath{X}}$,
and therefore, $\scalebox{0.95}{\ensuremath{X}}$ is $\mathbf{w}$-stable
if and only if the expression codimension of $\scalebox{0.95}{\ensuremath{X}}$
is \emph{not} less than the codimension of $\scalebox{0.95}{\ensuremath{\mathrm{BP}_{X}}}$
in $\scalebox{0.95}{\ensuremath{\Delta_{\mathbf{w}|_{E\left(X\right)}}}}$.
 Thus,  Definition \ref{def:w-stability} generalizes Definition \ref{def:stability}.

\section{\label{sec:Hyperplanes}Abstract Hyperplane Arrangement and Grassmann
Geometry}

In this section, we study the abstract hyperplane arrangements in
further depth. In particular, we discover that the \emph{degeneration
of abstract hyperplanes} is equivalent to cutting off corners of a
base polytope to obtain a smaller base polytope. We also construct
realizable line arrangements, whose puzzle-pieces will be used as
building blocks in Section \ref{sec:Tiling Extension}. All matroidal
arrangements are assumed to be labelled.

\subsection{\label{subsec:gen-pos}Abstract hyperplanes in general position}

Recall that if $\scalebox{0.95}{\ensuremath{M}}$ is a loopless matroid,
the abstract hyperplane arrangement $\scalebox{0.95}{\ensuremath{\mathcal{S}^{\lessdot M}=\mathrm{HA}_{M}}}$
has a natural label map $\scalebox{0.95}{\ensuremath{s\mapsto\ensuremath{M/\overline{\left\{ s\right\} }}}}$
for $\scalebox{0.95}{\ensuremath{s\in E\left(M\right)}}$. Because
$\scalebox{0.95}{\ensuremath{\mathcal{S}}}$ is coatomistic, every
subspace $\scalebox{0.95}{\ensuremath{M/F}}$ is an intersection of
hyperplanes: $\scalebox{0.95}{\ensuremath{M/F=\mathlarger{\mathlarger{\mathlarger{\varowedge}}}_{s\in F}\left.M/\overline{\left\{ s\right\} }\right.}}.$
\begin{defn}
\label{def:general position}Let $\scalebox{0.95}{\ensuremath{M}}$
be a rank-$k$ matroid. For any nonempty $\scalebox{0.95}{\ensuremath{J\subseteq E(M)}}$,
we say (labelled) hyperplanes $\scalebox{0.95}{\ensuremath{M/\overline{\{j\}}}}$
for $\scalebox{0.95}{\ensuremath{j\in J}}$ are \textbf{in general
position} if 
\[
\scalebox{0.95}{\ensuremath{M|_{J}=U_{J}^{\min\{k,\left|J\right|\}}}}.
\]
\end{defn}

\begin{lem}
\label{lem:general position}Let $\scalebox{0.95}{\ensuremath{M}}$
be a rank-$k$ loopless matroid. If $\scalebox{0.95}{\ensuremath{\mathrm{HA}_{M}}}$
has $k+1$ hyperplanes in general position, then $\scalebox{0.95}{\ensuremath{M}}$
is connected.
\end{lem}

\begin{proof}
Assume $\scalebox{0.95}{\ensuremath{M/\overline{\left\{ j\right\} }}}$
with $\scalebox{0.95}{\ensuremath{j\in J}}$ are $k+1$ hyperplanes
in general position. If $\scalebox{0.95}{\ensuremath{M}}$ is disconnected,
there is a nontrivial separator $\scalebox{0.95}{\ensuremath{T}}$
with $\scalebox{0.95}{\ensuremath{M|_{J}=M|_{J\cap T}\oplus M|_{J\cap T^{c}}}}$,
and hence $\scalebox{0.95}{\ensuremath{J\cap T=\emptyset}}$ or $\scalebox{0.95}{\ensuremath{J\cap T^{c}=\emptyset}}$
by the connectedness of $\scalebox{0.95}{\ensuremath{M|_{J}\simeq U_{k+1}^{k}}}$.
Without loss of generality, assume $\scalebox{0.95}{\ensuremath{J\cap T^{c}=\emptyset}}$,
i.e., $\scalebox{0.95}{\ensuremath{T\supseteq J}}$, then 
\[
\scalebox{0.95}{\ensuremath{k=r\left(T\right)+r\left(T^{c}\right)\ge r\left(J\right)+r\left(T^{c}\right)=k+r\left(T^{c}\right)}}.
\]
 Therefore, $\scalebox{0.95}{\ensuremath{r\left(T^{c}\right)=0}}$,
and $\scalebox{0.95}{\ensuremath{T^{c}=\emptyset}}$ since $\scalebox{0.95}{\ensuremath{M}}$
is loopless. This contradicts that $\scalebox{0.95}{\ensuremath{T}}$
is a nontrivial separator of $\scalebox{0.95}{\ensuremath{M}}$. Thus,
$\scalebox{0.95}{\ensuremath{M}}$ is connected.
\end{proof}
\begin{rem}
\label{rem:general position}When $\scalebox{0.95}{\ensuremath{k=2,3}}$,
the converse of Lemma~\ref{lem:general position} is true. However,
when $\scalebox{0.95}{\ensuremath{k\ge4}}$, connectedness does not
promise existence of $\scalebox{0.95}{\ensuremath{k+1}}$ hyperplanes
in general position: Consider the graphic matroid $\scalebox{0.95}{\ensuremath{M[G]}}$
of graph $\scalebox{0.95}{\ensuremath{G}}$ given in Figure~\ref{fig:gen-pos-counter-ex},
which is connected since $\scalebox{0.95}{\ensuremath{G}}$ is $\scalebox{0.95}{\ensuremath{2}}$-connected.
All of its circuits have size $\scalebox{0.95}{\ensuremath{4}}$ while
$\scalebox{0.95}{\ensuremath{U_{5}^{4}}}$ has a unique circuit $\scalebox{0.95}{\ensuremath{\left[5\right]}}$
of size $\scalebox{0.95}{\ensuremath{5}}$. So, $\scalebox{0.95}{\ensuremath{M[G]}}$
has no submatroid isomorphic to $\scalebox{0.95}{\ensuremath{U_{5}^{4}}}$.
One can also compute with the given matrix. Here, over any field $\Bbbk$,
the matrix with $\scalebox{0.95}{\ensuremath{1=1_{\Bbbk}}}$ represents
the same matroid $\scalebox{0.95}{\ensuremath{M[G]}}$. We call such
a matrix a \textbf{regular  realization}.
\begin{figure}[H]
\noindent \centering{}\noindent \begin{center}
\def\size{0.9}


\par\end{center}\caption{\label{fig:gen-pos-counter-ex}A rank-4 connected graphic matroid
without any 5 lines in general position and its \emph{regular} realization.}
\end{figure}
\end{rem}

\subsection{A classification of point and line arrangements}

Figures \ref{fig:HA-1-dim} and \ref{fig:HA-2-dim} classify the point
and line arrangements according to the value of $\scalebox{0.95}{\ensuremath{\kappa\left(M\right)}}$.
\begin{figure}[H]
\noindent \begin{centering}
\noindent \begin{center}
\def\size{0.3}
%
\par\end{center}\vspace{-0.4cm}
\par\end{centering}
\noindent \centering{}\caption{\label{fig:HA-1-dim}A classification of point arrangements.}
\end{figure}
\begin{figure}[H]
\noindent \begin{centering}
\noindent \begin{center}
\def\sizea{0.22}
%
\par\end{center}\vspace{-0.4cm}
\par\end{centering}
\noindent \centering{}\caption{\label{fig:HA-2-dim}A classification of line arrangements.}
\end{figure}

\subsection{Realizable hyperplane arrangement}

For any subset $\scalebox{0.95}{\ensuremath{A\subseteq E(M)}}$, we
have $\scalebox{0.95}{\ensuremath{r\left(A\right)={\rm scodim}_{M}\bigl(\mathlarger{\mathlarger{\mathlarger{\varowedge}}}_{s\in A}\left.M/\overline{\left\{ s\right\} }\right.\bigr)}}$.
This is equivalent to \eqref{eq:realizable} if and only if $\scalebox{0.95}{\ensuremath{M}}$
is realizable.
\begin{example}
\label{exa:rank2-realizable}
\begin{enumerate}
\item The uniform matroid $\scalebox{0.95}{\ensuremath{U_{n}^{k}}}$ is
realizable over a large enough field since there are $n$ hyperplanes
of $\scalebox{0.95}{\ensuremath{\mathbb{P}^{k-1}}}$ in general position.
\item \label{enu:realizable}Let $\scalebox{0.95}{\ensuremath{M}}$ be a
loopless rank-$2$ matroid on $\scalebox{0.95}{\ensuremath{S}}$ and
$\scalebox{0.95}{\ensuremath{A_{1}\cup\cdots\cup A_{m}}}$ be the
partition of $\scalebox{0.95}{\ensuremath{S}}$ into the rank-$1$
flats of $\scalebox{0.95}{\ensuremath{M}}$. Consider a line $\scalebox{0.95}{\ensuremath{\mathbb{P}^{1}}}$
over a sufficiently large field, e.g.~an infinite field. For each
$\scalebox{0.95}{\ensuremath{j\in\left[m\right]}}$, pick a point
$\scalebox{0.95}{\ensuremath{P_{j}}}$ on $\scalebox{0.95}{\ensuremath{\mathbb{P}^{1}}}$
avoiding duplication. Every $\scalebox{0.95}{\ensuremath{i\in S}}$
is contained in one and only one $\scalebox{0.95}{\ensuremath{A_{j}}}$
for $\scalebox{0.95}{\ensuremath{j=j\left(i\right)}}$. Then, $\scalebox{0.95}{\ensuremath{\left\{ P_{j}:j\in\left[m\right]\right\} }}$
with label map $\scalebox{0.95}{\ensuremath{\lambda:i\mapsto P_{j\left(i\right)}}}$
has matroid structure $\scalebox{0.95}{\ensuremath{M}}$ and is a
realization of $\scalebox{0.95}{\ensuremath{\mathrm{HA}_{M}}}$. Thus,
rank-$2$ matroids are realizable.
\end{enumerate}
\end{example}

\subsection{\label{subsec:Degeneration}Degeneration of abstract hyperplanes\protect\footnote{This is a continuation of the \emph{degeneration of matroids} discussed
in \cite[Section 5.2]{GGMS}.}}

Suppose $\scalebox{0.95}{\ensuremath{M}}$ is a matroid on $\scalebox{0.95}{\ensuremath{S}}$
of rank $\scalebox{0.95}{\ensuremath{k\ge2}}$. Let $\scalebox{0.95}{\ensuremath{F_{1},\dots,F_{m}}}$
be nonempty proper subsets of $\scalebox{0.95}{\ensuremath{S}}$ and
$j_{1},\dots,j_{m}$ be integers in $\left[k-1\right]$ with $\scalebox{0.95}{\ensuremath{j_{i}<r\left(F_{i}\right)}}$.
Denote
\[
\scalebox{0.95}{\ensuremath{{\displaystyle \delta_{F_{1},\dots,F_{m}}^{j_{1},\dots,j_{m}}\left(\mathrm{BP}_{M}\right)=\mathrm{BP}_{M}\cap\bigcap_{i\in\left[m\right]}\left\{ x\left(F_{i}\right)\le j_{i}\right\} }}}.
\]
If $\scalebox{0.95}{\ensuremath{\delta_{F_{1},\dots,F_{m}}^{j_{1},\dots,j_{m}}\left(\mathrm{BP}_{M}\right)}}$
is a base polytope, by abuse of notation, we denote by $\scalebox{0.95}{\ensuremath{\delta_{F_{1},\dots,F_{m}}^{j_{1},\dots,j_{m}}\left(M\right)}}$
its matroid and by $\scalebox{0.95}{\ensuremath{\delta_{F_{1},\dots,F_{m}}^{j_{1},\dots,j_{m}}\left(\mathrm{HA}_{M}\right)}}$
the corresponding hyperplane arrangement, and call the operation $\scalebox{0.95}{\ensuremath{\delta_{F_{1},\dots,F_{m}}^{j_{1},\dots,j_{m}}}}$
a \textbf{degeneration }(of hyperplanes) for $\scalebox{0.95}{\ensuremath{\mathrm{HA}_{M}}}$
or for $\scalebox{0.95}{\ensuremath{M}}$.

This definition is well-defined because a set of describing inequalities
\eqref{eq:def-ineq-flat} of an independence polytope $\scalebox{0.95}{\ensuremath{\mathrm{IP}_{M}}}$,
written in terms of flats, tells us that hyperplane degeneration of
$\scalebox{0.95}{\ensuremath{M}}$ is equivalent to cutting off corners
of the base polytope $\scalebox{0.95}{\ensuremath{\mathrm{BP}_{M}}}$
to obtain a smaller base polytope. For any permutation $\scalebox{0.95}{\ensuremath{\pi\in\mathfrak{S}_{m}}}$,
we have
\[
\scalebox{0.95}{\ensuremath{\delta_{F_{1},\dots,F_{m}}^{j_{1},\dots,j_{m}}=\delta_{F_{\pi\left(m\right)}}^{j_{\pi\left(m\right)}}\circ\cdots\circ\delta_{F_{\pi\left(1\right)}}^{j_{\pi\left(1\right)}}}}.
\]
 Here, some $\scalebox{0.95}{\ensuremath{\delta_{F_{i}}^{j_{i}}}}$
are possibly not degenerations even if $\scalebox{0.95}{\ensuremath{\delta_{F_{1},\dots,F_{m}}^{j_{1},\dots,j_{m}}}}$
is a degeneration. But, they clearly commute, and degenerations commute.\smallskip{}

Some degenerations do not happen at the same time:
\begin{example}
Consider the hypersimplex $\scalebox{0.95}{\ensuremath{\Delta_{4}^{2}}}$
and let $\scalebox{0.95}{\ensuremath{F_{1}=\left\{ 1,2\right\} }}$,
$\scalebox{0.95}{\ensuremath{F_{2}=\left\{ 1,3\right\} }}$ and $j_{1}=j_{2}=1$,
then $\scalebox{0.95}{\ensuremath{\Delta_{4}^{2}\cap\left(\bigcap_{i\in\left[2\right]}\left\{ x(F_{i})\le1\right\} \right)}}$
is not a base polytope.
\end{example}

\begin{defn}
For two polytopes $\scalebox{0.95}{\ensuremath{P}}$ and $\scalebox{0.95}{\ensuremath{P'}}$,
we say they are \textbf{face-fitting} or one of them is \textbf{face-fitting
to }the other if $\scalebox{0.95}{\ensuremath{P\cap P'}}$ is empty
or a common face of them.
\end{defn}

\begin{lem}
\label{lem:Cut-w-1}Given a rank-$k$ matroid $\scalebox{0.95}{\ensuremath{M=\left(S,r\right)}}$,
if $\scalebox{0.95}{\ensuremath{F\subset S}}$ has a positive rank,
then both $\scalebox{0.95}{\ensuremath{\delta_{F}^{1}}}$ and $\scalebox{0.95}{\ensuremath{\delta_{S-F}^{k-1}}}$
are degenerations for $\scalebox{0.95}{\ensuremath{M}}$, and $\scalebox{0.95}{\ensuremath{\{\delta_{F}^{1}(\mathrm{BP}_{M}),\delta_{S-F}^{k-1}(\mathrm{BP}_{M})\}}}$
is a matroid subdivision of $\scalebox{0.95}{\ensuremath{\mathrm{BP}_{M}}}$. 
\end{lem}

\begin{proof}
By assumption, $\scalebox{0.95}{\ensuremath{1\le r\left(F\right)\le\left|F\right|}}$,
and we  may assume $\scalebox{0.95}{\ensuremath{k\ge1}}$. Define
a map $f$ on $\scalebox{0.95}{\ensuremath{S}}$ by $\scalebox{0.95}{\ensuremath{f|_{S-F}\left(s\right)=s}}$
and $\scalebox{0.95}{\ensuremath{f|_{F}\left(s\right)=F}}$. Then,
$\scalebox{0.95}{\ensuremath{\delta_{F}^{1}\left(M\right)=f^{\ast}f_{\ast}\left(M\right)}}$
is a matroid, and $\scalebox{0.95}{\ensuremath{\delta_{F}^{1}}}$
is a degeneration. Now, $\scalebox{0.95}{\ensuremath{\delta_{F}^{1}(\mathrm{BP}_{M})}}$
and $\scalebox{0.95}{\ensuremath{\delta_{S-F}^{k-1}(\mathrm{BP}_{M})}}$
are face-fitting and their union is $\scalebox{0.95}{\ensuremath{\mathrm{BP}_{M}}}$.
Because every edge of $\scalebox{0.95}{\ensuremath{\delta_{S-F}^{k-1}(\mathrm{BP}_{M})}}$
is an edge of either $\scalebox{0.95}{\ensuremath{\mathrm{BP}_{M}}}$
or the common facet $\scalebox{0.95}{\ensuremath{\delta_{F}^{1}(\mathrm{BP}_{M})\cap\delta_{S-F}^{k-1}(\mathrm{BP}_{M})}}$,
which is a base polytope, $\scalebox{0.95}{\ensuremath{\delta_{S-F}^{k-1}(\mathrm{BP}_{M})}}$
satisfies the edge length property and is a base polytope. Thus, $\scalebox{0.95}{\ensuremath{\delta_{S-F}^{k-1}}}$
is also a degeneration.
\end{proof}
\begin{cor}
Let $\scalebox{0.95}{\ensuremath{M}}$ be a matroid on $\scalebox{0.95}{\ensuremath{S}}$
with subsets $\scalebox{0.95}{\ensuremath{S_{1},\dots,S_{l}}}$. Cutting
$\scalebox{0.95}{\ensuremath{\mathrm{BP}_{M}}}$ with hyperplanes
$\scalebox{0.95}{\ensuremath{\left\{ x\left(S_{j}\right)=1\right\} }}$
for $\scalebox{0.95}{\ensuremath{j\in\left[l\right]}}$ results in
a matroid subdivision.
\end{cor}

Some degenerations are described by the \emph{free product} of matroids.
\begin{lem}
\label{lem:Cutting-1}For a uniform matroid $\scalebox{0.95}{\ensuremath{U_{S}^{k}}}$
with $\scalebox{0.95}{\ensuremath{2\le k<\left|S\right|}}$, let $\scalebox{0.95}{\ensuremath{F\subset S}}$
be a subset of size $\scalebox{0.95}{\ensuremath{\ge2}}$ and $\rho$
be a positive integer with $\scalebox{0.95}{\ensuremath{k-\left|S\right|+\left|F\right|<\rho<\min\left\{ k,\left|F\right|\right\} }}$.
Then, $\scalebox{0.95}{\ensuremath{\delta_{F}^{\rho}}}$ is a degeneration
for $\scalebox{0.95}{\ensuremath{U_{S}^{k}}}$ with $\scalebox{0.95}{\ensuremath{\delta_{F}^{\rho}\left(U_{S}^{k}\right)=U_{F}^{\rho}\boxempty U_{S-F}^{k-\rho}}}.$
This is a rank-$k$ connected matroid on $\scalebox{0.95}{\ensuremath{S}}$,
and $\scalebox{0.95}{\ensuremath{F}}$ is its unique non-degenerate
flat of size $\scalebox{0.95}{\ensuremath{\ge2}}$, whose rank is
$\scalebox{0.95}{\ensuremath{\rho}}$.
\end{lem}

\begin{proof}
Since the bases of $\scalebox{0.95}{\ensuremath{U_{F}^{\rho}\boxempty U_{S-F}^{k-\rho}}}$
are the bases of $\scalebox{0.95}{\ensuremath{U_{S}^{k}}}$ with $\scalebox{0.95}{\ensuremath{x\left(F\right)\le\rho}}$,
one has $\scalebox{0.95}{\ensuremath{\delta_{F}^{\rho}\left(U_{S}^{k}\right)=U_{F}^{\rho}\boxempty U_{S-F}^{k-\rho}}}$,
which is a connected matroid since $\scalebox{0.95}{\ensuremath{\delta_{F}^{\rho}\left(\Delta_{S}^{k}\right)}}$
is full-dimensional. Two uniform matroids $\scalebox{0.95}{\ensuremath{U_{F}^{\rho}=\left.\delta_{F}^{\rho}\left(U_{S}^{k}\right)\right|_{F}}}$
and $\scalebox{0.95}{\ensuremath{U_{S-F}^{k-\rho}=\delta_{F}^{\rho}\left(U_{S}^{k}\right)/F}}$
are connected since $\scalebox{0.95}{\ensuremath{0<\rho<\left|F\right|}}$
and $\scalebox{0.95}{\ensuremath{0<k-\rho<\left|S-F\right|}}$. So,
$\scalebox{0.95}{\ensuremath{F}}$ is the unique non-degenerate flat
of $\scalebox{0.95}{\ensuremath{\delta_{F}^{\rho}\left(U_{S}^{k}\right)}}$
of size $\scalebox{0.95}{\ensuremath{\ge2}}$, whose rank is $\scalebox{0.95}{\ensuremath{\rho}}$.
\end{proof}
\begin{prop}
\label{prop:Cutting-2}For a uniform matroid $\scalebox{0.95}{\ensuremath{U_{S}^{k}}}$
with $\scalebox{0.95}{\ensuremath{3\le k<\left|S\right|}}$, let $\scalebox{0.95}{\ensuremath{F}}$
and $\scalebox{0.95}{\ensuremath{L}}$ be subsets of $\scalebox{0.95}{\ensuremath{S}}$
with $\scalebox{0.95}{\ensuremath{F\cup L=S}}$, $\scalebox{0.95}{\ensuremath{F\cap L\neq\emptyset}}$,
$\scalebox{0.95}{\ensuremath{F-L\neq\emptyset}}$, and $\scalebox{0.95}{\ensuremath{L-F\neq\emptyset}}$.
Let $\rho_{F}$ and $\rho_{L}$ be positive integers with $\scalebox{0.95}{\ensuremath{k-\left|S\right|+\left|X\right|<\rho_{X}<\min\left\{ k,\left|X\right|\right\} }}$
for $\scalebox{0.95}{\ensuremath{X=F,L}}$ and $\scalebox{0.95}{\ensuremath{\rho_{F}+\rho_{L}<k+\left|F\cap L\right|}}$.
Then, $\scalebox{0.95}{\ensuremath{\delta_{F,L}^{\rho_{F},\rho_{L}}}}$
is a degeneration for $\scalebox{0.95}{\ensuremath{U_{S}^{k}}}$ with
\[
\scalebox{0.95}{\ensuremath{\delta_{F,L}^{\rho_{F},\rho_{L}}\bigl(U_{S}^{k}\bigr)=U_{F\cap L}^{\rho_{F}+\rho_{L}-k}\boxempty\bigl(U_{L-F}^{k-\rho_{F}}\oplus U_{F-L}^{k-\rho_{L}}\bigr)}}.
\]
 This is a connected matroid on $\scalebox{0.95}{\ensuremath{S}}$
of rank $k$, and $\scalebox{0.95}{\ensuremath{F}}$ and $\scalebox{0.95}{\ensuremath{L}}$
are non-degenerate flats of rank $\scalebox{0.95}{\ensuremath{\rho_{F}}}$
and $\scalebox{0.95}{\ensuremath{\rho_{L}}}$, respectively. However,
$\scalebox{0.95}{\ensuremath{F\cap L}}$ is a degenerate flat of rank
$\scalebox{0.95}{\ensuremath{\rho_{F}+\rho_{L}-k}}$ with $\scalebox{0.95}{\ensuremath{\kappa\bigl(\delta_{F,L}^{\rho_{F},\rho_{L}}\bigl(U_{S}^{k}\bigr)\left(F\cap L\right)\bigr)=3}}$.
In particular, $\delta_{F,L}^{\rho_{F},\rho_{L}}\bigl(U_{S}^{k}\bigr)=\delta_{F\cap L,F,L}^{\rho_{F}+\rho_{L}-k,\rho_{F},\rho_{L}}\bigl(U_{S}^{k}\bigr)$.
\end{prop}

\begin{proof}
The bases of $\scalebox{0.95}{\ensuremath{U_{F\cap L}^{\rho_{F}+\rho_{L}-k}\boxempty\bigl(U_{L-F}^{k-\rho_{F}}\oplus U_{F-L}^{k-\rho_{L}}\bigr)}}$
are bases of $\scalebox{0.95}{\ensuremath{U_{S}^{k}}}$ with $\scalebox{0.95}{\ensuremath{x\left(L\right)\le\rho_{L}}}$
and $\scalebox{0.95}{\ensuremath{x\left(F\right)\le\rho_{F}}}$. Conversely,
let $\scalebox{0.95}{\ensuremath{B}}$ be a base of $\scalebox{0.95}{\ensuremath{U_{S}^{k}}}$
with $\scalebox{0.95}{\ensuremath{x\left(L\right)\le\rho_{L}}}$ and
$\scalebox{0.95}{\ensuremath{x\left(F\right)\le\rho_{F}}}$. Then,
$\scalebox{0.95}{\ensuremath{\left|B\cap L\right|=1^{B}\left(L\right)\le\rho_{L}}}$
and $\scalebox{0.95}{\ensuremath{\left|B\cap F\right|=1^{B}\left(F\right)\le\rho_{F}}}$,
and 
\[
\scalebox{0.95}{\ensuremath{\left|B\cap\left(F\cap L\right)\right|}}=\scalebox{0.95}{\ensuremath{\left|B\cap F\right|+\left|B\cap L\right|-\left|B\cap\left(F\cup L\right)\right|}}\le\scalebox{0.95}{\ensuremath{\rho_{F}+\rho_{L}-k}}.
\]
 So, $\scalebox{0.95}{\ensuremath{\left|B-\left(F\cap L\right)\right|\ge2k-\rho_{F}-\rho_{L}}}$,
and $\scalebox{0.95}{\ensuremath{\left|B\cap\left(S-X\right)\right|\ge k-\rho_{X}}}$
for $\scalebox{0.95}{\ensuremath{X=F,L}}$. Therefore, $\scalebox{0.95}{\ensuremath{B-\left(F\cap L\right)}}$
is a spanning set of $\scalebox{0.95}{\ensuremath{U_{L-F}^{k-\rho_{F}}\oplus U_{F-L}^{k-\rho_{L}}}}$,
and $\scalebox{0.95}{\ensuremath{B}}$ is a base of the free product
$\scalebox{0.95}{\ensuremath{U_{F\cap L}^{\rho_{F}+\rho_{L}-k}\boxempty\bigl(U_{L-F}^{k-\rho_{F}}\oplus U_{F-L}^{k-\rho_{L}}\bigr)}}$.
Thus, $\scalebox{0.95}{\ensuremath{\delta_{F,L}^{\rho_{F},\rho_{L}}}}$
is a degeneration for $\scalebox{0.95}{\ensuremath{U_{S}^{k}}}$ with
$\scalebox{0.95}{\ensuremath{\delta_{F,L}^{\rho_{F},\rho_{L}}(U_{S}^{k})=U_{F\cap L}^{\rho_{F}+\rho_{L}-k}\boxempty\bigl(U_{L-F}^{k-\rho_{F}}\oplus U_{F-L}^{k-\rho_{L}}\bigr)}}$.

Choose $\scalebox{0.95}{\ensuremath{k-\rho_{F}+1}}$, $\scalebox{0.95}{\ensuremath{k-\rho_{L}+1}}$,
and $\scalebox{0.95}{\ensuremath{\rho_{F}+\rho_{L}-k-1}}$ elements
from $\scalebox{0.95}{\ensuremath{L-F}}$, $\scalebox{0.95}{\ensuremath{F-L}}$,
and $\scalebox{0.95}{\ensuremath{F\cap L}}$, respectively, and let
$\scalebox{0.95}{\ensuremath{A}}$ be the set of these $\scalebox{0.95}{\ensuremath{k+1}}$
elements. Then, $\scalebox{0.95}{\ensuremath{\delta_{F,L}^{\rho_{F},\rho_{L}}\bigl(U_{S}^{k}\bigr)|_{A}\simeq U_{k+1}^{k}}}$,
and $\scalebox{0.95}{\ensuremath{\delta_{F,L}^{\rho_{F},\rho_{L}}\bigl(U_{S}^{k}\bigr)}}$
is connected by Lemma \ref{lem:general position}.

Now, $\scalebox{0.95}{\ensuremath{\delta_{F,L}^{\rho_{F},\rho_{L}}\bigl(\Delta_{S}^{k}\bigr)=\delta_{F}^{\rho_{F}}\bigl(\Delta_{S}^{k}\bigr)\cap\delta_{L}^{\rho_{L}}\bigl(\Delta_{S}^{k}\bigr)}}$,
which is full-dimensional. For $\scalebox{0.95}{\ensuremath{X=F,L}}$,
none of $\scalebox{0.95}{\ensuremath{\delta_{X}^{\rho_{X}}(\Delta_{S}^{k})}}$
contains the other, which makes $\scalebox{0.95}{\ensuremath{\delta_{F,L}^{\rho_{F},\rho_{L}}(\Delta_{S}^{k})\left(X\right)}}$
distinct facets of $\scalebox{0.95}{\ensuremath{\delta_{F,L}^{\rho_{F},\rho_{L}}(\Delta_{S}^{k})}}$
and $\scalebox{0.95}{\ensuremath{X}}$ non-degenerate flats of $\scalebox{0.95}{\ensuremath{\delta_{F,L}^{\rho_{F},\rho_{L}}(U_{S}^{k})}}$
with ranks $\scalebox{0.95}{\ensuremath{\rho_{X}}}$. Moreover, we
have $\scalebox{0.95}{\ensuremath{\delta_{F,L}^{\rho_{F},\rho_{L}}(U_{S}^{k})|_{F\cap L}=U_{F\cap L}^{\rho_{F}+\rho_{L}-k}}}$
and $\scalebox{0.95}{\ensuremath{\delta_{F,L}^{\rho_{F},\rho_{L}}(U_{S}^{k})/(F\cup L)=U_{L-F}^{k-\rho_{F}}\oplus U_{F-L}^{k-\rho_{L}}}}$
where the uniform matroids $\scalebox{0.95}{\ensuremath{U_{\left(F\cup L\right)-X}^{k-\rho_{X}}}}$
are connected since $\scalebox{0.95}{\ensuremath{0<k-\rho_{X}<\left|\left(F\cup L\right)-X\right|}}$.
So, $\scalebox{0.95}{\ensuremath{\kappa\bigl(\delta_{F,L}^{\rho_{F},\rho_{L}}(U_{S}^{k})\left(F\cap L\right)\bigr)=3}}$,
and $\scalebox{0.95}{\ensuremath{F\cap L}}$ is a degenerate flat
of $\scalebox{0.95}{\ensuremath{\delta_{F,L}^{\rho_{F},\rho_{L}}(U_{S}^{k})}}$
with rank $\scalebox{0.95}{\ensuremath{\rho_{F}+\rho_{L}-k}}$. This
implies $\scalebox{0.95}{\ensuremath{\delta_{F,L}^{\rho_{F},\rho_{L}}\bigl(U_{S}^{k}\bigr)=\delta_{F\cap L,F,L}^{\rho_{F}+\rho_{L}-k,\rho_{F},\rho_{L}}\bigl(U_{S}^{k}\bigr)}}$,
cf. Lemma \ref{lem:Ineq improved}.
\end{proof}
\begin{example}
\label{exa:degenerations}Let $\scalebox{0.95}{\ensuremath{M=U_{S}^{3}}}$
with $\scalebox{0.95}{\ensuremath{\left|S\right|\ge5}}$. Suppose
that $\scalebox{0.95}{\ensuremath{F}}$ and $\scalebox{0.95}{\ensuremath{L}}$
are subsets of $\scalebox{0.95}{\ensuremath{S}}$ with $\scalebox{0.95}{\ensuremath{F\cup L=S}}$,
$\scalebox{0.95}{\ensuremath{F\cap L\neq\emptyset}}$, $\scalebox{0.95}{\ensuremath{\left|F-L\right|\ge2}}$,
and $\scalebox{0.95}{\ensuremath{\left|L-F\right|\ge2}}$.
\begin{enumerate}
\item \label{enu:degen-1}Because $\scalebox{0.95}{\ensuremath{3-\left|S\right|+\left|F\right|<2<\min\left\{ 3,\left|F\right|\right\} }}$,
the operation $\scalebox{0.95}{\ensuremath{\delta_{F}^{2}}}$ is a
degeneration for $\scalebox{0.95}{\ensuremath{M}}$, and $\scalebox{0.95}{\ensuremath{\delta_{F}^{2}\left(M\right)}}$
is a connected matroid, cf. Lemma \ref{lem:Cutting-1}. See Figure
\ref{fig:degenerations-1}.
\item \label{enu:degen-2}$\scalebox{0.95}{\ensuremath{\delta_{F\cap L,F,L}^{1,2,2}}}$
and $\scalebox{0.95}{\ensuremath{\delta_{F,L}^{2,2}}}$ are the same
degeneration and $\scalebox{0.95}{\ensuremath{\delta_{F\cap L,F,L}^{1,2,2}\left(M\right)=\delta_{F,L}^{2,2}\left(M\right)}}$
is a connected matroid, cf. Proposition \ref{prop:Cutting-2}. See
Figure \ref{fig:degenerations-2}. 
\end{enumerate}
\end{example}

\begin{rem}
\label{rem:realizable}In Lemmas \ref{lem:Cutting-1} and \ref{prop:Cutting-2},
the constructed matroids are realizable over a sufficiently large
field because uniform matroids are realizable over such a field and
so is their free product, cf. \cite[Proposition 4.13]{Crapo1}.
\end{rem}

\begin{figure}[H]
\noindent \centering{}\noindent \begin{center}
\def\sizea{0.22}
\def\size{0.25}
\def\ratioa{0.901}
\def\ratior{0.851}
\def\ratiow{0.761}
\def\ratioe{0.881}


\par\end{center}\caption{\label{fig:degenerations-1}Degeneration of Lines I}
\end{figure}
\begin{figure}[H]
\noindent \centering{}\noindent \begin{center}
\def\sizea{0.22}
\def\size{0.25}
\def\ratioa{0.901}
\def\ratior{0.851}
\def\ratiow{0.761}
\def\ratioe{0.881}

%
\par\end{center}\caption{\label{fig:degenerations-2}Degeneration of Lines II}
\end{figure}

\subsection{\label{subsec:Construct-LA}Construction of realizable line arrangements}

Let $\scalebox{0.95}{\ensuremath{M}}$ and $\scalebox{0.95}{\ensuremath{N}}$
be two loopless matroids on $\scalebox{0.95}{\ensuremath{S}}$ of
rank $3$ and dimension $1$, then $\scalebox{0.95}{\ensuremath{\left\lfloor M\right\rfloor }}$
and $\scalebox{0.95}{\ensuremath{\left\lfloor N\right\rfloor }}$
are connected matroids of rank $2$ and dimension $1$. Suppose $\scalebox{0.95}{\ensuremath{E\left(\left\lfloor M\right\rfloor \right)\cup E\left(\left\lfloor N\right\rfloor \right)=S}}$
and let $\scalebox{0.95}{\ensuremath{F_{0},F_{1},\dots,F_{m}}}$ and
$\scalebox{0.95}{\ensuremath{L_{0},L_{1},\dots,L_{n}}}$ be the rank-$1$
flats of $\scalebox{0.95}{\ensuremath{\left\lfloor M\right\rfloor }}$
and $\scalebox{0.95}{\ensuremath{\left\lfloor N\right\rfloor }}$,
respectively. Then, $\scalebox{0.95}{\ensuremath{m,n\ge2}}$, and
we  have $\scalebox{0.95}{\ensuremath{E\left(\left\lfloor M\right\rfloor \right)=\bigcup_{i=0}^{m}F_{i}}}$
and $\scalebox{0.95}{\ensuremath{E\left(\left\lfloor N\right\rfloor \right)=\bigcup_{j=0}^{n}L_{j}}}$,
which are \emph{disjoint} unions.
\begin{enumerate}[label=(LA\arabic*),topsep=0.5em]
\item \label{enu:(LA1)}Suppose $\scalebox{0.95}{\ensuremath{E\left(\left\lfloor M\right\rfloor \right)=S-L_{0}}}$
and $\scalebox{0.95}{\ensuremath{E\left(\left\lfloor N\right\rfloor \right)=S-F_{0}}}$.
Figure \ref{fig:LA-1} depicts how to obtain a rank-$3$ loopless
matroid on $\scalebox{0.95}{\ensuremath{S}}$ which we denote by $\scalebox{0.95}{\ensuremath{\mathrm{MA}_{M,N}^{1,1}}}$.
Because it has $4$ lines in general position, it is connected by
Lemma \ref{lem:general position}. We describe below its nonempty
proper flats.
\begin{enumerate}
\item Its rank-$1$ flats are $\scalebox{0.95}{\ensuremath{F_{0},L_{0}}}$
and all nonempty intersections of $\scalebox{0.95}{\ensuremath{F_{i}}}$
and $\scalebox{0.95}{\ensuremath{L_{j}}}$ for $\scalebox{0.95}{\ensuremath{i\in\left[m\right]}}$
and $\scalebox{0.95}{\ensuremath{j\in\left[n\right]}}$.
\item Its rank-$2$ flats are $\scalebox{0.95}{\ensuremath{F_{0}\cup L_{0}}}$,
$\scalebox{0.95}{\ensuremath{F_{i}}}$ for $\scalebox{0.95}{\ensuremath{i\in\left[m\right]}}$,
$\scalebox{0.95}{\ensuremath{L_{j}}}$ for $\scalebox{0.95}{\ensuremath{j\in\left[n\right]}}$,
and all unions of two rank-$1$ flats that are not contained in any
of the aforementioned rank-$2$ flats.
\end{enumerate}
The rank-$1$ flats $\scalebox{0.95}{\ensuremath{F_{0}}}$ and $\scalebox{0.95}{\ensuremath{L_{0}}}$
are non-degenerate with $\scalebox{0.95}{\ensuremath{\left\lfloor M\right\rfloor =\mathrm{MA}_{M,N}^{1,1}/L_{0}}}$
and $\scalebox{0.95}{\ensuremath{\left\lfloor N\right\rfloor =\mathrm{MA}_{M,N}^{1,1}/F_{0}}}$,
while $\scalebox{0.95}{\ensuremath{M=\mathrm{MA}_{M,N}^{1,1}(L_{0})}}$
and $\scalebox{0.95}{\ensuremath{N=\mathrm{MA}_{M,N}^{1,1}(F_{0})}}$.
\begin{figure}[H]
\noindent \centering{}\noindent \begin{center}
\def\sizea{0.22}
\def\size{0.25}
\def\ratioa{0.901}
\def\ratior{0.851}
\def\ratiow{0.761}
\def\ratioe{0.881}


\par\end{center}\caption{\label{fig:LA-1}Construction of Line Arrangements I}
\end{figure}

\end{enumerate}
Now, let $\scalebox{0.95}{\ensuremath{f,g}}$ be the simplification
maps for $\scalebox{0.95}{\ensuremath{\left\lfloor M\right\rfloor }}$
and $\scalebox{0.95}{\ensuremath{\left\lfloor N\right\rfloor }}$,
respectively. Define a map $\scalebox{0.95}{\ensuremath{h}}$ on $\scalebox{0.95}{\ensuremath{S}}$
such that $\scalebox{0.95}{\ensuremath{h|_{S-E\left(\left\lfloor N\right\rfloor \right)}=f|_{S-E\left(\left\lfloor N\right\rfloor \right)}}}$
and $\scalebox{0.95}{\ensuremath{h|_{E\left(\left\lfloor N\right\rfloor \right)}=g}}$.
For the following constructions, see Figures \ref{fig:LA-2}, \ref{fig:LA-3},
and \ref{fig:LA-4} in due order.
\begin{enumerate}[label=(LA\arabic*),start=2,topsep=0.5em,itemsep=0.5em]
\item \label{enu:(LA2)}Suppose $\scalebox{0.95}{\ensuremath{E\left(\left\lfloor M\right\rfloor \right)\cap E\left(\left\lfloor N\right\rfloor \right)=F_{0}=L_{0}}}$.
Then, the following is a rank-$3$ connected matroid on $\scalebox{0.95}{\ensuremath{S}}$:
\[
\scalebox{0.95}{\ensuremath{\mathrm{MA}_{M,N}^{2,2}:=h^{\ast}\bigl(\delta_{h\left(E\left(\left\lfloor M\right\rfloor \right)\right),h\left(E\left(\left\lfloor N\right\rfloor \right)\right)}^{2,2}\bigl(U_{h(S)}^{3}\bigr)\bigr)=\delta_{E\left(\left\lfloor M\right\rfloor \right),E\left(\left\lfloor N\right\rfloor \right)}^{2,2}\bigl(h^{\ast}\bigl(U_{h(S)}^{3}\bigr)\bigr).}}
\]

\begin{enumerate}
\item Its rank-$1$ flats are $\scalebox{0.95}{\ensuremath{F_{0}=L_{0}}}$,
$\scalebox{0.95}{\ensuremath{F_{i}}}$ for $\scalebox{0.95}{\ensuremath{i\in\left[m\right]}}$,
and $\scalebox{0.95}{\ensuremath{L_{j}}}$ for $\scalebox{0.95}{\ensuremath{j\in\left[n\right]}}$.
\item Its rank-$2$ flats are $\scalebox{0.95}{\ensuremath{E\left(\left\lfloor M\right\rfloor \right)}}$,
$\scalebox{0.95}{\ensuremath{E\left(\left\lfloor N\right\rfloor \right)}}$,
and $\scalebox{0.95}{\ensuremath{F_{i}\cup L_{j}}}$ for $\scalebox{0.95}{\ensuremath{i\in\left[m\right]}}$
and $\scalebox{0.95}{\ensuremath{j\in\left[n\right]}}$.
\end{enumerate}
The flats $\scalebox{0.95}{\ensuremath{E\left(\left\lfloor M\right\rfloor \right)}}$
and $\scalebox{0.95}{\ensuremath{E\left(\left\lfloor N\right\rfloor \right)}}$
are non-degenerate with $\scalebox{0.95}{\ensuremath{\left\lfloor M\right\rfloor =\mathrm{MA}_{M,N}^{2,2}|_{E\left(\left\lfloor M\right\rfloor \right)}}}$
and $\scalebox{0.95}{\ensuremath{\left\lfloor N\right\rfloor =\mathrm{MA}_{M,N}^{2,2}|_{E\left(\left\lfloor N\right\rfloor \right)}}}$.
\begin{figure}[H]
\noindent \centering{}\noindent \begin{center}
\def\sizea{0.22}
\def\size{0.25}
\def\ratioa{0.901}
\def\ratior{0.851}
\def\ratiow{0.761}
\def\ratioe{0.881}


\par\end{center}\caption{\label{fig:LA-2}Construction of Line Arrangements II}
\end{figure}

\item \label{enu:(LA3)}Suppose $\scalebox{0.95}{\ensuremath{E\left(\left\lfloor M\right\rfloor \right)=S-L_{0}}}$
and $\scalebox{0.95}{\ensuremath{F_{0}=\bigcup_{j\in\left[n\right]}L_{j}}}$
with $\scalebox{0.95}{\ensuremath{E\left(\left\lfloor N\right\rfloor \right)=F_{0}\cup L_{0}}}$.
Then, the following is a rank-$3$ connected matroid on $\scalebox{0.95}{\ensuremath{S}}$:
\[
\scalebox{0.95}{\ensuremath{\mathrm{MA}_{M,N}^{1,2}=\mathrm{MA}_{N,M}^{2,1}:=h^{\ast}\bigl(\delta_{h\left(E\left(\left\lfloor N\right\rfloor \right)\right)}^{2}\bigl(U_{h\left(S\right)}^{3}\bigr)\bigr)=\delta_{E\left(\left\lfloor N\right\rfloor \right)}^{2}\bigl(h^{\ast}\bigl(U_{h\left(S\right)}^{3}\bigr)\bigr).}}
\]

\begin{enumerate}
\item Its rank-$1$ flats are $\scalebox{0.95}{\ensuremath{L_{0}}}$, $\scalebox{0.95}{\ensuremath{F_{i}}}$
for $\scalebox{0.95}{\ensuremath{i\in\left[m\right]}}$, and $\scalebox{0.95}{\ensuremath{L_{j}}}$
for $\scalebox{0.95}{\ensuremath{j\in\left[n\right]}}$.
\item Its rank-$2$ flats are $\scalebox{0.95}{\ensuremath{F_{0}\cup L_{0}}}$,
$\scalebox{0.95}{\ensuremath{F_{i}\cup L_{0}}}$ for $\scalebox{0.95}{\ensuremath{i\in\left[m\right]}}$,
$\scalebox{0.95}{\ensuremath{F_{i}\cup L_{j}}}$ for $\scalebox{0.95}{\ensuremath{i\in\left[m\right]}}$
and $\scalebox{0.95}{\ensuremath{j\in\left[n\right]}}$, and $\scalebox{0.95}{\ensuremath{F_{i}\cup F_{l}}}$
for $\scalebox{0.95}{\ensuremath{i,l\in\left[m\right]}}$ with $\scalebox{0.95}{\ensuremath{i\neq l}}$.
\end{enumerate}
The flats $\scalebox{0.95}{\ensuremath{L_{0}}}$ and $\scalebox{0.95}{\ensuremath{E\left(\left\lfloor N\right\rfloor \right)}}$
are non-degenerate with $\scalebox{0.95}{\ensuremath{\left\lfloor M\right\rfloor =\mathrm{MA}_{M,N}^{1,2}/L_{0}}}$
and $\scalebox{0.95}{\ensuremath{\left\lfloor N\right\rfloor =\mathrm{MA}_{M,N}^{1,2}|_{E\left(\left\lfloor N\right\rfloor \right)}}}$.
\begin{figure}[H]
\noindent \centering{}\noindent \begin{center}
\def\sizea{0.22}
\def\size{0.25}
\def\ratioa{0.901}
\def\ratior{0.851}
\def\ratiow{0.761}
\def\ratioe{0.881}


\par\end{center}\caption{\label{fig:LA-3}Construction of Line Arrangements III}
\end{figure}

\item \label{enu:(LA4)}Suppose $\scalebox{0.95}{\ensuremath{E\left(\left\lfloor M\right\rfloor \right)=S-L_{0}}}$
and $\scalebox{0.95}{\ensuremath{F_{0}=\bigcup_{j\in\left[n\right]}L_{j}}}$
with $\scalebox{0.95}{\ensuremath{E\left(\left\lfloor N\right\rfloor \right)=F_{0}\cup L_{0}}}$.
Then, the following is a rank-$3$ connected matroid on $\scalebox{0.95}{\ensuremath{S}}$:
\[
\scalebox{0.95}{\ensuremath{h^{\ast}\bigl(\delta_{h\left(L_{n}\cup\left(E\left(\left\lfloor M\right\rfloor \right)-F_{0}\right)\right),h\left(E\left(\left\lfloor N\right\rfloor \right)\right)}^{2,2}\bigl(U_{h\left(S\right)}^{3}\bigr)\bigr)=\delta_{L_{n}\cup\left(E\left(\left\lfloor M\right\rfloor \right)-F_{0}\right)}^{2}\bigl(\mathrm{MA}_{M,N}^{1,2}\bigr).}}
\]

\begin{enumerate}
\item Its rank-$1$ flats are $\scalebox{0.95}{\ensuremath{L_{0}}}$, $\scalebox{0.95}{\ensuremath{F_{i}}}$
for $\scalebox{0.95}{\ensuremath{i\in\left[m\right]}}$, and $\scalebox{0.95}{\ensuremath{L_{j}}}$
for $\scalebox{0.95}{\ensuremath{j\in\left[n\right]}}$.
\item Its rank-$2$ flats are $\scalebox{0.95}{\ensuremath{F_{0}\cup L_{0}}}$,
$\scalebox{0.95}{\ensuremath{L_{n}\cup(E\left(\left\lfloor M\right\rfloor \right)-F_{0})}}$,
and $\scalebox{0.95}{\ensuremath{F_{i}\cup L_{j}}}$ for $i\in\left[m\right]$
and $j=0,1,\dots,n-1$.
\begin{figure}[H]
\noindent \centering{}\noindent \begin{center}
\def\sizea{0.22}
\def\size{0.25}
\def\ratioa{0.901}
\def\ratior{0.851}
\def\ratiow{0.761}
\def\ratioe{0.881}


\par\end{center}\caption{\label{fig:LA-4}Construction of Line Arrangements IV}
\end{figure}
\end{enumerate}
\end{enumerate}
\begin{rem}
\label{rem:realizable-1}The above matroids are realizable over a
sufficiently large field. The loopless matroid $\scalebox{0.95}{\ensuremath{M\cap N}}$
has dimension $0$, and $\scalebox{0.95}{\ensuremath{\mathrm{PZ}_{M\cap N}}}$
is a point-piece.
\end{rem}

\subsection{A characteristic property of abstract line arrangements}

The line arrangements produced in \ref{enu:(LA2)} and \ref{enu:(LA4)}
can also be achieved by applying Proposition \ref{prop:Cutting-2}
and then using an appropriate pullback. These line arrangements have
rank-$\scalebox{0.95}{\ensuremath{1}}$ degenerate flats. The following
lemma indicates that any line arrangement can possess at most one
rank-$\scalebox{0.95}{\ensuremath{1}}$ degenerate flat, which is
a characteristic feature of line arrangements.
\begin{lem}
\label{lem:degen-rank1-flat}Let $\scalebox{0.95}{\ensuremath{M}}$
be a rank-$\scalebox{0.95}{\ensuremath{3}}$ connected matroid on
$\scalebox{0.95}{\ensuremath{S}}$. Then, $\scalebox{0.95}{\ensuremath{M}}$
has two proper flats $\scalebox{0.95}{\ensuremath{F}}$ and $\scalebox{0.95}{\ensuremath{L}}$
with $\scalebox{0.95}{\ensuremath{F\cap L\neq\emptyset}}$ and $\scalebox{0.95}{\ensuremath{F\cup L=S}}$
if and only if $\scalebox{0.95}{\ensuremath{M}}$ has a degenerate
flat $\scalebox{0.95}{\ensuremath{T}}$ of rank $\scalebox{0.95}{\ensuremath{1}}$.
In this case, $\scalebox{0.95}{\ensuremath{F}}$ and $\scalebox{0.95}{\ensuremath{L}}$
are the only two non-degenerate flats of rank $\scalebox{0.95}{\ensuremath{2}}$,
and $\scalebox{0.95}{\ensuremath{T=F\cap L}}$ is the unique degenerate
flat of rank $\scalebox{0.95}{\ensuremath{1}}$. 
\end{lem}

\begin{proof}
To prove the only if direction, suppose that $\scalebox{0.95}{\ensuremath{M}}$
has two proper flats $\scalebox{0.95}{\ensuremath{F}}$ and $\scalebox{0.95}{\ensuremath{L}}$
with $\scalebox{0.95}{\ensuremath{F\cap L\neq\emptyset}}$ and $\scalebox{0.95}{\ensuremath{F\cup L=S}}$.
Then, $\scalebox{0.95}{\ensuremath{F}}$ and $\scalebox{0.95}{\ensuremath{L}}$
are rank-$2$ flats. If $\scalebox{0.95}{\ensuremath{F}}$ is disconnected,
then $\scalebox{0.95}{\ensuremath{F=\left(F\cap L\right)\cup T}}$
for a rank-$1$ flat $\scalebox{0.95}{\ensuremath{T}}$ with $\scalebox{0.95}{\ensuremath{S=T\cup L}}$.
So, $\scalebox{0.95}{\ensuremath{M}}$ is disconnected, a contradiction.
 Thus, $\scalebox{0.95}{\ensuremath{F}}$ is connected and so is $\scalebox{0.95}{\ensuremath{L}}$
by the same manner. Then, since $\scalebox{0.95}{\ensuremath{M/F}}$
and $\scalebox{0.95}{\ensuremath{M/L}}$ are loopless rank-$1$ matroids,
they are connected and therefore $\scalebox{0.95}{\ensuremath{F}}$
and $\scalebox{0.95}{\ensuremath{L}}$ are non-degenerate. Further,
if $\scalebox{0.95}{\ensuremath{A=\left(A\cap F\right)\cup\left(A\cap L\right)}}$
is a rank-$2$ flat that is different from $\scalebox{0.95}{\ensuremath{F}}$
and $\scalebox{0.95}{\ensuremath{L}}$, then it is disconnected. Therefore,
$\scalebox{0.95}{\ensuremath{F}}$ and $\scalebox{0.95}{\ensuremath{L}}$
are the only two non-degenerate flats of rank $2$. Then, $\scalebox{0.95}{\ensuremath{M/\left(F\cap L\right)=M/F\oplus M/L}}$
by Lemma \ref{lem:Ridges of BP}, and $\scalebox{0.95}{\ensuremath{F\cap L}}$
is a rank-$1$ degenerate flat.

To prove the if direction, let $\scalebox{0.95}{\ensuremath{T}}$
be a rank-$1$ degenerate flat, then there are rank-$2$ flats $\scalebox{0.95}{\ensuremath{F}}$
and $\scalebox{0.95}{\ensuremath{L}}$ with $\scalebox{0.95}{\ensuremath{M/T=\left(M|_{F}/T\right)\oplus\left(M|_{L}/T\right)}}$
so that $\scalebox{0.95}{\ensuremath{T=F\cap L}}$ and $\scalebox{0.95}{\ensuremath{S=F\cup L}}$.
The above argument tells that $\scalebox{0.95}{\ensuremath{T}}$ is
the unique degenerate flat of rank $\scalebox{0.95}{\ensuremath{1}}$.
\end{proof}
\begin{example}
The matroid of Example \ref{exa:meet} has the unique rank-$\scalebox{0.95}{\ensuremath{1}}$
degenerate flat which is $\scalebox{0.95}{\ensuremath{\left\{ 3\right\} }}$.
\end{example}

\begin{rem}
\label{rem:degenerate-flat}When $\scalebox{0.95}{\ensuremath{k\ge4}}$,
however, the uniqueness of rank-$\scalebox{0.95}{\ensuremath{1}}$
degenerate flat fails. For example, consider the rank-$\scalebox{0.95}{\ensuremath{4}}$
graphic matroid $\scalebox{0.95}{\ensuremath{M[G]}}$ shown in Figure~\ref{fig:Counter-to-degen lem}.
Since the graph $\scalebox{0.95}{\ensuremath{G}}$ is $\scalebox{0.95}{\ensuremath{2}}$-connected,
$\scalebox{0.95}{\ensuremath{M[G]}}$ is connected. Now, contracting
either edge $e_{6}$ or $e_{7}$ produces a loopless graph $\scalebox{0.95}{\ensuremath{G'}}$
that is only $\scalebox{0.95}{\ensuremath{1}}$-connected, and $\scalebox{0.95}{\ensuremath{M[G']}}$
is disconnected. So, $\scalebox{0.95}{\ensuremath{\left\{ 6\right\} }}$
and $\scalebox{0.95}{\ensuremath{\left\{ 7\right\} }}$ are two rank-$\scalebox{0.95}{\ensuremath{1}}$
degenerate flats of $\scalebox{0.95}{\ensuremath{M[G]}}$.\vspace{-0.2cm}

\begin{figure}[H]
\noindent \centering{}\noindent \begin{center}
\def\size{0.75}

\begin{tikzpicture}[font=\footnotesize]

\begin{scope}[line cap=round,rotate=0,scale=\size,xshift=0cm,yshift=0cm,>=stealth]
 \path (0:0) coordinate (O); \fill (O) circle (2.5pt);
 \foreach \x [count=\xi] in {A0,A1,A2,A3}{
      \path (-60+60*\xi:2) coordinate (\x);
      \fill (\x) circle (2.3pt);}
 \draw [thick] (A1)--node[midway,left=-2pt]{6}(O)--node[midway,below=-1pt]{4}(A0)
                    --node[midway,right=-2pt]{3}(A1)--node[midway,above=-2pt]{1}(A2)
                    --node[midway,left=-2pt]{2}(A3)--node[midway,below=-1pt]{5}(O)
                    --node[midway,right=-2pt]{7}(A2);
\end{scope}

\path (6,0.6)node[]{$\begin{pmatrix*}[r] &  &  & 1 & 1 & 1 & 1\\-1 & -1 &  &  &  &  & 1\\1 &  & -1 &  &  & 1\\ & 1 &  &  & 1\\ &  & 1 & 1\end{pmatrix*}_{5\times7}$};

\end{tikzpicture}
\par\end{center}\caption{\label{fig:Counter-to-degen lem}A rank-4 connected graphic matroid
with two distinct degenerate rank-1 flats and its regular realization.}
\end{figure}
\end{rem}

\subsection{\label{subsec:minimal-model}Minimal model program for line arrangements}

Suppose that $\scalebox{0.95}{\ensuremath{M}}$ is a loopless matroid
of rank $\scalebox{0.95}{\ensuremath{3}}$. As indicated in Subsection
\ref{subsec:MMP}, we can obtain the puzzle-piece $\scalebox{0.95}{\ensuremath{\mathrm{PZ}_{M}}}$
from the line arrangement $\scalebox{0.95}{\ensuremath{\mathrm{HA}_{M}}}$
applying the MMP. It is worth noting that we do not have to blow up
everything if we only want a lattice isomorphic to the intersection
lattice $\scalebox{0.95}{\ensuremath{\mathcal{F}\left(M\right)}}$
of $\scalebox{0.95}{\ensuremath{\mathrm{PZ}_{M}}}$: If an unstable
line exists, which is unique according to Lemma \ref{lem:degen-rank1-flat},
it suffices to make it collapsible and collapse the collapsible line.
If every line is stable, we collapsibilize the points $\scalebox{0.95}{\ensuremath{M/F}}$
for connected rank-$\scalebox{0.95}{\ensuremath{2}}$ flats $\scalebox{0.95}{\ensuremath{F}}$
and collapse them.
\begin{example}
\label{exa:minimal-models}Let $\scalebox{0.95}{\ensuremath{M}}$
be a rank-$\scalebox{0.95}{\ensuremath{3}}$ loopless matroid on $\scalebox{0.95}{\ensuremath{S}}$.
\begin{enumerate}
\item For a partition $\scalebox{0.95}{\ensuremath{S=\bigcup_{i\in\left[3\right]}S_{i}}}$,
let $\scalebox{0.95}{\ensuremath{M=U_{S_{1}}^{1}\oplus U_{S_{2}}^{1}\oplus U_{S_{3}}^{1}}}$.
Then, $\scalebox{0.95}{\ensuremath{\dim M=0}}$, and $\scalebox{0.95}{\ensuremath{\mathrm{HA}_{M}}}$
has only three line loci: $\scalebox{0.95}{\ensuremath{M/S_{1}=U_{S_{2}}^{1}\oplus U_{S_{3}}^{1}}}$,
$\scalebox{0.95}{\ensuremath{M/S_{2}=U_{S_{1}}^{1}\oplus U_{S_{3}}^{1}}}$,
and $\scalebox{0.95}{\ensuremath{M/S_{3}=U_{S_{1}}^{1}\oplus U_{S_{2}}^{1}}}$;
see Figure \ref{fig:HA-2-dim}. The puzzle-piece $\scalebox{0.95}{\ensuremath{\mathrm{PZ}_{M}}}$
is a point-piece, see Figure \ref{fig:mmp-1}. 
\begin{figure}[H]
\noindent \centering{}\noindent \begin{center}
\def\sizea{0.22}
\def\size{0.25}
\def\ratioa{0.901}
\def\ratior{0.851}
\def\ratiow{0.761}
\def\ratioe{0.881}

\begin{tikzpicture}[blowup line/.style={line width=2pt,gray},font=\scriptsize]

\begin{scope}[line cap=round,rotate=0,scale=\sizea,xshift=0cm,yshift=0cm,rounded corners]
 \foreach \x [count=\xi] in {B0,C0,A0}{
      \path (90+120*\xi:4) coordinate (\x);}
 \draw [blue,thick] (A0)++(60:1)--++(240:8.928);
 \draw [red,thick]  (A0)++(120:1)--++(-60:8.928);
 \draw [green,thick](B0)++(180:1)--++(0:8.928);

\path (0,6.5) node[]{$\mathrm{HA}_{U_{S_{1}}^{1}\oplus U_{S_{2}}^{1}\oplus U_{S_{3}}^{1}}$};

\end{scope}

\begin{scope}[xshift=2.5cm,yshift=0cm]
 \draw [->]     (1,0) -- (-1,0);
 \path (0,0.4) node[]{Blowups};
\end{scope}

\begin{scope}[line cap=round,rotate=0,xshift=5cm,yshift=0cm,scale=\sizea,rounded corners]
 \foreach \x [count=\xi] in {B0,C0,A0}{
      \path (90+120*\xi:4) coordinate (\x);}
 \draw [blue,thick] (A0)++(60:1)--++(240:8.928);
 \draw [red,thick]  (A0)++(120:1)--++(-60:8.928);
 \draw [green,thick](B0)++(180:1)--++(0:8.928);
\end{scope}

\begin{scope}[xshift=7.5cm,yshift=0cm]
 \draw [->]     (-1,0) -- (1,0);
 \path (0,0.4) node[]{Collapsing};
\end{scope}

\begin{scope}[line cap=round,rotate=0,xshift=10cm,yshift=0cm,scale=\sizea,rounded corners]

\fill (0,0) circle(8pt);

\path (0,6.5) node[]{$\mathrm{PZ}_{U_{S_{1}}^{1}\oplus U_{S_{2}}^{1}\oplus U_{S_{3}}^{1}}$};
\end{scope}

\begin{scope}[xshift=0cm,yshift=0cm]
 \draw [->]     (1.4,1.3) .. node[above]{MMP} controls (5,1.6) .. (8.6,1.3);
\end{scope}

\end{tikzpicture}
\par\end{center}\caption{\label{fig:mmp-1}A minimal model example I.}
\end{figure}
\item For a partition $\scalebox{0.95}{\ensuremath{S=S_{1}\cup S_{2}}}$
with $\scalebox{0.95}{\ensuremath{\left|S_{2}\right|\ge3}}$, let
$\scalebox{0.95}{\ensuremath{M=U_{S_{1}}^{1}\oplus U_{S_{2}}^{2}}}$.
Then, $\scalebox{0.95}{\ensuremath{\dim M=1}}$, and $\scalebox{0.95}{\ensuremath{\mathrm{PZ}_{M}}}$
is a line-piece, see Figure \ref{fig:mmp-2}.
\begin{figure}[H]
\noindent \centering{}\noindent \begin{center}
\def\sizea{0.22}
\def\size{0.25}
\def\ratioa{0.901}
\def\ratior{0.851}
\def\ratiow{0.761}
\def\ratioe{0.881}


\par\end{center}\caption{\label{fig:mmp-2}A minimal model example II.}
\end{figure}
\item \label{enu:MMP3}Let $\scalebox{0.95}{\ensuremath{M=\delta_{F}^{2}\left(U_{S}^{3}\right)}}$
be the matroid of Example \ref{exa:degenerations}\eqref{enu:degen-1},
then $\scalebox{0.95}{\ensuremath{\dim\mathrm{PZ}_{M}=2}}$, see Figure
\ref{fig:mmp-3}.
\begin{figure}[H]
\noindent \begin{centering}
\noindent \begin{center}
\def\sizea{0.22}
\def\size{0.25}
\def\ratioa{0.901}
\def\ratior{0.851}
\def\ratiow{0.761}
\def\ratioe{0.881}

%
\par\end{center}\vspace{-0.2cm}
\par\end{centering}
\noindent \centering{}\caption{\label{fig:mmp-3}A minimal model example III.}
\end{figure}
\item \label{enu:P1timesP1}Let $\scalebox{0.95}{\ensuremath{M=\delta_{F,L}^{2,2}\left(U_{S}^{3}\right)}}$
be the matroid of Example \ref{exa:degenerations}\eqref{enu:degen-2}.
The line $\scalebox{0.95}{\ensuremath{M/(F\cap L)}}$ has only two
intersection points and is unstable; so is the expression $\scalebox{0.95}{\ensuremath{M\left(F\cap L\right)}}$,
while the point-piece $\scalebox{0.95}{\ensuremath{\phi\left(M\left(F\cap L\right)\right)}}$
is stable, see Figure \ref{fig:mmp-4}. 
\begin{figure}[H]
\noindent \begin{centering}
\noindent \begin{center}
\def\sizea{0.22}
\def\size{0.25}
\def\ratioa{0.901}
\def\ratior{0.851}
\def\ratiow{0.761}
\def\ratioe{0.881}


\par\end{center}\vspace{-0.2cm}
\par\end{centering}
\noindent \centering{}\caption{\label{fig:mmp-4}A minimal model example IV.}
\end{figure}
\end{enumerate}
\end{example}

\subsection{\label{subsec:Grassmann}Grassmann geometry}

Fix an algebraically closed field of arbitrary characteristic. For
the Grassmannian $\scalebox{0.95}{\ensuremath{\mathrm{G}\left(k,n\right)}}$,
the natural action of the algebraic torus $\scalebox{0.95}{\ensuremath{H=\mathbb{G}_{m}^{n}}}$
on $\scalebox{0.95}{\ensuremath{\mathbb{P}^{n-1}}}$ extends to a
torus action on it. The weights of the Chow form of $\scalebox{0.95}{\ensuremath{X\in\mathrm{G}\left(k,n\right)}}$
are identified with the indicator vectors of the $k$-element subsets
$\scalebox{0.95}{\ensuremath{B}}$ of $\left[n\right]$ with \emph{nonzero}
Plücker coordinate $\scalebox{0.95}{\ensuremath{p_{B}\left(X\right)}}$.
These $k$-element subsets form the base collection of a matroid,
say $\scalebox{0.95}{\ensuremath{M}}$. If $\scalebox{0.95}{\ensuremath{X}}$
is not contained in the coordinate hyperplanes of $\scalebox{0.95}{\ensuremath{\mathbb{P}^{n-1}}}$,
then the matroid $\scalebox{0.95}{\ensuremath{M}}$ is loopless. Moreover,
it induces $n$ hyperplanes in $\scalebox{0.95}{\ensuremath{\mathbb{P}^{k-1}}}$
determined by $n$ linear equations $\scalebox{0.95}{\ensuremath{\left\langle \mathbf{v}_{i},\underline{\hspace{0.7em}}\,\right\rangle =0}}$
for $\scalebox{0.95}{\ensuremath{i\in\left[n\right]}}$ where 
\[
X=\left(\begin{array}{c|c|c}
 & \\
\mathbf{v}_{1} & \cdots & \mathbf{v}_{n}\\
 & \\
\end{array}\right)_{k\times n}
\]
 which are identified with the $n$ intersections of those $n$ coordinate
hyperplanes of $\scalebox{0.95}{\ensuremath{\mathbb{P}^{n-1}}}$ and
$\scalebox{0.95}{\ensuremath{X\subset\mathbb{P}^{n-1}}}$. Let $\scalebox{0.95}{\ensuremath{\mathcal{H}}}$
denote this hyperplane arrangement in $\scalebox{0.95}{\ensuremath{\mathbb{P}^{k-1}}}$,
then it is a realization (over the given field) of the abstract hyperplane
arrangement $\scalebox{0.95}{\ensuremath{\mathrm{HA}_{M}}}$. Here,
we know that $\scalebox{0.95}{\ensuremath{M}}$ is connected if $\scalebox{0.95}{\ensuremath{\mathcal{H}}}$
has $\scalebox{0.95}{\ensuremath{k+1}}$ hyperplanes in general position,
but not the other way around, cf. Lemma \ref{lem:general position}
and Remark \ref{rem:general position}. Further:
\begin{prop}[{Gelfand-Serganova, \cite[Theorems 1.1 and 2.2]{GS87}}]
\label{prop:GS-thm-2} Let $\scalebox{0.95}{\ensuremath{H\cdot\left[X\right]}}$
be the orbit of $\scalebox{0.95}{\ensuremath{\left[X\right]\in\mathrm{G}\left(k,n\right)}}$
under the action of $\scalebox{0.95}{\ensuremath{H}}$. Then, there
is a one-to-one correspondence from the collection of the  orbits
$\scalebox{0.95}{\ensuremath{H\cdot\left[Y\right]}}$ in $\scalebox{0.95}{\ensuremath{\overline{H\cdot\left[X\right]}}}$
with $\scalebox{0.95}{\ensuremath{Y}}$ not contained in the coordinate
hyperplanes of $\scalebox{0.95}{\ensuremath{\mathbb{P}^{n-1}}}$ to
the intersection lattice of the puzzle-piece $\scalebox{0.95}{\ensuremath{\mathrm{PZ}_{M}}}$.
This correspondence preserves the codimension.
\end{prop}

Recall that $\scalebox{0.95}{\ensuremath{\mathrm{MA}_{\mathrm{P}}}}$
for a base polytope $\scalebox{0.95}{\ensuremath{\mathrm{P}}}$ denotes
its matroid. One can regard $\scalebox{0.95}{\ensuremath{\mathrm{MA}}}$
as a map defined on the collection of all base polytopes and $\scalebox{0.95}{\ensuremath{\mathrm{PZ}}}$
as a map taking the puzzle-piece. Then, $\scalebox{0.95}{\ensuremath{\mathrm{PZ}\circ\mathrm{MA}\circ\mu}}$
is the above correspondence which maps orbits to puzzle-pieces, where
$\mu$ is the \emph{moment mapping}: 
\[
\scalebox{0.95}{\ensuremath{{\displaystyle \mu\left(Y\right)=\frac{\sum_{B\in\binom{\left[n\right]}{k}}\left|p_{B}\left(Y\right)\right|^{2}\cdot1^{B}}{\sum_{B\in\binom{\left[n\right]}{k}}\left|p_{B}\left(Y\right)\right|^{2}}}}}.
\]
Here, $\scalebox{0.95}{\ensuremath{\binom{\left[n\right]}{k}}}$ denotes
the collection of all $\scalebox{0.95}{\ensuremath{k}}$-element subsets
of $\scalebox{0.95}{\ensuremath{\left[n\right]}}$.\vspace{2pt}

For moduli theorists, we remark that $\scalebox{0.95}{\ensuremath{\mathcal{H}}}$
has the \emph{log canonical model}, which is a realization of $\scalebox{0.95}{\ensuremath{\mathrm{PZ}_{M}}}$,
a normal projective toric variety. We will describe a specific method
for obtaining $\scalebox{0.95}{\ensuremath{\mathrm{PZ}_{M}}}$ from
$\scalebox{0.95}{\ensuremath{\mathrm{HA}_{M}}}$. Using this method,
every intermediate matroidal lattice that arises in between is coatomistic,
and one can obtain the log canonical model by applying corresponding
algebro-geometric operations followed by taking the coatoms:
\begin{enumerate}
\item \label{enu:degen}Collapsibilize all degenerate hyperplane expressions,
that is, $\scalebox{0.95}{\ensuremath{M/F}}$ for all degenerate rank-1
flats $\scalebox{0.95}{\ensuremath{F}}$, and collapse them. The lattice
obtained after the process is coatomistic; no unstable member is a
coatom.
\item \label{enu:nondeg}Collapsibilize the  subspace expressions $\scalebox{0.95}{\ensuremath{M/F}}$
for all non-degenerate flats $\scalebox{0.95}{\ensuremath{F}}$ and
collapse them. Then, the coatoms of the resulting lattice $\scalebox{0.95}{\ensuremath{\mathcal{A}}}$
are precisely $\scalebox{0.95}{\ensuremath{\phi\left(M\left(F\right)\right)}}$
for all non-degenerate flats $\scalebox{0.95}{\ensuremath{F}}$; hence,
$\scalebox{0.95}{\ensuremath{\mathcal{A}=\mathcal{F}\left(M\right)}}$.
\end{enumerate}
\begin{rem}[]
\begin{enumerate}
\item Lemma \ref{lem:degen-rank1-flat} and Example \ref{exa:minimal-models}\eqref{enu:P1timesP1}
show that the log canonical model of a line arrangement is either
$\scalebox{0.95}{\ensuremath{\mathbb{P}^{1}\times\mathbb{P}^{1}}}$
or $\scalebox{0.95}{\ensuremath{\mathrm{Bl}_{\mathrm{pts}}\mathbb{P}^{2}}}$.
\item Our stability notion in Definitions \ref{def:stability} and \ref{def:w-stability}
generalizes the GIT stability of Keel and Tevelev, \cite[Definition 1.25]{Chow}.
\end{enumerate}
\end{rem}

\section{\label{sec:Matroid Tilings}Matroid Tiling}

We now consider extending \emph{matroid tilings}, which are generalized
matroid subdivisions. One of the major challenges of the tiling extension
is that it is a global process and can have a huge computational complexity.
We introduce a fundamental theorem that addresses the problem: it
tells us that if we extend a matroid tiling in a \textquotedblleft locally
convex\textquotedblright{} manner, the extension is a matroid subdivision,
while the process is a sequence of local operations. We investigate
gluing of base polytopes and provide practical tools for the extension.
We also study weighted tilings.

\subsection{\label{subsec:Tiling}Tiling and semitiling}

Let $\scalebox{0.95}{\ensuremath{\Sigma}}$ be a finite collection
of convex polytopes contained in the $\scalebox{0.95}{\ensuremath{\left(k,S\right)}}$-hypersimplex
$\scalebox{0.95}{\ensuremath{\Delta_{S}^{k}}}$. The \textbf{support}
of $\scalebox{0.95}{\ensuremath{\Sigma}}$ is the union of its members:
\[
\scalebox{0.95}{\ensuremath{\left|\Sigma\right|:=\bigcup_{P\in\Sigma}P}}.
\]
 The \textbf{dimension} of $\scalebox{0.95}{\ensuremath{\Sigma}}$
is defined as the dimension of its support:
\[
\scalebox{0.95}{\ensuremath{\dim\Sigma:=\dim\left|\Sigma\right|}}.
\]
 We say that $\scalebox{0.95}{\ensuremath{\Sigma}}$ is \textbf{equidimensional}
if all of its members have the same dimension and \textbf{full-dimensional}
if $\scalebox{0.95}{\ensuremath{\dim\Sigma=\dim\Delta=\left|S\right|-1}}$.
Henceforth we assume that $\scalebox{0.95}{\ensuremath{\Sigma}}$
is equidimensional and full-dimensional unless otherwise noted.

A nonempty face of a member of $\scalebox{0.95}{\ensuremath{\Sigma}}$
is said to be a \textbf{cell} of $\scalebox{0.95}{\ensuremath{\Sigma}}$.
The empty set $\emptyset$ is regarded as a cell of $\scalebox{0.95}{\ensuremath{\Sigma}}$
with dimension $\scalebox{0.95}{\ensuremath{\dim\emptyset:=-1}}$.
If $\scalebox{0.95}{\ensuremath{Q}}$ is a nonempty common face of
all members of $\scalebox{0.95}{\ensuremath{\Sigma}}$, then we say
that $\scalebox{0.95}{\ensuremath{Q}}$ is a \textbf{common cell}
of $\scalebox{0.95}{\ensuremath{\Sigma}}$. The \textbf{codimension}
of a nonempty cell $\scalebox{0.95}{\ensuremath{Q}}$ of $\scalebox{0.95}{\ensuremath{\Sigma}}$
is:
\[
\scalebox{0.95}{\ensuremath{\mathrm{codim}_{\Sigma}\left.Q\right.:=\dim\Sigma-\dim Q}}.
\]
The codimension of the empty cell $\emptyset$ is $\scalebox{0.95}{\ensuremath{\mathrm{codim}_{\Sigma}\left.\emptyset\right.:=\dim\Sigma+1}}$.
\begin{note}
When mentioning cells, we identify $\scalebox{0.95}{\ensuremath{\Sigma}}$
with the collection of all of its cells.
\end{note}

We say that $\scalebox{0.95}{\ensuremath{\Sigma}}$ is \textbf{connected
in codimension $c$} if, for any two polytopes $\scalebox{0.95}{\ensuremath{P,P'\in\Sigma}}$,
there is a sequence of distinct polytopes $\scalebox{0.95}{\ensuremath{P_{1},\dots,P_{l}\in\Sigma}}$
for some $\scalebox{0.95}{\ensuremath{l\ge1}}$ with $\scalebox{0.95}{\ensuremath{P_{0}=P}}$
and $\scalebox{0.95}{\ensuremath{P_{l}=P'}}$ such that $\scalebox{0.95}{\ensuremath{P_{i-1}\cap P_{i}}}$
for each $\scalebox{0.95}{\ensuremath{i\in\left[l\right]}}$ is a
common face of $\scalebox{0.95}{\ensuremath{P_{i-1}}}$ and $\scalebox{0.95}{\ensuremath{P_{i}}}$
with $\scalebox{0.95}{\ensuremath{1\le\mathrm{codim}\left(P_{i-1}\cap P_{i}\right)\le c}}$.
\begin{defn}
A $\scalebox{0.95}{\ensuremath{\left(k,S\right)}}$\textbf{-tiling}
or simply a \textbf{tiling} is a face-fitting collection of convex
polytopes in $\scalebox{0.95}{\ensuremath{\Delta_{S}^{k}}}$ that
is connected in codimension $1$. For a tiling $\Sigma$, we say that
it is a \textbf{subdivision} of $\scalebox{0.95}{\ensuremath{\left|\Sigma\right|}}$.
\end{defn}

By definition, a single convex polytope is a tiling. The collection
of all cells of a tiling is a polyhedral complex. We call a tiling
with convex support a \textbf{convex tiling}. Some authors mean by
a convex subdivision a polyhedral subdivision induced by a convex
or concave function, which we call \textbf{coherent}. We call a tiling
with a common cell $\scalebox{0.95}{\ensuremath{Q}}$ of codimension
$\ge1$ a \textbf{tiling at $\scalebox{0.95}{\ensuremath{Q}}$}.

Let $\scalebox{0.95}{\ensuremath{\Sigma}}$ be a tiling. A codimension-$1$
cell of $\scalebox{0.95}{\ensuremath{\Sigma}}$ is said to be a \textbf{facet}
of $\scalebox{0.95}{\ensuremath{\Sigma}}$ if it is \emph{not} a common
face of two members of $\scalebox{0.95}{\ensuremath{\Sigma}}$. A
facet of $\scalebox{0.95}{\ensuremath{\Sigma}}$ is either a facet
of $\scalebox{0.95}{\ensuremath{\left|\Sigma\right|}}$ or contained
in a facet of $\scalebox{0.95}{\ensuremath{\left|\Sigma\right|}}$.
The \textbf{boundary }$\scalebox{0.95}{\ensuremath{\partial\Sigma}}$
of $\scalebox{0.95}{\ensuremath{\Sigma}}$ is the collection of all
facets of $\scalebox{0.95}{\ensuremath{\Sigma}}$. We call a face
of a facet of $\scalebox{0.95}{\ensuremath{\Sigma}}$ a \textbf{boundary
cell}.

Let $\scalebox{0.95}{\ensuremath{\Sigma_{0}}}$ be a subcollection
of $\scalebox{0.95}{\ensuremath{\Sigma}}$ with a common cell $\scalebox{0.95}{\ensuremath{Q}}$
of codimension $\scalebox{0.95}{\ensuremath{\ge1}}$. Then, $\scalebox{0.95}{\ensuremath{\Sigma_{0}}}$
is a tiling at $\scalebox{0.95}{\ensuremath{Q}}$,  and we  call $\scalebox{0.95}{\ensuremath{\Sigma_{0}}}$
a \textbf{subtiling at $\scalebox{0.95}{\ensuremath{Q}}$}. We say
that $\scalebox{0.95}{\ensuremath{\Sigma_{0}}}$ is \textbf{maximal
at $\scalebox{0.95}{\ensuremath{Q}}$} if there is no larger subtiling
of $\scalebox{0.95}{\ensuremath{\Sigma}}$ at $\scalebox{0.95}{\ensuremath{Q}}$.
An (inclusionwise) maximal subtiling of $\scalebox{0.95}{\ensuremath{\Sigma}}$
at $\scalebox{0.95}{\ensuremath{Q}}$ is unique and is denoted by
$\scalebox{0.95}{\ensuremath{\Sigma_{Q}}}$.
\begin{notation}
Let $\scalebox{0.95}{\ensuremath{R}}$ and $\scalebox{0.95}{\ensuremath{Q}}$
be polyhedra. We write $\scalebox{0.95}{\ensuremath{Q<R}}$ if $\scalebox{0.95}{\ensuremath{Q}}$
is a proper face of $\scalebox{0.95}{\ensuremath{R}}$ and $\scalebox{0.95}{\ensuremath{Q\lessdot R}}$
if $\scalebox{0.95}{\ensuremath{Q<R}}$ and $\scalebox{0.95}{\ensuremath{\dim R=\dim Q+1}}$.
\end{notation}

For a tiling $\scalebox{0.95}{\ensuremath{\Sigma}}$ with a codimension-$\scalebox{0.95}{\ensuremath{2}}$
cell $\scalebox{0.95}{\ensuremath{Q}}$, let $\scalebox{0.95}{\ensuremath{R\nsubseteq\partial\Delta}}$
be a facet of $\scalebox{0.95}{\ensuremath{\Sigma}}$ and $\scalebox{0.95}{\ensuremath{\Sigma_{Q}}}$
with $\scalebox{0.95}{\ensuremath{Q\lessdot R}}$, then $\scalebox{0.95}{\ensuremath{R\lessdot P}}$
for some $\scalebox{0.95}{\ensuremath{P\in\Sigma_{Q}}}$. Suppose
that $\scalebox{0.95}{\ensuremath{P'}}$ is a full-dimensional polytope
with $\scalebox{0.95}{\ensuremath{R=P\cap P'}}$ and $\scalebox{0.95}{\ensuremath{\Sigma_{Q}\cup\left\{ P'\right\} }}$
being a tiling. If $\scalebox{0.95}{\ensuremath{\Sigma_{Q'}\cup\left\{ P'\right\} }}$
for any $\scalebox{0.95}{\ensuremath{\Sigma_{Q'}}}$ with $\scalebox{0.95}{\ensuremath{Q'\lessdot R}}$
is a tiling, we call $\scalebox{0.95}{\ensuremath{\Sigma'=\Sigma\cup\left\{ P'\right\} }}$
or obtaining it a \textbf{single-member extension} of $\scalebox{0.95}{\ensuremath{\Sigma}}$,
a tiling or not. Further, a maximal subtiling of $\scalebox{0.95}{\ensuremath{\Sigma'}}$
at another codimension-$\scalebox{0.95}{\ensuremath{2}}$ cell $\scalebox{0.95}{\ensuremath{Q'}}$
can also be defined so we can perform single-member extension recursively.
We call $\scalebox{0.95}{\ensuremath{Q'}}$ the \textbf{center} of
the maximal subtiling of $\scalebox{0.95}{\ensuremath{\Sigma'}}$.
Note that a single-member extension is a local operation.
\begin{defn}
A $\scalebox{0.95}{\ensuremath{\left(k,S\right)}}$\textbf{-semitiling
}for $\scalebox{0.95}{\ensuremath{k\ge2}}$ or simply a \textbf{semitiling}
is a collection of polytopes that is obtained from a $\scalebox{0.95}{\ensuremath{\left(k,S\right)}}$-tiling
by recursive single-member extensions.
\end{defn}

A tiling is a semitiling, and a semitiling is connected in codimension
$\scalebox{0.95}{\ensuremath{1}}$.

A maximal subtiling of semitiling $\scalebox{0.95}{\ensuremath{\Sigma}}$
at codimension-$\scalebox{0.95}{\ensuremath{2}}$ cell $\scalebox{0.95}{\ensuremath{Q}}$
is defined, but we do not know whether it exists uniquely and do not
use the notation $\scalebox{0.95}{\ensuremath{\Sigma_{Q}}}$.

For two semitilings $\scalebox{0.95}{\ensuremath{\Sigma}}$ and $\scalebox{0.95}{\ensuremath{\Sigma'}}$
with $\scalebox{0.95}{\ensuremath{\Sigma\subseteq\Sigma'}}$, we call
$\scalebox{0.95}{\ensuremath{\Sigma'}}$ an \textbf{extension} of
$\scalebox{0.95}{\ensuremath{\Sigma}}$.
\begin{defn}
A \textbf{matroid semitiling} is a semitiling whose members are base
polytopes. A matroid tiling whose support is a base polytope is said
to be a\textbf{ matroid subdivision} of the base polytope.
\end{defn}

The matroid semitiling is a well-defined notion due to Proposition
\ref{prop:base-poly-cpx}. Henceforth we assume that a semitiling
is a matroid semitiling.

\subsection{Quotient figure}

Let $\scalebox{0.95}{\ensuremath{P=\mathrm{BP}_{M}}}$ be a full-dimensional
$\scalebox{0.95}{\ensuremath{\left(k,S\right)}}$-polytope with $\scalebox{0.95}{\ensuremath{k\ge2}}$.
Then, $\scalebox{0.95}{\ensuremath{\mathrm{MA}_{P}=M}}$ is a connected
matroid. If $\scalebox{0.95}{\ensuremath{Q}}$ is a codimension-$\scalebox{0.95}{\ensuremath{\left(c-1\right)}}$
face with $\scalebox{0.95}{\ensuremath{2\le c\le k}}$, then $\scalebox{0.95}{\ensuremath{\kappa\left(\mathrm{MA}_{Q}\right)=c}}$.
Let $\scalebox{0.95}{\ensuremath{M_{i}}}$, $\scalebox{0.95}{\ensuremath{i\in\left[c\right]}}$,
be the connected components of $\scalebox{0.95}{\ensuremath{\mathrm{MA}_{Q}}}$
with ground sets $\scalebox{0.95}{\ensuremath{S_{i}=E\left(M_{i}\right)}}$.
Then, $\scalebox{0.95}{\ensuremath{r\left(M_{i}\right)=k_{i}}}$ for
$\scalebox{0.95}{\ensuremath{0\le k_{i}\le k}}$ with $\scalebox{0.95}{\ensuremath{\sum_{i\in\left[c\right]}k_{i}=k}}$,
and $\scalebox{0.95}{\ensuremath{\mathrm{MA}_{Q}=\bigoplus_{i\in\left[c\right]}M_{i}}}$.
Let $\scalebox{0.95}{\ensuremath{\pi_{Q}:\mathbb{R}^{S}\rightarrow\mathbb{R}^{c}}}$
be a linear map defined by
\[
\scalebox{0.95}{\ensuremath{{\displaystyle \sum_{s\in S}x_{s}\mathbf{e}_{s}\mapsto\sum_{i\in\left[c\right]}\left(\sum_{s\in S_{i}}x_{s}\right)\mathbf{e}_{i}}}}
\]
 where $\scalebox{0.95}{\ensuremath{\mathbf{e}_{s}\in\mathbb{R}^{S}}}$
and $\scalebox{0.95}{\ensuremath{\mathbf{e}_{i}\in\mathbb{R}^{c}}}$
are standard basis vectors. Then:
\begin{itemize}
\item The hyperplane $\scalebox{0.95}{\ensuremath{\left\{ x\left(S_{i}\right)=k_{i}\right\} }}$
of $\scalebox{0.95}{\ensuremath{\mathbb{R}^{S}}}$ maps to the hyperplane
$\scalebox{0.95}{\ensuremath{\left\{ y_{i}=k_{i}\right\} }}$ of $\scalebox{0.95}{\ensuremath{\mathbb{R}^{c}}}$.
\item The hyperplane $\scalebox{0.95}{\ensuremath{\left\{ x\left(S\right)=k\right\} }}$
maps to $\scalebox{0.95}{\ensuremath{\left\{ y\left(\left[c\right]\right)=k\right\} }}$.
\item $\scalebox{0.95}{\ensuremath{Q}}$ maps to a point: $\scalebox{0.95}{\ensuremath{\pi_{Q}\left(Q\right)=\left\{ \left(k_{1},\dots,k_{c}\right)\right\} }}$.
\item If $\scalebox{0.95}{\ensuremath{\left|S_{i}\right|\ge k}}$ for all
$\scalebox{0.95}{\ensuremath{i\in\left[c\right]}}$, then $\scalebox{0.95}{\ensuremath{\pi_{Q}\left(\Delta_{S}^{k}\right)}}$
is a regular $\scalebox{0.95}{\ensuremath{\left(c-1\right)}}$-simplex.
\end{itemize}
For a point $\scalebox{0.95}{\ensuremath{x\in Q}}$, take an open
ball $\scalebox{0.95}{\ensuremath{B_{\epsilon}\left(x\right)}}$ centered
at $\scalebox{0.95}{\ensuremath{x}}$ with radius $\scalebox{0.95}{\ensuremath{0<\epsilon\ll1}}$
in $\scalebox{0.95}{\ensuremath{\Delta}}$. The image of this ball
under $\scalebox{0.95}{\ensuremath{\pi_{Q}}}$, which does not depend
on the choice of $\scalebox{0.95}{\ensuremath{x\in Q}}$, is referred
to as the \textbf{quotient figure of $\scalebox{0.95}{\ensuremath{P}}$
at $\scalebox{0.95}{\ensuremath{Q}}$}.

For a tiling $\scalebox{0.95}{\ensuremath{\Sigma}}$ at a cell $\scalebox{0.95}{\ensuremath{Q}}$,
the \textbf{quotient figure of $\scalebox{0.95}{\ensuremath{\Sigma}}$
at }$\scalebox{0.95}{\ensuremath{Q}}$ is defined as:
\[
\scalebox{0.95}{\ensuremath{\left\{ \text{the quotient figure of }P\text{ at }Q:P\in\Sigma\right\} }}.
\]
We use the same radius for the quotient figure of each $\scalebox{0.95}{\ensuremath{P\in\Sigma}}$
and refer to the union of these figures as the quotient figure of
$\scalebox{0.95}{\ensuremath{\Sigma}}$, for example, when discussing
convexity. We say that a tiling $\scalebox{0.95}{\ensuremath{\Sigma}}$
at $\scalebox{0.95}{\ensuremath{Q}}$ is \textbf{locally convex at
}$\scalebox{0.95}{\ensuremath{Q}}$ if the quotient figure of $\scalebox{0.95}{\ensuremath{\Sigma}}$
at $\scalebox{0.95}{\ensuremath{Q}}$ is convex. Then, $\scalebox{0.95}{\ensuremath{\Sigma}}$
becomes automatically locally convex at any codimension-$1$ cell.

\subsection{Extension of semitiling}

Let $\scalebox{0.95}{\ensuremath{\Sigma}}$ be a semitiling of rank
$\scalebox{0.95}{\ensuremath{k\ge3}}$. We say that $\scalebox{0.95}{\ensuremath{\Sigma}}$
is \textbf{locally convex }if it is locally convex at any codimension-$\scalebox{0.95}{\ensuremath{2}}$
cell. It is generally difficult to determine whether a semitiling
is a tiling. However, the following theorem together with Lemma \ref{lem:locally-convex}
provides a specific extension method to guarantee that an extension
obtained in this way is a matroid tiling, thereby making the warning
of Mnëv's universality theorem obsolete for the extension problem.
\begin{thm}
\label{thm:locally-convex}A locally convex semitiling is a convex
tiling.
\end{thm}

\begin{proof}
Let $\scalebox{0.95}{\ensuremath{\Sigma}}$ be a locally convex semitiling.
To show that $\scalebox{0.95}{\ensuremath{\left|\Sigma\right|}}$
is convex, consider two distinct points $\scalebox{0.95}{\ensuremath{x}}$
and $\scalebox{0.95}{\ensuremath{y}}$ in $\scalebox{0.95}{\ensuremath{\left|\Sigma\right|}}$.
If there exists a polytope in $\scalebox{0.95}{\ensuremath{\Sigma}}$
that contains both $\scalebox{0.95}{\ensuremath{x}}$ and $\scalebox{0.95}{\ensuremath{y}}$,
then it also contains the line segment $\scalebox{0.95}{\ensuremath{\overline{xy}}}$,
and so does $\scalebox{0.95}{\ensuremath{\left|\Sigma\right|}}$.
Suppose that no polytope in $\scalebox{0.95}{\ensuremath{\Sigma}}$
contains both $\scalebox{0.95}{\ensuremath{x}}$ and $\scalebox{0.95}{\ensuremath{y}}$.
Choose a path $\scalebox{0.95}{\ensuremath{\alpha:\left[0,1\right]\rightarrow\left|\Sigma\right|}}$
from $\scalebox{0.95}{\ensuremath{\alpha\left(0\right)=x}}$ to $\scalebox{0.95}{\ensuremath{\alpha\left(1\right)=y}}$,
of the smallest length. Then, every point $\scalebox{0.95}{\ensuremath{\alpha\left(t\right)}}$
has a relatively open, convex neighborhood that is contained in the
support of a subtiling of $\scalebox{0.95}{\ensuremath{\Sigma}}$
at a codimension-2 cell. We claim that $\scalebox{0.95}{\ensuremath{\alpha}}$
is a line segment. Suppose not. Then, there exists a $\scalebox{0.95}{\ensuremath{t\in\left(0,1\right)}}$
such that, for any relatively open, convex neighborhood $\scalebox{0.95}{\ensuremath{N_{t}}}$
at $\scalebox{0.95}{\ensuremath{\alpha\left(t\right)}}$, the arc
$\scalebox{0.95}{\ensuremath{\left(N_{t}\cup\partial N_{t}\right)\cap\alpha\left(\left[0,1\right]\right)}}$
is not a line segment. We can choose $\scalebox{0.95}{\ensuremath{N_{t}}}$
to be small enough so that $\scalebox{0.95}{\ensuremath{\partial N_{t}\cap\alpha\left(\left[0,1\right]\right)}}$
is a set with two points, say $\scalebox{0.95}{\ensuremath{x'}}$
and $\scalebox{0.95}{\ensuremath{y'}}$. The line segment $\scalebox{0.95}{\ensuremath{\overline{x'y'}}}$
is contained in $\scalebox{0.95}{\ensuremath{N_{t}\cup\partial N_{t}}}$,
which contradicts the length minimality of $\alpha$. Hence, $\alpha$
must be a line segment, and this proves the convexity of $\scalebox{0.95}{\ensuremath{\left|\Sigma\right|}}$.

To show that $\scalebox{0.95}{\ensuremath{\Sigma}}$ is a tiling,
suppose that $\scalebox{0.95}{\ensuremath{P}}$ and $\scalebox{0.95}{\ensuremath{P'}}$
are two elements of $\scalebox{0.95}{\ensuremath{\Sigma}}$ that are
not face-fitting. Then, by definition, they have a nonempty intersection,
and for any point $\scalebox{0.95}{\ensuremath{x}}$ in this intersection,
there exists a nontrivial closed path $\scalebox{0.95}{\ensuremath{\alpha:\left[0,1\right]\rightarrow\left|\Sigma\right|}}$
with $\scalebox{0.95}{\ensuremath{\alpha\left(0\right)=\alpha\left(1\right)=x}}$,
of the shortest length. This is a contradiction, as $\scalebox{0.95}{\ensuremath{\left|\Sigma\right|}}$
is convex, and such a path must be a point. Hence, $\scalebox{0.95}{\ensuremath{P}}$
and $\scalebox{0.95}{\ensuremath{P'}}$ must be face-fitting. This
proves that $\scalebox{0.95}{\ensuremath{\Sigma}}$ is a tiling.
\end{proof}
Therefore, a semitiling that has a convex extension is itself a tiling.
Moreover, to determine the convexity of a semitiling, it suffices
to verify the local convexity of the semitiling only at relevant boundary
cells. This is because the quotient figure of a semitiling at an irrelevant
codimension-$\scalebox{0.95}{\ensuremath{2}}$ cell is already convex.
This observation immediately leads to the following lemma.
\begin{lem}
\label{lem:loc-conv-relev}A semitiling without any relevant boundary
codimension-$\scalebox{0.95}{\ensuremath{2}}$ cell is convex.
\end{lem}

Combining Theorem \ref{thm:locally-convex} and the following lemma
produces Corollary \ref{cor:locally-convex}.
\begin{lem}
\label{lem:locally-convex}A convex tiling is a matroid subdivision.
More specifically, if $\scalebox{0.95}{\ensuremath{\Sigma}}$ is a
convex tiling, then $\scalebox{0.95}{\ensuremath{M=\bigcup_{P\in\Sigma}\mathrm{MA}_{P}}}$
is a matroid with $\scalebox{0.95}{\ensuremath{\left|\Sigma\right|=\mathrm{BP}_{M}}}$.
\end{lem}

\begin{proof}
Let $\scalebox{0.95}{\ensuremath{\Sigma=\left\{ \mathrm{BP}_{M_{1}},\dots,\mathrm{BP}_{M_{l}}\right\} }}$.
Because no vertex of a 0/1-polytope is a convex combination of the
other vertices, no edge of the 0/1-polytope $\scalebox{0.95}{\ensuremath{\left|\Sigma\right|=\mathrm{P}_{\bigcup_{i\in\left[l\right]}M_{i}}}}$
is a nontrivial union of edges of $\scalebox{0.95}{\ensuremath{\mathrm{BP}_{M_{1}},\dots,\mathrm{BP}_{M_{l}}}}$.
So, every edge of $\scalebox{0.95}{\ensuremath{\left|\Sigma\right|}}$
is an edge of some $\scalebox{0.95}{\ensuremath{\mathrm{BP}_{M_{i}}}}$,
and $\scalebox{0.95}{\ensuremath{\left|\Sigma\right|}}$ has the edge
length property.  Thus, $\scalebox{0.95}{\ensuremath{\bigcup_{i\in\left[l\right]}M_{i}}}$
is a matroid with $\scalebox{0.95}{\ensuremath{\left|\Sigma\right|}}$
being its base polytope, and $\scalebox{0.95}{\ensuremath{\Sigma}}$
is a matroid subdivision.
\end{proof}
\begin{cor}
\label{cor:locally-convex}A semitiling is convex if and only if it
is a matroid subdivision.
\end{cor}

\begin{defn}
A\textbf{ }semitiling $\scalebox{0.95}{\ensuremath{\Sigma}}$ is said
to be \textbf{complete} if $\scalebox{0.95}{\ensuremath{\bigcup\left(\partial\Sigma\right)\subseteq\partial\Delta}}$.
\end{defn}

\begin{example}
A complete semitiling is a matroid subdivision of $\scalebox{0.95}{\ensuremath{\Delta}}$.
\end{example}

Now, we examine how base polytopes are glued together. It suffices
to check for two face-fitting base polytopes in $\scalebox{0.95}{\ensuremath{\Delta_{S}^{k}}}$,
say $\scalebox{0.95}{\ensuremath{\mathrm{BP}_{M}}}$ and $\scalebox{0.95}{\ensuremath{\mathrm{BP}_{N}}}$.
Formula \eqref{eq:defining-ineq} implies that there is a hyperplane
of the form $\scalebox{0.95}{\ensuremath{\left\{ x(A)=a\right\} }}$
for a subset $\scalebox{0.95}{\ensuremath{A\subset S}}$ and an integer
$\scalebox{0.95}{\ensuremath{a\in\left[k\right]}}$ such that $\scalebox{0.95}{\ensuremath{\mathrm{BP}_{M}\subseteq\left\{ x(A)\le a\right\} }}$
and $\scalebox{0.95}{\ensuremath{\mathrm{BP}_{N}\subseteq\left\{ x(S-A)\le k-a\right\} }}$,
i.e., the hyperplane separates $\scalebox{0.95}{\ensuremath{\mathrm{BP}_{M}}}$
and $\scalebox{0.95}{\ensuremath{\mathrm{BP}_{N}}}$. Assuming $\scalebox{0.95}{\ensuremath{\mathrm{codim}\left(\mathrm{BP}_{M}\cap\mathrm{BP}_{N}\right)=1}}$,
$\scalebox{0.95}{\ensuremath{\mathrm{BP}_{M}\cap\mathrm{BP}_{N}}}$
is their common facet, and there are non-degenerate flats $\scalebox{0.95}{\ensuremath{F}}$
and $\scalebox{0.95}{\ensuremath{L}}$ of $\scalebox{0.95}{\ensuremath{M}}$
and $\scalebox{0.95}{\ensuremath{N}}$, respectively, with $\scalebox{0.95}{\ensuremath{\mathrm{BP}_{M}\cap\mathrm{BP}_{N}=\mathrm{BP}_{M(F)}=\mathrm{BP}_{N(L)}}}$,
cf. Lemma \ref{lem:Faces of BP}. In the above formulas, we can choose
$\scalebox{0.95}{\ensuremath{A=F}}$ and $\scalebox{0.95}{\ensuremath{S-A=L}}$
with $\scalebox{0.95}{\ensuremath{a=r_{M}\left(F\right)}}$ and $\scalebox{0.95}{\ensuremath{k-a=r_{N}\left(L\right)}}$,
respectively (these choices are unique). In particular 
\begin{equation}
\begin{cases}
\scalebox{0.95}{\ensuremath{M|_{F}=N/L}},\\
\scalebox{0.95}{\ensuremath{M/F=N|_{L}}}.
\end{cases}\label{eq:face-fitting}
\end{equation}
 Conversely, suppose that \eqref{eq:face-fitting} holds for subsets
$\scalebox{0.95}{\ensuremath{F,L\subseteq S}}$, then
\[
\scalebox{0.95}{\ensuremath{\mathrm{BP}_{M}\cap\mathrm{BP}_{N}=\mathrm{BP}_{M(F)}=\mathrm{BP}_{N(L)}}}
\]
 which is a common face of $\scalebox{0.95}{\ensuremath{\mathrm{BP}_{M}}}$
and $\scalebox{0.95}{\ensuremath{\mathrm{BP}_{N}}}$. These two polytopes
are face-fitting through it. If one of the two subsets  $\scalebox{0.95}{\ensuremath{F}}$
and $\scalebox{0.95}{\ensuremath{L}}$ is non-degenerate, then so
is the other, and $\scalebox{0.95}{\ensuremath{\mathrm{codim}\left(\mathrm{BP}_{M}\cap\mathrm{BP}_{N}\right)=1}}$.

\subsection{\label{subsec:Rank-2-tilings}Semitilings of rank two }

Every $\scalebox{0.95}{\ensuremath{\left(2,S\right)}}$-semitiling
$\scalebox{0.95}{\ensuremath{\Sigma}}$ is a convex tiling by Lemma
\ref{lem:loc-conv-relev} since all its codimension-$\scalebox{0.95}{\ensuremath{2}}$
cells are irrelevant. If it is not complete, we can extend it to a
complete tiling as follows. Let $\scalebox{0.95}{\ensuremath{R}}$
be a relevant facet of it. Then, $\scalebox{0.95}{\ensuremath{R\lessdot\mathrm{BP}_{M}\in\Sigma}}$
for a rank-$\scalebox{0.95}{\ensuremath{2}}$ connected matroid $\scalebox{0.95}{\ensuremath{M}}$.
Moreover, $\scalebox{0.95}{\ensuremath{R=\mathrm{BP}_{M(L)}}}$ for
a rank-$\scalebox{0.95}{\ensuremath{1}}$ flat $\scalebox{0.95}{\ensuremath{L}}$
of $\scalebox{0.95}{\ensuremath{M}}$ with size $\scalebox{0.95}{\ensuremath{\ge2}}$
(a rank-$\scalebox{0.95}{\ensuremath{1}}$ flat of a rank-$\scalebox{0.95}{\ensuremath{2}}$
connected matroid is non-degenerate). Let $\scalebox{0.95}{\ensuremath{F=S-L}}$.
The number of rank-$\scalebox{0.95}{\ensuremath{1}}$ flats of $\scalebox{0.95}{\ensuremath{M|_{F}}}$
is at least $\scalebox{0.95}{\ensuremath{2}}$, and $\scalebox{0.95}{\ensuremath{\left|F\right|\ge2}}$.
Therefore, by Lemma \ref{lem:Cutting-1}, $\scalebox{0.95}{\ensuremath{N=\delta_{F}^{1}(U_{S}^{2})}}$
is a rank-$\scalebox{0.95}{\ensuremath{2}}$ connected matroid with
$\scalebox{0.95}{\ensuremath{M|_{L}=N/F}}$ and $\scalebox{0.95}{\ensuremath{M/L=N|_{F}}}$.
Hence, $\scalebox{0.95}{\ensuremath{\mathrm{BP}_{M}}}$ and $\scalebox{0.95}{\ensuremath{\mathrm{BP}_{N}}}$
are face-fitting through $\scalebox{0.95}{\ensuremath{\mathrm{BP}_{M(L)}=\mathrm{BP}_{N\left(F\right)}}}$,
and $\scalebox{0.95}{\ensuremath{\Sigma\cup\bigl\{\mathrm{BP}_{N}\bigr\}}}$
is a convex tiling again by Lemma \ref{lem:loc-conv-relev}. We can
repeat this process until we get a complete extension of $\scalebox{0.95}{\ensuremath{\Sigma}}$.

\subsection{Semitilings of rank at least three}

For a full-dimensional $\scalebox{0.95}{\ensuremath{\left(k,S\right)}}$-polytope
$\scalebox{0.95}{\ensuremath{P=\mathrm{BP}_{M}}}$ with $\scalebox{0.95}{\ensuremath{k\ge3}}$,
let $\scalebox{0.95}{\ensuremath{Q}}$ be a codimension-$\scalebox{0.95}{\ensuremath{2}}$
face. Then, $\scalebox{0.95}{\ensuremath{M}}$ is a connected matroid,
and $\scalebox{0.95}{\ensuremath{\kappa\left(\mathrm{MA}_{Q}\right)=3}}$.
Let $\scalebox{0.95}{\ensuremath{M_{i}}}$, $\scalebox{0.95}{\ensuremath{i\in\left[3\right]}}$,
be the connected components of $\scalebox{0.95}{\ensuremath{\mathrm{MA}_{Q}}}$
with ground sets $\scalebox{0.95}{\ensuremath{S_{i}=E(M_{i})}}$.
Then, $\scalebox{0.95}{\ensuremath{\mathrm{MA}_{Q}}}$ is written
as:
\[
\scalebox{0.95}{\ensuremath{\mathrm{MA}_{Q}=M_{1}\oplus M_{2}\oplus M_{3}}}.
\]
 There are precisely two facets of $\scalebox{0.95}{\ensuremath{P}}$
containing $\scalebox{0.95}{\ensuremath{Q}}$, say $\scalebox{0.95}{\ensuremath{R}}$
and $\scalebox{0.95}{\ensuremath{R'}}$. Then, $\scalebox{0.95}{\ensuremath{R=\mathrm{BP}_{M\left(F\right)}}}$
and $\scalebox{0.95}{\ensuremath{R'=\mathrm{BP}_{M\left(L\right)}}}$
for non-degenerate subsets $\scalebox{0.95}{\ensuremath{F}}$ and
$\scalebox{0.95}{\ensuremath{L}}$ of $\scalebox{0.95}{\ensuremath{M}}$,
and 
\[
\scalebox{0.95}{\ensuremath{Q=\mathrm{BP}_{M\left(F\right)\cap M\left(L\right)}=\mathrm{BP}_{M\left(F\right)\left(L\right)}=\mathrm{BP}_{M\left(L\right)\left(F\right)}}}.
\]
 The matroid $\scalebox{0.95}{\ensuremath{\mathrm{MA}_{Q}}}$ has
at most two loops, and there are three cases as follows.
\begin{description}
\item [{$\scalebox{0.95}{\ensuremath{\left|\overline{\emptyset}_{\mathrm{MA}_{Q}}\right|=0}}$}] The
non-degenerate subsets $\scalebox{0.95}{\ensuremath{F}}$ and $\scalebox{0.95}{\ensuremath{L}}$
are flats of $\scalebox{0.95}{\ensuremath{M}}$. By Lemma \ref{lem:Ridges of BP},
the number of $\scalebox{0.95}{\ensuremath{S_{i}}}$ contained in
$\scalebox{0.95}{\ensuremath{X}}$ for $\scalebox{0.95}{\ensuremath{X\in\left\{ F,L\right\} }}$
is $1$ or $2$.
\item [{$\scalebox{0.95}{\ensuremath{\left|\overline{\emptyset}_{\mathrm{MA}_{Q}}\right|=1}}$}] Let
$\scalebox{0.95}{\ensuremath{\overline{\emptyset}_{\mathrm{MA}_{Q}}=\left\{ e\right\} }}$,
then $\scalebox{0.95}{\ensuremath{\overline{\emptyset}_{\mathrm{MA}_{Q}}\in\left\{ S_{1},S_{2},S_{3}\right\} }}$,
and one of $\scalebox{0.95}{\ensuremath{F}}$ and $\scalebox{0.95}{\ensuremath{L}}$
does not contain $\scalebox{0.95}{\ensuremath{e}}$, say $\scalebox{0.95}{\ensuremath{e\notin F}}$.
Then, $\scalebox{0.95}{\ensuremath{F=S-\left\{ e\right\} }}$ since
$\scalebox{0.95}{\ensuremath{F}}$ is non-degenerate. Here, $\scalebox{0.95}{\ensuremath{F}}$
is a non-flat, but $\scalebox{0.95}{\ensuremath{L}}$ is a flat with
$\scalebox{0.95}{\ensuremath{L\neq\left\{ e\right\} }}$. The number
of $\scalebox{0.95}{\ensuremath{S_{i}}}$ contained in $\scalebox{0.95}{\ensuremath{X}}$
for $\scalebox{0.95}{\ensuremath{X\in\left\{ F,L\right\} }}$ is $1$
or $2$.
\item [{$\scalebox{0.95}{\ensuremath{\left|\overline{\emptyset}_{\mathrm{MA}_{Q}}\right|=2}}$}] Let
$\scalebox{0.95}{\ensuremath{\overline{\emptyset}_{\mathrm{MA}_{Q}}=\left\{ d,e\right\} }}$,
then $\scalebox{0.95}{\ensuremath{\left\{ \left\{ d\right\} ,\left\{ e\right\} ,S-\left\{ d,e\right\} \right\} =\left\{ S_{1},S_{2},S_{3}\right\} }}$.
Similarly as above, $\scalebox{0.95}{\ensuremath{\left\{ S-\left\{ d\right\} ,S-\left\{ e\right\} \right\} =\left\{ F,L\right\} }}$.
Here, both $\scalebox{0.95}{\ensuremath{F}}$ and $\scalebox{0.95}{\ensuremath{L}}$
are non-flats. The number of $\scalebox{0.95}{\ensuremath{S_{i}}}$
contained in $\scalebox{0.95}{\ensuremath{X}}$ for $\scalebox{0.95}{\ensuremath{X\in\left\{ F,L\right\} }}$
is $2$.
\end{description}
In any case, the number of $\scalebox{0.95}{\ensuremath{S_{i}}}$
contained in $\scalebox{0.95}{\ensuremath{F}}$ is $1$ or $2$. We
denote this number by $\scalebox{0.95}{\ensuremath{\mathrm{typ}_{P,Q}\left(R\right)}}$
and call the \textbf{type of $\scalebox{0.95}{\ensuremath{R}}$ in
$\scalebox{0.95}{\ensuremath{P}}$ at $\scalebox{0.95}{\ensuremath{Q}}$},
or the \textbf{type of $\scalebox{0.95}{\ensuremath{Q\lessdot R\lessdot P}}$}.
\begin{note}
\label{note:type}From now on, even though the type of $\scalebox{0.95}{\ensuremath{Q\lessdot R\lessdot P}}$
can be defined at any $\scalebox{0.95}{\ensuremath{Q}}$, we will
only consider the type at a \emph{loopless} $\scalebox{0.95}{\ensuremath{Q}}$.
When drawing quotient figures of $\scalebox{0.95}{\ensuremath{P}}$
at $\scalebox{0.95}{\ensuremath{Q}}$, which are two-dimensional,
we use a \emph{solid line} to represent $\scalebox{0.95}{\ensuremath{R}}$
with $\scalebox{0.95}{\ensuremath{\mathrm{typ}_{P,Q}\left(R\right)=1}}$
and a \emph{dashed line} to represent $\scalebox{0.95}{\ensuremath{R}}$
with $\scalebox{0.95}{\ensuremath{\mathrm{typ}_{P,Q}\left(R\right)=2}}$.
Note that for $\scalebox{0.95}{\ensuremath{k=3}}$, we have $\scalebox{0.95}{\ensuremath{\mathrm{typ}_{P,Q}\left(R\right)=r_{M}\left(F\right)}}$
and $\scalebox{0.95}{\ensuremath{\mathrm{typ}_{P,Q}\left(R'\right)=r_{N}\left(L\right)}}$,
which are independent of $\scalebox{0.95}{\ensuremath{Q}}$. However,
this is not the case for $\scalebox{0.95}{\ensuremath{k>3}}$.
\end{note}

The (loopless) codimension-$2$ face $\scalebox{0.95}{\ensuremath{Q}}$
is written as follows: 
\[
\scalebox{0.95}{\ensuremath{Q=P\cap\left\{ x\left(S_{1}\right)=k_{1}\right\} \cap\left\{ x\left(S_{2}\right)=k_{2}\right\} \cap\left\{ x\left(S_{3}\right)=k_{3}\right\} }}
\]
for a partition $\scalebox{0.95}{\ensuremath{k=\sum_{i\in\left[3\right]}k_{i}}}$.
See Figure \ref{fig:flake-Grid} for $\scalebox{0.95}{\ensuremath{\pi_{Q}\left(\Delta_{S}^{k}\right)}}$
and $\scalebox{0.95}{\ensuremath{\pi_{Q}\left(\left\{ x\left(S_{i}\right)=k_{i}\right\} \right)}}$
near the point $\scalebox{0.95}{\ensuremath{\pi_{Q}\left(Q\right)}}$.
By Lemma \ref{lem:Ridges of BP} again, without loss of generality:\smallskip{}

\noindent \begin{center}
\begin{tabular}{c||c|c|c|c}
\hline 
\scalebox{0.9}{$F$} & \scalebox{0.9}{$S_1$} & \scalebox{0.9}{$S_1\cup S_3$} & \scalebox{0.9}{$S_1$} & \scalebox{0.9}{$S_1\cup S_3$}\tabularnewline
\hline 
\scalebox{0.9}{$L$} & \scalebox{0.9}{$S_2$} & \scalebox{0.9}{$S_1\cup S_2$} & \scalebox{0.9}{$S_1\cup S_2$} & \scalebox{0.9}{$S_1$}\tabularnewline
\hline 
\end{tabular}
\par\end{center}

\begin{figure}[H]
\noindent \centering{}\noindent \begin{center}
\def\size{0.75}
\def\gapa{0.1}
\def\gapb{0.15}

\begin{tikzpicture}

\begin{scope}[line cap=round,rotate=0,scale=\size,xshift=0cm,yshift=0cm,>=stealth,
  decoration={markings,mark=at position 1 with {\arrow[line width=1pt]{stealth};}},font=\footnotesize]
 \foreach \x [count=\xi] in {green,blue,red}{
      \draw [\x] (0,0)++(300-120*\xi:2.3)--++(120-120*\xi:4.6);}
 \path (0:2.3) node[right=-2pt]{$\pi_Q\left(\left\{ x(S_1)=k_1\right\}\right) $};
 \path (120:2.3) node[above left=-4pt]{$\pi_Q\left(\left\{ x(S_2)=k_2\right\}\right) $};
 \path (240:2.3) node[below left=-4pt]{$\pi_Q\left(\left\{ x(S_3)=k_3\right\}\right) $};

 \draw [postaction={decorate},gray,font=\scriptsize] (-2,0)--++(0,0.7) node[above=-2pt ]{$\pi_Q\left(1^{S_2\cup S_3}\right)$};
 \draw [postaction={decorate},font=\scriptsize] (-2,0)--++(0,-0.7) node[below=-2pt ]{$\pi_Q\left(1^{S_1}\right)$};

 \draw [postaction={decorate},font=\scriptsize] (300:2)--++(30:0.7) node[right=-2pt ]{$\pi_Q\left(1^{S_2}\right)$};
 \draw [postaction={decorate},gray,font=\scriptsize] (300:2)--++(30:-0.7) node[below=-2pt ]{$\pi_Q\left(1^{S_1\cup S_3}\right)$};

 \draw [postaction={decorate},font=\scriptsize] (60:2)--++(150:0.7) node[above=-2pt ]{$\pi_Q\left(1^{S_3}\right)$};
 \draw [postaction={decorate},gray,font=\scriptsize] (60:2)--++(150:-0.7) node[right=-2pt ]{$\pi_Q\left(1^{S_1\cup S_2}\right)$};

\fill [black,opacity=1](0,0) circle (2.7pt);
 \draw [ultra thin] (0,0)--++(-30:0.7) node[below right=-6pt ]{$\pi_Q\left(Q\right)$};

\end{scope}

\end{tikzpicture}
\par\end{center}\caption{\label{fig:flake-Grid}$\pi_{Q}(\Delta_{S}^{k})$ and $\pi_{Q}(\{x(S_{i})=k_{i}\})$
near the point $\pi_{Q}(Q)$.}
\end{figure}

\noindent Using the formula \eqref{eq:def-ineq-flat} and Figure \ref{fig:flake-Grid},
we obtain four corresponding quotient figures of $\scalebox{0.95}{\ensuremath{P}}$
at $\scalebox{0.95}{\ensuremath{Q}}$; see Figure \ref{fig:local-figure}
where ``$\scalebox{0.95}{\ensuremath{\pi_{Q}}}$'' is dropped everywhere
for simplicity.
\begin{figure}[H]
\noindent \centering{}\noindent \begin{center}
\def\size{0.55}
\def\gapa{0.1}
\def\gapb{0.15}


\par\end{center}\caption{\label{fig:local-figure}Quotient Figures I}
\end{figure}
So, it is natural to define the \textbf{angle} \textbf{of} $\scalebox{0.95}{\ensuremath{P}}$
\textbf{at} $\scalebox{0.95}{\ensuremath{Q}}$ as follows: 
\[
\scalebox{0.95}{\ensuremath{\mathrm{ang}_{Q}\left(P\right):=\begin{cases}
1 & \mbox{if }\mathrm{typ}_{P,Q}\left(R\right)=\mathrm{typ}_{P,Q}\left(R'\right),\\
2 & \mbox{otherwise.}
\end{cases}}}
\]

Let $\scalebox{0.95}{\ensuremath{P_{1}=\mathrm{BP}_{M_{1}}}}$ and
$\scalebox{0.95}{\ensuremath{P_{2}=\mathrm{BP}_{M_{2}}}}$ be full-dimensional
$\scalebox{0.95}{\ensuremath{\left(k,S\right)}}$-polytopes that are
face-fitting through $\scalebox{0.95}{\ensuremath{R=P_{1}\cap P_{2}=P_{1}(F)}}$
for a subset $\scalebox{0.95}{\ensuremath{F\subset S}}$. In this
case, $\scalebox{0.95}{\ensuremath{R=P_{2}\left(S-F\right)}}$ and
$\scalebox{0.95}{\ensuremath{\mathrm{typ}_{P_{1},Q}\left(R\right)+\mathrm{typ}_{P_{2},Q}\left(R\right)=3}}.$
Without loss of generality, assume $\scalebox{0.95}{\ensuremath{\mathrm{typ}_{P_{1},Q}\left(R\right)=1}}$,
so $\scalebox{0.95}{\ensuremath{\mathrm{typ}_{P_{2},Q}\left(R\right)=2}}$.
Let $\scalebox{0.95}{\ensuremath{R'}}$ and $\scalebox{0.95}{\ensuremath{R''}}$
be other facets of $\scalebox{0.95}{\ensuremath{P_{1}}}$ and $\scalebox{0.95}{\ensuremath{P_{2}}}$,
respectively, containing $\scalebox{0.95}{\ensuremath{Q}}$. We write
$\scalebox{0.95}{\ensuremath{R'=P_{1}\left(L\right)}}$ and $\scalebox{0.95}{\ensuremath{R''=P_{2}\left(X\right)}}$
for some $\scalebox{0.95}{\ensuremath{L,X\subset S}}$. Then, one
of the following four cases happens:\medskip{}

\noindent \begin{center}
\begin{tabular}{c|c|c|c|c}
\hline 
\scalebox{0.9}{$X$} & \scalebox{0.9}{$\mathrm{typ}_{P_{1},Q}\left(R'\right)$} & \scalebox{0.9}{$\mathrm{typ}_{P_{2},Q}\left(R''\right)$} & \scalebox{0.9}{$\mathrm{ang}_{Q}\left(P_{1}\right)$} & \scalebox{0.9}{$\mathrm{ang}_{Q}\left(P_{2}\right)$}\tabularnewline
\hline 
\hline 
\scalebox{0.9}{$F\cup L$} & \scalebox{0.9}{$1$} & \scalebox{0.9}{$2$} & \scalebox{0.9}{$1$} & \scalebox{0.9}{$1$}\tabularnewline
\hline 
\scalebox{0.9}{$L$} & \scalebox{0.9}{$1$} & \scalebox{0.9}{$1$} & \scalebox{0.9}{$1$} & \scalebox{0.9}{$2$}\tabularnewline
\hline 
\scalebox{0.9}{$L$} & \scalebox{0.9}{$2$} & \scalebox{0.9}{$2$} & \scalebox{0.9}{$2$} & \scalebox{0.9}{$1$}\tabularnewline
\hline 
\scalebox{0.9}{$L-F$} & \scalebox{0.9}{$2$} & \scalebox{0.9}{$1$} & \scalebox{0.9}{$2$} & \scalebox{0.9}{$2$}\tabularnewline
\hline 
\end{tabular}\medskip{}
\par\end{center}

\noindent See Figure \ref{fig:face-fitting} for the quotient figure
of the tiling $\scalebox{0.95}{\ensuremath{\left\{ P_{1},P_{2}\right\} }}$
at $\scalebox{0.95}{\ensuremath{Q}}$.
\begin{figure}[H]
\noindent \centering{}\noindent \begin{center}
\def\size{0.55}
\def\gapa{0.1}
\def\gapb{0.15}

%
\par\end{center}\caption{\label{fig:face-fitting}Quotient Figures II}
\end{figure}

Fix $\scalebox{0.95}{\ensuremath{k\ge3}}$. Let $\scalebox{0.95}{\ensuremath{\Sigma}}$
be a $\scalebox{0.95}{\ensuremath{\left(k,S\right)}}$-tiling at a
loopless codimension-$\scalebox{0.95}{\ensuremath{2}}$ cell $\scalebox{0.95}{\ensuremath{Q}}$.
We define the \textbf{angle of $\scalebox{0.95}{\ensuremath{\Sigma}}$
at $\scalebox{0.95}{\ensuremath{Q}}$} as: 
\[
\scalebox{0.95}{\ensuremath{{\displaystyle \mathrm{ang}_{Q}\left(\Sigma\right)=\sum_{P\in\Sigma}\mathrm{ang}_{Q}\left(P\right)}}}.
\]
 Theorem \ref{thm:Uniform-finiteness} below tells us that the number
of full-dimensional base polytopes face-fitting at $\scalebox{0.95}{\ensuremath{Q}}$
is at most $\scalebox{0.95}{\ensuremath{3}}$ with $\scalebox{0.95}{\ensuremath{{\displaystyle \mathrm{ang}_{Q}\left(\Sigma\right)\le3}}}$
if $\scalebox{0.95}{\ensuremath{Q}}$ is irrelevant and at most $\scalebox{0.95}{\ensuremath{6}}$
with $\scalebox{0.95}{\ensuremath{{\displaystyle \mathrm{ang}_{Q}\left(\Sigma\right)\le6}}}$
if $\scalebox{0.95}{\ensuremath{Q}}$ is relevant. We define the\textbf{
deficiency} \textbf{of $\scalebox{0.95}{\ensuremath{\Sigma}}$ at
$\scalebox{0.95}{\ensuremath{Q}}$} as: 
\[
\scalebox{0.95}{\ensuremath{\mathrm{def}_{Q}\left(\Sigma\right):=\begin{cases}
3-\mathrm{ang}_{Q}\left(\Sigma\right) & \text{if }Q\text{ is irrelevant},\\
6-\mathrm{ang}_{Q}\left(\Sigma\right) & \text{if }Q\text{ is relevant}.
\end{cases}}}
\]

\begin{thm}[Uniform finiteness of matroid tiling]
\label{thm:Uniform-finiteness} Let $\scalebox{0.95}{\ensuremath{\Sigma}}$
be a tiling of rank $\scalebox{0.95}{\ensuremath{\ge3}}$ at a loopless
codimension-$\scalebox{0.95}{\ensuremath{2}}$ cell $\scalebox{0.95}{\ensuremath{Q}}$.
The quotient figure of $\scalebox{0.95}{\ensuremath{\Sigma}}$ at
$\scalebox{0.95}{\ensuremath{Q}}$, up to symmetry, can be classified
as a subcollection of one of the  quotient figures shown either in
Figure \ref{fig:Classfy tiling'} or in Figure \ref{fig:Classfy tiling},
depending on whether $\scalebox{0.95}{\ensuremath{Q}}$ is irrelevant
or relevant, respectively.
\begin{figure}[H]
\noindent \centering{}\noindent \begin{center}
\def\size{0.5}
\def\gapa{0.1}
\def\gapb{0.15}


\par\end{center}\caption{\label{fig:Classfy tiling'}Quotient Figures III}
\end{figure}
\begin{figure}[H]
\noindent \centering{}\noindent \begin{center}
\def\size{0.5}
\def\gapa{0.1}
\def\gapb{0.15}

%
\par\end{center}\caption{\label{fig:Classfy tiling}Quotient Figures IV}
\end{figure}
\end{thm}

\begin{proof}
Use Figures \ref{fig:flake-Grid}, \ref{fig:local-figure}, and \ref{fig:face-fitting}.
\end{proof}
\begin{rem}
For $\scalebox{0.95}{\ensuremath{\Sigma}}$ and $\scalebox{0.95}{\ensuremath{Q}}$
as in Theorem \ref{thm:Uniform-finiteness}, if $\scalebox{0.95}{\ensuremath{\Sigma}}$
is a tiling at $\scalebox{0.95}{\ensuremath{Q}}$, then $\scalebox{0.95}{\ensuremath{\mathrm{ang}_{Q}\left(\Sigma\right)\le3}}$
if and only if $\scalebox{0.95}{\ensuremath{\Sigma}}$ is locally
convex at $\scalebox{0.95}{\ensuremath{Q}}$, which can be easily
read off from the quotient figure of $\scalebox{0.95}{\ensuremath{\Sigma}}$
at $\scalebox{0.95}{\ensuremath{Q}}$. Moreover, $\scalebox{0.95}{\ensuremath{\Sigma}}$
is complete if and only if every subtiling at any $\scalebox{0.95}{\ensuremath{Q}}$
has deficiency $\scalebox{0.95}{\ensuremath{0}}$ at $\scalebox{0.95}{\ensuremath{Q}}$.
\end{rem}

\subsection{\label{subsec:w-tiling}Weighted tiling}
\begin{defn}
\label{def:w-tiling}For a $\scalebox{0.95}{\ensuremath{\left(k,S\right)}}$-tiling
$\scalebox{0.95}{\ensuremath{\Sigma}}$ and a weight vector $\scalebox{0.95}{\ensuremath{\mathbf{w}=\left(w_{s}\right)_{s\in S}\in\mathbb{R}^{S}}}$,
we call $\scalebox{0.95}{\ensuremath{\Sigma}}$ a\textbf{ tiling weighted
by $\mathbf{w}$} or simply a \textbf{$\mathbf{w}$-tiling} if $\scalebox{0.95}{\ensuremath{\Delta_{\mathbf{w}}\subseteq\left|\Sigma\right|}}$
and $\scalebox{0.95}{\ensuremath{P\cap\mathrm{int}\left(\Delta_{\mathbf{w}}\right)\neq\emptyset}}$
for all $\scalebox{0.95}{\ensuremath{P\in\Sigma}}$. Here, $\scalebox{0.95}{\ensuremath{\mathrm{int}\left(\Delta_{\mathbf{w}}\right)=\Delta_{\mathbf{w}}-\partial\Delta_{\mathbf{w}}}}$.
\end{defn}

A weighted tiling need not be convex.

Let $\scalebox{0.95}{\ensuremath{\Sigma}}$ be a $\scalebox{0.95}{\ensuremath{\left(k,S\right)}}$-tiling
with $\scalebox{0.95}{\ensuremath{k\ge3}}$, weighted by $\mathbf{w}$.
For a relevant cell $\scalebox{0.95}{\ensuremath{Q}}$ of codimension
$\scalebox{0.95}{\ensuremath{2}}$, there exist unique partitions
$\scalebox{0.95}{\ensuremath{\bigcup_{i\in\left[3\right]}S_{i}}}$
and $\scalebox{0.95}{\ensuremath{\sum_{i\in\left[3\right]}k_{i}}}$
of $\scalebox{0.95}{\ensuremath{S}}$ and $k$, respectively, with
\[
\scalebox{0.95}{\ensuremath{Q\subset\bigcap_{i\in\left[3\right]}\left\{ x\left(S_{i}\right)=k_{i}\right\} }}.
\]
 If $\scalebox{0.95}{\ensuremath{\mathbf{w}\left(S_{i}\right)\le k_{i}}}$
for some $\scalebox{0.95}{\ensuremath{i\in\left[3\right]}}$, then
using Figure \ref{fig:flake-Grid} and Theorem \ref{thm:Uniform-finiteness},
we deduce $\scalebox{0.95}{\ensuremath{\Delta_{\mathbf{w}}\cap\left\{ x\left(S_{i}\right)>k_{i}\right\} =\emptyset}}$
and $\scalebox{0.95}{\ensuremath{\mathrm{def}_{Q}\left(\Sigma\right)>0}}$.
If $\scalebox{0.95}{\ensuremath{\mathbf{w}\left(S-S_{i}\right)\le k-k_{i}}}$
for some $\scalebox{0.95}{\ensuremath{i\in\left[3\right]}}$, then
$\scalebox{0.95}{\ensuremath{\Delta_{\mathbf{w}}\cap\left\{ x\left(S_{i}\right)<k_{i}\right\} =\emptyset}}$
and $\scalebox{0.95}{\ensuremath{\mathrm{def}_{Q}\left(\Sigma\right)>0}}$.
The contrapositive statements of these two provide information about
the weight vector $\mathbf{w}$:
\begin{itemize}
\item If $\scalebox{0.95}{\ensuremath{\Delta_{\mathbf{w}}\cap\left\{ x\left(S_{i}\right)>k_{i}\right\} \neq\emptyset}}$,
then $\scalebox{0.95}{\ensuremath{\mathbf{w}\left(S_{i}\right)>k_{i}}}$.
\item If $\scalebox{0.95}{\ensuremath{\Delta_{\mathbf{w}}\cap\left\{ x\left(S_{i}\right)<k_{i}\right\} \neq\emptyset}}$,
then $\scalebox{0.95}{\ensuremath{\mathbf{w}\left(S-S_{i}\right)>k-k_{i}}}$.
\item If $\scalebox{0.95}{\ensuremath{\mathrm{def}_{Q}\left(\Sigma\right)=0}}$,
then $\scalebox{0.95}{\ensuremath{\mathbf{w}\left(S_{i}\right)>k_{i}}}$
for all $\scalebox{0.95}{\ensuremath{i\in\left[3\right]}}$.
\item If $\scalebox{0.95}{\ensuremath{\mathrm{def}_{Q}\left(\Sigma\right)=0}}$,
then $\scalebox{0.95}{\ensuremath{\mathbf{w}\left(S-S_{i}\right)>k-k_{i}}}$
for all $\scalebox{0.95}{\ensuremath{i\in\left[3\right]}}$.
\end{itemize}
The third statement implies the fourth one, but not the other way
around. The converse of the third one holds when $\scalebox{0.95}{\ensuremath{k=3}}$:
If $\scalebox{0.95}{\ensuremath{k=3}}$, then $\scalebox{0.95}{\ensuremath{k_{i}=1}}$
and $\scalebox{0.95}{\ensuremath{\left|S_{i}\right|\ge2}}$ for all
$\scalebox{0.95}{\ensuremath{i\in\left[3\right]}}$ because $\scalebox{0.95}{\ensuremath{Q}}$
is relevant, and $\scalebox{0.95}{\ensuremath{Q=\bigcap_{i\in\left[3\right]}\left\{ x\left(S_{i}\right)=1\right\} }}$.
Suppose $\scalebox{0.95}{\ensuremath{\mathbf{w}\left(S_{i}\right)>1}}$
for all $\scalebox{0.95}{\ensuremath{i\in\left[3\right]}}$, then
there is a point $\scalebox{0.95}{\ensuremath{\mathbf{v}=\left(v_{s}\right)_{s\in S}\in Q}}$
with $\scalebox{0.95}{\ensuremath{\mathbf{v}\left(S_{i}\right)=1}}$
and $\scalebox{0.95}{\ensuremath{0<v_{s}<w_{s}}}$ for all $\scalebox{0.95}{\ensuremath{s\in S}}$.
So, $\scalebox{0.95}{\ensuremath{\mathbf{v}\in Q\cap\mathrm{int}\left(\Delta_{\mathbf{w}}\right)}}$
and $\scalebox{0.95}{\ensuremath{\mathrm{def}_{Q}\left(\Sigma_{Q}\right)=0}}$.\vspace{2pt}

In the next proposition, we restrict ourselves to the rank-$\scalebox{0.95}{\ensuremath{3}}$
case and investigate the boundary shape of a weighted tiling.
\begin{prop}
\label{prop:rank3-weighted}For a $\scalebox{0.95}{\ensuremath{\left(3,S\right)}}$-tiling
$\scalebox{0.95}{\ensuremath{\Sigma}}$ weighted by $\mathbf{w}$,
assume that $\scalebox{0.95}{\ensuremath{Q_{0}}}$ is a relevant codimension-$\scalebox{0.95}{\ensuremath{2}}$
cell at which $\scalebox{0.95}{\ensuremath{\Sigma_{Q_{0}}}}$ is not
locally convex. Then, there is a partition $\scalebox{0.95}{\ensuremath{S=\bigcup_{i\in\left[3\right]}S_{i}}}$
with $\scalebox{0.95}{\ensuremath{Q_{0}=\bigcap_{i\in\left[3\right]}\left\{ x\left(S_{i}\right)=1\right\} }}$
where $\scalebox{0.95}{\ensuremath{\left|S_{i}\right|\ge2}}$ for
all $\scalebox{0.95}{\ensuremath{i\in\left[3\right]}}$ and $\scalebox{0.95}{\ensuremath{\mathbf{w}\left(S_{j}\right)\le1}}$
for some $\scalebox{0.95}{\ensuremath{j\in\left[3\right]}}$, say
$\scalebox{0.95}{\ensuremath{j=1}}$.
\begin{enumerate}
\item \label{enu:non-convex}The quotient figure of $\scalebox{0.95}{\ensuremath{\Sigma_{Q_{0}}}}$
at $\scalebox{0.95}{\ensuremath{Q_{0}}}$ is one of the five figures
shown in Figure \ref{fig:non-convex-1} with $\scalebox{0.95}{\ensuremath{\mathbf{w}\left(S_{i}\right)>1}}$
for $\scalebox{0.95}{\ensuremath{i=2,3}}$.
\item \label{enu:parallel}In \eqref{enu:non-convex}, there always exists
a type-$\scalebox{0.95}{\ensuremath{2}}$ facet containing $\scalebox{0.95}{\ensuremath{Q_{0}}}$,
say $\scalebox{0.95}{\ensuremath{R_{0}}}$. Let $\scalebox{0.95}{\ensuremath{P_{0}}}$
be a member of $\scalebox{0.95}{\ensuremath{\Sigma}}$ that contains
$\scalebox{0.95}{\ensuremath{R_{0}}}$, then $\scalebox{0.95}{\ensuremath{Q_{0}\lessdot R_{0}\lessdot P_{0}\in\Sigma}}$
and $\scalebox{0.95}{\ensuremath{\mathrm{ang}_{Q_{0}}(P_{0})=2}}$.
Let $\scalebox{0.95}{\ensuremath{Q\lessdot R_{0}}}$ be another codimension-$\scalebox{0.95}{\ensuremath{2}}$
cell, then the number of those $\scalebox{0.95}{\ensuremath{Q}}$
with $\scalebox{0.95}{\ensuremath{\mathrm{ang}_{Q}\left(P_{0}\right)=1}}$
and the number of those $\scalebox{0.95}{\ensuremath{Q}}$ with $\scalebox{0.95}{\ensuremath{\mathrm{ang}_{Q}\left(\Sigma_{Q}\right)=3}}$
are both at most $\scalebox{0.95}{\ensuremath{1}}$. The quotient
figure of $\scalebox{0.95}{\ensuremath{\Sigma_{Q}}}$ at $\scalebox{0.95}{\ensuremath{Q}}$
is one of the figures shown either in Figure \ref{fig:non-convex-2}
or in Figure \ref{fig:non-convex-3}, depending on whether $\scalebox{0.95}{\ensuremath{\mathrm{ang}_{Q}\left(\Sigma_{Q}\right)=3}}$
or $\scalebox{0.95}{\ensuremath{\mathrm{ang}_{Q}\left(\Sigma_{Q}\right)\le2}}$,
respectively. 
\begin{figure}[H]
\noindent \centering{}\noindent \begin{center}
\def\size{0.5}
\def\gapa{0.1}
\def\gapb{0.15}


\par\end{center}\caption{\label{fig:non-convex-1}Non-convex rank-3 weighted tilings, I.}
\end{figure}
\begin{figure}[H]
\noindent \centering{}\noindent \begin{center}
\def\size{0.5}
\def\gapa{0.1}
\def\gapb{0.15}

%
\par\end{center}\caption{\label{fig:non-convex-2}Non-convex rank-3 weighted tilings, II.}
\end{figure}
\begin{figure}[H]
\noindent \centering{}\noindent \begin{center}
\def\size{0.5}
\def\gapa{0.1}
\def\gapb{0.15}

%
\par\end{center}\caption{\label{fig:non-convex-3}Non-convex rank-3 weighted tilings, III.}
\end{figure}
\end{enumerate}
\end{prop}

\begin{proof}
Since $\scalebox{0.95}{\ensuremath{\Sigma_{Q_{0}}}}$ at $\scalebox{0.95}{\ensuremath{Q_{0}}}$
is non-convex, $\scalebox{0.95}{\ensuremath{\mathrm{ang}_{Q_{0}}(\Sigma_{Q_{0}})=4,5}}$.
Since $\scalebox{0.95}{\ensuremath{\mathbf{w}\left(S_{1}\right)\le1}}$,
one has $\scalebox{0.95}{\ensuremath{\mathbf{w}\left(S_{2}\right)>1}}$
and $\scalebox{0.95}{\ensuremath{\mathbf{w}\left(S_{3}\right)>1}}$.
Moreover, $\scalebox{0.95}{\ensuremath{\Delta_{\mathbf{w}}\cap\left\{ x\left(S_{1}\right)>1\right\} =\emptyset}}$.
Using Figure \ref{fig:Classfy tiling}, we can obtain the five figures
in Figure \ref{fig:non-convex-1} for the quotient figure of $\scalebox{0.95}{\ensuremath{\Sigma_{Q_{0}}}}$
at $\scalebox{0.95}{\ensuremath{Q_{0}}}$, which proves \eqref{enu:non-convex}.
Next, let $\scalebox{0.95}{\ensuremath{R_{0}=P_{0}\left(F\right)}}$.
Then, we have either $\scalebox{0.95}{\ensuremath{F=S_{1}\cup S_{2}}}$
or $\scalebox{0.95}{\ensuremath{F=S_{1}\cup S_{3}}}$. These two cases
are symmetric, so we can assume the former case. In this case, $\scalebox{0.95}{\ensuremath{F}}$
is a rank-$\scalebox{0.95}{\ensuremath{2}}$ non-degenerate flat of
$\scalebox{0.95}{\ensuremath{\mathrm{MA}_{P_{0}}}}$. By Lemma \ref{lem:Ridges of BP},
$\scalebox{0.95}{\ensuremath{S_{2}}}$ is a rank-$\scalebox{0.95}{\ensuremath{1}}$
non-degenerate flat of $\scalebox{0.95}{\ensuremath{\mathrm{MA}_{P_{0}}}}$,
and $\scalebox{0.95}{\ensuremath{P_{0}\left(S_{2}\right)}}$ is a
facet of $\scalebox{0.95}{\ensuremath{P_{0}}}$ with $\scalebox{0.95}{\ensuremath{Q_{0}=P_{0}\left(F\right)\left(S_{2}\right)=P_{0}\left(S_{2}\right)\left(F\right)}}$.

To prove \eqref{enu:parallel}, consider a codimension-$\scalebox{0.95}{\ensuremath{2}}$
cell $\scalebox{0.95}{\ensuremath{Q\lessdot R_{0}}}$ that is not
$\scalebox{0.95}{\ensuremath{Q_{0}}}$. Then, we can write $\scalebox{0.95}{\ensuremath{Q=R_{0}\left(A\right)}}$
for a rank-$\scalebox{0.95}{\ensuremath{1}}$ non-degenerate flat
$\scalebox{0.95}{\ensuremath{A}}$ of $\scalebox{0.95}{\ensuremath{\mathrm{MA}_{R_{0}}}}$.
Since $\scalebox{0.95}{\ensuremath{\mathrm{MA}_{R_{0}}}}$ is loopless,
we have $\scalebox{0.95}{\ensuremath{A\cap S_{2}=\emptyset}}$ and
$\scalebox{0.95}{\ensuremath{A\subsetneq S_{1}}}$. On the other hand,
we also have $\scalebox{0.95}{\ensuremath{Q=R_{0}\left(L\right)}}$
for a non-degenerate flat $\scalebox{0.95}{\ensuremath{L}}$ of $\scalebox{0.95}{\ensuremath{\mathrm{MA}_{P_{0}}}}$.
Again, by Lemma \ref{lem:Ridges of BP}, we have either $\scalebox{0.95}{\ensuremath{L=A}}$
or $\scalebox{0.95}{\ensuremath{L=A\cup S_{3}}}$ with rank $\scalebox{0.95}{\ensuremath{1}}$
or $\scalebox{0.95}{\ensuremath{2}}$, respectively. 
\begin{itemize}
\item If $\scalebox{0.95}{\ensuremath{L=A}}$, then $\scalebox{0.95}{\ensuremath{\Delta_{\mathbf{w}}\cap\left\{ x\left(A\right)>1\right\} =\emptyset}}$
and $\scalebox{0.95}{\ensuremath{\mathrm{ang}_{Q}(\Sigma_{Q})=\mathrm{ang}_{Q}(P_{0})=2}}$.
This leads to the first and last figures in Figure \ref{fig:non-convex-3}.
\item If $\scalebox{0.95}{\ensuremath{L=A\cup S_{3}}}$, then $\scalebox{0.95}{\ensuremath{A}}$
is a rank-$\scalebox{0.95}{\ensuremath{1}}$ degenerate flat of $\scalebox{0.95}{\ensuremath{\mathrm{MA}_{P_{0}}}}$,
and $\scalebox{0.95}{\ensuremath{Q}}$ is the unique cell with $\scalebox{0.95}{\ensuremath{\mathrm{ang}_{Q}(P_{0})=1}}$
(Lemma \ref{lem:degen-rank1-flat}). We obtain two figures in Figure
\ref{fig:non-convex-2} if $\scalebox{0.95}{\ensuremath{\mathrm{ang}_{Q}(\Sigma_{Q})=3}}$,
and four in the middle of Figure \ref{fig:non-convex-3} if $\scalebox{0.95}{\ensuremath{\mathrm{ang}_{Q}(\Sigma_{Q})\le2}}$. 
\end{itemize}
This completes the proof.
\end{proof}

\section{\label{sec:Puzzle}Puzzle and Compatible Arrangements}

This section presents the concept of a semipuzzle, a matroidal equivalent
of a semitiling. By introducing coordinate charts, we can view a semipuzzle
as a topological manifold. We also define compatible arrangements
as ``ready-to-glue'' arrangements, where their associated puzzle-pieces
form a semipuzzle.

\subsection{\label{subsec:Puzzle}Puzzle and semipuzzle}
\begin{defn}
A $\scalebox{0.95}{\ensuremath{\left(k,S\right)}}$\textbf{-semipuzzle}
$\scalebox{0.95}{\ensuremath{\Psi}}$ is the collection of puzzle-pieces
whose base polytopes form a $\scalebox{0.95}{\ensuremath{\left(k,S\right)}}$-semitiling
$\scalebox{0.95}{\ensuremath{\Sigma}}$. We call $\scalebox{0.95}{\ensuremath{\Sigma}}$
the \textbf{associated semitiling} of $\scalebox{0.95}{\ensuremath{\Psi}}$,
and $\scalebox{0.95}{\ensuremath{\Psi}}$ the \textbf{associated semipuzzle}
of $\scalebox{0.95}{\ensuremath{\Sigma}}$. A $\scalebox{0.95}{\ensuremath{\left(k,S\right)}}$\textbf{-puzzle}
is a $\scalebox{0.95}{\ensuremath{\left(k,S\right)}}$-semipuzzle
whose associated semitiling is a tiling. We can omit ``$\scalebox{0.95}{\ensuremath{\left(k,S\right)}}$-''
for simplicity.
\end{defn}

We use the same terminology for semipuzzles as we do for semitilings.
However, we only concern ourselves with loopless objects when dealing
with semipuzzles, and the terminology for semipuzzles reflects this
restriction. For instance, the \textbf{boundary} of a semipuzzle is
defined as the collection of corresponding puzzle-pieces of the \emph{loopless}
facets of its associated semitiling. From this point on, we will assume
that semipuzzles are both equidimensional and full-dimensional, unless
otherwise noted. This is the same assumption we make for semitilings.

Because base polytopes glue through faces that are not contained in
the boundary of the hypersimplex, extending a tiling can be transformed
to extending a puzzle. This results in a dimension drop by $\scalebox{0.95}{\ensuremath{n-k}}$,
where $\scalebox{0.95}{\ensuremath{n}}$ is the size of the ground
set and $\scalebox{0.95}{\ensuremath{k}}$ is the rank. See also Remark
\ref{rem:why-pp} for additional benefits of switching to puzzles.

\subsection{\label{subsec:Coord-charts}Coordinate charts for semipuzzle}

Consider a hypersimplex $\scalebox{0.95}{\ensuremath{\Delta_{S}^{k}}}$.
A partition of $\scalebox{0.95}{\ensuremath{S}}$ into $k$ (nonempty)
subsets is said to be a \textbf{$k$-partition} of $\scalebox{0.95}{\ensuremath{S}}$.
For any two $k$-partitions of $\scalebox{0.95}{\ensuremath{S}}$,
say $\scalebox{0.95}{\ensuremath{\bigcup_{i\in\left[k\right]}A_{i}}}$
and $\scalebox{0.95}{\ensuremath{\bigcup_{i\in\left[k\right]}B_{i}}}$,
the \textbf{distance} between two point-pieces $\scalebox{0.95}{\ensuremath{\mathrm{PZ}_{\bigoplus_{i\in\left[k\right]}U_{A_{i}}^{1}}}}$
and $\scalebox{0.95}{\ensuremath{\mathrm{PZ}_{\bigoplus_{i\in\left[k\right]}U_{B_{i}}^{1}}}}$
is defined as 
\[
\scalebox{0.95}{\ensuremath{{\displaystyle \min_{\sigma\in\mathfrak{S}_{k}}\biggl\{\sum_{i\in\left[k\right]}d\left(A_{i},B_{\sigma(i)}\right)\biggr\}}}}
\]
 where $d$ is the metric on $\scalebox{0.95}{\ensuremath{2^{S}}}$,
given in \eqref{eq:metric on U(n,n)}. So, if $\scalebox{0.95}{\ensuremath{\Psi}}$
is a $\scalebox{0.95}{\ensuremath{\left(k,S\right)}}$-semipuzzle,
the collection of point-pieces of $\Psi$ is a metric space. The convex
hull of $k$ points $\scalebox{0.95}{\ensuremath{\mathbbm1+\left(\left|S\right|-k\right)\cdot\mathbf{e}_{i}\in\mathbb{R}^{k}}}$
for all $\scalebox{0.95}{\ensuremath{i\in\left[k\right]}}$ is a $\scalebox{0.95}{\ensuremath{\left(k-1\right)}}$-simplex
whose intersection with $\scalebox{0.95}{\ensuremath{\mathbb{Z}^{k}}}$
works as a\textbf{ coordinate chart} for $\Psi$ with barycentric
coordinates. For example, $\scalebox{0.95}{\ensuremath{\left(\left|A_{1}\right|,\dots,\left|A_{k}\right|\right)}}$
is the coordinate of the point-piece $\scalebox{0.95}{\ensuremath{\mathrm{PZ}_{\bigoplus_{i\in\left[k\right]}U_{A_{i}}^{1}}}}$.

When $\scalebox{0.95}{\ensuremath{k=3}}$, quotient figures at codimension-$\scalebox{0.95}{\ensuremath{2}}$
cells and puzzle-pieces are both $\scalebox{0.95}{\ensuremath{2}}$-dimensional.
So, we incorporate a quotient figure with a puzzle-piece figure. We
may add or drop some information as doing that; see Figure \ref{fig:coordinate}.
\begin{figure}[H]
\noindent \centering{}\noindent \begin{center}
\def\size{0.45}
\def\gapa{0.1}
\def\gapb{0.15}

\begin{tikzpicture}[font=\scriptsize]

\matrix[column sep=2.7cm, row sep=0cm]{

\begin{scope}[line cap=round,scale=\size,xshift=0cm,yshift=0cm,>=stealth]
 \foreach \x [count=\xi] in {0,1,2}{
   \path (0,3.464)--++(240:2*\x) coordinate (A\x);
   \foreach \y in {1,...,\xi}{
      \path (A\x)--++(0:2*\y-2) coordinate (B\y);
      \draw [ultra thin] (B\y)--++(240:2)--++(0:2)--cycle;
  }}

 \fill [black,opacity=1](0,0) circle (2.8pt);
 \foreach \x/\y [count=\xi] in {6/2,3/3.464}{
   \foreach \z in {1,...,\x}{
     \fill [black,opacity=1](\z*360/\x+90*\xi-90:\y) circle (2.8pt);
  }}
\end{scope}

&

\begin{scope}[line cap=round,scale=\size,xshift=0cm,yshift=0cm,>=stealth]
 \foreach \x [count=\xi] in {0,1,2}{
   \path (0,3.464)--++(240:2*\x) coordinate (A\x);
   \foreach \y in {1,...,\xi}{
      \path (A\x)--++(0:2*\y-2) coordinate (B\y);
      \draw [ultra thin] (B\y)--++(240:2)--++(0:2)--cycle;
  }}

 \fill [black,opacity=1](0,0) circle (2.8pt);
 \foreach \x/\y [count=\xi] in {6/2,3/3.464}{
   \foreach \z in {1,...,\x}{
     \fill [black,opacity=1](\z*360/\x+90*\xi-90:\y) circle (2.8pt);
  }}

 \path (60:\gapa*1.155) coordinate (O);
 \foreach \x [count=\xi] in {A1,A2,A3}{
      \path (-60+60*\xi:2) coordinate (\x);}
 \path (A1)++(150:2*\gapa) coordinate (A1);
 \path (A2)++(240:1.155*\gapa) coordinate (A2);
 \path (A3)++(330:2*\gapa) coordinate (A3);
 \path ($(O)!.5!(A1)$) coordinate (B1);
 \path ($(O)!.5!(A3)$) coordinate (B3);

 \fill [gray,opacity=0.3] (O)--(A1)--(A2)--(A3)--cycle;
 \draw [very thick,densely dashed] (A3)--(O)(A3)--(A2);
 \draw [red,very thick] (A1)--(A2);
 \draw [green,very thick] (O)--(A1);

\end{scope}

\\
};
\end{tikzpicture}
\par\end{center}\caption{\label{fig:coordinate}A coordinate chart for $(3,S)$-puzzle with
$|S|=6$.}
\end{figure}
 We note that if the number of lines (rank-$3$ hyperplanes are lines)
is large, more than one coordinate chart may be necessary, and it
may be impossible to draw them all on one sheet of paper. Also, for
technical reasons, the drawn angle between two lines may not be a
multiple of $\scalebox{0.95}{\ensuremath{\frac{\pi}{3}}}$. These
factors necessitate variations; see Figure \ref{fig:MaxIrr (3,6)}
for instance. For simplicity, we will not include coordinate charts
in our drawings.

\subsection{\label{subsec:Compt-arr}Compatible arrangements with gluing and
merging}

Let $\scalebox{0.95}{\ensuremath{M_{1},\dots,M_{m}}}$ be rank-$k$
connected matroids on $\scalebox{0.95}{\ensuremath{S}}$. Let $\scalebox{0.95}{\ensuremath{\Psi=\bigl\{\mathcal{A}_{i}^{\lessdot M_{i}}:i\in\left[m\right]\bigr\}}}$
be a collection of $\scalebox{0.95}{\ensuremath{\mathbbm1}}$-arrangements,
i.e., each arrangement $\scalebox{0.95}{\ensuremath{\mathcal{A}_{i}^{\lessdot M_{i}}}}$
is obtained from $\scalebox{0.95}{\ensuremath{\mathrm{HA}_{M_{i}}}}$
by a sequence of blowups and collapsings, cf. Subsection \ref{subsec:w-arrangement}.
We say that those arrangements are \textbf{compatible} if the following
hold:
\begin{enumerate}[label=(GA\arabic*)]
\item \label{enu:ass-tiling} $\scalebox{0.95}{\ensuremath{\Sigma=\left\{ \mathrm{BP}_{M_{i}}:i\in\left[m\right]\right\} }}$
is a semitiling called the \textbf{associated semitiling}.
\item \label{enu:common-face}For any tiling $\scalebox{0.95}{\ensuremath{\left\{ \mathrm{BP}_{M_{i}},\mathrm{BP}_{M_{j}}\right\} \subseteq\Sigma}}$
with $\scalebox{0.95}{\ensuremath{i\neq j}}$ ($\scalebox{0.95}{\ensuremath{\mathrm{codim}(\mathrm{BP}_{M_{i}}\cap\mathrm{BP}_{M_{j}})=1}}$),
there exist $\scalebox{0.95}{\ensuremath{X\in\mathcal{A}_{i}^{\lessdot M_{i}}}}$
and $\scalebox{0.95}{\ensuremath{Y\in\mathcal{A}_{j}^{\lessdot M_{j}}}}$
with $\scalebox{0.95}{\ensuremath{\phi\left(X\right)=\phi\left(Y\right)=M_{i}\cap M_{j}}}$
and a lattice isomorphism $\scalebox{0.95}{\ensuremath{\psi_{ij}:\mathcal{A}_{i}|_{X}\rightarrow\mathcal{A}_{j}|_{Y}}}$.
\end{enumerate}
Note that uniqueness is not included in the definition and different
collections of compatible arrangements can have the same associated
semitiling.
\begin{example}
A puzzle is a compatible collection of matroidal arrangements.
\end{example}

Let $\scalebox{0.95}{\ensuremath{\Psi=\bigl\{\mathcal{A}_{i}^{\lessdot M_{i}}:i\in\left[m\right]\bigr\}}}$
be a compatible collection of arrangements with label maps $\scalebox{0.95}{\ensuremath{\lambda_{i}:\Lambda_{i}\rightarrow\mathcal{A}_{i}^{\lessdot M}}}$.
We \textbf{glue} the coatomistic lattices $\scalebox{0.95}{\ensuremath{\mathcal{A}_{i}}}$
by identifying $\scalebox{0.95}{\ensuremath{\mathcal{A}_{i}|_{X}}}$
with $\scalebox{0.95}{\ensuremath{\mathcal{A}_{j}|_{Y}}}$ via the
lattice isomorphism $\scalebox{0.95}{\ensuremath{\psi_{ij}:\mathcal{A}_{i}|_{X}\rightarrow\mathcal{A}_{j}|_{Y}}}$
for any $\scalebox{0.95}{\ensuremath{i}}$ and $\scalebox{0.95}{\ensuremath{j}}$
with \ref{enu:common-face}, and denote it by $\scalebox{0.95}{\ensuremath{\bigl(\bigcup_{i\in\left[m\right]}M_{i},\mathlarger{\mathlarger{\#}}_{i\in\left[m\right]}\mathcal{A}_{i}\bigr)}}$
or simply by $\scalebox{0.95}{\ensuremath{\mathlarger{\mathlarger{\#}}_{i\in\left[m\right]}\mathcal{A}_{i}}}$.
Define a map $\scalebox{0.95}{\ensuremath{\tilde{\lambda}}}$ on $\scalebox{0.95}{\ensuremath{\bigcup_{i\in\left[m\right]}\Lambda_{i}}}$
as follows: for all $\scalebox{0.95}{\ensuremath{i,j\in\left[m\right]}}$
\begin{enumerate}
\item identify $\scalebox{0.95}{\ensuremath{\lambda_{i}\left(s\right)}}$
with $\scalebox{0.95}{\ensuremath{\lambda_{j}\left(s\right)}}$ for
$\scalebox{0.95}{\ensuremath{s\in\Lambda_{i}\cap\Lambda_{j}}}$, and 
\item identify $\scalebox{0.95}{\ensuremath{\lambda_{i}\left(t\right)}}$
with $\scalebox{0.95}{\ensuremath{\lambda_{j}\left(u\right)}}$ for
$\scalebox{0.95}{\ensuremath{\psi_{ij}\left(\lambda_{i}\left(t\right)\right)=\lambda_{j}\left(u\right)}}$. 
\end{enumerate}
We define the \textbf{gluing} of the arrangements of $\scalebox{0.95}{\ensuremath{\Psi}}$
with label map $\scalebox{0.95}{\ensuremath{\tilde{\lambda}}}$ as
follows: 
\[
\scalebox{0.95}{\ensuremath{\mathrm{GA}_{\Psi}:=\bigl\{\tilde{\lambda}\left(s\right):s\in\bigcup_{i\in\left[m\right]}\Lambda_{i}\bigr\}}}.
\]
Taking $\scalebox{0.95}{\ensuremath{\mathrm{GA}_{\Psi}}}$ is also
said to be \textbf{gluing} the arrangements of $\scalebox{0.95}{\ensuremath{\Psi}}$.
Let $\scalebox{0.95}{\ensuremath{\Psi_{\mathrm{int}}}}$ be the collection
of all those identified coatoms in \ref{enu:common-face}. Define:
\[
\scalebox{0.95}{\ensuremath{{\displaystyle \left|\Psi\right|=\left|\mathrm{GA}_{\Psi}\right|:=\mathrm{GA}_{\Psi}-\Psi_{\mathrm{int}}}}}.
\]
 The map $\scalebox{0.95}{\ensuremath{\tilde{\lambda}}}$ naturally
induces a label map for $\scalebox{0.95}{\ensuremath{\left|\Psi\right|}}$.
We call $\scalebox{0.95}{\ensuremath{\left|\Psi\right|}}$ the \textbf{support}
of $\scalebox{0.95}{\ensuremath{\Psi}}$ or of $\scalebox{0.95}{\ensuremath{\mathrm{GA}_{\Psi}}}$.
Taking $\scalebox{0.95}{\ensuremath{\left|\Psi\right|}}$ is said
to be \textbf{merging} (the arrangements of) $\scalebox{0.95}{\ensuremath{\Psi}}$
or $\scalebox{0.95}{\ensuremath{\mathrm{GA}_{\Psi}}}$. Note that
the arrangement $\scalebox{0.95}{\ensuremath{\left|\Psi\right|}}$
need not be matroidal. However, if the semitiling $\scalebox{0.95}{\ensuremath{\Sigma=\left\{ \mathrm{BP}_{M_{i}}:i\in\left[m\right]\right\} }}$
is convex, then $\scalebox{0.95}{\ensuremath{M=\bigcup_{i\in\left[m\right]}M_{i}}}$
is a matroid by Lemma \ref{lem:locally-convex}, and $\scalebox{0.95}{\ensuremath{\left|\Psi\right|}}$
is matroidal.
\begin{example}
If $\scalebox{0.95}{\ensuremath{\Psi}}$ is a convex puzzle, then
$\scalebox{0.95}{\ensuremath{\left|\Psi\right|=\mathrm{PZ}_{M}}}$.
\end{example}

Let $\scalebox{0.95}{\ensuremath{\Sigma=\left\{ \mathrm{BP}_{M_{i}}:i\in\left[m\right]\right\} }}$
be a convex tiling in $\scalebox{0.95}{\ensuremath{\Delta_{S}^{k}}}$,
then $\scalebox{0.95}{\ensuremath{\left|\Sigma\right|=\mathrm{BP}_{M}}}$
for a matroid $\scalebox{0.95}{\ensuremath{M}}$. For any tiling $\scalebox{0.95}{\ensuremath{\bigl\{\mathrm{BP}_{M_{i}},\mathrm{BP}_{M_{j}}\bigr\}\subseteq\Sigma}}$
with $\scalebox{0.95}{\ensuremath{i\neq j}}$, there are non-degenerate
flats $\scalebox{0.95}{\ensuremath{F}}$ and $\scalebox{0.95}{\ensuremath{L}}$
of $\scalebox{0.95}{\ensuremath{M_{i}}}$ and $\scalebox{0.95}{\ensuremath{M_{j}}}$,
respectively, with: 
\[
\scalebox{0.95}{\ensuremath{M_{i}\cap M_{j}=\phi\left(M_{i}(F)\right)=\phi\left(M_{j}(L)\right)}}.
\]
 For each $\scalebox{0.95}{\ensuremath{i\in\left[m\right]}}$, blow
up all those $\scalebox{0.95}{\ensuremath{M_{i}/F\in\mathcal{S}\left(M_{i}\right)}}$
and let $\scalebox{0.95}{\ensuremath{\mathcal{A}_{i}}}$ be the resulting
lattice which is matroidal. Let $\scalebox{0.95}{\ensuremath{\Phi=\bigl\{\mathcal{A}_{i}^{\lessdot M_{i}}:i\in\left[m\right]\bigr\}}}$,
then $\scalebox{0.95}{\ensuremath{\Phi}}$ is a compatible collection
of matroidal arrangements. Define 
\[
\scalebox{0.95}{\ensuremath{\mathrm{HA}_{\Sigma}:=\mathrm{GA}_{\Phi}}}.
\]
 Then, $\scalebox{0.95}{\ensuremath{\left|\mathrm{HA}_{\Sigma}\right|}}$
is an arrangement isomorphic to $\scalebox{0.95}{\ensuremath{\mathrm{HA}_{M}}}$
with $\scalebox{0.95}{\ensuremath{M=\bigcup_{i\in\left[m\right]}M_{i}}}$.
With this, we will abuse terminology and refer to $\scalebox{0.95}{\ensuremath{\mathrm{HA}_{\Sigma}}}$
and $\scalebox{0.95}{\ensuremath{\left|\mathrm{HA}_{\Sigma}\right|}}$
as the \textbf{gluing} and \textbf{merging} of the hyperplane arrangements
$\scalebox{0.95}{\ensuremath{\mathrm{HA}_{M_{i}}}}$, respectively.
We will also say that $\scalebox{0.95}{\ensuremath{\mathrm{HA}_{M_{i}}}}$
are \textbf{compatible}.
\begin{example}
\label{exa:GA for D(2,4)}Let $\scalebox{0.95}{\ensuremath{M}}$ and
$\scalebox{0.95}{\ensuremath{N}}$ be connected $\scalebox{0.95}{\ensuremath{\left(2,4\right)}}$-matroids
whose collections of rank-$1$ flats are $\scalebox{0.95}{\ensuremath{\left\{ 1,2,34\right\} }}$
and $\scalebox{0.95}{\ensuremath{\left\{ 12,3,4\right\} }}$, respectively,
where we write $\scalebox{0.95}{\ensuremath{34}}$ and $\scalebox{0.95}{\ensuremath{12}}$
for $\scalebox{0.95}{\ensuremath{\left\{ 3,4\right\} }}$ and $\scalebox{0.95}{\ensuremath{\left\{ 1,2\right\} }}$.\footnote{For $\scalebox{0.95}{\ensuremath{\left[n\right]}}$ with $\scalebox{0.95}{\ensuremath{n\le9}}$,
we denote a subset $\scalebox{0.95}{\ensuremath{\left\{ s_{1},\dots,s_{m}\right\} }}$
by $\scalebox{0.95}{\ensuremath{s_{1}\cdots s_{m}}}$ for convenience.} The tiling $\scalebox{0.95}{\ensuremath{\Sigma=\left\{ \mathrm{BP}_{M},\mathrm{BP}_{N}\right\} }}$
is a complete tiling which is, up to symmetry, the unique nontrivial
matroid subdivision of $\scalebox{0.95}{\ensuremath{\Delta_{4}^{2}}}$.
This is because there is, up to symmetry, only one partition of $\scalebox{0.95}{\ensuremath{\left[4\right]}}$
into subsets of size $\scalebox{0.95}{\ensuremath{\ge2}}$. See Figure
\ref{fig:Glue(2,4)}.
\begin{figure}[H]
\noindent \begin{centering}
\noindent \begin{center}
\def\sizea{0.65}


\par\end{center}\vspace{-0.25em}
\par\end{centering}
\noindent \centering{}\caption{\label{fig:Glue(2,4)}Gluing and merging of compatible $(2,4)$-arrangements.}
\end{figure}
\end{example}

Let $\scalebox{0.95}{\ensuremath{\Sigma}}$ be a semitiling and $\scalebox{0.95}{\ensuremath{\Psi}}$
its associated semipuzzle. The \textbf{dual graph} of $\scalebox{0.95}{\ensuremath{\Psi}}$
is defined as a graph that has a vertex corresponding to each polytope
of $\scalebox{0.95}{\ensuremath{\Sigma}}$ and an edge joining two
polytopes having a common facet.
\begin{rem}
\label{rem:Rank-2-tiling}For a rank-$2$ tiling $\scalebox{0.95}{\ensuremath{\Sigma=\left\{ \mathrm{BP}_{M},\mathrm{BP}_{N}\right\} }}$,
any two realizations of $\scalebox{0.95}{\ensuremath{\mathrm{HA}_{M}}}$
and $\scalebox{0.95}{\ensuremath{\mathrm{HA}_{N}}}$, respectively,
can be glued so that their gluing is also a realization of $\scalebox{0.95}{\ensuremath{\mathrm{HA}_{\Sigma}}}$,
cf. Example \ref{exa:rank2-realizable}\eqref{enu:realizable} and
Subsection \ref{subsec:Rank-2-tilings}. Because every rank-$2$ semitiling
is a coherent convex tiling, every rank-$2$ \emph{tropical linear
space} arises as a polyhedral complex for which the bounded part of
its $1$-skeleton is combinatorially the dual graph of a certain rank-$2$
puzzle, and each ray of the unbounded part is outward-pointing normal
to a facet of the puzzle. Here, the dual graph is a tree.
\end{rem}

\begin{example}
\label{exa:GA for D(3,5)}Consider the line arrangements ($\scalebox{0.95}{\ensuremath{k=3}}$)
$\scalebox{0.95}{\ensuremath{\mathrm{HA}_{M_{i}}}}$, $\scalebox{0.95}{\ensuremath{i\in\left[3\right]}}$,
shown in Figure \ref{fig:Glue(3,5)},\footnote{In the figure for $\scalebox{0.95}{\ensuremath{\mathrm{HA}_{M_{1}}}}$,
a double line is drawn to represent $\scalebox{0.95}{\ensuremath{M_{1}/\left\{ 1,2\right\} }}$.
For a rank-$3$ matroid $\scalebox{0.95}{\ensuremath{M}}$, we will
draw an $m$-fold line to represent a line $\scalebox{0.95}{\ensuremath{M/F}}$
or a line-piece $\scalebox{0.95}{\ensuremath{M\left(F\right)}}$ with
$\scalebox{0.95}{\ensuremath{\left|F\right|=m}}$, which makes sense
because parallel for lines in a projective plane means identical.} let $\scalebox{0.95}{\ensuremath{\Sigma=\left\{ \mathrm{BP}_{M_{i}}:i\in\left[3\right]\right\} }}$,
then $\scalebox{0.95}{\ensuremath{\Sigma}}$ is a complete tiling
with $\scalebox{0.95}{\ensuremath{\mathrm{MA}_{\left|\Sigma\right|}=U_{5}^{3}}}$.
In the meanwhile, all arrangements appearing in this example are realizable,
cf.~Remarks \ref{rem:realizable} and \ref{rem:realizable-1}. See
also Figure \ref{fig:mmp-4}. 
\end{example}

\begin{defn}
A semitiling is said to be \textbf{indecomposable} if no member can
be further subdivided into smaller ones to produce another semitiling.
\end{defn}

\begin{example}
\label{exa:GA for D(3,6)}Assume that $\scalebox{0.95}{\ensuremath{\Phi}}$
is a $\scalebox{0.95}{\ensuremath{\left(3,6\right)}}$-semipuzzle
with a relevant point-piece. There are exactly four $\scalebox{0.95}{\ensuremath{\left(3,6\right)}}$-puzzle-pieces
that contain the point-piece, up to symmetry, shown in Figure \ref{fig:Splitting-1}.
The first one has exactly two nontrivial subdivisions, as illustrated
in Figure \ref{fig:Splitting}. Applying Theorem \ref{thm:Uniform-finiteness},
we obtain exactly six (inclusionwise) maximal indecomposable semipuzzles
shown in Figure \ref{fig:MaxRel (3,6)},\footnote{In the bottom-center puzzle-piece of the puzzle at the $\scalebox{0.95}{\ensuremath{\left(2,2\right)}}$-position,
a $\scalebox{0.95}{\ensuremath{Y}}$-shaped figure is drawn to represent
a line-piece for a technical reason.} up to symmetry, which are complete puzzles; $\scalebox{0.95}{\ensuremath{\Phi}}$
is obtained from one of them by selecting puzzle-pieces from it and
merging two of them if necessary, as depicted in Figure \ref{fig:Splitting}
(subdividing is the inverse operation of merging). This implies that
every $\scalebox{0.95}{\ensuremath{\left(3,6\right)}}$-semipuzzle
with a relevant point-piece is a puzzle, and the number of relevant
point-pieces it contains is $1$.\vspace{-0.5cm}

\begin{figure}[H]
\noindent \begin{centering}
\noindent \begin{center}
\def\sizea{0.24}


\par\end{center}\vspace{-0.25cm}
\par\end{centering}
\noindent \centering{}\caption{\label{fig:Glue(3,5)}Gluing and merging of compatible $(3,5)$-arrangements.}
\end{figure}
\vspace{-0.25cm}

\begin{figure}[H]
\noindent \begin{centering}
\noindent \begin{center}
\def\size{0.28}
\def\ratioa{0.901}
\def\ratior{0.851}
\def\ratiow{0.761}
\def\ratioe{0.881}

%
\par\end{center}\vspace{-0.2em}
\par\end{centering}
\noindent \centering{}\caption{\label{fig:MaxRel (3,6)}All indecomposable complete $(3,6)$-semipuzzles
that contain a relevant point-piece, up to symmetry.}
\end{figure}
\vspace{-0.25cm}

\begin{figure}[H]
\noindent \begin{centering}
\noindent \begin{center}
\def\size{0.25}
\def\sizo{0.23}
\def\ratioa{0.901}
\def\ratior{0.851}
\def\ratiow{0.761}
\def\ratioe{0.881}

%
\par\end{center}\vspace{-6pt}
\par\end{centering}
\noindent \centering{}\caption{\label{fig:Splitting-1}All $(3,6)$-puzzle-pieces that contain a
relevant point-piece, up to symmetry.}
\end{figure}
\vspace{-0.25cm}

\begin{figure}[H]
\noindent \begin{centering}
\noindent \begin{center}
\def\size{0.25}
\def\sizo{0.23}
\def\ratioa{0.901}
\def\ratior{0.851}
\def\ratiow{0.761}
\def\ratioe{0.881}

%
\par\end{center}\vspace{-8pt}
\par\end{centering}
\noindent \centering{}\caption{\label{fig:Splitting}All nontrivial subdivisions of a $(3,6)$-puzzle-piece.}
\end{figure}

\end{example}

\begin{example}
\label{exa:GA for D(3,6)-1}There is exactly one maximal indecomposable
$\scalebox{0.95}{\ensuremath{\left(3,6\right)}}$-semipuzzle, up to
symmetry, that contains no relevant point-piece; see Figure \ref{fig:MaxIrr (3,6)}.
This is a complete puzzle.\vspace{-0.25em}

\begin{figure}[H]
\noindent \centering{}\noindent \begin{center}
\def\size{0.28}
\def\ratioa{0.901}
\def\ratior{0.851}
\def\ratioh{1.851}
\def\ratiow{0.761}
\def\ratioe{0.881}


\par\end{center}\caption{\label{fig:MaxIrr (3,6)}The indecomposable complete $(3,6)$-semipuzzle
that contains \emph{no} relevant point-piece, up to symmetry.}
\end{figure}

\end{example}

\begin{rem}[]
\label{rem:puzzle}
\begin{enumerate}
\item Examples \ref{exa:GA for D(3,6)} and \ref{exa:GA for D(3,6)-1} imply
that every $\scalebox{0.95}{\ensuremath{\left(3,6\right)}}$-semipuzzle
is a puzzle and a coherent puzzle that extends to a complete puzzle.
\item The dual graphs of seven puzzles of Examples \ref{exa:GA for D(3,6)}
and \ref{exa:GA for D(3,6)-1} are the seven types of generic tropical
planes in $\scalebox{0.95}{\ensuremath{\mathbb{TP}^{5}}}$ of \cite[Figure 1]{Tropical}.
Here, we note that our result is not computer-aided and simple.
\item Figures \ref{fig:Glue(2,4)}, \ref{fig:Glue(3,5)}, \ref{fig:MaxRel (3,6)},
and \ref{fig:MaxIrr (3,6)} visualize that some algebraic varieties
improve their singularities by getting glued together to pass their
singularities into better ones.
\end{enumerate}
\end{rem}

\subsection{\label{subsec:weighted-stable}Abstract weighted stable hyperplane
arrangements}

Consider rank-$k$ connected matroids $\scalebox{0.95}{\ensuremath{M_{1},\dots,M_{m}}}$
on $\scalebox{0.95}{\ensuremath{S}}$ and assume $\scalebox{0.95}{\ensuremath{\Psi=\bigl\{\mathcal{A}_{i}^{\lessdot M_{i}}:i\in\left[m\right]\bigr\}}}$
is a collection of compatible matroidal arrangements. For a weight
vector $\mathbf{w}$, if the associated semitiling of $\scalebox{0.95}{\ensuremath{\Psi}}$
is a $\mathbf{w}$-tiling and every member of $\scalebox{0.95}{\ensuremath{\bigcup_{i\in\left[m\right]}\mathcal{A}_{i}}}$
is $\mathbf{w}$-stable (cf. Definitions \ref{def:w-stability} and
\ref{def:w-tiling}), we refer to $\scalebox{0.95}{\ensuremath{\mathrm{GA}_{\Psi}}}$
as an \textbf{abstract stable hyperplane arrangement weighted by $\mathbf{w}$},
or simply an \textbf{abstract $\mathbf{w}$-stable hyperplane arrangement}.
When $\scalebox{0.95}{\ensuremath{\mathbf{w}=\mathbbm1}}$, we can
omit ``weighted by $\scalebox{0.95}{\ensuremath{\mathbbm1}}$''
or ``$\scalebox{0.95}{\ensuremath{\mathbbm1}}$-.''

\section{\label{sec:Tiling Extension}Extension of Tiling and Puzzle}

In Subsection \ref{subsec:Rank-2-tilings}, we achieved a complete
understanding of the extension of rank-$\scalebox{0.95}{\ensuremath{2}}$
matroid tilings. Nevertheless, extending tilings of higher rank is
generally a difficult task. In this section, we focus on the extension
of rank-$\scalebox{0.95}{\ensuremath{3}}$ tilings, which is already
considered interesting due to Mnëv's universality theorem. Since puzzles
become two-dimensional when $\scalebox{0.95}{\ensuremath{k=3}}$,
we can perform manual computations and enjoy the task of extending
matroid tilings as if they were fun jigsaw puzzles. However, not all
rank-$\scalebox{0.95}{\ensuremath{3}}$ tilings are extendable,  and
we  limit our discussion to a specific class of rank-$\scalebox{0.95}{\ensuremath{3}}$
tilings that we call \emph{admissible}.

Throughout this section, facets and codimension-$\scalebox{0.95}{\ensuremath{2}}$
cells will be considered loopless by default unless otherwise stated.
Additionally, given a semitiling $\scalebox{0.95}{\ensuremath{\Sigma}}$,
we assume the prior existence of a semitiling $\scalebox{0.95}{\ensuremath{\Sigma'\subseteq\Sigma}}$
with prescribed conditions. This semitiling $\scalebox{0.95}{\ensuremath{\Sigma'}}$
is a union of maximal subtilings of $\scalebox{0.95}{\ensuremath{\Sigma}}$
at codimension-$\scalebox{0.95}{\ensuremath{2}}$ cells (centers).
If all involved semitilings are tilings, then a maximal subtiling
at a center exists uniquely, and this assumption is unnecessary.

\subsection{Admissible tiling}

We define the concept of ``admissibility'' for rank-$\scalebox{0.95}{\ensuremath{3}}$
tilings by emulating weighted tilings. Defining ``admissibility''
for higher ranks is a problem that should be approached with caution
in the future.\vspace{2pt}

Let $\scalebox{0.95}{\ensuremath{\Sigma}}$ be a rank-$\scalebox{0.95}{\ensuremath{3}}$
semitiling with ground set $\scalebox{0.95}{\ensuremath{S}}$. The
all-one vector $\mathbbm1\in\mathbb{R}^{S}$ is normal to the hypersimplex
$\scalebox{0.95}{\ensuremath{\Delta_{S}^{3}}}$. We say that two facets
$\scalebox{0.95}{\ensuremath{R}}$ and $\scalebox{0.95}{\ensuremath{R'}}$
of $\scalebox{0.95}{\ensuremath{\Sigma}}$ are \textbf{parallel} if
their normal vectors are parallel modulo $\mathbbm1$. An ordered
collection $\scalebox{0.95}{\ensuremath{\left(R_{1},\dots,R_{m}\right)}}$
of parallel facets of $\scalebox{0.95}{\ensuremath{\Sigma}}$ with
$\scalebox{0.95}{\ensuremath{m\ge2}}$ is said to be a\textbf{ parallel
sequence} if the intersection of two adjacent cells is a codimension-$\scalebox{0.95}{\ensuremath{2}}$
cell (assuming that $\scalebox{0.95}{\ensuremath{R_{i}}}$ for $\scalebox{0.95}{\ensuremath{i\in\left[m\right]}}$
are facets of a sub-semitiling of $\scalebox{0.95}{\ensuremath{\Sigma}}$
that is a union of maximal subtilings of $\scalebox{0.95}{\ensuremath{\Sigma}}$
at those centers). A singleton $\scalebox{0.95}{\ensuremath{\left(R_{1}\right)}}$
is regarded as a parallel sequence.

A maximal subtiling of $\scalebox{0.95}{\ensuremath{\Sigma}}$ is
referred to as a \textbf{socket}. We call a socket \textbf{non-convex}
if its center is \emph{relevant} and it has a deficiency of $\scalebox{0.95}{\ensuremath{1}}$
or $\scalebox{0.95}{\ensuremath{2}}$. A non-convex socket provides
sufficient information for constructing a line arrangement of \ref{enu:(LA1)},
\ref{enu:(LA2)}, \ref{enu:(LA3)}, or \ref{enu:(LA4)}. Note that
if an arrangement of \ref{enu:(LA1)} or \ref{enu:(LA2)} is constructed
from a socket, it is unique. Also, two arrangements of \ref{enu:(LA3)}
and \ref{enu:(LA4)}, respectively, can be obtained from the same
socket.

A \textbf{filler} is a puzzle-piece or a   base polytope whose matroid
is the same as that of the line arrangement constructed from a socket,
non-convex or not. In particular, a filler is said to be of \textbf{kind-}\textbf{\emph{i}}
if its associated line arrangement is constructed from $\left(\text{LA}i\right)$.
Given a socket, a filler is constructed so that the socket together
with the filler becomes a tiling with less deficiency at the center.
This process is repeated until the deficiency becomes $\scalebox{0.95}{\ensuremath{0}}$,
which we call \textbf{socket-filling}.

For a nontrivial extension $\scalebox{0.95}{\ensuremath{\Sigma\cup\left\{ P\right\} }}$
of $\scalebox{0.95}{\ensuremath{\Sigma}}$, if $\scalebox{0.95}{\ensuremath{R}}$
is a facet of $\scalebox{0.95}{\ensuremath{\Sigma}}$, but not a facet
of the extension, we say that $\scalebox{0.95}{\ensuremath{R}}$ is
\textbf{saturated}. If the center of a socket in $\scalebox{0.95}{\ensuremath{\Sigma}}$
with nonzero deficiency becomes a $\scalebox{0.95}{\ensuremath{0}}$-deficiency
center, we say that the socket is \textbf{saturated} and the center
is \textbf{complete}. Multiple sockets can be filled simultaneously.
Note that we fill sockets to get an extension of a semitiling.
\begin{defn}
\label{def:regular semitiling} A rank-$\scalebox{0.95}{\ensuremath{3}}$
semitiling is called \textbf{admissible} if 
\begin{enumerate}
\item \label{enu:regular-cond1}there is no non-convex socket of deficiency
$\scalebox{0.95}{\ensuremath{1}}$, and 
\item \label{enu:regular-cond2}for any parallel collection of type-$\scalebox{0.95}{\ensuremath{2}}$
facets connected in codimension~$\scalebox{0.95}{\ensuremath{2}}$,
the number of non-convex sockets of deficiency $\scalebox{0.95}{\ensuremath{2}}$
with centers contained in their support is at most $\scalebox{0.95}{\ensuremath{1}}$.
\end{enumerate}
\end{defn}

Because a weighted tiling can have a non-convex socket of deficiency
$\scalebox{0.95}{\ensuremath{1}}$, not every weighted tiling is an
admissible tiling. However:
\begin{lem}
\label{lem:wt<adm}Every rank-$\scalebox{0.95}{\ensuremath{3}}$ weighted
tiling extends to an admissible semitiling.
\end{lem}

\begin{proof}
Let $\scalebox{0.95}{\ensuremath{\Sigma}}$ be a weighted $\scalebox{0.95}{\ensuremath{\left(3,S\right)}}$-tiling.
A non-convex socket of deficiency $\scalebox{0.95}{\ensuremath{1}}$
is, up to symmetry, one of the four figures shown in Figure \ref{fig:Socket-filling},
cf. Figures \ref{fig:non-convex-1}, \ref{fig:non-convex-2}, and
\ref{fig:non-convex-3}. We focus on the ``leg'' shaded blue.
\begin{figure}[H]
\begin{spacing}{0}
\noindent \begin{centering}
\noindent \begin{center}
\def\size{0.5}
\def\gapa{0.1}
\def\gapb{0.15}


\par\end{center}\vspace{-16pt}
\par\end{centering}
\noindent \centering{}\caption{\label{fig:Socket-filling}Extending a weighted tiling to an admissible
semitiling.}
\end{spacing}
\end{figure}
\vspace{6pt}

Given a non-convex socket of deficiency $\scalebox{0.95}{\ensuremath{1}}$,
let $\scalebox{0.95}{\ensuremath{R_{0}}}$ be a type-$\scalebox{0.95}{\ensuremath{2}}$
facet containing the center and $\scalebox{0.95}{\ensuremath{P_{0}\in\Sigma}}$
the member containing $\scalebox{0.95}{\ensuremath{R_{0}}}$. The
socket determines a (unique) kind-$\scalebox{0.95}{\ensuremath{1}}$
filler, say $\scalebox{0.95}{\ensuremath{P}}$,  and we  fill the
socket with this filler. By Lemma \ref{lem:degen-rank1-flat}, we
know that the number of $\scalebox{0.95}{\ensuremath{Q\lessdot R_{0}}}$
with $\scalebox{0.95}{\ensuremath{\mathrm{ang}_{Q}(P_{0})=1}}$ is
at most $\scalebox{0.95}{\ensuremath{1}}$. According to Proposition
\ref{prop:rank3-weighted}, the number of $\scalebox{0.95}{\ensuremath{Q\lessdot R_{0}}}$
with $\scalebox{0.95}{\ensuremath{\mathrm{ang}_{Q}(\Sigma)=3}}$ is
also at most $\scalebox{0.95}{\ensuremath{1}}$. If the number is
$\scalebox{0.95}{\ensuremath{1}}$, we have either $\scalebox{0.95}{\ensuremath{\mathrm{ang}_{Q}(P)=1}}$
or $\scalebox{0.95}{\ensuremath{\mathrm{ang}_{Q}(P)=2}}$ with $\scalebox{0.95}{\ensuremath{\mathrm{ang}_{Q}(\Sigma\cup\left\{ P\right\} )=4}}$
or $\scalebox{0.95}{\ensuremath{\mathrm{ang}_{Q}(\Sigma\cup\left\{ P\right\} )=5}}$,
respectively. In the latter case, $\scalebox{0.95}{\ensuremath{Q}}$
is the center of the new socket with deficiency $\scalebox{0.95}{\ensuremath{1}}$,
 and we  repeat the previously described socket-filling. In the former
case, we move on to other socket with deficiency $\scalebox{0.95}{\ensuremath{1}}$.
Note that each single-member extension produces at most two sockets
with deficiency $\scalebox{0.95}{\ensuremath{1}}$. This process terminates
due to the finiteness of a semitiling. The resulting semitiling has
no non-convex sockets of deficiency $\scalebox{0.95}{\ensuremath{1}}$
and satisfies Definition \ref{def:regular semitiling}\eqref{enu:regular-cond2},
which makes it admissible.
\end{proof}
Let $\scalebox{0.95}{\ensuremath{\Sigma}}$ be a rank-$\scalebox{0.95}{\ensuremath{3}}$
admissible semitiling on $\scalebox{0.95}{\ensuremath{S}}$ with a
relevant facet $\scalebox{0.95}{\ensuremath{R_{1}}}$. Then, $\scalebox{0.95}{\ensuremath{\left\lfloor \mathrm{MA}_{R_{1}}\right\rfloor }}$
is a rank-$\scalebox{0.95}{\ensuremath{2}}$ connected matroid. Let
$f_{1}$ be the simplification map for $\scalebox{0.95}{\ensuremath{\left\lfloor \mathrm{MA}_{R_{1}}\right\rfloor }}$.
Let $\scalebox{0.95}{\ensuremath{L_{0}:=S-E\left(\left\lfloor \mathrm{MA}_{R_{1}}\right\rfloor \right)}}$
and $\scalebox{0.95}{\ensuremath{R_{1}\lessdot P_{1}\in\Sigma}}$.
We construct a filler $\scalebox{0.95}{\ensuremath{\mathrm{BP}_{N}}}$
that is face-fitting to $\scalebox{0.95}{\ensuremath{P_{1}}}$ as
follows.
\begin{enumerate}[label=(LA\arabic*),start=5]
\item \label{enu:(LA5)}If $\scalebox{0.95}{\ensuremath{R_{1}}}$ is a
type-$\scalebox{0.95}{\ensuremath{2}}$ facet, let $\scalebox{0.95}{\ensuremath{f_{0}:L_{0}\rightarrow\left\{ L_{0}\right\} }}$
be a constant map. Let $f:=f_{0}\oplus f_{1}$ which is a map on $\scalebox{0.95}{\ensuremath{S}}$.
Let $\scalebox{0.95}{\ensuremath{N:=f^{\ast}\left(f_{\ast}\left(U_{S}^{3}\right)\right)}}$. 
\item \label{enu:(LA6)}If $\scalebox{0.95}{\ensuremath{R_{1}}}$ is a type-$\scalebox{0.95}{\ensuremath{1}}$
facet, let $\scalebox{0.95}{\ensuremath{f_{0}}}$ be the identity
map on $\scalebox{0.95}{\ensuremath{L_{0}}}$. Let $f:=f_{0}\oplus f_{1}$
and $\scalebox{0.95}{\ensuremath{N:=f^{\ast}\left(f_{\ast}\left(U_{S}^{3}\right)\right)}}$.
\end{enumerate}
In either case, $\scalebox{0.95}{\ensuremath{N}}$ is a rank-$\scalebox{0.95}{\ensuremath{3}}$
connected matroid on $\scalebox{0.95}{\ensuremath{S}}$, and $\scalebox{0.95}{\ensuremath{\mathrm{BP}_{N}}}$
is face-fitting to $\scalebox{0.95}{\ensuremath{P_{1}}}$. Moreover,
$\scalebox{0.95}{\ensuremath{N}}$ is a free product of matroids and
is realizable because any matroid of rank $\scalebox{0.95}{\ensuremath{1}}$
or $\scalebox{0.95}{\ensuremath{2}}$ is realizable, cf. Example \ref{exa:rank2-realizable}
and Remark \ref{rem:realizable}. Again, a filler is said to be of
\textbf{kind-}\textbf{\emph{i}} if it is constructed from $\left(\text{LA}i\right)$.

\subsection{Extension of admissible tiling without any forked facets}

The number of relevant codimension-$\scalebox{0.95}{\ensuremath{2}}$
cells is a crucial factor. For example, when this number is $\scalebox{0.95}{\ensuremath{0}}$,
the given semitiling is a convex tiling (Lemma \ref{lem:loc-conv-relev}).
Let $\scalebox{0.95}{\ensuremath{\Sigma}}$ be a rank-$\scalebox{0.95}{\ensuremath{3}}$
semitiling with a codimension-$\scalebox{0.95}{\ensuremath{1}}$ cell
$\scalebox{0.95}{\ensuremath{R}}$. We say that $\scalebox{0.95}{\ensuremath{R}}$
is \textbf{forked} if it has at least three relevant codimension-$\scalebox{0.95}{\ensuremath{2}}$
cells, and \textbf{planar} if it does not.

Note that for an admissible semitiling without any forked facets,
every collection of parallel facets connected in codimension $\scalebox{0.95}{\ensuremath{2}}$
arises as a parallel sequence. A parallel sequence $\scalebox{0.95}{\ensuremath{\left(R_{1},\dots,R_{m}\right)}}$
is said to be \textbf{maximal at }$\scalebox{0.95}{\ensuremath{Q\lessdot R_{m}}}$
if, for the socket centered at $\scalebox{0.95}{\ensuremath{Q}}$,
the angle is different from $\scalebox{0.95}{\ensuremath{3}}$.

Algorithm~\ref{alg:extend_tilings} below consists of three steps
that produce a complete extension of an admissible semitiling with
no forked facets. In the first two steps, the algorithm finds an extension
that does not contain any type-$\scalebox{0.95}{\ensuremath{2}}$
(relevant) facets. In the final step, it finds an extension without
any type-$\scalebox{0.95}{\ensuremath{2}}$ or type-$\scalebox{0.95}{\ensuremath{1}}$
relevant facets, which is therefore a complete tiling.
\begin{lyxalgorithm}
\textup{\label{alg:extend_tilings}Let $\scalebox{0.95}{\ensuremath{\Sigma}}$
be an admissible semitiling without any forked facets.}
\end{lyxalgorithm}

\begin{enumerate}[label=\textbf{Step\thinspace\arabic*}]
\item \label{enu:Alg Step 1}If $\scalebox{0.95}{\ensuremath{\Sigma}}$
has no non-convex sockets (of deficiency $\scalebox{0.95}{\ensuremath{2}}$),
then go to \ref{enu:Alg Step 2}. Otherwise, for a non-convex socket,
let $\scalebox{0.95}{\ensuremath{Q_{1}}}$ be its center and $\scalebox{0.95}{\ensuremath{Q_{m}}}$
be the center of a socket connected to $\scalebox{0.95}{\ensuremath{Q_{1}}}$
by a parallel sequence $\scalebox{0.95}{\ensuremath{\left(R_{1},\dots,R_{m-1}\right)}}$
of type-$\scalebox{0.95}{\ensuremath{1}}$ facets with $\scalebox{0.95}{\ensuremath{Q_{1}\lessdot R_{1}}}$
and $\scalebox{0.95}{\ensuremath{Q_{m}\lessdot R_{m-1}}}$ that is
maximal at $\scalebox{0.95}{\ensuremath{Q_{m}}}$; see Figure \ref{fig:Step1-1-1}.
Recursively fill the sockets with kind-$\scalebox{0.95}{\ensuremath{4}}$
fillers as illustrated in Figure \ref{fig:Step1-2}. All facets of
those fillers are planar. Also, any irrelevant center $\scalebox{0.95}{\ensuremath{Q}}$
shared by the fillers has deficiency $\scalebox{0.95}{\ensuremath{0}}$.
The deficiency of the new socket at $\scalebox{0.95}{\ensuremath{Q_{m}}}$
is $\scalebox{0.95}{\ensuremath{1}}$ or $\scalebox{0.95}{\ensuremath{\ge3}}$. 
\begin{enumerate}
\item In the latter case, the extension is admissible. Repeat \ref{enu:Alg Step 1}.
\item If the former case, recursively fill the sockets of deficiency $\scalebox{0.95}{\ensuremath{1}}$
with kind-$\scalebox{0.95}{\ensuremath{1}}$ fillers as in Lemma \ref{lem:wt<adm}
until an admissible extension is obtained. Repeat \ref{enu:Alg Step 1}.
\begin{figure}[H]
\begin{spacing}{0}
\noindent \centering{}\noindent \begin{center}
\def\size{0.5}
\def\rat{0.85}


\par\end{center}\caption{\label{fig:Step1-1-1}Step 1-1}
\end{spacing}
\end{figure}
\begin{figure}[H]
\begin{spacing}{0}
\noindent \centering{}\noindent \begin{center}
\def\size{0.5}
\def\rat{0.85}

%
\par\end{center}\caption{\label{fig:Step1-2}Step 1-2}
\end{spacing}
\end{figure}
\vspace{6pt}
\end{enumerate}
\item \label{enu:Alg Step 2}The given semitiling $\scalebox{0.95}{\ensuremath{\Sigma}}$
has no non-convex sockets and is locally convex. So, it is a convex
tiling. If it has no type-$\scalebox{0.95}{\ensuremath{2}}$ facets
as well, go to \ref{enu:Alg Step 3}. Otherwise, let $\scalebox{0.95}{\ensuremath{\left(R_{m-1},\dots,R_{1}\right)}}$
be a parallel sequence of type-$\scalebox{0.95}{\ensuremath{2}}$
facets that is maximal at $\scalebox{0.95}{\ensuremath{Q_{1}}}$;
see Figure \ref{fig:Step2-1}. Construct a kind-$\scalebox{0.95}{\ensuremath{5}}$
filler $\scalebox{0.95}{\ensuremath{P}}$ from $\scalebox{0.95}{\ensuremath{R_{1}}}$
and add it to $\scalebox{0.95}{\ensuremath{\Sigma}}$, then $\scalebox{0.95}{\ensuremath{\Sigma\cup\left\{ P\right\} }}$
is admissible; see Figure \ref{fig:Step2-2}. No forked facets are
created. Go to \ref{enu:Alg Step 1}. 
\begin{figure}[H]
\noindent \centering{}\noindent \begin{center}
\def\size{0.5}
\def\rat{0.85}


\par\end{center}\caption{\label{fig:Step2-1}Step 2-1}
\end{figure}
\begin{figure}[H]
\noindent \centering{}\noindent \begin{center}
\def\size{0.5}
\def\rat{0.85}

%
\par\end{center}\caption{\label{fig:Step2-2}Step 2-2}
\end{figure}
\item \label{enu:Alg Step 3}The given semitiling $\scalebox{0.95}{\ensuremath{\Sigma}}$
is a convex tiling without any type-$\scalebox{0.95}{\ensuremath{2}}$
facets. If it has no type-$\scalebox{0.95}{\ensuremath{1}}$ relevant
facets, then it is complete and \emph{terminate} the algorithm. Otherwise,
let $\scalebox{0.95}{\ensuremath{\left(R_{m-1},\dots,R_{1}\right)}}$
be a parallel sequence of type-$\scalebox{0.95}{\ensuremath{1}}$
relevant facets that is maximal at $\scalebox{0.95}{\ensuremath{Q_{1}}}$
with $\scalebox{0.95}{\ensuremath{\left(R_{1},\dots,R_{m-1}\right)}}$
maximal at $\scalebox{0.95}{\ensuremath{Q_{m}}}$; see Figure \ref{fig:Step3-1}.
\begin{enumerate}[topsep=3pt]
\item Construct a kind-$\scalebox{0.95}{\ensuremath{6}}$ filler $\scalebox{0.95}{\ensuremath{P}}$
from $\scalebox{0.95}{\ensuremath{R_{1}}}$ and add it to $\scalebox{0.95}{\ensuremath{\Sigma}}$.
No forked facets are created.
\item If $\scalebox{0.95}{\ensuremath{m=2}}$, then $\scalebox{0.95}{\ensuremath{\Sigma\cup\left\{ P\right\} }}$
is a convex tiling without any type-$\scalebox{0.95}{\ensuremath{2}}$
facets. If $\scalebox{0.95}{\ensuremath{m\ge3}}$, recursively fill
the sockets with kind-$\scalebox{0.95}{\ensuremath{2}}$ fillers until
a convex extension without any type-$\scalebox{0.95}{\ensuremath{2}}$
facets is obtained. In either case, all facets of the extension are
planar; see Figure \ref{fig:Step3-2}. Repeat \ref{enu:Alg Step 3}.
\begin{figure}[H]
\begin{spacing}{0}
\noindent \centering{}\noindent \begin{center}
\def\size{0.5}
\def\rat{0.85}


\par\end{center}\caption{\label{fig:Step3-1}Step 3-1}
\end{spacing}
\end{figure}
\begin{figure}[H]
\noindent \centering{}\noindent \begin{center}
\def\size{0.5}
\def\rat{0.85}

%
\par\end{center}\caption{\label{fig:Step3-2}Step 3-2}
\end{figure}
\end{enumerate}
\end{enumerate}
Thus, we have the following theorem.
\begin{thm}
\label{thm:unforked}An admissible semitiling without any forked facets
is a tiling.
\end{thm}

\begin{proof}
By Algorithm \ref{alg:extend_tilings}, every admissible semitiling
without any forked facets can be extended to a complete tiling, which
makes it a tiling.
\end{proof}

\subsection{Complete extension of admissible tilings}

We demonstrate that every admissible $\scalebox{0.95}{\ensuremath{\left(3,n\right)}}$-tiling
with $\scalebox{0.95}{\ensuremath{n=9}}$ extends to a complete tiling.
The method used for $\scalebox{0.95}{\ensuremath{n=9}}$ is applicable
to all smaller values of $n$. We perform the extension process manually,
using only pen and paper. To avoid redundancy, we will not reproduce
information that is already presented pictorially.

\smallskip{}

Let $\scalebox{0.95}{\ensuremath{\Sigma}}$ be an admissible $\scalebox{0.95}{\ensuremath{\left(3,9\right)}}$-tiling
and $\scalebox{0.95}{\ensuremath{\Psi}}$ its associated puzzle. If
$\scalebox{0.95}{\ensuremath{\Sigma}}$ does not have any forked facets,
we can extend it to a complete tiling using Algorithm \ref{alg:extend_tilings}.
Otherwise, suppose that $\scalebox{0.95}{\ensuremath{\Sigma}}$ has
a forked facet $\scalebox{0.95}{\ensuremath{\mathrm{BP}_{M\left(F\right)}}}$
for $\scalebox{0.95}{\ensuremath{\mathrm{BP}_{M}\in\Sigma}}$ and
a non-degenerate flat $\scalebox{0.95}{\ensuremath{F}}$ of $\scalebox{0.95}{\ensuremath{M}}$.
By Lemma \ref{lem:general position}, there are exactly six cases
for $\scalebox{0.95}{\ensuremath{M/F}}$ as presented in Figure \ref{fig:forked-facet}
where in the first three cases, $\scalebox{0.95}{\ensuremath{F=L_{1}\cup\cdots\cup L_{m}}}$
with $\scalebox{0.95}{\ensuremath{m=3,4}}$ is a rank-$\scalebox{0.95}{\ensuremath{2}}$
flat of $\scalebox{0.95}{\ensuremath{M}}$. In the last three cases,
$\scalebox{0.95}{\ensuremath{F}}$ is a rank $\scalebox{0.95}{\ensuremath{1}}$
flat of $\scalebox{0.95}{\ensuremath{M}}$ with $\scalebox{0.95}{\ensuremath{\left[9\right]=F\cup L_{1}\cup\cdots\cup L_{m}}}$.
Below, we assume that the first case of Figure \ref{fig:forked-facet}
occurs and examine the shape of the semipuzzle $\scalebox{0.95}{\ensuremath{\Psi}}$.
\begin{figure}[H]
\noindent \centering{}\noindent \begin{center}
\def\size{0.5}
\def\rat{0.85}


\par\end{center}\caption{\label{fig:forked-facet}$M/F$ with forked facet $\mathrm{BP}_{M(F)}$.}
\end{figure}

Because the forked facet $\scalebox{0.95}{\ensuremath{\mathrm{BP}_{M\left(F\right)}}}$
is a type-$\scalebox{0.95}{\ensuremath{2}}$ facet, if it contains
a non-convex socket, due to the admissibility of $\scalebox{0.95}{\ensuremath{\Sigma}}$,
all other sockets with centers on the facet are convex. In this case,
let $\scalebox{0.95}{\ensuremath{\mathrm{BP}_{M\left(F\right)\left(L_{1}\right)}}}$
be the center of the non-convex socket (the other cases are similar).
See Figure \ref{fig:type2branch} for the puzzle-piece $\scalebox{0.95}{\ensuremath{\mathrm{PZ}_{M}}}$
in $\scalebox{0.95}{\ensuremath{\Psi}}$, where the third and fourth
are symmetric, while the first has only two nontrivial subdivisions
in Figure \ref{fig:type2branch-2}, up to symmetry. Then, Figures
\ref{fig:type2branch-3}, \ref{fig:type2branch-4}, \ref{fig:type2branch-5},
and \ref{fig:type2branch-6} describe what $\scalebox{0.95}{\ensuremath{\Psi}}$
looks like locally and hence globally, up to merging and symmetry.
\begin{figure}[H]
\noindent \centering{}\noindent \begin{center}
\def\size{0.5}
\def\rat{0.85}


\par\end{center}\caption{\label{fig:type2branch}A puzzle-piece $\mathrm{PZ}_{M}$ in $\Psi$.}
\end{figure}
\vspace{-0.25cm}

\begin{figure}[H]
\noindent \begin{centering}
\noindent \begin{center}
\def\size{0.5}
\def\rat{0.85}

%
\par\end{center}\vspace{-0.1cm}
\par\end{centering}
\noindent \centering{}\caption{\label{fig:type2branch-2}Subdivisions of a puzzle-piece in $\Psi$.}
\end{figure}
\vspace{-0.25cm}

\begin{figure}[H]
\noindent \begin{centering}
\noindent \begin{center}
\def\size{0.5}
\def\rat{0.85}

%
\par\end{center}\vspace{-0.1cm}
\par\end{centering}
\noindent \centering{}\caption{\label{fig:type2branch-3}An admissible puzzle $\Psi$, Part 1.}
\end{figure}
\vspace{-0.25cm}

\begin{figure}[H]
\noindent \centering{}\noindent \begin{center}
\def\size{0.5}
\def\rat{0.85}

%
\par\end{center}\caption{\label{fig:type2branch-4}An admissible puzzle $\Psi$, Part 2-1.}
\end{figure}

\begin{figure}[H]
\noindent \centering{}\noindent \begin{center}
\def\size{0.45}
\def\rat{0.85}

%
\par\end{center}\caption{\label{fig:type2branch-5}An admissible puzzle $\Psi$, Part 2-2.}
\end{figure}
\begin{figure}[H]
\noindent \begin{centering}
\noindent \begin{center}
\def\size{0.45}
\def\rat{0.85}

%
\par\end{center}\caption{\label{fig:type2branch-6}An admissible puzzle $\Psi$, Part 3.}
\end{figure}

Now, we will show how to fill the non-convex sockets. If $\scalebox{0.95}{\ensuremath{\Psi}}$
has no forked facet with more than one non-convex socket, we can apply
Algorithm \ref{alg:extend_tilings} as if it had no forked facet.
Therefore, suppose that $\scalebox{0.95}{\ensuremath{\Psi}}$ has
such a forked facet and assume that the facet has a maximum number
of non-convex sockets. This number is three, and there are precisely
eight cases for such a facet, up to symmetry, as shown in Figures
\ref{fig:forked-3-sockets-1} and \ref{fig:forked-3-sockets-2}. Figure
 \ref{fig:forked-3-sockets-3} illustrates the extension process for
Figure \ref{fig:forked-3-sockets-1}. We saturate the four facets
in Figure \ref{fig:forked-3-sockets-2} with the puzzle-pieces of
the four line arrangements in Figure \ref{fig:forked-3-sockets-4},
respectively, then all three non-convex sockets at each facet are
saturated at the same time. Combining this with Algorithm \ref{alg:extend_tilings},
we can extend $\scalebox{0.95}{\ensuremath{\Psi}}$ to an admissible
semipuzzle and to a complete puzzle. Thus, we conclude: 
\begin{thm}
\label{thm:complete-ext}Any admissible $\scalebox{0.95}{\ensuremath{\left(3,n\right)}}$-semitiling
with $\scalebox{0.95}{\ensuremath{n\le9}}$ can be extended to a complete
tiling.
\end{thm}

\begin{figure}[H]
\noindent \begin{centering}
\noindent \begin{center}
\def\size{0.48}
\def\rat{0.85}


\par\end{center}\vspace{-0.6cm}
\par\end{centering}
\noindent \centering{}\caption{\label{fig:forked-3-sockets-1}Forked facets with at least three non-convex
sockets, I.}
\end{figure}
\begin{figure}[H]
\noindent \begin{centering}
\noindent \begin{center}
\def\size{0.48}
\def\sizea{0.22}
\def\rat{0.85}

%
\par\end{center}\vspace{-0.4cm}
\par\end{centering}
\noindent \centering{}\caption{\label{fig:forked-3-sockets-2}Forked facets with at least three non-convex
sockets, II.}
\end{figure}
\vspace{-0.5cm}

\begin{figure}[H]
\noindent \begin{centering}
\noindent \begin{center}
\def\size{0.5}
\def\sizea{0.22}
\def\rat{0.85}

%
\par\end{center}\vspace{-0.3cm}
\par\end{centering}
\noindent \centering{}\caption{\label{fig:forked-3-sockets-3}Saturating forked facets, I.}
\end{figure}
\vspace{-0.2cm}

\noindent 
\begin{figure}[H]
\noindent \begin{centering}
\noindent \begin{center}
\def\sizea{0.25}

%
\par\end{center}\vspace{-0.5cm}
\par\end{centering}
\noindent \centering{}\caption{\label{fig:forked-3-sockets-4}Saturating forked facets, II.}
\end{figure}

We prove that the bound $\scalebox{0.95}{\ensuremath{n=9}}$ of Theorem
\ref{thm:complete-ext} is sharp.
\begin{thm}
\label{thm:Counterexample}A weighted admissible $\scalebox{0.95}{\ensuremath{\left(3,10\right)}}$-tiling
with no complete extension exists.
\end{thm}

\begin{proof}
Write $\scalebox{0.95}{\ensuremath{\left[10\right]=\left\{ 1,\dots,9,0\right\} }}$.
Consider the admissible puzzle $\scalebox{0.95}{\ensuremath{\Psi}}$
shown in Figure \ref{fig:counterex-puzzle} and let $\scalebox{0.95}{\ensuremath{\Sigma}}$
be its associated tiling. Observe that $\scalebox{0.95}{\ensuremath{\mathrm{BP}_{M_{0}(7890)}}}$
is a relevant type-$\scalebox{0.95}{\ensuremath{1}}$ facet of $\scalebox{0.95}{\ensuremath{\Sigma}}$
with three non-convex sockets. If $\scalebox{0.95}{\ensuremath{\Sigma}}$
extends to a complete tiling, this facet must be saturated. Suppose
that $\scalebox{0.95}{\ensuremath{\Sigma\cup\left\{ \mathrm{BP}_{M}\right\} }}$
is a tiling with the facet saturated, then $\scalebox{0.95}{\ensuremath{\mathrm{BP}_{M_{0}(7890)}=\mathrm{BP}_{M(123456)}}}$
for a rank-$\scalebox{0.95}{\ensuremath{2}}$ non-degenerate flat
$\scalebox{0.95}{\ensuremath{\left[6\right]=123456}}$ of $\scalebox{0.95}{\ensuremath{M}}$.
Moreover, $\scalebox{0.95}{\ensuremath{12}}$, $\scalebox{0.95}{\ensuremath{34}}$,
and $\scalebox{0.95}{\ensuremath{56}}$ are rank-$\scalebox{0.95}{\ensuremath{1}}$
flats of $\scalebox{0.95}{\ensuremath{M}}$, at most one of which
is degenerate by Lemma \ref{lem:degen-rank1-flat}. So, there are
exactly four candidates in Figure \ref{fig:the4candidates} for the
collection $\scalebox{0.95}{\ensuremath{\left\{ M/12,M/34,M/56,M/\left[6\right]\right\} }}$.
However, all of those four contain two distinct lines passing through
two distinct points, as shown in Figure \ref{fig:the4candidates-1}.
This is a contradiction, cf.~Example \ref{exa:2-Lines}. Hence, $\scalebox{0.95}{\ensuremath{\Sigma}}$
has no complete extension.

We now show $\scalebox{0.95}{\ensuremath{\Sigma}}$ is a tiling weighted
by a vector $\scalebox{0.95}{\ensuremath{\mathbf{w}=\left(1,1,1,1,1,1,\frac{1}{4},\frac{1}{4},\frac{1}{4},\frac{1}{4}\right)}}$.
For a point $\scalebox{0.95}{\ensuremath{\mathbf{v}=\left(\upsilon_{1},\dots,\upsilon_{9},\upsilon_{0}\right)\in\Delta_{10}^{3}}}$,
we write $\scalebox{0.95}{\ensuremath{\mathbf{v}_{A}=\mathbf{v}\left(A\right)}}$
and $\scalebox{0.95}{\ensuremath{x_{A}=x\left(A\right)}}$ for $\scalebox{0.95}{\ensuremath{A\subseteq\left[10\right]}}$.
Take $\scalebox{0.95}{\ensuremath{\mathbf{v}=\left(\frac{1}{3},\frac{1}{3},\frac{1}{3},\frac{1}{3},\frac{1}{3},\frac{1}{3},\frac{1}{4},\frac{1}{4},\frac{1}{4},\frac{1}{4}\right)\in{\rm BP}_{M_{0}}}}$.
We can decrease entries $\upsilon_{7},\upsilon_{8},\upsilon_{9},\upsilon_{0}$
by a sufficiently small number $\epsilon>0$ and increase $\upsilon_{1},\dots,\upsilon_{6}$
by $\frac{4\epsilon}{6}$ so that the new point is still contained
in both $\scalebox{0.95}{\ensuremath{{\rm BP}_{M_{0}}}}$ and $\scalebox{0.95}{\ensuremath{\mathrm{int}\left(\Delta_{\mathbf{w}}\right)}}$,
and hence, $\scalebox{0.95}{\ensuremath{{\rm BP}_{M_{0}}\cap\mathrm{int}\left(\Delta_{\mathbf{w}}\right)\neq\emptyset}}$.
Similarly, $\scalebox{0.95}{\ensuremath{\mathrm{BP}_{M_{i}}\cap\mathrm{int}\left(\Delta_{\mathbf{w}}\right)\neq\emptyset}}$
for $\scalebox{0.95}{\ensuremath{i=1,2,3}}$. Suppose $\scalebox{0.95}{\ensuremath{\Delta_{\mathbf{w}}-\left|\Sigma\right|\neq\emptyset}}$.
Then, there exists a point $\mathbf{v}\in\Delta_{\mathbf{w}}$ that
violates at least one of the describing inequalities of $\scalebox{0.95}{\ensuremath{{\rm BP}_{M_{i}}}}$
for each $\scalebox{0.95}{\ensuremath{i=0,1,2,3}}$, listed in Figure
\ref{fig:counterex-puzzle}. Since $\scalebox{0.95}{\ensuremath{\mathbf{v}_{7890}\le\mathbf{w}_{7890}=1}}$,
the point $\mathbf{v}$ must violate all three inequalities $\scalebox{0.95}{\ensuremath{x_{3456}\le1}}$,
$\scalebox{0.95}{\ensuremath{x_{1256}\le1}}$, and $\scalebox{0.95}{\ensuremath{x_{1234}\le1}}$
of $\scalebox{0.95}{\ensuremath{{\rm BP}_{M_{1}}}}$, $\scalebox{0.95}{\ensuremath{{\rm BP}_{M_{2}}}}$,
and $\scalebox{0.95}{\ensuremath{{\rm BP}_{M_{3}}}}$, respectively.
Also, $\mathbf{v}$ must violate at least one of the three inequalities
of $\scalebox{0.95}{\ensuremath{\mathrm{BP}_{M_{0}}}}$ apart from
$\scalebox{0.95}{\ensuremath{x_{7890}\le1}}$. But whichever is violated,
one reaches a contradiction:
\noindent \begin{center}

\par\end{center}\vspace{-0.5cm}
\par\end{centering}
\noindent \centering{}\caption{\label{fig:counterex-puzzle}A weighted admissible $(3,10)$-puzzle
$\Psi$.}
\end{figure}
\vspace{-0.5cm}

\noindent 
\begin{figure}[H]
\noindent \begin{centering}
\noindent \begin{center}
\def\sizea{0.26}

%
\par\end{center}\vspace{-0.4cm}
\par\end{centering}
\noindent \centering{}\caption{\label{fig:the4candidates}The four candidates for $\{M/12,M/34,M/56,M/[6]\}$.}
\end{figure}
\vspace{-0.5cm}

\noindent 
\begin{figure}[H]
\noindent \begin{centering}
\noindent \begin{center}
\def\sizea{0.26}

%
\par\end{center}\vspace{-0.4cm}
\par\end{centering}
\noindent \centering{}\caption{\label{fig:the4candidates-1}Non-matroidal arrangements.}
\end{figure}

\subsection{\label{subsec:Realizable tiling}Realizable complete extension of
admissible puzzle}

In this subsection and the next, we assume that $\Bbbk$ is an algebraically
closed field of arbitrary characteristic. We say that a puzzle $\scalebox{0.95}{\ensuremath{\Psi}}$
is \textbf{realizable} over $\Bbbk$ if there exists an algebraic
variety over $\Bbbk$ that can be identified as a realization of $\scalebox{0.95}{\ensuremath{\mathrm{GA}_{\Psi}}}$,
which we call a \textbf{realization} of $\scalebox{0.95}{\ensuremath{\Psi}}$.
Even if $\scalebox{0.95}{\ensuremath{\Psi}}$ can be extended to a
complete puzzle, it is not guaranteed that a realization of $\scalebox{0.95}{\ensuremath{\mathrm{GA}_{\Psi}}}$
can be extended to a realization of the complete puzzle. This is because
the lattice isomorphism $\psi_{ij}$ in \ref{enu:common-face} might
not be realizable, meaning that there might be no corresponding gluing
automorphism.

However, note that any two realizations (over $\Bbbk$) of a point
arrangement $\scalebox{0.95}{\ensuremath{\mathrm{HA}_{U_{3}^{2}}}}$
are isomorphic. This property also holds for $\scalebox{0.95}{\ensuremath{\mathrm{HA}_{N}}}$,
where $\scalebox{0.95}{\ensuremath{N}}$ is any rank-$\scalebox{0.95}{\ensuremath{2}}$
matroid with precisely three rank-$\scalebox{0.95}{\ensuremath{1}}$
flats. So, if $\scalebox{0.95}{\ensuremath{\Phi=\left\{ \mathrm{PZ}_{M_{1}},\mathrm{PZ}_{M_{2}}\right\} }}$
is a rank-$\scalebox{0.95}{\ensuremath{3}}$ puzzle with $\scalebox{0.95}{\ensuremath{M_{1}\cap M_{2}\simeq N}}$,
then any realizations of $\scalebox{0.95}{\ensuremath{\mathrm{PZ}_{M_{1}}}}$
and $\scalebox{0.95}{\ensuremath{\mathrm{PZ}_{M_{2}}}}$ can be glued
into a realization of $\scalebox{0.95}{\ensuremath{\Phi}}$. Moreover,
the fillers we construct in \ref{enu:(LA1)} through \ref{enu:(LA6)},
which are all realizable over $\Bbbk$ because $\Bbbk$ is an infinite
field, will not have forked facets. Hence, if $\scalebox{0.95}{\ensuremath{\Psi}}$
has no forked facets, then according to Algorithm \ref{alg:extend_tilings},
any realization of it can be extended to a realization of a complete
extension of it. Further:
\begin{thm}
\label{thm:realizable-complete}For any admissible $\scalebox{0.95}{\ensuremath{\left(3,n\right)}}$-semipuzzle
with $\scalebox{0.95}{\ensuremath{n\le9}}$, any realization of it
can be extended to a realization of a complete puzzle. The bound $\scalebox{0.95}{\ensuremath{n=9}}$
is sharp.
\end{thm}

\begin{proof}
Let $\scalebox{0.95}{\ensuremath{\Psi}}$ be an admissible $\scalebox{0.95}{\ensuremath{\left(3,n\right)}}$-semipuzzle
with $\scalebox{0.95}{\ensuremath{n\le9}}$. For $\scalebox{0.95}{\ensuremath{n\le7}}$,
$\scalebox{0.95}{\ensuremath{\Psi}}$ has no forked facet. For $\scalebox{0.95}{\ensuremath{n=8}}$,
$\scalebox{0.95}{\ensuremath{\Psi}}$ has at most one forked facet
with a maximum number of non-convex sockets in Figure \ref{fig:forked-facet-38},
up to symmetry, which can be saturated as illustrated. Since every
relevant line-piece of the filler has precisely three point-pieces,
any realization of $\scalebox{0.95}{\ensuremath{\Psi}}$ can be extended
to a realization of a complete puzzle.

For $\scalebox{0.95}{\ensuremath{n=9}}$, we already know all possible
forked facets with a maximum number of non-convex sockets, up to symmetry:
Figures \ref{fig:forked-3-sockets-1} and \ref{fig:forked-3-sockets-2},
while Figures \ref{fig:forked-3-sockets-3} and \ref{fig:forked-3-sockets-4}
describe the fillers that fit in. The fillers in Figure \ref{fig:forked-3-sockets-3}
have no relevant facets with more than three point-pieces, causing
no issues. However, the fillers in Figure \ref{fig:forked-3-sockets-4}
do have such facets. We explain the realizable complete extension
for the first case in Figure \ref{fig:forked-3-sockets-2}, and the
remaining cases can be handled similarly.

Assume that the first case in Figure \ref{fig:forked-3-sockets-2}
occurs. Let $\scalebox{0.95}{\ensuremath{\gamma_{1}}}$, $\scalebox{0.95}{\ensuremath{\gamma_{2}}}$,
and $\scalebox{0.95}{\ensuremath{\gamma_{3}}}$ be realizations in
$\scalebox{0.95}{\ensuremath{\mathbb{P}^{2}}}$ of the type-$\scalebox{0.95}{\ensuremath{2}}$
facets of the three non-convex sockets, respectively, as shown in
Figure \ref{fig:forked-socket-39} (in which the puzzle figure is
presented up to merging). If none of them are isomorphic, there is
a realization of the filler in Figure \ref{fig:forked-realizable-1}
for which the realizations of the three type-$\scalebox{0.95}{1}$
relevant facets are isomorphic to them, respectively. If two of them
are isomorphic, say $\scalebox{0.95}{\ensuremath{\gamma_{1}}}$ and
$\scalebox{0.95}{\ensuremath{\gamma_{2}}}$, then there are fillers
with realizations that fit into the given realization of $\scalebox{0.95}{\ensuremath{\mathrm{GA}_{\Psi}}}$,
as shown in Figure \ref{fig:forked-realizable-2}. In either case,
the realizations glue to form a realization of a larger admissible
semipuzzle, and no relevant facets with more than three point-pieces
are produced. In this way, using Algorithm \ref{alg:extend_tilings},
we can extend any realization of $\scalebox{0.95}{\ensuremath{\Psi}}$
to a realization of a complete puzzle. Because Theorem \ref{thm:Counterexample}
provides a counterexample when $\scalebox{0.95}{\ensuremath{n=10}}$,
the bound $\scalebox{0.95}{\ensuremath{n=9}}$ is tight.
\end{proof}
\vspace{-0.3cm}

\begin{figure}[H]
\noindent \begin{centering}
\noindent \begin{center}
\def\size{0.5}
\def\sizea{0.22}
\def\rat{0.85}


\par\end{center}\vspace{-0.4cm}
\par\end{centering}
\noindent \centering{}\caption{\label{fig:forked-facet-38}Saturating a forked facet of an admissible
$(3,8)$-puzzle.}
\end{figure}

\begin{figure}[H]
\noindent \begin{centering}
\noindent \begin{center}
\def\size{0.5}
\def\sizea{0.22}
\def\rat{0.85}

%
\par\end{center}\vspace{-0.1cm}
\par\end{centering}
\noindent \centering{}\caption{\label{fig:forked-socket-39}A forked facet with three non-convex
sockets of an admissible $(3,9)$-puzzle.}
\end{figure}
\begin{figure}[H]
\noindent \begin{centering}
\noindent \begin{center}
\def\size{0.5}
\def\sizea{0.22}
\def\rat{0.85}

%
\par\end{center}\vspace{-0.3cm}
\par\end{centering}
\noindent \centering{}\caption{\label{fig:forked-realizable-1}Realizable extension of an admissible
$(3,9)$-puzzle, I.}
\end{figure}
\begin{figure}[H]
\noindent \begin{centering}
\noindent \begin{center}
\def\size{0.5}
\def\sizea{0.22}
\def\rat{0.85}

%
\par\end{center}\vspace{-0.5cm}
\par\end{centering}
\noindent \centering{}\caption{\label{fig:forked-realizable-2}Realizable extension of an admissible
$(3,9)$-puzzle, II.}
\end{figure}

\subsection{\label{subsec:Reduction-morphism}Surjectivity of reduction morphism}

Given a weight vector $\scalebox{0.95}{\ensuremath{\mathbf{w}\in\left[0,1\right]^{n}}}$,
we refer to a realization of an abstract $\mathbf{w}$-stable $\scalebox{0.95}{\ensuremath{\left(k,n\right)}}$-hyperplane
arrangement as a \textbf{$\mathbf{w}$-stable $\scalebox{0.95}{\ensuremath{\left(k,n\right)}}$-hyperplane
arrangement} (without the term ``abstract''). The moduli space of
these arrangements is denoted by $\scalebox{0.95}{\ensuremath{\overline{\mathrm{M}}_{\mathbf{w}}(k,n)}}$.
For any two weight vectors $\mathbf{w}$ and $\mathbf{v}$ with $\scalebox{0.95}{\ensuremath{\mathbf{w}>\mathbf{v}}}$,
there exists a natural morphism $\scalebox{0.95}{\ensuremath{\rho_{\mathbf{w},\mathbf{v}}:\overline{\mathrm{M}}_{\mathbf{w}}(k,n)\rightarrow\overline{\mathrm{M}}_{\mathbf{v}}(k,n)}}$,
called the \emph{reduction morphism}, \cite{Ale08}. We fix $\scalebox{0.95}{\ensuremath{k=3}}$
and consider weighted stable line arrangements. We will discuss the
surjectivity of the reduction morphism below.
\begin{thm}
\label{thm:Surjectivity}For $\scalebox{0.95}{\ensuremath{n\le9}}$
and any two weight vectors $\scalebox{0.95}{\ensuremath{\mathbf{w}}}$
and $\scalebox{0.95}{\ensuremath{\mathbf{v}}}$ in $\scalebox{0.95}{\ensuremath{\left[0,1\right]^{n}}}$
with $\scalebox{0.95}{\ensuremath{\mathbf{w}>\mathbf{v}}}$, the reduction
morphism $\scalebox{0.95}{\ensuremath{\rho_{\mathbf{w},\mathbf{v}}:\overline{\mathrm{M}}_{\mathbf{w}}(3,n)\rightarrow\overline{\mathrm{M}}_{\mathbf{v}}(3,n)}}$
between the moduli spaces of weighted stable $n$-line arrangements
is surjective. The bound $\scalebox{0.95}{\ensuremath{n=9}}$ is sharp.
\end{thm}

\begin{proof}
Any geometric fiber $\scalebox{0.95}{\ensuremath{X}}$ over a geometric
point $\scalebox{0.95}{\ensuremath{\left[X\right]}}$ of $\scalebox{0.95}{\ensuremath{\overline{\mathrm{M}}_{\mathbf{v}}(3,n)}}$
is obtained from a $\mathbf{v}$-stable line arrangement that is a
realization of a weighted puzzle, by contracting curves and/or blowing
up points in each component. Here, at most one blowup is needed due
to Lemma \ref{lem:degen-rank1-flat}. Lemma \ref{lem:wt<adm} implies
the existence of an extension $\scalebox{0.95}{\ensuremath{X'}}$
of $\scalebox{0.95}{\ensuremath{X}}$, trivial or not, that is a realization
of an admissible semipuzzle, i.e., $\scalebox{0.95}{\ensuremath{X'}}$
can be obtained from $\scalebox{0.95}{\ensuremath{X}}$ by blowing
up points and/or contracting curves followed by gluing with an appropriate
variety. Note that if $\scalebox{0.95}{\ensuremath{X'\neq X}}$, then
$\scalebox{0.95}{\ensuremath{X'}}$ is not a $\mathbf{v}$-stable
line arrangement. Theorem \ref{thm:realizable-complete} ensures that
there exists a $\scalebox{0.95}{\ensuremath{\mathbbm1}}$-stable line
arrangement $\scalebox{0.95}{\ensuremath{Z}}$ that extends $\scalebox{0.95}{\ensuremath{X'}}$,
with $\scalebox{0.95}{\ensuremath{\left[X\right]=\rho_{\mathbbm{1},\mathbf{v}}\left(\left[Z\right]\right)}}$.
Let $\scalebox{0.95}{\ensuremath{Y}}$ be a $\mathbf{w}$-stable line
arrangement with $\scalebox{0.95}{\ensuremath{\left[Y\right]=\rho_{\mathbbm{1},\mathbf{w}}\left(\left[Z\right]\right)}}$.
Now, the following commutative diagram holds:\vspace{-0.5cm}

\noindent \noindent \begin{center}
\[
\xymatrix{ & \overline{\mathrm{M}}_{\mathbbm{1}}(3,n)\ar[dl]_{\rho_{\mathbbm{1},\mathbf{w}}}\ar[dr]^{\rho_{\mathbbm{1},\mathbf{v}}}\\
\overline{\mathrm{M}}_{\mathbf{w}}(3,n)\ar[rr]^{\rho_{\mathbf{w},\mathbf{v}}} &  & \overline{\mathrm{M}}_{\mathbf{v}}(3,n)
}
\]
\par\end{center}

\noindent This tells us that: 
\[
\scalebox{0.95}{\ensuremath{\rho_{\mathbf{w},\mathbf{v}}\left(\left[Y\right]\right)=\rho_{\mathbf{w},\mathbf{v}}\left(\rho_{\mathbbm{1},\mathbf{w}}\left(\left[Z\right]\right)\right)=\left(\rho_{\mathbf{w},\mathbf{v}}\circ\rho_{\mathbbm{1},\mathbf{w}}\right)\left(\left[Z\right]\right)=\rho_{\mathbbm{1},\mathbf{v}}\left(\left[Z\right]\right)=\left[X\right]}}.
\]
 Thus, we conclude that the reduction morphism $\scalebox{0.95}{\ensuremath{\rho_{\mathbf{w},\mathbf{v}}}}$
is surjective. Theorem \ref{thm:Counterexample} provides a counterexample,
and it is also shown that the bound $\scalebox{0.95}{\ensuremath{n=9}}$
is sharp.
\end{proof}

\end{document}